\documentclass[reqno]{amsart}

\usepackage{bbm, amsfonts, amsthm, amssymb}
\usepackage{amsmath}
\usepackage{dsfont}
\usepackage{bm}
\usepackage[margin=3cm]{geometry}
\usepackage{xcolor}
\usepackage{graphicx,tabularx}
\usepackage{mathrsfs}
\usepackage{enumerate}
\usepackage{tikz-cd}
\usepackage[euler]{textgreek}

\usepackage{mathtools}
\graphicspath{{Figures/}}

\usepackage[textsize=tiny]{todonotes}
\newcommand{\EE}{\mathbb{E}}

\newcommand{\NN}{\mathbb{N}}
\newcommand{\PP}{\mathbb{P}}
\newcommand{\QQ}{\mathbb{Q}}
\newcommand{\RR}{\mathbb{R}}

\newcommand{\D}{\mathrm{d}}
\newcommand{\ds}{\mathrm{d}s}
\newcommand{\dr}{\mathrm{d}r}
\newcommand{\dt}{\mathrm{d}t}
\newcommand{\du}{\mathrm{d}u}

\newcommand{\dx}{\mathrm{d}x}
\newcommand{\dW}{\mathrm{d}W}

\newcommand{\E}{\mathrm{e}}

\newcommand{\Bb}{\mathcal{B}}
\newcommand{\Cc}{\mathcal{C}}
\newcommand{\Dd}{\mathcal{D}}
\newcommand{\Ee}{\mathcal{E}}
\newcommand{\Ff}{\mathcal{F}}
\newcommand{\Hh}{\mathcal{H}}
\newcommand{\Ii}{\mathcal{I}}

\newcommand{\Vv}{\mathcal{V}}
\newcommand{\Ww}{\mathcal{W}}
\newcommand{\Xx}{\mathcal{X}}
\newcommand{\ep}{\varepsilon}

\newcommand{\vspan}{\textnormal{span}}

\newcommand{\half}{\frac{1}{2}}
\newcommand{\one}{\mathbbm{1}}

\newcommand{\abs}[1]{\left\lvert#1\right\rvert}
\newcommand{\norm}[1]{\left\lVert#1\right\rVert}

\newcommand{\notthis}[1]{}
\newtheorem{theorem}{Theorem}[section]
\newtheorem*{theorem1}{Theorem A.}
\newtheorem*{theorem2}{Theorem B.}
\newtheorem*{theorem3}{Theorem C.}
\newtheorem{corollary}[theorem]{Corollary}
\newtheorem{lemma}[theorem]{Lemma}
\newtheorem{proposition}[theorem]{Proposition}

\theoremstyle{definition}
\newtheorem{definition}[theorem]{Definition}
\newtheorem{remark}[theorem]{Remark}
\newtheorem{assumption}[theorem]{Assumption}
\theoremstyle{plain}
\newtheorem{example}[theorem]{Example}
\numberwithin{equation}{section}

\definecolor{ocean}{rgb}{0,0.1,0.6}

\definecolor{imperialGreen}{RGB}{2,137,59}
\definecolor{imperialBlue}{RGB}{0, 62, 116}
\definecolor{imperialBrick}{RGB}{165,25,0}
\definecolor{imperialProcess}{RGB}{0,133,202}

\usepackage{hyperref}

\title{Kolmogorov equations for stochastic Volterra processes with singular kernels}

\author{Ioannis Gasteratos}
\address{Technische Universit\"at Berlin, Institute of Mathematics}
\email{i.gasteratos@tu-berlin.de}

\author{Alexandre Pannier} 
\address{Université Paris Cité, Laboratoire de Probabilités Statistique et Modélisation}
\email{pannier@lpsm.paris}
 
\thanks{\emph{Email addresses}: {\tt i.gasteratos@tu-berlin.de, pannier@lpsm.paris}}
 \subjclass[2020]{60G22, 60H15, 60H20, 45D05, 35K10, 91G20}
\keywords{Stochastic Volterra equations, Volterra processes, stochastic partial differential equations, Kolmogorov equations, fractional Brownian motion, rough volatility, Markovian lift}

\begin{document}

\begin{abstract} 
We associate backward and forward Kolmogorov equations to a class of fully nonlinear Stochastic Volterra Equations (SVEs) with convolution kernels $K$ that are singular at the origin. Working on a carefully chosen Hilbert space $\Hh_1$, we rigorously establish a link between solutions of SVEs and Markovian mild solutions of a Stochastic Partial Differential Equation (SPDE) of transport-type. Then, we obtain two novel Itô formulae for functionals of mild solutions and, as a byproduct, show that their laws solve corresponding Fokker--Planck equations. Finally, we introduce a natural notion of ``singular" directional derivatives along $K$ and prove that (conditional) expectations of SVE solutions can be expressed in terms of the unique solution to a backward Kolmogorov equation on~$\Hh_1$. Our analysis relies on stochastic calculus in Hilbert spaces, the reproducing kernel property of the state space $\Hh_1,$ as well as crucial invariance and smoothing properties that are specific to the SPDEs of interest. In the special case of singular power-law kernels, our conditions guarantee well-posedness of the backward equation either for all values of the Hurst parameter $H,$ when the noise is additive, or for all $H>1/4$ when the noise is multiplicative.
\end{abstract}

\maketitle

\tableofcontents

\section{Introduction}\label{sec:Intro}

We leverage stochastic analysis on Hilbert spaces to
extend the deep and fruitful connection between It\^o diffusions and Partial Differential Equations to the context of stochastic Volterra processes. The latter are neither Markovian nor semimartingales and are given as solutions to Stochastic Volterra Equations (SVEs) of the form
\begin{equation}\label{eq:SVE}
    X_t = X_0(t) + \int_0^t K(t-s) b(X_s) \ds + \int_0^t K(t-s) \sigma(X_s) \, \D W_s,\quad t\ge0,
\end{equation}
where $W$ is a standard $m-$dimensional Wiener process with respect to a filtration~$(\Ff_t)_{t\ge0}$, $b:\RR^d\rightarrow\RR^d, \sigma:\RR^d\rightarrow\RR^{d\times m}$ are measurable maps, $K:\RR^+\rightarrow\RR^{d\times d}$ is a locally square integrable kernel and $X_0:\RR^+\rightarrow\RR^d$ is an $\Ff_0$-measurable initial curve.

Via the choice of~$K,b,\sigma$ and $X_0$, \eqref{eq:SVE} includes many well known examples such as Mandelbrot--van Ness (type I) and Riemann--Liouville (type II) fractional Brownian motions (fBms), rough volatility models~\cite{bayer2023rough},  SDEs driven by additive fBm ~\cite{hairer2005ergodicity}, volatility modulated Volterra processes \cite{barndorff2008time}, affine and polynomial Volterra processes~\cite{abi2019affine,cuchiero2020generalized,abi2024polynomial}, fractional diffusions \cite{henry2010introduction,lototsky2020classical} (we refer to Section~\ref{subsec:examples} for the specifications), while ambit fields describe an extension with spatial interactions~\cite{podolskij2015ambit,barndorff2018ambit}. Due to their ubiquity, Volterra processes are found in a wide array of applications including volatility modeling~\cite{bayer2023rough}, stochastic game theory~\cite{abi2023equilibrium},  electricity price modeling \cite{barndorff2013modelling, bennedsen2022rough}, the study of turbulent flow velocities \cite{barndorff2008stochastic, chevillard2017regularized}, physical systems subject to non-Markovian forces~\cite{henry2010introduction,chen2015fractional}, and climate science \cite{eichinger2020sample}. In turn, their importance in applications has sparked interest in more theoretical aspects such as well-posedness in Banach spaces and connections with superprocesses \cite{zhang2010stochastic, mytnik2015uniqueness}, (functional) limit theorems and homogenization \cite{bai2015convergence, darses2010limit, gehringer2022functional}, stochastic integration \cite{decreusefond2005stochastic}, extensions of rough paths theory \cite{harang2021volterra}, functional inequalities \cite{gasteratos2023transportation} and generalisations to anticipative coefficients \cite{pardoux1990stochastic, alos1997anticipating} to name only a few.

Applications to rough volatility, which prompted considerable research in recent years, including the present work, require the kernel to be singular at the origin i.e. $\lim_{t\downarrow0} \abs{K(t)}=+\infty$. In terms of the popular power-law kernel~$K(t)=t^{H-\half}$, which will be our guiding example in this paper, this corresponds to the ``rough" regime~$H\in(0,\half)$. After closing the door of Markovianity behind us, we are thus thrown out of the realm of semimartingales as well.

\subsection{Goals and obstacles}\label{subsec:Goals}
In this new territory, located beyond the scope of many established methods, we aim to provide a dynamical description of conditional expectations~$(\EE[\phi(X_T)\lvert\Ff_t])_{t\in[0,T]}$ which, in rough volatility modelling, can be thought of as option prices with payoff~$\phi$ and maturity~$T$. To this end, we develop forward and backward Kolmogorov equations, along with their well-posedness theory, and characterise respectively the conditional laws and conditional expectations of Volterra processes. These equations and the Itô formula leading to them lie at the core of the theory of Markov processes and semimartingales, respectively.

We 
reclaim Markovianity with the well-known idea of enlarging the state space of a non-Markovian process to include (relevant information about) its past; we call this a \emph{lifting} procedure. Since the works of Carmona--Coutin~\cite{carmona1998fractional} and Hairer~\cite{hairer2005ergodicity} this approach encountered many successes and allowed to build upon results and intuition from Markov theory.  
The subsequent literature bears witness to the diversity of available lifts and state spaces for each of them; we refer to Section~\ref{subsec:literature} for an overview.
It also suggests that none of the proposed solutions subsumed their predecessors. Notably, Viens and Zhang \cite{viens2019martingale} derive a functional Itô formula for Volterra processes and exhibit the associated Kolmogorov backward PDE. Despite its significance, well-posedness of the latter is only obtained in the case of regular kernels~\cite{wang2022path,barrasso2022gateaux} ($H\ge1/2$) or Gaussian processes~\cite{viens2019martingale,bonesini2023rough,pannier2023path} ($b=0,\sigma=1$ in~\eqref{eq:SVE}). For completeness we also mention the early work~\cite{vyoral2005kolmogorov} for Gaussian processes with regular kernels.

Generally speaking, a Markov process~$(\lambda(t))_{t\ge0},$ evolving on a Polish space $\Xx,$ is called a \emph{lift} of~$(X_t)_{t\ge0}$
if there exists a (continuous) projection operator~$\pi:\Xx\rightarrow\RR^d$ that recovers the Volterra process~$\pi(\lambda(t))=X_t$. A natural idea, put forth in \cite{abi2019markovian,cuchiero2020generalized}, consists in introducing an additional shift variable~$x\ge0$ and defining
\begin{align*}
    \lambda(t,x):= \lambda_0(t+x) + \int_0^t K(t-s+x) b(X_s)\ds + \int_0^t K(t-s+x)\sigma(X_s)\D W_s. 
\end{align*}
It is clear that we should have~$X_t=\lambda(t,0)$ and, since~$X_{t+x}=\lambda(t,x)+X^{\perp}(t,x)$ where $X^{\perp}(t,x)$ is independent from~$\Ff_t$, we also expect that~$(\lambda(t,x))_{x\in[0,T-t]}$ is a ``path-dependent" state variable which adequately captures the history of $X.$ 
This parametrisation is reminiscent of the Heath--Jarrow--Morton--Musiela (HJMM)~\cite{heath1992bond,musiela1993stochastic} framework for forward rates in mathematical finance. Besides, $\lambda$ is linked (at least when $b$ is linear) to the so-called forward variance curves~$(\EE[X_{t+x}\lvert\Ff_t])_{x\ge0}.$ For these reasons we call $\lambda$ the \emph{curve lift} of $X.$

In order to study $\lambda$ independently of $X$ we shall write an equation for its dynamics. Informally we observe that $\lambda$ ought to satisfy the stochastic evolution equation
\begin{align}\label{eq:SPDE_mild_intro}
    \lambda(t) = S(t)\lambda_0 + \int_0^t S(t-s)K b(ev_0(\lambda(s)))\ds + \int_0^t S(t-s)K\sigma(ev_0(\lambda(s)))\D W_s,
\end{align}
where~$f\mapsto S(t)f(\cdot):=f(t+\cdot)$ is the shift semigroup, we have suppressed the $x$--dependence and the point evaluation map~$ev_0$ at $x=0$ plays the role of the projection~$\pi$ introduced above. The lift property is then satisfied as~$ev_0(\lambda(t))=X_t.$ In turn, our attention may shift from $X$ to the Markovian (yet infinite dimensional) $\lambda$ since properties of the former can be obtained from the latter after evaluating at~$x=0.$ Specifically, $\lambda$ embeds all the relevant path-dependent information of $X$ since, for all sufficiently regular test functions $\phi,$ there exists  some measurable function $u$ that satisfies
\begin{equation}\label{eq:IntroSVEConditionalExpectation}
    \EE[\phi(X_T)\lvert\Ff_t]= u(t,\lambda(t)).
\end{equation}
Up to this point, none of the above is rigorous and we have conveniently avoided to specify the state space $\Xx$ of our lift; this is however a non-trivial affair. To make matters worse, the singularity of the kernel leads to a cascade of three puzzles, and hence the three objectives of this work.
\begin{enumerate} \setlength\itemsep{5pt}
    \item[\textbf{A)}] On the one hand, the lift property is satisfied if the evaluation functionals~$ev_x:\Xx\rightarrow\RR^d, x\in\RR^+$ are continuous. On the other hand, the  singular kernel~$K$ cannot be an element of~$\Xx$ since $ev_0(K)=\infty.$\\
    $\boldsymbol{\Longrightarrow}$ Find a Hilbert space with continuous point evaluations (a Reproducing Kernel Hilbert Space, or RKHS) where~\eqref{eq:SPDE_mild_intro} is well-defined.
    \item[\textbf{B)}] The integral equation \eqref{eq:SPDE_mild_intro} is a mild solution of the first order (or transport-type) Stochastic Partial Differential Equation (SPDE)
\begin{equation}\label{eq:strong_SPDE_intro}
    \left\{\begin{aligned}
        &\partial_t\lambda(t,x)=\partial_x \lambda(t,x)+Kb(\lambda(t,0))+K\sigma(\lambda(t,0))\dot{W}_t,\;\;t>0,\; x\in\RR^+\\&
        \lambda(0,x)=\lambda_0(x),\;\; x\in\RR^+.
\end{aligned}\right.
\end{equation}
However, this SPDE has singular coefficients ($K$ is only locally square integrable) and does not admit any useful \footnote[1]{To be a bit more precise, \eqref{eq:strong_SPDE_intro} cannot admit $\Xx-$valued solutions that are simultaneously 
\textit{analytically strong} and \emph{lifts} of $X$ in the sense described above. } solution in semimartingale form.
\\
$\boldsymbol{\Longrightarrow}$ Establish an Itô formula for the mild solution~\eqref{eq:SPDE_mild_intro} that is conducive to deriving Kolmogorov PDEs. 
    \item[\textbf{C)}] In order to derive a Kolmogorov backward PDE, one needs to establish smoothness of the value function~$ y\mapsto u(t,y):=\EE[\varphi(\lambda^{t,y}(T))\lvert \Ff_t],$ where $\varphi\in\Cc^2_b$ and $\lambda$ is the solution of \eqref{eq:SPDE_mild_intro} starting from time $t$ at $y.$ Moreover, directional derivatives of $u$ along the ``singular" directions~$K$ (which, as we mentioned, cannot live in the same space as $y$)  must be introduced and well-defined, thereby adding another level of complexity.      
   \\
    $\boldsymbol{\Longrightarrow}$ Introduce an appropriate notion of regularity and prove that $u$ is in a corresponding class of $\Cc^{1,2}$ functions.
\end{enumerate}

\subsection{Contributions}\label{sec:contributions}
To the best of our knowledge, the present work is the first to address all three of these objectives in the setting of fully non-linear SVEs with singular kernels. Our main results are summarised below (complete versions of the theorems can be found in subsequent sections):

\bigskip
\textbf{A) The state space.} We identify a weighted Sobolev space~$\Xx\equiv\Hh_1:=H^1_w(\RR_+;\RR^d)$ as a natural state space for $\lambda$. $\Hh_1$ is a separable Hilbert space and the shift semigroup~$S$ has the particular property of regularising the kernel since~$K(t+\cdot)$ is continous for any~$t>0.$ Working on this space we obtain mild well-posedness of~\eqref{eq:strong_SPDE_intro}, when $b,\sigma$ are Lipschitz continuous and grow at most linearly, even though the coefficients do not belong to~$H^1_w:$
\begin{theorem1}[See Theorem~\ref{thm:SPDE_wellposedness}]
    The SPDE~\eqref{eq:strong_SPDE_intro} has a unique $\Hh_1-$valued mild solution. 
\end{theorem1}
The weights $w$ that we use, and will later be called admissible (Definition \ref{assumption:admissibleweights}), are chosen to compensate for singularities of $K$ both at  the origin and at infinity. In addition, they must satisfy certain technical regularity conditions ensuring that $K$ and its derivatives are square integrable when shifted away from the origin. 
In the special case of the power-law kernel~$K(t)=t^{H-\half},H\in(0,1/2)$, we prove that~$w(x)=x^\beta\E^{-x},$ for~$\beta\in(1-2H,1),$ provides just the right amount of regularity around zero (for more details we refer the reader to Example \ref{example:powerlaw} below). 

Riding on this success, the flow, Markov, Feller and generalised Feller properties quickly follow. We further obtain, even without semimartingality, a mild Itô formula for~$\lambda$ which relies on recent work of Da Prato, Jentzen and R\"ockner  \cite{da2019mild}. The latter gives rise to an Itô--Volterra formula for $X$ itself, as well as a Fokker--Planck (or forward Kolmogorov) equation, see Sections \ref{subsubsec:ItoSVE},\ref{sec:FPE}.
\bigskip

\textbf{B) A singular Itô formula.}
Unfortunately, the mild Itô formula does not lead to the derivation of a backward PDE and, as alluded to earlier, $\Hh_1-$valued solutions of~\eqref{eq:strong_SPDE_intro} can only be defined in a mild, non-semimartingale form (recall that~$K, \partial_x\lambda(t) \notin \Hh_1$). Once again, we leverage the crucial smoothing property of shifts acting on $K,$ this time to identify an \emph{invariant subspace}~$\mathscr{K}_1\subset \Hh_1$ for the paths of~$\lambda$. The space $\mathscr{K}_1$ \eqref{eq:yspace} is defined as a subspace of functions $y$ such that $y, \partial_xy$ are ``smoothened" when shifted by any positive $t>0.$ Indeed it turns out that, if, for all $\delta>0,$~$S(\delta)\partial_x\lambda(t)\in \Hh_1$ at~$t=0$ then the same holds for all~$t>0.$ Restricting initial conditions to $\mathscr{K}_1,$ we are able to obtain a ``semimartingale-type" It\^o formula for certain functionals of mild solutions $\lambda$ \eqref{eq:SPDE_mild_intro} which reads as follows:
\begin{theorem2}[See Theorem~\ref{thm:SingularIto}]
    Let $F:[0,T]\times\mathscr{K}_1\to\RR$ be sufficiently regular, then for all $t\in[0,T]$ and $\lambda_0\in\mathscr{K}_1$ we have $\PP-$almost surely
\begin{equation*}
\begin{aligned}
    F(t,\lambda(t))&=F(0, \lambda_0)+\int_0^{t}\bigg[ \partial_t F(s, \lambda(s))+\mathcal{D}F(s,\lambda(s))\Big( \partial_x\lambda(s)+K b(ev_0(\lambda(s)))\Big)\bigg]\ds
    \\&\quad +\int_0^t\frac{1}{2}\textnormal{Tr}\Big[ \mathcal{D}^2F(s,\lambda(s))K \sigma(ev_0(\lambda(s)))(K \sigma(ev_0(\lambda(s)))^*\Big]\ds\\&
    \quad     +\int_0^t \mathcal{D}F(s,\lambda(s))\big(K \sigma(ev_0(\lambda(s)))\dW_s\big).
\end{aligned}
   \end{equation*}
\end{theorem2}
Again, an important observation here is that differentials $\Dd F(y)(K),\Dd^2 F(y)(K,K)$ must be taken along directions $K$ that are not elements of $\Hh_1$ or $\mathscr{K}_1.$ For this reason, we define the latter as limits of classical Gateaux derivatives~$D F(y)(S(\delta)K),D^2 F(y)(S(\delta)K,S(\delta)K),$ as $\delta\to 0,$ and name them \emph{singular directional derivatives}. This notion is both natural and necessary to obtain this ``singular It\^o formula"; a precise definition of singular derivatives, along a class of admissible directions $\mathscr{K},$ and several examples are provided in Section \ref{sec:SingularIto}.  Indeed our proof strategy relies on the application of a standard Hilbert space Itô formula to the semimartingale~$(S(\delta)\lambda(t))_{t\ge0}$ and then a limiting argument as~$\delta\to 0.$ Finally, Theorem B. leads directly to a singular (non-mild) Fokker--Planck equation~\eqref{eq:singular_FPE} and plays a major role in our subsequent analysis of the backward PDE. 

\bigskip
\textbf{C) The singular Kolmogorov backward PDE.} 
In virtue of the singular Itô formula, we obtain a Feynman--Kac formula and thus unique probabilistic representations of solutions to a novel singular Kolmogorov backward PDE on $\Hh_1.$ The challenging part consists in proving that $u$ is in fact twice ``singularly" differentiable and that $u,\mathcal{D}u, \mathcal{D}^2u$ are sufficiently regular. This requires a fine analysis of the derivatives of~$\lambda,$ with respect to initial data, and their convergence properties as~$S(\delta)K$ tends to $K$. A complete study of the resulting variation equations and corresponding tangent processes cannot be found elsewhere in the literature. On a slightly technical note, our work here is facilitated once again by the special structure of~\eqref{eq:strong_SPDE_intro} since the nonlinear coefficients $b, \sigma$ only depend on finite-dimensional projections of the solution. Thus, in sharp contrast to SPDEs with Nemytskii (or composition) operators, the nonlinearities here are in fact Fr\'echet differentiable; see also Remark \ref{rem:NemytskiiRem} below for a more detailed discussion. More importantly, unless $\sigma$ is constant, our proof highlights that~$K$ must be in $L^q([0,T];\RR^{d\times d})$ with~$q>4.$ This threshold seems to be structural and not an artefact of our method, as we explain in Remark~\ref{rem:Honefourth}. 
For the power-law kernel~$K(t)=t^{H-\half}$, this corresponds to~$H>1/4.$\footnote[2]{It is worth noting that this ``1/4" condition has appeared in the completely different context of geometric rough path constructions above fBm; see e.g. the classical work \cite{coutin2002stochastic}.} 
\begin{theorem3}[See Theorem \ref{thm:backward_equation_singular}, Corollary \ref{cor:SVEConditionalExp}]\label{th:BackwardPDE_intro}
If $\varphi\in \Cc^2_b(\Hh_1)$, and $K\in L^q([0,T])$ with~$q>4$ or $\sigma$ is constant, then $(t,y)\mapsto u(t,y)=\EE[\varphi(\lambda^{t,y}(T))\lvert \Ff_t]$ is the unique pointwise solution in a subclass of $\Cc^{1,2}_b([0,T]\times\Hh_1)$ to
\begin{equation}\label{eq:BackwardPDE_intro}
    \left\{\begin{aligned}
        &\partial_t u(t,y) + \mathcal{D}u(t,y)\Big( \partial_x y+K b(ev_0(y))  \Big)+\frac{1}{2}\textnormal{Tr}\Big[ \mathcal{D}^2u(t,y)K \sigma(ev_0(y))(K \sigma(ev_0(y)))^*\Big],\;t<T\\&
        u(T,y)=\varphi(y),
    \end{aligned}\right.
\end{equation}
for all~$y\in\mathscr{K}_1$. 
Moreover, it holds that $
\EE[\phi(X_T)\lvert\Ff_t]= u(t,\lambda(t))$ for all $\phi\in\Cc^2_b(\RR^d).$
\end{theorem3}

Consequently, we complete the desired connection between SVEs and PDEs by identifying $u$ in \eqref{eq:IntroSVEConditionalExpectation} as the unique solution of the singular Kolmogorov equation that is posed on a Hilbert space.

\subsection{Relevant literature}\label{subsec:literature} 
Our proposed lift, especially written in the strong formulation~\eqref{eq:strong_SPDE_intro}, is closely related to the HJMM model for forward rates in mathematical finance~\cite{heath1992bond,musiela1993stochastic}, albeit with singular coefficients. Several theoretical results have been established for this Musiela SPDE; these include e.g. well-posedness with different types of noises \cite{filipovic2001consistency,filipovic2003existence, marinelli2010local}, Kolmogorov equations \cite{goldys2001infinite}, invariant measures \cite{tehranchi2005note, carmona2006interest} to name only a few. In particular, Filipovic identifies a class of weighted RKHS as natural state spaces for the solution~\cite[Chapter 5]{filipovic2001consistency}. The conditions he imposes on the weight are financially-driven but unfit for our purposes. These spaces are proposed by Cuchiero and Teichmann~\cite{cuchiero2019markovian} as a state space for Markovian lifts of affine Volterra processes. There, the authors define weakly mild solutions to the equation and derive an exponential affine formula which is essentially a backward PDE with a particular type of test function. 

Simultaneously, Abi Jaber and El Euch \cite{abi2019markovian} introduce the curve lift to discuss the Markov and invariant properties (in $\RR\times\RR_+$) of the Volterra Heston model. The emphasis is on determining the geometry of the state space to guarantee non-negativity of the variance process. 

If one does not wish to exploit the SPDE structure of the curve lift then it is also possible to work without the Musiela parametrisation and define the lift~$\tilde{\lambda}(t,T)=\lambda(t,T-t)$. From an option pricing viewpoint, this shifts the focus from time \emph{to} maturity~$T-t$ to time \emph{of} maturity~$T$. In this context we can mention~\cite{bondi2023kolmogorov} where the authors derive, in the additive noise case, the existence of a backward PDE in an $L^2(\RR^+)$ space directly from a Taylor expansion. If one enforces the assumption~$K(0)<+\infty$ then most of the obstacles outlined above disappear and more tools become available, opening the gates to stochastic control~\cite{possamai2024optimal}. As discussed above, in the closest paper to ours, Viens and Zhang~\cite{viens2019martingale} prove a functional Itô formula for~$\tilde{\lambda}(t)$. Instead of a superficial comparison, we provide in Section~\ref{sec:comparison_VZ} a comprehensive review of their results and of the developments offered by the present paper. 

The other most popular lift in the literature, initiated by Carmona and Coutin in~\cite{carmona1998fractional} for fBms and extended to SVEs by Harms and Stefanovits~\cite{harms2019affine}, is the Ornstein--Uhlenbeck (OU) lift. The latter can be embedded in $L^p(\mu)$ spaces with different measures/weights $\mu$~\cite{hamaguchi2023markovian}, spaces of measures~\cite{cuchiero2019markovian} or negative Sobolev spaces~\cite{huber2024markovian}. Notably, this approach led to advances in stochastic control for Volterra processes~\cite{abi2021linear,hamaguchi2025global} but its wide recognition comes (in our humble opinion) from the Markovian multifactor approximation it entails~\cite{carmona1998fractional,abi2019multifactor}. The lifting method adopted by Hamaguchi in particular~\cite{hamaguchi2023markovian} seems to have similar features as ours. We demonstrate in Section~\ref{subsec:OUlift} how one can pass from the curve lift to the OU lift and vice-versa.

\subsection{Perspectives}
Having resolved (under certain conditions) the three problems outlined above, we believe this paper presents a framework that is not an end in itself but a base camp from which we can meet further objectives. 
For the first time, we have access to a complete and rigorous, Hilbert space based theory that delivers Itô formula and Kolmogorov PDEs for Volterra processes.   
These can be used as tools in applications
such as stochastic Taylor expansions and cubature formulae~\cite{feng2023cubature}, weak error rates analysis~\cite{bonesini2023rough}, support characterisation~\cite{kalinin2021support} and non-linear filtering~\cite{crisan2024uniqueness}. Numerically solving the backward PDE is interesting to the mathematical finance and rough volatility communities where conditional expectations correspond to option prices, see for instance~\cite{pannier2024path} and Section~\ref{sec:option_pricing}. 

A higher summit that sparks a lot of interest is the study of controlled Volterra processes. With the foundations laid out here, we will tackle Hamilton--Jacobi--Bellman equations and their viscosity solutions in the future. Such an expedition has only ever been attempted with the regular kernel oxygen bottle.

The two Fokker--Planck equations, mild and singular, presented in this paper are a central tool for the analysis of invariant measures and (a non-linear version thereof) shall characterise the law of mean-field SVEs as it does in the semimartingale case~\cite{chaintron2022propagation}.

On the foundational level, theoretical work is still required to relax the assumptions made in this paper and to extend the theory to the multiplicative noise case with~$H\le 1/4$ (in the spirit of~\cite{brehier2018kolmogorov}), to non-smooth test functions~$\varphi$ (checking whether the Markov semigroup regularises the test function), to path-dependent test functions and to non-convolution kernels. We believe the latter extension not to require radically new ideas: the change of parametrisation $K(t,s)\sigma(x)=\tilde\sigma(t,s,x)=\hat\sigma(t-s,s,x)$ allows to develop similar arguments as in this paper at the expense of working with time-dependent coefficients since $S(t-s)K\sigma\circ ev_0(\lambda)$ shall be replaced by~$S(t-s)\hat\sigma(\cdot,s,ev_0(\lambda))$ in~\eqref{eq:SPDE_mild_intro}.

In mathematical finance, rough volatility and forward variance models are natural fields of applications but other markets with term structure models may benefit from this forward curve formalism, such as foreign exchange or energy markets~\cite{fontana2024real}. We refer to Section~\ref{sec:option_pricing} for a short overview of option pricing in these different contexts.

\bigskip
\subsection{Examples and applications}\label{subsec:SVEexamples}
\subsubsection{Examples of Volterra processes}\label{subsec:examples} The SVE \eqref{eq:SVE} is multidimensional, fully nonlinear and allows for matrix-valued kernels. Due to this level of generality, our results apply to many Volterra processes that have been widely studied in the literature. Some examples, along with the necessary specifications in terms of the parameters $X_0,b,\sigma, K,$ are given below. In what follows, $\Gamma$ denotes the Euler Gamma function.
\begin{itemize}
    \item The \textbf{Mandelbrot-van Ness fBm}    (type I) can be obtained by taking $d=m=1, b=0, \sigma=1, K(t)=\tfrac{1}{\Gamma(H+1/2)}t^{H-1/2},H\in (0,1)$ and $$X_0(t)=\int_{-\infty}^{0}\bigg[ (t-s)^{H-1/2}-(-s)^{H-1/2}     \bigg]\dW_s\;, t\geq 0$$
    where $W$ is a two--sided, one-dimensional Brownian motion.
    The resulting process 
    \begin{equation}\label{eq:MandelNessRep}
        W^H_t=X_0(t)+\frac{1}{\Gamma(H+1/2)}\int_0^t(t-s)^{H-\half}\dW_s
    \end{equation}
     is the Mandelbrot-van Ness representation of a one-dimensional (standard) fBm with Hurst parameter $H.$ We remark here that the initial curve $X_0$ is smooth in $t\in(0,\infty)$ and posseses the same integrability properties as the kernel $K.$ Thus, after choosing an appropriate family of weighted Sobolev spaces corresponding to $K$ (see Example \ref{example:powerlaw}), all the results of this work can be (trivially) applied to $W^H.$ Using this example as a starting point, our framework may be applied to solutions $X$ of SDEs driven by additive fBm (such as the ones considered in~\cite{hairer2005ergodicity} or the fractional OU process considered in \cite{li2023non}).
     Indeed, taking different kernels in front of the drift and diffusion fits into our framework by adding an extra dimension as follows:
     $$
K(t)= \begin{pmatrix}
    K_1(t) & 1 \\ 0&0
\end{pmatrix},\; b(x_1,x_2)=\begin{pmatrix}
    0 \\b_1(x_1)
\end{pmatrix},\;\sigma(x_1,x_2)\equiv I_{2}, 
     $$
where we set $m=d=2.$ 
Using the initial curve $X_0+x,$ for some $x\in\RR^d$ and $X_0$ as above, the first coordinate of the resulting SVE reads
$$
X^1_t=x+X_0^1(t) + \int_0^t b_1(X_s^1)\ds + \int_0^t K_1(t-s)\D W_s^1=x+\int_0^t b_1(X_s^1)\ds+W^H_t.
$$
     
    \item  The \textbf{Riemann-Liouville fBm} (or type II) is given by the second term on the right-hand side of \eqref{eq:MandelNessRep}. This process is widely used to model rough volatility and differs from type I fBm in that its increments are not stationary. It is obtained from \eqref{eq:SVE} by setting  $d=m=1, X_0=0, b=0, \sigma=1, K(t)=\tfrac{1}{\Gamma(H+1/2)}t^{H-1/2}, H\in (0,1).$ As in the previous example $H$ is called again the Hurst parameter. 
    
    More generally, many of our results (in particular well-posedness of the associated SPDE, It\^o formulae and Fokker-Planck equations) apply to the larger class of \textbf{Brownian semistationary processes} (see e.g. \cite[Chapter 1.5]{barndorff2018ambit}) of the form
    $X_t=\int_{\infty}^tK(t-s)\sigma_s\dW_s$ where $\{\sigma_t\}_{t\geq 0}$ is a stochastic process with locally bounded trajectories, adapted to the filtration generated by $W.$ Apart from taking $\sigma=\sigma_t,$ these processes are obtained with the same specifications as above.  Of course, our results which rely on Markovianity of the lift (i.e. well-posedness and derivation of the backward Kolmogorov PDE) apply to the state-dependent case $\sigma_t=\sigma(X_t)$ for some sufficiently well-behaved $\sigma.$
    
    \item As an important example of Volterra processes with a smooth kernel we mention the \textbf{stationary OU process}. It is obtained  from \eqref{eq:SVE} by the choices  $d=m=1, b=0, \sigma=1, K(t)=e^{-t}$
    $$X_0(t)=\int_{-\infty}^{0}e^{-(t-s)}\dW_s.$$
    \item Lastly, we mention the class of \textbf{Rough volatility models} for the dynamics of an asset price $S$ with stochastic volatility $V$ following Volterra dynamics. We refer to~\cite{bayer2023rough} for an overview of the field.  
    The corresponding joint log-price/volatility process $X_t=(\log(S_t/ S_0), V_t), t\geq 0$ can be obtained from \eqref{eq:SVE} by setting $d=m=2,$ $\rho^2=1-\bar{\rho}^2\in[0,1],$  \begin{align}\label{eq:specs_roughvol}
    & X_0=\begin{pmatrix}0\\ V_0\end{pmatrix},\quad K(t)=\begin{pmatrix}  1 & 0\\ 0 &\tfrac{t^{H-1/2}}{\Gamma(H+1/2)}   
    \end{pmatrix},\\
    &b(x_1, x_2)= \begin{pmatrix}   -\psi^2(x_2)/2 \\ \vartheta(x_2)     
    \end{pmatrix},\quad \sigma(x_1, x_2)= \begin{pmatrix}\rho\psi(x_2) &\bar{\rho}\psi(x_2)\\ 0 &\varsigma(x_2)   
    \end{pmatrix},
    \end{align}
    where $\rho\in[-1,1]$ is a correlation parameter between the Brownian noises that drive $S$ and $V.$ Popular examples include the rough Bergomi model~\cite{bayer2016pricing} where $V$ is given by (a type II) fBm with Hurst parameter $H\in (0, 1/2)$, hence~$\psi(x)=\E^{x},\,\vartheta(x)=0,\,\varsigma(x)=1$, and the rough Heston model~\cite{el2019characteristic} where~$V$ follows fractional CIR dynamics, hence~$\psi(x)=\sqrt{x},\,\vartheta(x)=a_1+a_2x,\,\varsigma(x)=a_3\sqrt{x}$. Extensions of \eqref{eq:specs_roughvol} may involve more sophisticated kernels and higher dimensions.

\end{itemize}
\subsubsection{Option pricing applications}\label{sec:option_pricing}
As one can infer from the last example, a number of asset price models in mathematical finance are accurately represented by SVEs such as~\eqref{eq:SVE}. In turn, as alluded to in the introduction, options prices are conditional expectations of a certain payoff function, represented here by~$\varphi(\lambda(T))$. In this context, such prices are thus characterised by the backward PDE \eqref{eq:BackwardPDE_intro}.

We highlight below that, for a variety of underlying assets of interest, the relevant payoff functions do not only depend on~$X_T$ but on the forward curve~$(\EE[X_{T+x}\lvert \Ff_t])_{x\in[a,b]}$, for some interval~$[a,b]\subset\RR_+$.
We thus start by showcasing the situations in which $\lambda$ can be represented as the conditional expectation $\lambda(t,x)=\EE[X_{t+x}\lvert\Ff_t]$, $t,x\ge0$, making our choice of lift particulary suited for the pricing task at hand. 

Let $b(y)= ay,a\in\RR^{d\times d}$,
in \eqref{eq:SVE} and $R_a$ be the resolvent of the second kind of~$aK$, defined as the solution of~$aK-R_a=aK\star R_a$ where~$\star$ denotes the convolution (this resolvent exists and is unique under Assumption~\ref{assumption:Kernel}). A variation of constants formula of Volterra type~\cite[Chapter 2, Theorem 3.5]{gripenberg1990volterra} yields the expression
\begin{align}\label{eq:variationofconstants}
        X_t = X_0(t) + \int_0^t R_a(t-s)X_0(s)\ds + \frac{1}{a}\int_0^t R_a(t-s)\sigma(X_s)\D W_s.
\end{align}
It follows that $\lambda(t,x)=\EE[X_{t+x}\lvert\Ff_t]$ for all $x\ge0$.

We highlight four model specifications below. The first three are displayed under a risk-neutral measure and with zero interest rate.

\begin{itemize}
    \item \textbf{Rough volatility.} In the rough volatility model described in the last example of Section~\ref{subsec:examples}, $X_t=(\log(S_t/S_0),V_t),\,t\ge0$ is a model for the joint log-price and volatility. For an option of payoff~$\phi\in \Cc^2_b(\RR;\RR)$ over the asset price~$S$, the option price at time~$t\in[0,T]$ is represented by the conditional expectation~$\EE[\phi(S_T)\lvert\Ff_t]$. The latter fits into Theorem~\ref{thm:backward_equation_singular} upon setting~$\varphi(y_1,y_2):=\phi(S_0 \exp\circ \,ev_0(y_1))$ and is therefore characterised by~$u(t,\lambda(t))$ where~$u$ is the unique solution to the PDE~\eqref{eq:BackwardPDE_intro}. A path-dependent PDE for this option price was already derived in~\cite{bonesini2023rough} in the setting of the rough Bergomi model where $V$ is a Gaussian process. This paper makes the leap to fully non-linear Volterra dynamics.
    \item \textbf{Forward variance.} Let the instantaneous variance of an asset be represented by~$(X_t)_{t\ge0}$ as in~\eqref{eq:variationofconstants} such that the forward variance~$(\lambda(t,x))_{x\ge0}$ at time~$t$ satisfies~$\lambda(t,x)=\EE[X_{t+x}\lvert\Ff_t]$. Options on the forward variance, represented by~$\EE[\varphi(\lambda_T)\lvert\Ff_t]$ at time $t$, immediately embed into Theorem~\ref{thm:backward_equation_singular}. As a primary example, the Volatility Index (VIX) at time~$T$ is given by
    $$
\textrm{VIX}_T := \sqrt{\frac{1}{\Delta} \int_T^{T+\Delta} \EE[X_s\lvert\Ff_T] \D s}=\sqrt{\frac{1}{\Delta} \int_0^{\Delta} \lambda(T,x) \D x},
    $$
    where~$\Delta$ is 30 days. It is clear that $\textrm{VIX}_T $ can be represented as~$G(\lambda(T))$ for some~$G:\Hh_1\to\RR$ and thus \eqref{eq:BackwardPDE_intro} characterises VIX options (upon modulating the boundedness conditions imposed on the payoff's derivatives). A path-dependent PDE for VIX options was derived in~\cite{pannier2023path} where $X$ is a Gaussian process.
    \item \textbf{Energy markets.} 
    The empirical studies carried out in~\cite{bennedsen2022rough,mishura2024gaussian} suggested electricity prices to have rough trajectories and to be well modeled by Volterra processes. The prices themselves are not traded (and hence not observed on the market), instead the traded assets are futures with delivery over the time interval~$(\tau_1,\tau_2]$. The future's price at time~$t\le\tau_1$ is given by
    \begin{align}
        F(t,\tau_1,\tau_2) = \int_{\tau_1}^{\tau_2} w(s,\tau_1,\tau_2) f(t,s)\D s,
    \end{align}
    where, letting $(X_t)_{t\ge0}$ be the electricity price in the form~\eqref{eq:variationofconstants}, $f(t,s)=\EE[X_s\lvert\Ff_t]=\lambda(t,s-t),s\ge t,$ are the instantaneous forward prices, and hence there exists some function~$G:\Hh_1\to\RR$ such that 
    $$F(t,\tau_1,\tau_2)
    = \int_{\tau_1-t}^{\tau_2-t} w(t+x,\tau_1,\tau_2)\lambda(t,x)\dx 
    = G(\lambda(t)).$$
    We conclude that option prices on the future~$F(t,\tau_1,\tau_2)$ are solutions to the PDE~\eqref{eq:BackwardPDE_intro}. 
    This modeling setup can be seen as an extension of the approach taken in~\cite{benth2015derivatives} to the Volterra case and it can be generalised to other energy markets with similar traded assets.
    \item \textbf{Interest rates.}
    We eventually come back to forward rate models in the spirit of the HJMM framework.
    Denote by~$r_t(x)$ the forward rate at time~$t$ with maturity~$t+x$; as in the previous example, this process is not traded. The most fundamental contract is the zero-coupon bond which pays one at maturity~$T$. Its price at time~$t<T$ is given by
    $$
P(t,T)=\exp\left(-\int_0^{T-t} r_t(x)\dx\right).
    $$
    We model $(r_t)_{t\ge0}$ as the solution of the mild SPDE~\eqref{eq:SPDE_mild_intro}. This is an extension of the semimartingale HJMM model to the singular Volterra case, but restricts the coefficients to be deterministic and locally state-dependent. Under the $T$-forward measure $\QQ^T$ (see \cite[Chapter 7]{filipovic2009term}), the forward rate is the conditional expectation of the future short rate, hence it reads
    $$
r_t(x) = r_0(t+x) + \int_0^t K(t-s+x) \sigma(r_s(0))\D W_s^T,
    $$
     where~$W^T$ is a $\QQ^T$--Brownian motion. For some $\delta>0$, traded contracts include for instance options on the simple forward rate~$F(T,T+\delta)=\frac{1}{\delta}(\frac{1}{P(T,T+\delta)}-1)$.    
     More generally we consider options with payoff~$\phi$ on the zero-coupon bond~$P(T,T+\delta)$, which price reads
    $$
\EE\left[ \E^{-\int_t^{T+\delta} r_s(0)\ds} \phi(P(T,T+\delta)) \lvert \Ff_t\right]
= P(t,T+\delta) \EE_{\QQ^T} \big[\phi(P(T,T+\delta)) \lvert \Ff_t\big].
    $$
    We then identify $r_t(x)$ with $\lambda(t,x)$ and notice that~$P(T,T+\delta)$ is a smooth map of~$r_T=\lambda(T)$. Finally,
    the conditional expectation on the right side can be represented as~$u(t,\lambda(t))$ where $u$ solves the PDE~\eqref{eq:BackwardPDE_intro}. We refer to \cite{goldys2001infinite} for an exposition of the PDE in the classical HJMM framework.      Our setup also effortlessly scales to multiple dimensions which was the focus of recent studies~\cite{cuchiero2016general,fontana2024real}.
     \end{itemize}

\subsection{Outline} The rest of this paper is organised as follows:

In Section \ref{sec:Notation} we fix all the necessary notation and introduce our standing assumptions and function spaces that will be used throughout this work. We also define a class of admissible weights $\Ww_a$ (Definition \ref{assumption:admissibleweights}) and corresponding scales of weighted Sobolev spaces $\Hh, \Hh_m, m\in\NN$ which we use for the well-posedness of \eqref{eq:strong_SPDE_intro}. Furthermore, we elaborate on the validity of our assumptions for the particular yet central case of power-law kernels (Example \ref{example:powerlaw}).

Section \ref{sec:SPDE_lift} is devoted to the analysis of the SPDE~\ref{eq:strong_SPDE_intro}. Properties of the shift semigroup $\{S(t)\}_{t\geq 0}$ on $\Hh_1$, as well as existence and uniqueness of $\Hh_1-$valued mild solutions (Theorem~\ref{thm:SPDE_wellposedness}) are proved in Subsection~\ref{sec:MildWellPosedness}. Then, we prove lift and temporal regularity properties of mild solutions $\lambda$~\eqref{eq:SPDE_mild_intro} in Subsection~\ref{Sec:LiftRegularity}. In Subsection~\ref{sec:invariance} we show that the subspace $\mathscr{K}_1\subset\Hh_1,$ as described in Subsection~\ref{sec:contributions}, is invariant for the SPDE dynamics.  Flow and regularity properties of $\lambda,$ as well as (generalized) Feller properties of the corresponding Markov semigroup, are then obtained in Subsections~\ref{subsec:MarkovSemigroup}, \ref{sec:Feller} respectively.

 With the basic properties of the underlying Markovian dynamics at hand, we set out to prove several It\^o formulae for functionals of the state process $\lambda$ in  Section \ref{sec:ito}. First, we remark that analytically strong SPDE solutions (and consequently a strong It\^o formula) are only available on the larger Hilbert space $\Hh\supset\Hh_1.$ Since the latter is not an RKHS, the resulting strong It\^o formula (Proposition \ref{prop:StrongIto}, Subsection \ref{subsec:StrongIto}) cannot be projected to $\RR^d$ via evaluation and thus is not useful for the study of the SVE \eqref{eq:SVE}. After these preliminary observations, we invoke results of \cite{da2019mild} to obtain a mild It\^o formula (Theorem \ref{thm:mildIto}) for smooth functions of $\lambda$ in Subsection \ref{subsec:MildIto}. Even though the latter does not lead to a well-posed backward Kolmogorov PDE, it does lead to an It\^o formula for \eqref{eq:SVE} as well as a mild Fokker--Planck (or forward) equation for the laws $\{\mu_t\}_{t\geq 0}\subset\mathscr{P}(\Hh_1)$ of $\lambda.$ These applications are presented in the first two parts of Subsection \ref{sec:applications_mildito}. Finally, Subsection \ref{sec:SingularIto} contains our results on the singular It\^o formula (Theorem \ref{thm:SingularIto}) which is the backbone of the singular Kolmogorov PDE. Here, we define the concept of singular directional derivatives (Definition \ref{dfn:CKclass_alt}) along a set of directions $\mathscr{K}\subset\Hh$ \eqref{eq:SingularDirections} and present several examples and non-examples that shed light on this notion of smoothness. Then, we introduce a class $\Cc^{1,2}_{T,\mathscr{K}}$ (Definition \ref{dfn:CKclass_alt}) of singularly differentiable functionals and proceed to the proof of Theorem \ref{thm:SingularIto} for functionals of this class. As mentioned above, the latter leads to a singular Fokker--Planck equation for the laws  $\{\mu_t\}_{t\geq 0}\subset\mathscr{P}(\Hh_1),$ when tested against $\Cc^{1,2}_{T,\mathscr{K}}$ functions. The latter is presented as a third and last application in the last part of Section \ref{sec:applications_mildito}. Our mild and singular It\^o formulae apply to different classes of test functions~$F$ and lead to Volterra and semimartingale representations of~$F(\lambda)$, respectively. 
 That being said, in this work we only provide a complete well-posedness theory for the corresponding singular Kolmogorov equations. A more detailed comparison between the two can be found in Remark \ref{rem:FPEcomparison}.

Section \ref{sec:BackwardEquation} is entirely devoted to the study of the backward Kolmogorov equation \eqref{eq:BackwardPDE_intro}. The main result of this paper, Theorem \ref{thm:backward_equation_singular}, is stated in the beginning of this section and asserts that a strong (pointwise) solution of \eqref{eq:BackwardPDE_intro} is given by the Markov semigroup of $\lambda.$ The same theorem shows that this probabilistic solution is unique among the class of  $\Cc^{1,2}_{T,\mathscr{K}}$ functions. The characterisation \eqref{eq:IntroSVEConditionalExpectation} for conditional expectations of $X$ is proved in the subsequent Corollary \ref{cor:SVEConditionalExp}. The proof of Theorem~\ref{thm:backward_equation_singular} is postponed to Subsection~\ref{subsec:BackwardProof} and is preceded by several directional differentiability results of~\eqref{eq:SPDE_mild_intro} with respect to initial data in $\Hh_1$. In particular: We obtain well-posedness of extended (in the sense that the  directions lie in $\mathscr{K}$)  first and second variation equations in Subsection \ref{subsec:SingularTangents}; see Lemmas \ref{lem:zetaKWP} and \ref{lem:zetaKKWP}. For ``regular" directions $h, h'\in\Hh_1,$ the corresponding mild solutions $\zeta_h, \zeta_{h, h'}$ coincide with the Gateaux derivatives $D_y\lambda^y(h), D_y^2\lambda^y(h, h')$ (per Corollary \ref{cor:regularTangentProcesses}). In Subsection \ref{subsec:singularDifferentiablity} we show that the map $y\longmapsto\lambda^y$ is ``singularly" differentiable in the direction of $K\notin\Hh_1$ and identify the singular differentials $\Dd\lambda^y(K), \Dd^2\lambda^y(K,K)$ with the processes $\zeta_K, \zeta_{K, K'}.$ Finally, Subsection \ref{subsec:ConvergenceDeltaTo0} contains several auxiliary convergence results for  ``mollified" variants of $\lambda, \zeta_h, \zeta_{h, h'}.$ To be  slightly more precise, we substitute various instances of the singular kernel $K$ (or, more generally, a singular direction $h\in\mathscr{K}),$ appearing either as coefficient of equations or as initial condition of the tangent process $\zeta_{K},$ with the shifted smooth kernels $K_\delta=S(\delta)K, \delta>0,$ and show that the corresponding processes converge as $\delta\to 0$ to their well-defined ``singular" counterparts. These results are necessary for the proof of Theorem \ref{thm:backward_equation_singular} which requires two approximation procedures by smooth functions (in the level of both initial data and directions).

As we have already argued, we believe that the SPDE framework put forth in this paper is both natural and fruitful for the study of SVEs with singular kernels. Its effectiveness in the derivation of Kolmogorov PDEs for SVEs is clearly showcased by our main results.  Indeed, the use of shift semigroups (or transport SPDEs) properly encodes the ``smoothing" of singular kernels away from the origin. Most of the results and constructions stemming from the developed theory are essentially manifestations of this phenomenon. That being said, there are (at least) two other popular and successful choices of Markovian lifts for SVEs in the relevant literature: the Ornstein-Uhlenbeck (OU) lift of Carmona--Coutin~\cite{carmona1998fractional} and the Path-Dependent PDEs (PPDEs) introduced by Viens--Zhang~\cite{viens2019martingale}. In Section~\ref{sec:Conclusion} we provide a detailed comparison of the curve lift and our results to those of the aforementioned papers. In Subsection \ref{subsec:OUlift} we show that the OU lift can be obtained, via Laplace transform, by an invertible transformation of the SPDE \eqref{eq:strong_SPDE_intro}. Ultimately, we discuss PPDEs in more detail in Subsection \ref{sec:comparison_VZ}.

A few auxiliary results related to properties of the state space $\Hh_1$ and the proof of a  Volterra--Gr\"onwall inequality are gathered in Appendix \ref{sec:appendix}.

\section{Notation, function spaces and assumptions}\label{sec:Notation}
The notation $\lesssim$ is used to denote inequality of real numbers up to unimportant constants. Throughout this work we fix a complete filtered probability space $(\Omega, \Ff, \{\Ff_t\}_{t\geq 0}, \PP)$ that supports an $m-$ dimensional standard Brownian motion $W.$ The nonnegative half-line $[0,\infty)$ is denoted by $\RR^+.$ For a compact set $M\subset\RR^d,$ $\Cc(M)$ is the Banach space of real-valued continuous functions on $M$ endowed with the topology of uniform convergence. For $m\in\NN$ and $U\subset\RR^d$ open, $\Cc_c^\infty(U), \Cc^m(U)$ are the classes of smooth, compactly supported and $m-$times continuously differentiable $\RR-$valued functions on $U$ respectively.

The absolute value, Euclidean norm in~$\RR^d$ and Frobenius norm in~$\RR^{d\times d}$ will all be denoted by~$\abs{\cdot}$. The gradient of a function $f:\RR^d\rightarrow \RR$ will be denoted by $\nabla f.$ Finally, for functions $f:[0,T]\times \RR^+\rightarrow\RR^d$ of time/space we use $\partial_t f$ or $\dot{f}$ and either $f'$ or $\partial_x f$ for the temporal/spatial partial derivatives respectively.

For two Banach spaces $E_1, E_2,$ $\mathscr{L}(E_1;E_2)$ is the space of bounded linear operators from $E_1$ to~$E_2.$ Throughout this work, all Hilbert spaces are implicitly assumed to be $\RR-$vector spaces. For two separable Hilbert spaces $H_1, H_2,$ $\mathscr{L}_2(H_1;H_2)$ denotes the space of Hilbert--Schmidt linear operators $A:H_1\rightarrow H_2.$ The latter is a Hilbert space and for $A_1, A_2\in \mathscr{L}_2(H_1;H_2),$ $\{e_k\}_{k\in\NN}\subset H_1$ an orthonormal basis the inner product is given by
$$\langle A_1, A_2\rangle_{\mathscr{L}_2(H_1;H_2)}=\sum_{k\in\NN}\langle A_1e_k, A_2e_k\rangle_{H_2}=:\textnormal{Tr}(A_1^*A_2),$$
where $A_1^*:H_2\rightarrow H_1$ is the adjoint operator. The transpose of a matrix $A$ will be denoted by $A^\top.$ The Lipschitz constant and supremum norms of a function $F:E_1\rightarrow E_2$ will be denoted by $|F|_{Lip}:=\sup_{x\neq y\in E_1}\tfrac{|F(x)-F(y)|_{E_2}}{|x-y|_{E_1}}$ and $|F|_\infty:=\sup_{x\in E_1}|F(x)|_{E_2}.$ The topological support of a Borel measure~$\mu$ on~$E_1$ and the set of all such probability measures on $E_1$ will be denoted by $\textnormal{supp}(\mu)$ and  $\mathscr{P}(E_1)$ respectively. Equality in law of stochastic process will be denoted by $\overset{\textnormal{law}}{=}$ while we shall reserve the notation $\mathcal{L}$ for generators of Markov processes.

In the sequel, and unless explicitly stated otherwise, we shall identify a Hilbert space with its continuous dual (via Riesz representation) without further explanation or changes in notation. The same principle will often be applied to Fr\'echet differentials of functions defined on Hilbert spaces and their corresponding gradients.

The domain of a (possibly unbounded) linear operator $A:E_1\rightarrow E_2$ is denoted by $Dom(A).$ For $\alpha\in(0,1)$ and $I\subset\RR^+$ a compact interval we denote by $\Cc^\alpha(I;E_1)$ the space of $\alpha-$H\"older continuous functions $f:I\rightarrow E_1.$ For $p\in[1,\infty),$ $L^p(\RR^+), L_{loc}^p(\RR^+)$ denote the spaces of $p-$integrable (respectively locally $p-$integrable) classes of real-valued measurable functions up to almost everywhere equality. Given a nonnegative function $w\in L^1_{loc}(\RR^+),$ we denote by $L^p_w(\RR^+)$ the space consisting of classes of real-valued measurable functions 
$f\in L^1_{loc}(\RR^+)$ such that 
$$|f|_{L^p_w}:=\bigg( \int_{\RR^+} |f(x)|^pw(x)\bigg)^{\frac{1}{p}}<\infty.$$
The space $L^p_w(\RR),$ endowed with the norm $|\cdot|_{L^p_w},$ is a Banach space.

Next, we define the weighted Sobolev space $W_w^{1,p}(\RR^+)$ as the vector space consisting of $f\in L^p_w(\RR^+)$ such that $f$ is weakly differentiable and the weak derivative $f'\in L^p(\RR^+).$  If for some $p>1$  
\begin{equation}\label{eq:weight-p-condition}
    w^{-\frac{1}{p-1}}\in L^1_{loc}(\RR^+)
\end{equation}
 then $W_w^{1,p}(\RR^+),$ endowed with the norm 
$|f|_{W^{1,p}_w}:=|f|_{L^p_w}+|f'|_{L^p_w},     $
is a Banach space; see Remark \ref{remark:admissible weights}. Clearly, when~$w=1,$ $W_w^{1,p}(\RR^+)$ coincides with the classical Sobolev space $W^{1,p}(\RR^+).$ The case $p=2$ is special since $H^1_w(\RR^+):=W^{1,2}_w(\RR^+)$ is a Hilbert space with respect to the inner product
$$\langle f, g\rangle_{H^1_w}:=\int_{\RR^+}fgw+\int_{\RR^+}f'g'w=:\langle f,g\rangle_w+\langle f',g'\rangle_w.$$ For $m\in\NN,$ we write $H^m_w(\RR^+):= W^{m,2}_w(\RR^+)$ for the corresponding weighted Sobolev spaces of $m-$times weakly differentiable functions with square integrable derivatives. Their topology is defined in analogy to $H^1_w(
\RR^+)$ and $H^{m}_w(\RR^+)$ is dense in~$L^2_w(\RR^+)$ and $H^{m-1}_w(\RR^+);$ see Corollary \ref{cor:HmDensity} below.

The vector-valued analogues of the weighted Lebesgue and Sobolev spaces are denoted respectively by $L_w^p(\RR^+;\RR^d), W_w^{1,p}(\RR^+; \RR^d)$ and $ H^m_w(\RR^+; \RR^d).$ 

\begin{remark}\label{remark:admissible weights} A proof that condition \eqref{eq:weight-p-condition} implies completeness for the spaces $W_w^{1,p}(\RR^+;\RR^d),$ as well as counterexamples when the latter is violated, can be found in \cite{kufner1984define} (and in particular Theorem 1.11 therein). Besides the latter, the interested reader is referred to \cite{kufner1980weighted} for a comprehensive study of weighted Sobolev spaces. In fact, the weights we use belong to the larger class of Muckenhoupt weights (see e.g. \cite[Chapter V]{stein2016harmonic}).
\end{remark}

Throughout the rest of this work, we fix the notations 
\begin{equation}\label{eq:HSpaceNotations}    \Hh\equiv\Hh_0:=L^2_w(\RR^+;\RR^d), \Hh_m:=H^m_w(\RR^+;\RR^d), m\in\NN.
\end{equation}
Furthermore, we consider, for each $p\geq 1,0\leq s<T$ the Banach space $\mathscr{H}^p_{s,T}(\Hh_1)$ of $\{\mathcal{F}_t\}_{t\geq 0}-$ adapted processes $u\in \Cc([s,T]; L^p(\Omega;\Hh_1)),$ endowed with the norm
\begin{equation}\label{eq:HpSpaces}
    |u|^p_{\mathscr{H}^p_{s,T}} :=\sup_{t\in [s,T]}\EE[|u(t)|^p_{\Hh_1}    ]
\end{equation}
(see e.g. \cite[Section 4.2]{cerrai2001second} for the use of similar spaces in the context of reaction-diffusion SPDEs).
These spaces will be used for the well-posedness of~\eqref{eq:SPDE} in Section~\ref{sec:SPDE_lift}. When the base point $s=0$ we shall simplify our notation and write $\mathscr{H}^p_{T}(\Hh_1)\equiv \mathscr{H}^p_{0,T}(\Hh_1).$ An $L^p-$in-time variant of these spaces, denoted by $\mathscr{Z}^{p}_{s,T}(\Hh_1),$ will be used to obtain well-posedness for singular tangent processes in Section \ref{subsec:SingularTangents}. The reader is referred to Lemmas \ref{lem:zetaKWP}, \ref{lem:zetaKKWP}  for precise definitions and further details. 

By a mild \textit{abuse of notation}, we will say that a matrix-valued function~$A:\RR^+\to\RR^{d\times d}$ belongs to $\Hh_m$ \eqref{eq:HSpaceNotations} if~$Ae_i\in\Hh_m$ for all~$i=1,..,d,$ where $\{e_i\}_{i=1}^{d}$ is the standard basis of $\RR^d,$  and the same convention will be used for membership to subspaces of $\Hh_m.$ Its norm~$\abs{A}_{\Hh_m}$ will stand for~$\sum_{i=1}^d \abs{Ae_i}_{\Hh_m}$ for any~$m\in\NN\cup\{0\}.$ 

 Throughout this work, $DF(x)(h)$ denotes the Gateaux derivative of a map $F:E_1\rightarrow E_2$ along the direction $h\in E_1$ and at the point $x\in E_1.$ For Fr\'echet derivatives $DF: E_1\rightarrow\mathscr{L}(E_1;E_2)$, when they exist, we will use the same notation. Moreover we will use, without exception, the notation $\Cc^{k}(E_1;E_2),\Cc^{k}_b(E_1;E_2), k\in\NN$ for the class of $k-$times continuously Fr\'echet differentiable functions $F:E_1\rightarrow E_2$ and the subset of bounded functions with bounded Fr\'echet derivatives of all orders.
 
 When $E_1=\Hh_1, E_2=\RR$ we will frequently make the following additional \textit{abuse of notation}:
 
 \noindent For any $h_1,h_2\in\Hh_1$ we denote the first and second order Gateaux derivatives~$\Hh_1\ni y\mapsto DF(y)(h_1)\in\RR$ and~$\Hh_1\ni y\mapsto D^2F(y)(h_1,h_2)\in\RR$. 
However,  we will also use the same notations with \textit{matrix-valued directions}~$ h_1,h_2\in H^1_w(\RR^+; \RR^{d\times d})$. In that context, they mean
    \begin{align*}
        DF(y)(h_1)=\big(DF(y)(h_1 e_j)\big)_{j=1}^d \quad \text{and}\quad 
        D^2F(y)(h_1,h_2)=\big(D^2F(y)(h_1e_i,h_2e_j)\big)_{i,j=1}^d,
    \end{align*}
    such that for any~$u,v,w\in\RR^d$ we have
    \begin{equation}\label{eq:AbuseofNotations}
    \begin{aligned}
         &DF(y)(h_1 u)=\sum_{j=1}^d DF(y)(h_1 e_j)u_j,
        \\&D^2F(y)(h_1v,h_2 w)
        = \sum_{i,j=1}^d D^2F(y)(h_1 e_i,h_2 e_j) v_i w_j. 
    \end{aligned}
    \end{equation}

In the sequel,  we will most frequently use the above relations notations with $h_1=h_2=K$ and with $u=b(y),v=\sigma(y)e_i,w=\sigma(y)e_j$.

Throughout the rest of this paper we shall exclusively work with weights that satisfy the following regularity conditions:
\begin{definition}\label{assumption:admissibleweights}
A weight function $w\in C^{\infty}(0,
\infty)$
is said to be \textit{admissible} if it satisfies~\eqref{eq:weight-p-condition} with~$p=2,$ $w^{-1}\in \Cc^{\infty}(0,
\infty)$ and moreover, for any $t\in\RR^+,$
$$\sup_{s\geq t} \frac{w(s-t)}{w(s)}<\infty.$$   
The class of admissible weights is denoted by $\Ww_a.$
\end{definition} 

\begin{remark} The requirement that $w, w^{-1}\in \Cc^{\infty}(0, \infty)$ is not necessary for our analysis and guarantees that $\Cc_c^\infty(0,\infty)$ is dense in $L^2_w(\RR^+;\RR^d)$ (see Corollary \ref{cor:HmDensity}). In fact, all results of this paper continue to hold for weights $w\in \Cc(0, \infty).$     
\end{remark}

Regarding the coefficients of the SVE \eqref{eq:SVE} we have the following assumptions.
\begin{assumption}\label{assumption:SVEassumptions} \
\begin{enumerate}[i)]
    \item There exists $p\geq 2$ such that, for all $t\in\RR^+$, the initial curve $X_0$ satisfies $$\int_0^t \EE\abs{X_0(s)}^p \ds <\infty.$$
    \item[ii)] $b:\RR^d\rightarrow\RR^d$ and $\sigma:\RR^d\rightarrow\mathscr{L}(\RR^m;\RR^d)$ are Lipschitz continuous. 
    \item[iii)] $b,\sigma$ are bounded, twice continuously differentiable with bounded derivatives and their second derivative is $\gamma_0$-H\"older continuous for some~$\gamma_0>0$. 
\end{enumerate}
\end{assumption}

\begin{remark}\label{rem:OptimalityofAssumptions} 
Our assumptions on the coefficients $b,\sigma$ are clearly not optimal and are chosen for a more direct exposition of the main results. Even though we do not aim for optimality in that sense, we discuss below the well-posedness of the SVE \eqref{eq:SVE} with more relaxed assumptions (see Remark \ref{rem:SVEwellposednessLiploc}).
Finally, we note here that Condition iii) is only relevant for the definition and study of the tangent processes (Section \ref{subsec:SingularTangents}) and analysis of the backward Kolmogorov equation (Theorem \ref{thm:backward_equation_singular} and Section \ref{subsec:BackwardProof}). Conditions i) and ii) are thus sufficient until the end of Section \ref{sec:ito}.
\end{remark}

The last set of assumptions is related to the Volterra kernel. 
\begin{assumption}\label{assumption:Kernel} There exists a weight $w\in\Ww_a$ such that 
$K\in L^2_w(\RR^+;\RR^{d\times d})$ and for all $t>0,$ $K(t+\cdot) \in H^1_w(\RR^+; \RR^{d\times d} ).$ Furthermore, there exists $q>2$ such that for any $T>0, h\geq 0,$ $K$ satisfies:
\begin{align}\label{eq:cond_K1}
    &i) \quad
    \int_0^T \abs{ K(t+\cdot)}_{\Hh_1}^q \dt <\infty 
    \\ 
    \label{eq:cond_K2}
    &ii) \quad
    \int_0^T \abs{K(t+\cdot)- K(t+h+\cdot)}^2_{\Hh_1}\dt \lesssim h^{1-2/q} 
    \\ \label{eq:cond_K3}
    &iii) \quad \int_0^t \abs{\partial_x K(s+h+\cdot)}_{\Hh_1}^2\ds<\infty\quad\text{and}\quad
\int_0^T \bigg(\int_0^t|\partial_x K(s+r+\cdot)|^2_{\Hh_1}\ds\bigg)^{1/2} \dr <\infty.
\end{align}

\end{assumption}
\begin{remark}
    While~\eqref{eq:cond_K1} is fundamental to the well-posedness of the SVE~\eqref{eq:SVE} and the SPDE~\eqref{eq:strong_SPDE_intro} (Theorem~\ref{thm:SPDE_wellposedness}), the other conditions are only relevant later on. Estimate~\eqref{eq:cond_K2} is necessary to the sample path continuity of~$\lambda$ (Lemma~\ref{lemma:pathregularity}). Finally, we use~\eqref{eq:cond_K3} to determine invariant properties of the solution (Proposition~\ref{prop:invariant subspaces}) and therefore all three conditions are crucial to obtain the singular Itô formula and backward PDE (Theorems~\ref{thm:SingularIto} and~\ref{thm:backward_equation_singular}).
\end{remark}

Next, we recall the notion of a Reproducing Kernel Hilbert space (RKHS) which is essential for the construction of a suitable state space for our SPDE lift of \eqref{eq:SVE}. See e.g. Section \ref{Sec:LiftRegularity} and in particular~\eqref{eq:liftproperty}.

\begin{definition} Let $H$ be a Hilbert space of functions $f:\RR^+\rightarrow \RR^d.$  $H$ is called a Reproducing Kernel Hilbert Space (RKHS) if, for all $x\in\RR^+,$ the evaluation operator
$$H\ni f\longmapsto ev_{x}(f):= f(x)\in\RR^d      $$
is continuous.
\end{definition}

We thus deduce a lemma that summarises a few useful properties of the weighted Sobolev spaces $H^1_w(\RR^+;\RR^d).$ Its proof is deferred to Appendix \ref{section:Appendix}.

\begin{lemma}\label{lemma:Hwproperties} If $w\in\Ww_a$ is an admissible weight then $H^1_w(\RR^+;\RR^d)$ is a separable Hilbert space and an RKHS.
\end{lemma}

 The standing assumptions \ref{assumption:SVEassumptions}, \ref{assumption:Kernel} on the coefficients and the kernel ensure global well-posedness of the SVE~\eqref{eq:SVE}. Such existence and uniqueness results were derived for initial conditions that are real deterministic~\cite{pardoux1990stochastic,abi2019affine}, or real random~\cite{jourdain2025convex}, but we rely on \cite{zhang2010stochastic} which allows for random initial curves. For this purpose, we check that~$\abs{K(s)}=\abs{ev_0(K(s+\cdot))}\le \abs{K(s+\cdot)}_{\Hh_1}$ for any~$t>0$ thanks to the RKHS property, and hence
   $$ \int_0^t \abs{K(s)}^2 \ds \le
   \int_0^t \abs{K(s+\cdot)}_{\Hh_1}^2\ds ,$$
which tends to zero as $t\to0$ by Cauchy--Schwarz inequality and~\eqref{eq:cond_K1}. Moreover, the linear growth of~$b,\sigma$ is a consequence of Assumption~\ref{assumption:SVEassumptions} ii). Therefore, the conditions of~\cite[Theorem 3.1]{zhang2010stochastic} are met and the SVE \eqref{eq:SVE} admits a unique strong solution that satisfies 
\begin{equation}\label{eq:SVEmoments}
\sup_{t\in[0,T]} \EE[\abs{X_t}^p]<\infty.
\end{equation}

\begin{remark}\label{rem:SVEwellposednessLiploc} Strong well-posedness for \eqref{eq:SVE} is also true under local Lipschitz and linear growth conditions on $b,\sigma,$ see e.g. \cite[Section 3.3]{zhang2010stochastic}. In the case of classical diffusions such statements can be found e.g. in \cite[Chapter V, Theorem 12.1]{rogers2000diffusions}. The proofs follow from the globally Lipschitz case after standard localization arguments.
\end{remark}


\begin{example}[Power-law kernel]\label{example:powerlaw}
    The power-law kernel $K(t)=t^{H-\half}$ with $H\in(0,1)$ is the most popular across the literature on SVEs. We associate to it the weight $w(x)=x^{\beta}\E^{-x}$, $x\in\RR^+$ with $\beta\in\big((1-2H)\wedge0,1\big)$ and show that all assumptions laid out in this section hold.  We emphasise that the condition $\beta>1-2H$ is necessary for the estimates to hold, which in turn guarantees that the kernel is sufficiently regular for equation~\eqref{eq:SPDE} to be well-posed; Theorem \ref{thm:SPDE_wellposedness}. Meanwhile, $\beta<1$ is required for $H^1_w$ to be an RKHS per Lemma \ref{lemma:Hwproperties}. In the regular case~$H\in[\half,1)$, $\beta$ can be set to zero.

    One can verify that $w^{-1}$ is continuous, locally integrable, thus meeting condition~\eqref{eq:weight-p-condition} and for any $0\le t<s$, $\frac{w(s-t)}{w(s)}= \frac{(s-t)^\beta}{s^{\beta}}\E^{t}\le \E^t$. 
    Lemma \ref{lemma:Hwproperties} then guarantees that $H^1_w$ is an RKHS, for any $H\in(0,1)$.
    
    We turn to Assumption \ref{assumption:Kernel} which requires to estimate the following integrals:
    \begin{align*}
        &\int_0^\infty K(x)^2 w(x)\dx 
        = \int_0^\infty x^{2H-1+\beta}\E^{-x}\dx =\Gamma(2H+\beta);\\
        &\int_0^\infty K'(t+x)^2 w(x) \dx 
        = \int_0^\infty (t+x)^{2H-3} w(x) \dx 
        \le t^{2H-3} \Gamma(\beta+1), \quad \text{for all  }t>0.
        \end{align*}
        This proves that $K\in L^2_w(\RR^+,\RR)$ and $K(t+\cdot)\in H^1_w(\RR^+,\RR)$.

        Furthermore \eqref{eq:cond_K1} holds since, for any $\ep>0$, we have the following estimate
        \begin{align*}
        & \int_0^T \abs{K'(s+\cdot)}^q_{\Hh} \ds = \int_0^T\bigg(\int_0^\infty (s+x)^{2H-3} x^\beta e^{-x}dx\bigg)^{q/2} \ds\leq \int_0^T\bigg(\int_0^\infty (s+x)^{2H+
        \beta-3} e^{-x}dx\bigg)^{q/2} \ds  \\&
        \leq \int_0^Te^{q s/2} s^{q(2H-\ep+\beta-2)/2}\bigg(\int_s^\infty x^{\ep-1}e^{-x}dx \bigg)^{q/2} \ds\lesssim \int_0^T s^{q(2H-\ep+\beta-2)/2}\ds=T^{q(2H-\ep+\beta-2)/2+1} ,
    \end{align*}
    where the last exponent is nonnegative if $q(2-\beta-2H+\ep)<2$. Since~$\ep>0$ is arbitrary, this estimate holds for any~$q<\frac{2}{2-2H-\beta}$ (which is greater than $2$ because $\beta>1-2H$) and we can similarly show that~$\int_0^T\abs{K(s+\cdot)}^q_{\Hh}\ds<\infty$.
    
    The next condition~\eqref{eq:cond_K2} also holds. 
    Indeed, for all~$s\in[0,t]$, $x>0$ and~$p>1$, H\"older inequality yields
\begin{align*}
    &\abs{K'(s+h+x)-K'(s+x)}
    \lesssim \int_0^h (s+u+x)^{H-5/2} \du 
    \le h^{1-1/p} \left(\int_0^h (s+u+x)^{p(H-5/2)}\du\right)^{1/p}\\
    &\lesssim h^{1-1/p} \abs{(h+x)^{p(H-5/2)+1}-x^{p(H-5/2)-1}}^{1/p}
    \lesssim h^{1-1/p} (s+x)^{H-5/2+1/p}.
\end{align*}
We observe that, by Fubini theorem and the estimate above
\begin{align*}
       &\int_0^t \abs{K'(s+h+\cdot) - K'(s+\cdot)}_{\Hh}^2 \ds
       \lesssim \int_0^\infty w(x) \left(\int_0^t h^{2-2/p} (s+x)^{2H-5+2/p} \ds \right)\dx\\
       &\lesssim  h^{2-2/p} \int_0^\infty \E^{-x}x^{\beta} x^{2H-4+2/p} \dx 
       = h^{\frac{2p-2}{p}} \Gamma(\beta+2H-3+2/p).
\end{align*}
Setting $q=\frac{2p}{2-p}$ we observe that the condition~$\beta+2H-3+2/p>0$ becomes $q<\frac{2}{2-2H-\beta}$, as the previous estimate. Since $p>1$ is arbitrary and~$\int_0^t \abs{K(s+h+\cdot)-K(s+\cdot)}_{\Hh}^2\ds\lesssim h^{\frac{2p-2}{p}}$, the estimate~\eqref{eq:cond_K2} is satisfied for any $q$ in the non-empty interval~$(2,\frac{2}{2-2H-\beta})$.

Lastly \eqref{eq:cond_K3} is satisfied thanks to \eqref{eq:cond_K1} and the following. For any $t,h>0$ we have 
    \begin{equation*}
        \begin{aligned}        \int_0^t|\partial_x^2S(s+h)K|^2_{\Hh}\ds
        =\int_0^t\int_0^\infty K''(s+h+x)^2w(x)dx\ds
    &=\int_0^\infty w(x)\int_0^t (s+h+x)^{2H-5}\ds \dx\\&
  \lesssim \int_0^\infty x^\beta e^{-x}(h+x)^{2H-4} \dx,
        \end{aligned}
    \end{equation*}
    where we used Fubini's theorem for the variables $x,s.$ It follows that
     \begin{equation*}
\int_0^t|\partial_x^2S(s+h)K|^2_{\Hh}\ds
\lesssim \int_0^\infty(h+x)^{\beta+2H-4} \E^{-x}\dx
\lesssim h^{\beta+2H-3}.
    \end{equation*}
This yields \eqref{eq:cond_K3} 
since the weight condition $\beta>1-2H$ implies $\beta+2H-3>-2$ and a similar estimate holds for~$\int_0^t|\partial_xS(s+h)K|^2_{\Hh}\ds$.
\end{example}

\section{Well-posedness and properties of the SPDE}
\label{sec:SPDE_lift}

This section is devoted to the study of the transport-type system of SPDEs 
\begin{equation}\label{eq:SPDE}
    \left\{\begin{aligned}
        &\partial_t\lambda(t,x)=\partial_x \lambda(t,x)+Kb(\lambda(t,0))+K\sigma (\lambda(t,0))\dot{W}_t\\&
        \lambda(0,x)=\lambda_0(x),\;\; x\in\RR^+,
    \end{aligned}\right.
\end{equation}
where $W$ is a standard $m-$dimensional Wiener process and the initial datum $\lambda_0$ is an $\Hh_1-$valued, $\Ff_0-$measurable random variable.
Before describing an appropriate notion of mild solutions, we introduce some relevant tools from semigroup theory.

\subsection{Shift semigroup on weighted Sobolev spaces and mild solutions}\label{sec:MildWellPosedness}

For $t\geq 0, f\in L^1_{loc}(\RR^+)$ we introduce the left-shift map
\begin{equation}\label{eq:ShiftSemigroup}
    S(t)f(x):=f(x+t),\;\;x\in\RR^+.
\end{equation}
To be slightly more precise, $S(t)$ can first be defined for $f\in \Cc(M)$ for any compact $M\subset\RR^+$ and then uniquely extended by density to $L^1(M)$.

\begin{lemma}\label{lemma:semigroup_pties} If $w\in\Ww_a,$ the following hold:
\begin{enumerate}
\item Let $t\geq 0,$ $C_t:=\sup_{s\geq t
}\frac{w(s-t)}{w(s)}.$ Then $S(t)\in\mathscr{L}(\Hh)$ and for all $f\in\Hh$
$$|S(t)f|_{\Hh}\leq C_t|f|_{\Hh}.$$
    \item The family $\{S(t)\}_{t\geq 0}\subset \mathscr{L}(\Hh)$ defines a strongly continuous semigroup.
    \item The derivative operator $\partial_x: Dom(\partial_x)\subset \Hh\rightarrow \Hh $ generates $S$ on $\Hh$ and $Dom(\partial_x),$ endowed with the graph norm,  coincides with $\Hh_1$ as a Banach space.    
\end{enumerate}
The same conclusions hold with the pair $(\Hh, \Hh_1)$ replaced by $(\Hh_m, \Hh_{m+1})$ for some $m\in\NN.$
\end{lemma}

We postopone the proof to Appendix \ref{app:shift}.

\begin{remark}\label{rem:H1Semigroup} From a glance at the proof of Lemma \ref{lemma:semigroup_pties} we see that the shift semigroup is strongly continuous when viewed as a family of operators in $\mathscr{L}(\Hh_1).$ Indeed for any $f\in\Hh_1,$ $S(t)f$ is weakly differentiable with $\partial_x S(t)f=S(t)f'.$ Moreover, since $f'\in\Hh,$ strong continuity of $S$ implies that $|S(t)f'|_{\Hh}\leq C_t|f'|_{\Hh}$ for a constant $C_t>0.$ Clearly, the same conclusions hold with $\Hh_1$ replaced by~$\Hh_m$ and $f'$ replaced by higher order derivatives.  
\end{remark}
\begin{corollary}\label{cor:SemigroupBound} There exist $c, C>0$ such that for all $t\geq 0, f\in\Hh_1$
$$|S(t)f|_{\Hh}\leq Ce^{c t}|f|_{\Hh}. $$
In particular, the constant $C_t$  introduced in Lemma \ref{lemma:semigroup_pties}(i) satisfies
$C_t\le C\E^{c t}.$ The same conclusion (albeit with possibly different constants) holds if $\Hh$ is replaced by $\Hh_1.$
\end{corollary}
\begin{proof} Both conclusions follow directly from the Hille-Yosida theorem (see e.g. \cite[Theorem A.3]{da2014stochastic}). 
\end{proof}

\begin{corollary}\label{cor:HmDensity} Let $w\in\mathcal{W}_a$ be an admissible weight per Definition \ref{assumption:admissibleweights} and $m\in\NN.$ Then, $\Cc_c^\infty(0,\infty)\subset \Hh_{m}$ is dense in $\Hh.$  Moreover, $\Hh_{m+1}$ is dense in $\Hh_m.$ 
\end{corollary}

The proof of this corollary is postponed to Appendix \ref{app:density}.

\begin{remark} We stress here that $\Cc_c^\infty(0,\infty)$ is not dense in the spaces $\Hh_m$ for $m\in\NN$ but it is dense as a subset of $\Hh.$ Indeed, if the former was true, then any $f\in\Hh_m$ would have a unique continuous representative with $ev_0f=f(0)=0.$ Since we do not want to prescribe fixed values at the origin, we also avoid defining $\Hh_m$ as closures of smooth, compactly supported test functions.
\end{remark}

Before we define an appropriate notion of solutions to \eqref{eq:SPDE}, we shall recast it into a stochastic evolution equation on the space $\Hh_1$ as follows:
\begin{equation}\label{eq:SPDE-evolution}
    \left\{\begin{aligned}
        &d\lambda(t)=\big(\partial_x \lambda(t)+Kb(ev_0\lambda(t))\big)\dt+K\sigma (ev_0\lambda(t))d W_t\\&
        \lambda(0)=\lambda_0\in\Hh_1,
    \end{aligned}\right.
\end{equation}
where the stochastic differential is interpreted in the It\^o sense. \noindent A few remarks on the form of \eqref{eq:SPDE-evolution} are in order. On the one hand we have, by virtue of Assumption \ref{assumption:SVEassumptions} and Lemma \ref{lemma:Hwproperties}, that the nonlinear operators 
$$ \Hh_1\ni x\longmapsto b(ev_0(x))\in\RR^d, \sigma(ev_0(x))\in\mathscr{L}(\RR^{m},\RR^d)      $$
are Lipschitz continuous. On the other hand, $\partial_x K$ is not necessarily in $L^2_w(\RR^+;\RR^{d\times d})$ (this can be readily checked in the case of power-law kernels; see Example \ref{example:powerlaw}). In particular, this means we only know that
$$  \Hh_1\ni x\longmapsto Kb(ev_0(x))\in\Hh, K\sigma(ev_0(x))\in L^2_w(\RR^+;\RR^{d\times m})$$
are Lipschitz continuous as $L^2-$valued maps.
Nevertheless, Assumption \ref{assumption:Kernel} is chosen to guarantee that the action of the shift operator on the kernel is regularizing, i.e. for all $t>0,$ $S(t)\partial_xK=\partial_x(S(t)K)\in L^2_w(\RR^+;\RR^{d\times d})$. Thus, even though the SPDE \eqref{eq:SPDE-evolution} is beyond the reach of the classical theory (which can be found, for example, in \cite[Theorem 7.2]{da2014stochastic}), it is still possible to obtain unique  $\Hh_1-$valued mild solutions.

For this reason, we present below a well-posedness result for \eqref{eq:SPDE-evolution} along with a detailed proof. Before doing so we introduce first the notion of mild solutions.

\begin{definition}\label{dfn:mildsolutions} Let $T>0$. We say that an adapted, $\Hh_1-$valued process $\{\lambda(t)\;;t\in[0,T]\}$ is a mild solution to \eqref{eq:SPDE-evolution}   if the following hold:
\begin{enumerate}[i)]
    \item There exists $p\geq 2$ such that $$\PP\bigg(\int_0^T|\lambda(t)|^p_{\Hh_1}\dt<\infty\bigg)=1.$$
    \item For any $t\in[0,T]$ we have 
\begin{equation}\label{eq:SPDE_mild}
\begin{aligned} \lambda(t)=S(t)\lambda_0+\int_0^tS(t-s)Kb_0\big(\lambda(s)\big)\ds+\int_0^tS(t-s)K\sigma_0\big(\lambda(s)\big)\dW_s\;\;\;\PP-\text{a.s.}
\end{aligned}
\end{equation}
with
$b_0:=b\circ ev_0$, $\sigma_0:=\sigma\circ ev_0$.
\end{enumerate}
If $\lambda$ is a unique mild solution with $\lambda(0)=\lambda_0$, we shall often use the notation $\lambda^{0,\lambda_0},$ instead of $\lambda,$ to emphasise the dependence on the initial condition. 
\end{definition}

\begin{theorem}[SPDE well-posedness]\label{thm:SPDE_wellposedness} Let $q>2$ as in Assumption \ref{assumption:Kernel}, $p\geq\frac{2q}{q-2},$ $\lambda_0\in L^p(\Omega;\Hh_1)$ be an $\Ff_0-$measurable initial condition, $T>0.$ Under Assumptions \ref{assumption:SVEassumptions} i)-ii) and \ref{assumption:Kernel} i), there exists a mild solution $\{\lambda(t); t\in[0,T]\}$ to~\eqref{eq:SPDE-evolution}  that satisfies
\begin{equation}\label{eq:lambdaestimate}
    \begin{aligned}
\sup_{t\in[0,T]}\EE\big[ |\lambda(t)|^p_{\Hh_1}   \big]\leq C\bigg(1+\EE|\lambda_0|^p_{\Hh_1}\bigg)    \end{aligned}
\end{equation}
for a deterministic constant $C=C(T,K,w, p)>0.$ Moreover, $\lambda$ is unique in the space $\mathscr{H}^p_T(\Hh_1)$ \eqref{eq:HpSpaces} in the sense that, for any two mild solutions $\lambda_1, \lambda_2$ to \eqref{eq:SPDE-evolution},  $|\lambda_2-\lambda_1|_{\mathscr{H}^p_T(\Hh_1)}=0.$
\end{theorem}

\begin{proof} Recall that $b_0,\sigma_0$ denote  $b\circ ev_0, \sigma\circ ev_0 $ respectively (as in Definition \ref{dfn:mildsolutions}). We start with the proof of existence of solutions which we show by a fixed point argument on the space $\mathscr{H}^p_T(\Hh_1)$ \eqref{eq:HpSpaces}. To this end we consider, for  $\lambda\in \mathscr{H}^p_T(\Hh_1), 
\lambda_0\in L^p(\Omega;\Hh_1)$ a map 
\begin{equation}\label{eq:lambdaSolutionMap}
    \Ii(\lambda_0, \lambda)(t):=S(t)\lambda_0+\int_0^t S(t-s)Kb_0(\lambda(s))\ds+ \int_0^t S(t-s)K\sigma_0(\lambda(s))\dW_s,\;\;t\geq 0   \end{equation}
    and show that, for each $\lambda_0\in L^p(\Omega;\Hh_1),$ $ \Ii(\lambda_0, \cdot): \mathscr{H}^p_T(\Hh_1)\rightarrow \mathscr{H}^p_T(\Hh_1)$ is a contraction. 

    First, we show that $\Ii(\lambda_0, \cdot)$ maps $\mathscr{H}^p_T(\Hh_1)$ to itself. 
    Indeed, in view of Assumption \ref{assumption:SVEassumptions}(ii), there exists a constant $C_l>0$ such that for all $x\in\RR^d$
    $|b(x)|+|\sigma(x)|\leq C_l(1+|x|).$ Thus, by continuity of the linear evaluation functional $ev_0$ on $\Hh_1,$ Lemma \ref{lemma:Hwproperties}, we have 
    \begin{equation}\label{eq:sigma0growth}
        |b_0(\lambda
    )|+|\sigma_0(\lambda)|\leq C_l(1+|ev_0\lambda|)\leq C'_l(1+|\lambda|_{\Hh_1})
    \end{equation}
    for all $\lambda\in\Hh_1.$ Moreover, Assumption \ref{assumption:Kernel} implies that $S(t-s)K\in\Hh_1$ for all $s\leq t\leq T$ and the almost sure estimate
    \begin{equation}\label{eq:SolutionMapBound}
        \begin{aligned}
|\Ii(\lambda_0, \lambda)&(t)|_{\Hh_1}\leq |S(t)\lambda_0|_{\Hh_1}+ \int_0^t |S(t-s)K|_{\Hh_1}|b_0(\lambda(s))|\ds+ \bigg|\int_0^t S(t-s)K\sigma_0(\lambda(s))\dW_s\bigg|_{\Hh_1}\\&
\lesssim |\lambda_0|_{\Hh_1}+\int_0^t |S(t-s)K|_{\Hh_1}|\big(1+|\lambda(s)|_{\Hh_1}\big)\ds+ \bigg|\int_0^t S(t-s)K\sigma_0(\lambda(s))\dW_s\bigg|_{\Hh_1}
        \end{aligned}
    \end{equation}
    follows by strong continuity of $S$ on $\Hh_1$ (Lemma \ref{lemma:semigroup_pties}). Turning to the drift term on the right-hand side, we use Assumption \ref{assumption:Kernel}(and in particular the fact that $K\in L^2_w(\RR^+;\RR^{d\times d})$), along with H\"older's inequality,   to obtain the bound
     \begin{equation*}
        \begin{aligned}
\int_0^t|S(t-s)K|_{\Hh_1}\big(1+|\lambda(s)|_{\Hh_1}\big)\ds&=\int_0^t\big(|S(t-s)K|_{\Hh}+ |\partial_x S(t-s)K|_{\Hh}|\big)\big(1+|\lambda(s)|_{\Hh_1}\big)\ds\\&
 \lesssim \bigg(\int_0^t|K|^q_{\Hh}+|\partial_x S(t-s)K|_{\Hh}^q\ds\bigg)^{\frac{1}{q}}  \bigg(\int_0^t 1+|\lambda(s)|^{\frac{q}{q-1}}_{\Hh_1}\ds\bigg)^{\frac{q-1}{q}}
\\&
 \lesssim 1+\bigg(\int_0^T|\lambda(s)|^{\frac{q}{q-1}}_{\Hh_1}\ds\bigg)^{\frac{q-1}{q}},
     \end{aligned}
    \end{equation*}
which holds up to constants that depend on $T, q$ and the term involving $K$ is finite from \eqref{eq:cond_K1}. Since $p\geq \tfrac{q}{q-1},$ we apply H\"older's inequality with exponent $p(q-1)/q$ on the last integral, raise to the $p-$th power and then take expectation to arrive at the bound  
 \begin{equation}\label{eq:SolutionMapDriftBnd}
        \begin{aligned}
\sup_{t\in[0,T]}\EE\bigg|\int_0^t S(t-s)Kb_0(\lambda(s))\ds\bigg|^p_{\Hh_1}\lesssim 1+\int_0^T\EE|\lambda(s)|^p_{\Hh_1}\ds\lesssim 1+ |\lambda|^p_{\mathscr{H}^p_T}<\infty.
         \end{aligned}
    \end{equation}
    As for the stochastic integral in \eqref{eq:SolutionMapBound}, an application of the BDG inequality and \eqref{eq:sigma0growth}, followed by H\"older's inequality with exponent $q/2>1$ for the Riemann integrals furnish
    \begin{equation*}
        \begin{aligned}
\EE\bigg|\int_0^t S(t-s)K&\sigma_0(\lambda(s))\dW_s\bigg|^p_{\Hh_1}\lesssim \EE\bigg( \int_0^t |S(t-s)K|^2_{\Hh_1}|\sigma_0(\lambda(s))|^2\ds   \bigg)^{\frac{p}{2}}
\\&\lesssim  \bigg( \int_0^t |S(t-s)K|^q_{\Hh_1}\ds\bigg)^{\frac{p}{q}} \EE\bigg( \int_0^T 1+|\lambda(s)|_{\Hh_1}^{\frac{2q}{q-2}}\ds\bigg)^{\frac{p(q-2)}{2q}},
         \end{aligned}
    \end{equation*}
where the integral involving $K$ is finite by \eqref{eq:cond_K1}. Since $p\geq\frac{2q}{q-2},$ we may apply H\"older's inequality with exponent $\frac{p(q-2)}{q}$ on the last integral and take supremum over $s\in[0,T]$ to obtain
\begin{equation}\label{eq:SolutionMapDiffBnd}
        \begin{aligned}
\sup_{t\in[0,T]}\EE\bigg|\int_0^t S(t-s)K&\sigma_0(\lambda(s))\dW_s\bigg|^p_{\Hh_1}\lesssim 1+|\lambda|^p_{\mathscr{H}^p_T}<\infty.
  \end{aligned}
    \end{equation}
Finally, from the above computations (and in particular from the strong continuity of $S$ on $\Hh_1,$ and BDG inequalities) it follows that 
$[s,T]\ni t\mapsto \Ii(\lambda_0, \lambda)(t)\in L^p(\Omega;\Hh_1)$
is continuous. This, along with~\eqref{eq:SolutionMapBound}, \eqref{eq:SolutionMapDriftBnd}, \eqref{eq:SolutionMapDiffBnd} imply that 
$$ |\Ii(\lambda_0, \lambda)|^p_{\mathscr{H}^p_T}\lesssim 1+|\lambda_0|^p_{L^p(\Omega)}+|\lambda|^p_{\mathscr{H}^p_T}<\infty   $$
for all $p\geq \frac{2q}{q-2}$ so that $\Ii(\lambda_0,\cdot)$ maps $\mathscr{H}^p_T$ to itself for all such $p.$

We proceed to show that the solution map $\Ii$ is a contraction. To this end, we fix $\lambda_1, \lambda_2\in\mathscr{H}^p_T.$ Starting from the drift we have 
    \begin{equation*}
        \begin{aligned}
            \bigg|\int_0^t S(t-s)K\bigg[b_0(\lambda_1(s))-b_0(\lambda_2(s))\bigg]\ds\bigg|^p_{\Hh}&\leq \bigg(\int_0^t \big|S(t-s)K\big|_{L^2_w}\big|b_0(\lambda_1(s))-b_0(\lambda_2(s))\big|\ds\bigg)^p\\&
            \leq C_T|K|^p_{L^2_w}|b|^p_{Lip} \bigg(\int_0^t|ev_0(\lambda_1(s)-\lambda_2(s))|\ds\bigg)^p\\&
            \lesssim  \bigg(\int_0^t|\lambda_1(s)-\lambda_2(s)|_{\Hh_1}\ds\bigg)^p,
        \end{aligned}
    \end{equation*}
    where we used Assumption \ref{assumption:SVEassumptions} and Lemma \ref{lemma:Hwproperties}. From Corollary \ref{cor:SemigroupBound} it follows that 
    \begin{equation} \label{eq:b-estimate}
        \begin{aligned}
            \EE\bigg|&\int_0^t S(t-s)K\bigg[b_0(\lambda_1(s))-b_0(\lambda_2(s))\bigg]\ds\bigg|^p_{\Hh}
            \lesssim |\lambda_1-\lambda_2|^p_{\mathscr{H}^p_T}.
            \end{aligned}
    \end{equation}

     Turning to the stochastic convolution, we apply the BDG inequality, along with Lemmas \ref{lemma:semigroup_pties}, \ref{lemma:Hwproperties} to obtain 
    \begin{equation}\label{eq:sigma-estimate}
        \begin{aligned}
            \EE\bigg| \int_0^t S(t-s)K\bigg[\sigma_0(\lambda_1(s))-\sigma_0(\lambda_2(s))\bigg]&\dW_s\bigg|^p_{\Hh}\lesssim \bigg(\int_0^t |S(t-s)K|^2_{L^2_w}\EE\big|\sigma_0(\lambda_1(s))-\sigma_0(\lambda_2(s))\big|^2\ds\bigg)^{p/2} \\&
            \leq C_T|K|^p_{L^2_w}|\sigma_0|^p_{Lip}\bigg(\int_0^t \EE\big|\lambda_1(s)-\lambda_2(s)\big|_{\Hh_1}^2 \ds\bigg)^{p/2}\\&
            \lesssim 
            \bigg(\int_0^t \EE\big|\lambda_1(s)-\lambda_2(s)\big|_{\Hh_1}^{2} \ds\bigg)^{p/2}\lesssim |\lambda_1-\lambda_2\big|_{\mathscr{H}^p_T}^p.
            \end{aligned}
    \end{equation}

As mentioned after Corollary \ref{cor:SemigroupBound}, $K'$ is not necessarily in $\Hh$ and this leads to slightly different arguments for estimating the spatial derivatives of the drift and diffusion terms. Following the rationale behind estimates \eqref{eq:SolutionMapDriftBnd} and \eqref{eq:SolutionMapDiffBnd},
 Cauchy--Schwarz and BDG inequalities yield
\begin{align*}
    &\abs{\int_0^t \partial_x S(t-s)K\Big[ b_0(x_1(s))-b_0(x_2(s))\Big]\ds }^p_{\Hh}
    \le t^{p/2} \left(\int_0^t  \abs{\partial_xS(t-s)K}_{\Hh}^2\abs{ b_0(\lambda_1(s))-b_0(\lambda_2(s))}^2\ds \right)^{p/2},
    \end{align*}
    \begin{align*}
    \EE\bigg|\int_0^t\partial_x S(t-s)K&\Big[ \sigma_0(\lambda_1(s))-\sigma_0(\lambda_2(s))\Big]\D W_s \bigg|^p_{\Hh}
    \\&\le C_p \EE \left(\int_0^t  \abs{\partial_xS(t-s)K}_{\Hh}^2\abs{ \sigma_0(\lambda_1(s))-\sigma_0(\lambda_2(s))}^2\ds \right)^{p/2}.
\end{align*}
By Lipschitz continuity of $b, \sigma$ and continuity of the linear functional $ev_0:\Hh_1\rightarrow \RR$ it follows that 
$$ |b_0(\lambda_2(s))- b_0(\lambda_1(s))|+ |\sigma_0(\lambda_2(s))- \sigma_0(\lambda_1(s))|\lesssim |ev_0(\lambda_2(s)-\lambda_1(s))|\lesssim |\lambda_2(s)-\lambda_1(s)|_{\Hh_1} $$ for all $s\in[0,t].$ From the latter, along with H\"older's inequality with exponent $q/2>1,$ both terms are bounded by
\begin{align*}
    C \left(\int_0^t \abs{\partial_xS(t-s)K}^q_{\Hh} \ds\right)^{p/q} \left(\int_0^t \abs{\lambda_1(s)-\lambda_2(s)}_{\Hh_1}^{\frac{2q}{q-2}}\ds \right)^{\frac{p}{2}\frac{q-2}{q}}.
\end{align*}
From H\"older's inequality with exponent $p(q-2)/2q>1$ we get 
\begin{align*}
    \left(\int_0^t \abs{\lambda_1(s)-\lambda_2(s)}_{\Hh_1}^{\frac{2q}{q-2}}\ds \right)^{\frac{p}{2}\frac{q-2}{q}}
    \le T^{\frac{p(q-2)}{2q}-1}\int_0^t \abs{\lambda_1(s)-\lambda_2(s)}_{\Hh_1}^{p}\ds, 
\end{align*}
while $\int_0^t \abs{\partial_x S(t-s)K}^q_{\Hh} \ds$ is finite by Assumption~\ref{assumption:Kernel}. Thus,   
    \begin{equation}\label{eq:b-sigma-derivative-estimate}
        \begin{aligned}      \sup_{t\in[0,T]}\EE\bigg|&\partial_x\int_0^t S(t-s)K\bigg[b_0(\lambda_1(s))-b_0(\lambda_2(s))\bigg]\ds\bigg|^p_{\Hh}\\&
        +\EE\abs{\int_0^t\partial_x S(t-s)K\Big[ \sigma_0(\lambda_1(s))-\sigma_0(\lambda_2(s))\Big]\D W_s }^p_{\Hh}
    \leq C_T|\lambda_1-\lambda_2|_{\mathscr{H}^p_T}^p. 
        \end{aligned}
    \end{equation}
    for a constant $C_T>0$ that is increasing in $T.$

    Combining \eqref{eq:b-estimate}, \eqref{eq:sigma-estimate}, \eqref{eq:b-sigma-derivative-estimate} we deduce that 
    \begin{equation*}
        \big|\Ii(\lambda_0, \lambda_1)-\Ii(\lambda_0, \lambda_2)\big|^p_{\mathscr{H}^p_T}\leq F(T)|\lambda_1-\lambda_2|_{\mathscr{H}^p_T}^p
    \end{equation*}
    where $F(T)$ is a continuous increasing function with $F(0)=0$. Hence, by choosing $T_0>0$ small so that $F(T_0)<1$ and $T=T_0$ it follows that 
$\Ii(\lambda_0,\cdot):\mathscr{H}^p_{T_0}\rightarrow\mathscr{H}^p_{T_0}$
    is a contraction mapping. By the Banach fixed point theorem it follows that, for all $\lambda_0,$ there exists a unique $\lambda\in\mathscr{H}^p_{T_0}(\Hh_1)$ such that
    $$\forall t\in[0,T_0],
    \;\;\PP\bigg(\Ii(\lambda_0,\lambda)(t)=\lambda(t)\bigg)=1.$$      The restriction on $T$ can then be removed by concatenating unique solutions on time-intervals of the form $[kT_0, (k+1)T_0]$ for $k\in\NN.$

    Finally, the estimate  \eqref{eq:lambdaestimate} follows from the estimates of the solution map. Indeed, from the proofs of \eqref{eq:SolutionMapDriftBnd}, \eqref{eq:SolutionMapDiffBnd} and the fixed point property of unique solutions we have 
    \begin{equation*}
        \begin{aligned}
        \sup_{t\in[0,T]}\EE |\lambda(t)|^p_{\Hh_1}=  |\lambda|^p_{\mathscr{H}^p_T} = |\Ii(\lambda_0, \lambda)|^p_{\mathscr{H}^p_T} \lesssim 1+\int_0^T\sup_{s\in[0,t]}\EE|\lambda(s)|^p_{\Hh_1}\ds
        \end{aligned}
    \end{equation*}
    and Gr\"onwall's inequality directly implies the desired bound. In turn, this shows that the unique solution satisfies the first requirement from Definition \ref{dfn:mildsolutions}.  The proof is complete.
\end{proof}

\begin{remark} We conclude this subsection with some observations on Theorem \ref{thm:SPDE_wellposedness}:
\begin{enumerate}
    \item[i)] We only proved uniqueness of mild solutions in the space $\mathscr{H}^p_T.$ In particular, this implies that any two mild solutions $\lambda_1, \lambda_2,$ per Definition \ref{dfn:mildsolutions}, must satisfy $\PP(\lambda_1(t)=\lambda_2(t))=1$ for all $t\geq 0.$ However, as we shall show in Lemma \ref{lemma:pathregularity} below, mild solutions are in fact continuous in $t.$ This in particular implies that $\lambda_1, \lambda_2$ are indistinguishable.
    \item[ii)] If the initial condition $\lambda_0$ is deterministic, then the a-priori bound  \eqref{eq:lambdaestimate} holds for all $p\geq 1.$ Indeed, for $p\geq 2q/(q-2)$ and any $p'<p$ we have 
    $$  \sup_{t\in[0,T]}\EE\big[ |\lambda(t)|^{p'}_{\Hh_1}   \big]\leq \sup_{t\in[0,T]}\big(\EE\big[ |\lambda(t)|^p_{\Hh_1}   \big]\big)^{\frac{p'}{p}}\lesssim \big(1+|\lambda_0|^p_{\Hh_1}\big)^{\frac{p'}{p}}\lesssim 1+|\lambda_0|^{p'}_{\Hh_1}.$$
\end{enumerate}
\end{remark}
\begin{remark}
    If instead one were to consider a regular kernel $K\in H_w^1(\RR^+,\RR^{d\times d})$ (e.g.  $K(t)=t^{H-\half},\,H\in[\half,1)$), then one would again fall back on the standard framework of~\cite{da2014stochastic} and recover semimartingality of the solution in~$\Hh_1$. 
\end{remark}

\subsection{Lift property and temporal regularity of mild solutions}\label{Sec:LiftRegularity}
In view of Theorem \ref{thm:SPDE_wellposedness}, we see that, for each $T>0$ the unique mild solution $\{\lambda(t);t\in[0,T]\},$ takes values in $L^\infty([0,T];\Hh_1)$ with probability $1$ (recall \eqref{eq:HpSpaces}) and defines a proper infinite-dimensional lift of the SVE~\eqref{eq:SVE}. Indeed, since $\Hh_1$ is an RKHS, the evaluation operator $ev_0:\Hh_1\rightarrow\RR^d$ at $x=0$ is continuous and for all $\lambda_0\in L^{\frac{2q}{q-2}}(\Omega;\Hh_1)$
\begin{equation} \label{eq:lift_property}
\begin{aligned} ev_0(\lambda(t))&=ev_0S(t)\lambda_0+\int_0^tev_0[S(t-s)K]b(ev_0\lambda(s))\big)\ds+\int_0^tev_0[S(t-s)K]\sigma(ev_0\lambda(s))\big)\dW_s\\&=\lambda_0(t)+\int_0^tK(t-s)b(ev_0\lambda(s))\big)\ds+\int_0^tK(t-s)\sigma(ev_0\lambda(s))\big)\dW_s.
\end{aligned}
\end{equation}
We assume from now on that $\lambda_0=X_0$; from uniqueness of solutions to \eqref{eq:SVE} and sample path continuity, it follows that
\begin{equation}\label{eq:liftproperty}
    \PP\bigg(\forall t\in[0,T],\;ev_0(\lambda(t))=X_t\bigg)=1.
\end{equation}
This property means $X$ is a projection of $\lambda$, in fact it also justifies that we call $\lambda$ a \emph{lift} of $X$.


\begin{lemma}\label{lemma:pathregularity}
Let $p\geq 2q/(q-2)$ and $\lambda_0\in L^p(\Omega;\Hh_1)$ be an $\Ff_0-$measurable random variable and $T>0.$ Under Assumptions \ref{assumption:SVEassumptions} i)-ii) and \ref{assumption:Kernel} i)-ii) the following hold:
    \begin{enumerate}
        \item $\lambda\in \Cc([0,T];\Hh)$ and $\lambda-S(\cdot)\lambda_0\in \Cc^\gamma([0,T];\Hh)$ for all~$\gamma<\half$, $\PP$-almost surely.
        \item 
        $\lambda\in \Cc([0,T];\Hh_1)$ and $\lambda-S(\cdot)\lambda_0\in \Cc^\gamma([0,T];\Hh_1)$ for all~$\gamma<\frac{q-2}{2q}$, $\PP$-almost surely. 
    \end{enumerate}
\end{lemma}

\begin{proof}
Throughout the proof we shall use the notation $b_0,\sigma_0$ for the terms $b\circ ev_0, \sigma\circ ev_0 $ respectively. 
    For all $p>2$, $h>0$ and~$t\in[0,T-h]$, we have
    \begin{align*}
        \EE&\abs{\big(\lambda(t+h)-S(t+h)\lambda_0\big)-\big(\lambda(t)-S(t)\lambda_0\big)}^p_{\Hh} 
        \\
       & \lesssim\;\EE \abs{ \int_0^t \big(S(t+h-s)K - S(t-s)K\big) b_0(\lambda(s))\ds }^p_{\Hh} + \EE\abs{\int_t^{t+h} S(t+h-s)K b_0(\lambda(s))\ds}^p_{\Hh} \\
        &+ \EE\abs{ \int_0^t \big(S(t+h-s)K - S(t-s)K\big) \sigma_0(\lambda(s))\D W_s }^p_{\Hh} + \EE\abs{\int_t^{t+h} S(t+h-s)K \sigma_0(\lambda(s))\D W_s}^p_{\Hh}.
    \end{align*}
By Jensen, BDG and Hölder inequalities
\begin{align*}
    &\EE \abs{ \int_0^t \big(S(t+h-s)K - S(t-s)K\big) b_0(\lambda(s))\ds }^p_{\Hh} +\EE \abs{ \int_0^t \big(S(t+h-s)K - S(t-s)K\big) \sigma_0(\lambda(s))\D W_s }^p_{\Hh}\\
    &\lesssim \EE \abs{ \int_0^t \abs{S(t+h-s)K - S(t-s)K}_{\Hh}^2\big(\abs{b_0(\lambda(s))}^2+\abs{\sigma_0(\lambda(s))}^2\big)\ds }^{p/2}\\
    &\le \abs{S(t+h-\cdot)K - S(t-\cdot)K}_{L^2((0,t),\Hh)}^{p-2} \int_0^t \abs{S(t+h-s)K - S(t-s)K}_{\Hh}^2 \EE\Big[\abs{b_0(\lambda(s))}^p+\abs{\sigma_0(\lambda(s))}^p\Big] \ds \\
    &\lesssim \left(1+\sup_{s\in[0,T]}\EE\abs{\lambda(s,0)}^p\right) \left(\int_0^t \abs{S(t+h-s)K - S(t-s)K}_{\Hh}^2 \ds \right)^{p/2},
\end{align*}
where we exploited the linear growth of $b$ and $\sigma$. Notice that, in view of \eqref{eq:SVEmoments},~$\EE\abs{\lambda(s,0)}^p=\EE\abs{X_s}^p$ is uniformly bounded over~$s\in[0,T]$.  Furthermore,
\begin{align}
    &\int_0^t \abs{S(t+h-s)K - S(t-s)K}_{\Hh}^2 \ds
    = \int_0^t \int_0^\infty w(x) \abs{K(t+h-s+x)-K(t-s+x)}^2\dx \ds \nonumber\\
    &\quad = \int_0^t \int_0^\infty w(x) \abs{\int_0^h K'(s+x+r)\dr }^2\dx \ds
    \le h \int_0^h \int_0^t \int_0^\infty w(x) \abs{K'(s+x+r)}^2 \dx \ds \dr \nonumber\\
    &\quad \le h^{2} \sup_{r\in[0,h)}\int_0^t \abs{K'(s+r+\cdot)}_{\Hh}^2\ds \lesssim h^{2},
    \label{eq:diff_StK_H}
\end{align}
where we used Assumption~\ref{assumption:Kernel}. Similar computations yield 
\begin{align*}
    \EE&\abs{\int_t^{t+h} S(t+h-s)K b_0(\lambda(s))\ds}^p_{\Hh} + \EE\abs{\int_t^{t+h} S(t+h-s)K \sigma_0(\lambda(s))\D W_s}^p_{\Hh}\\
    &\lesssim \left(\int_t^{t+h} \abs{S(t+h-s)K}^2_{\Hh} \ds\right)^{p/2}=  \left(\int_0^h  \abs{S(s)K}^2_{\Hh}\ds\right)^{p/2}
    \le h^{p/2} (C\E^{c h})^{p/2} \abs{K}^p_{\Hh},
\end{align*}
where we used Corollary~\ref{cor:SemigroupBound} for the last inequality. Overall, these amount to
$$
\EE\abs{\big(\lambda-\lambda_0\big)(t+h) - \big(\lambda-\lambda_0\big)(t)}^p_{\Hh} \lesssim h^{p/2}.
$$
Kolmogorov continuity criterion thus asserts that $\lambda-\lambda_0$ admits a version which is~$\gamma-$Hölder continuous for any~$\gamma<1/2-1/p$; since~$p$ can be taken arbitrarily large this concludes the first claim.

Let us turn to regularity in $\Hh_1$. The same computations as above lead to
\begin{align*}
\EE&\abs{\partial_x \big(\lambda(t+h)-\lambda_0(t+h)\big)-\partial_x\big(\lambda(t)-\lambda_0(t)\big)}^p_{\Hh} \\
&\lesssim \left(\int_0^t \abs{\partial_x S(s+h) K - \partial_x S(s) K}^2_{\Hh}\ds \right)^{p/2}
+ \left(\int_0^{h}\abs{\partial_x S(s) K}^2_{\Hh}\ds\right)^{p/2}.
\end{align*}
Condition \eqref{eq:cond_K2} asserts that the first term is smaller than $C h^{\frac{p(q-2)}{2q}}$. H\"older's inequality yields 
\begin{align*}
    \left(\int_0^{h}\abs{\partial_xS(s) K}^2_{\Hh}\ds\right)^{p/2}
    \le \left( h^{q/2-1}\int_0^{h}\abs{\partial_x S(s)K}^q_{\Hh}\ds\right)^{p/q}
    \le h^{\frac{p(q-2)}{2q}} \left(\int_0^{T}\abs{\partial_xS(s) K}^q_{\Hh}\ds\right)^{p/q},
\end{align*}
where the last integral is finite by~\eqref{eq:cond_K1}. 
By virtue of Kolmogorov continuity criterion this proves the claim.

 Finally, strong continuity of the semigroup (Lemma \ref{lemma:semigroup_pties}) entails that $\PP$-almost surely, in both cases described above, $t\mapsto S(t)\lambda_0$ is continuous and hence so is~$\lambda$.
\end{proof}
\begin{corollary}
    Let $p\ge 2$ and $\lambda_0\in L^p(\Omega;\Hh_1)$ be an $\Ff_0-$measurable random variable and $T>0.$ Let Assumption \ref{assumption:SVEassumptions} i)-ii) hold, let~$K$ be the power-law kernel~$K(t)=t^{H-\half}$ with $H\in(0,\half)$ and let the weight be~$w(x)=x^{\beta} \E^{-x}$ with $\beta\in(1-2H,1)$. Then  $\lambda-S(\cdot)\lambda_0\in \Cc^\gamma([0,T];\Hh_1)$ for all~$\gamma<H+\frac{\beta-1}{2}$, $\PP$-almost surely.
\end{corollary}
\begin{proof}
    The first two conditions \eqref{eq:cond_K1} and \eqref{eq:cond_K2} of Assumption~\ref{assumption:Kernel} hold for any $q\in(2,\frac{2}{2-2H-\beta})$ which is equivalent to $\frac{q-2}{2q}\in(0,H+\frac{\beta-1}{2})$.
\end{proof}
\begin{remark}
For an arbitrary $\delta>0$, the Volterra process $X$, solution to~\eqref{eq:SVE}, has $(H-\delta)$-Hölder continuous trajectories. 
Since one has to fix a value of $\beta=1-\ep$ with $\ep>0$, $\lambda-S(\cdot)\lambda_0$ achieves ($H-\ep/2-\delta$)-Hölder regularity in $\Hh_1$  (which depends on $\ep>0$), for any $\delta>0$.
\end{remark}

\subsection{Invariant subspaces}\label{sec:invariance}

In this section we prove invariance properties for the solution flow of \eqref{eq:SPDE}. These are crucial for both our study of the backward Kolomogorov equation (Theorem \ref{thm:backward_equation_singular}) as well as the proof of the singular It\^o formula (Theorem \ref{thm:SingularIto}).

In the sequel, we shall distinguish the following linear subspaces of $\Hh, \Hh_1:$

\begin{align}
\label{eq:SingularDirections}
     &\mathscr{K}:=\big\{ h\in\Hh\;|\; \forall t>0: S(t)h\in\Hh_1 \big\}=\bigcap_{t>0}S(t)^{-1}(\Hh_1)\subset\Hh,\\
\label{eq:yspace}
&\mathscr{K}_1:=\bigg\{  y\in\Hh_1\bigg|\;   \partial_x y\in \mathscr{K},  \abs{S(\cdot)\partial_x y}_{\Hh_1}\in L^1_{loc}(\RR^+)\bigg\}.
\end{align}
Let us also recall the notation~$\Hh_2:=  H^2_w (\RR^+;\RR^d)$ \eqref{eq:HSpaceNotations} for the weighted Sobolev space of twice weakly differentiable functions with square integrable derivatives.
\begin{remark}\label{rem:InvariantSpaces} Since $\Hh_1\subset\mathscr{K}$ and $\Hh_1$ is dense in $\Hh$ (per Corollary \ref{cor:HmDensity})  $\mathscr{K}$ is a dense subspace of $\Hh.$ Moreover, we have the inclusions
$ \Hh_2\subset\mathscr{K}_1\subset \Hh_1,$
the second of which is trivial, while the first follows from the strong continuity of $S,$ since  
$$\int_0^T|\partial_xS(t)y|_{\Hh_1}\dt\leq\int_0^T|S(t)y|_{\Hh_2}\dt\lesssim |y|_{\Hh_2}\int_0^Te^{ct}\dt<\infty$$
for any $y\in\Hh_2, T>0.$    
\end{remark}

We are now ready to show that $\mathscr{K}_1$ is invariant under the solution flow of the SPDE \ref{eq:SPDE}. The subspace $\mathscr{K}\subset\Hh$ mainly serves as the collection of ``singular" directions for our notion of singular directional derivatives (see Section \ref{subsec:singularDifferentiablity} below).

\begin{proposition}[Invariant subspace]\label{prop:invariant subspaces} Let Assumptions \ref{assumption:SVEassumptions} i)-ii) and \ref{assumption:Kernel} hold. The linear subspace $\mathscr{K}_1\subset\Hh_1$ \eqref{eq:yspace} is invariant for the unique mild solution \eqref{eq:SPDE_mild}. 
\end{proposition}
\begin{proof}
For any~$y\in\Hh_1$, $\lambda^{0,y}$ is the mild solution \eqref{eq:SPDE_mild} started at~$\lambda_0=y$. We will show that~$\mathscr{K}_1$ is invariant i.e., for any initial datum $y\in\mathscr{K}_1,$ 
$\lambda^{0,y}(T)\in \mathscr{K}_1$ almost surely for any $T>0.$  
To this end, we fix an initial datum $y\in \mathscr{K}_1$ and proceed by showing that each term on the right-hand side of \eqref{eq:SPDE_mild} is an element of $\mathscr{K}_1$ almost surely.

Indeed, for any $T>0$ we have 
$S(T)y\in\Hh_1, \partial_x S(T)y=S(T)\partial_xy\in\Hh.$ Moreover, for all $t>0$ we have $S(t)\partial_x S(T)y=S(T)[S(t)\partial_xy]\in\Hh_1,$ since $S(t)\partial_xy\in\Hh_1,$ and for any $T_0>0$ we have
$$\int_0^{T_0} |S(t)\partial_xS(T)y|_{\Hh_1}\dt\lesssim \int_0^{T_0} |S(t)\partial_xy|_{\Hh_1}\dt<\infty    $$
which follows from the fact that $y\in\mathscr{K}_1$ and Corollary \ref{cor:SemigroupBound}. Thus, $S(T)y\in \mathscr{K}_1$ for any $y\in\mathscr{K}_1.$ 

As for the drift term, the previous computations show that 
$\int_0^T S(T-s)Kb_0(\lambda^{0,y}(s))\ds\in\Hh_1,    $ hence $\partial_x\int_0^T S(T-s)Kb_0(\lambda^{0,y}(s))\ds\in\Hh$ almost surely and  we have by linear growth and  Assumption \ref{assumption:Kernel}
\begin{align*}   \EE\bigg|S(t)\partial_x\bigg(\int_0^T S(T-s)Kb_0(\lambda^{0,y}(s))\ds \bigg)\bigg|_{\Hh_1}
&\lesssim  
\sup_{s\in[0,T]}\EE\left[1+\abs{\lambda^{0,y}(s)}_{\Hh_1}\right] \int_0^T |S(t+s)K|_{\Hh_2}\ds\\
&\lesssim \bigg(\int_0^T |S(t+s)K|^2_{\Hh_2}\ds\bigg)^{\frac{1}{2}}<\infty
\end{align*}
Furthermore, for any $T_0>0$, Assumption \ref{assumption:Kernel} also entails
\begin{align*}
\int_0^{T_0}\bigg|S(t)\partial_x\bigg(\int_0^T &S(T-s)Kb_0(\lambda^{0,y}(s))\bigg)\ds\bigg|_{\Hh_1}\dt\lesssim \int_0^{T_0}\bigg(\int_0^T |S(t+s)K|^2_{\Hh_2}\ds\bigg)^\frac{1}{2}\dt<\infty.
\end{align*}
Thus, for any $T>0$ and $y\in \mathscr{K}_1,$  $\int_0^T S(T-s)Kb_0(\lambda^{0,y}(s))\ds\in\mathscr{K}_1  $ almost surely.

We conclude the proof by showing similar estimates for the stochastic integral. Indeed, we have already shown that $\int_0^T S(T-s)K\sigma_0(\lambda^{0,y}(s))\dW_s\in\Hh_1,    $ hence $\partial_x\int_0^T S(T-s)K\sigma_0(\lambda^{0,y}(s))\dW_s\in\Hh$ almost surely. By virtue of the BDG inequality and for any $t>0$ we have
\begin{equation*}
    \begin{aligned}    &\EE\bigg|S(t)\partial_x\int_0^T S(T-s)K\sigma_0(\lambda^{0,y}(s))\dW_s      \bigg|_{\Hh_1}\lesssim \bigg(\int_0^T |\partial_xS(t+T-s)K|^2_{\Hh_1}\EE|\sigma_0(\lambda^{0,y}(s))|^2\ds\bigg)^{1/2}\\&
    \quad 
        \lesssim \sup_{s\in[0,T]}\EE\left[1+\abs{\lambda^{0,y}(s)}_{\Hh_1}^2\right]^\half \bigg(\int_0^T |\partial_xS(t+T-s)K|^2_{\Hh_1}\ds\bigg)^{1/2}\leq \bigg(\int_0^T |S(t+s)K|^2_{\Hh_2}\ds\bigg)^{1/2},
    \end{aligned}
\end{equation*}
where we used linear growth of $\sigma$ and the last integral is finite by Assumption \ref{assumption:Kernel}. In a similar fashion, for each $t,T_0>0,$ 
\begin{equation*}
    \begin{aligned}
\EE\int_0^{T_0}&\bigg|S(t)\partial_x\int_0^T S(T-s)K\sigma_0(\lambda^{0,y}(s))\dW_s      \bigg|_{\Hh_1}\dt\\&=\int_0^{T_0}\EE\bigg|S(t)\partial_x\int_0^T S(T-s)K\sigma_0(\lambda^{0,y}(s))\dW_s      \bigg|_{\Hh_1}\dt\lesssim \int_0^{T_0} \bigg(\int_0^T |S(t+s)K|^2_{\Hh_2}\ds\bigg)^{1/2}   \dt,
     \end{aligned}
\end{equation*}
where we used Fubini's theorem for the expectation and Riemann integral with respect to $t.$ Again, the last integral is finite  by Assumption \ref{assumption:Kernel}. We have showed that for all $T>0$ and $y\in\mathscr{K}_1,$ $\int_0^T S(T-s)K\sigma_0(\lambda^{0,y}(s))\dW_s\in\mathscr{K}_1$ almost surely. In view of the preceding arguments, the same is true for $\lambda^{0,y}(T).$ The proof is complete.    \end{proof}

\subsection{The flow and Markov properties}\label{subsec:MarkovSemigroup}
In this section and the next, we are interested in the mild solution started at~$y\in L^p(\Omega;\Hh_1)$ at time~$s\ge0$ which reads
\begin{equation}\label{eq:generalised_SPDE}
    \lambda^{s,y}(t) = S(t-s)y + \int_s^t S(t-r)K b_0(\lambda^{s,y}(r))\dr + \int_s^t S(t-r)K \sigma_0(\lambda^{s,y}(r))\D W_r.
\end{equation}
Notice that the solution to~\eqref{eq:SPDE_mild} simply corresponds to~$\lambda^{0,\lambda_0}$ and that well-posedness and continuity of Theorem \ref{thm:SPDE_wellposedness} extends without modification to~\eqref{eq:generalised_SPDE}. 
From here onwards, we may omit the initial time if it is zero and write instead~$\lambda^y:=\lambda^{0,y}$.

Under the assumptions of Theorem~\ref{thm:SPDE_wellposedness}, the SPDE \eqref{eq:generalised_SPDE} has a unique probabilistically strong solution, hence there exists a measurable map~$F:\Hh_1\times \Cc(\RR^+;\RR^m)\to \Cc(\RR^+;\Hh_1)$ such that~$F(y,W):=\lambda^{0,y}$. 
Our first goal is to establish the following flow properties.
\begin{lemma}\label{lemma:flow_pties}
Let Assumptions \ref{assumption:SVEassumptions} i)-ii) and \ref{assumption:Kernel} i)-ii) hold for some $q>2$. Define the Brownian motion~$W^t_\tau := W_{\tau+t}-W_\tau$ for all~$\tau\ge0$.
    For all~$0\le s \le t$, $y\in L^p(\Omega;\Hh_1)$ with $p>\frac{2q}{q-2}$ the following relations hold $\PP$-almost surely in $\Cc(\RR^+;\Hh_1)$:
    \begin{enumerate}
        \item $\lambda^{s,y} =\lambda^{t,\lambda^{s,y}(t)}$;
        \item $\lambda^{s,y}(t+\cdot) =F(\lambda^{s,y}(t), W^t)$. 
    \end{enumerate}
\end{lemma}
\begin{proof}
For all~$0\le s \le t \le \theta\le T$, the semigroup property allows us to write
\begin{align*}
    \lambda^{s,y}(\theta) 
    &= S(\theta-t)\left(S(t-s)y + \int_s^t S(t-r) \Big[ K b_0(\lambda^{s,y}(r))\dr + K\sigma_0(\lambda^{s,y}(r))\D W_r\Big] \right)\\
    &\quad + \int_t^\theta S(\theta-r) \Big[  K b_0(\lambda^{s,y}(r))\dr+  K \sigma_0(\lambda^{s,y}(r))\D W_r\Big] \\
    &= S(\theta-t) \lambda^{s,y}(t)+ \int_t^\theta S(\theta-r) \Big[  K b_0(\lambda^{s,y}(r))\dr + K \sigma_0(\lambda^{s,y}(r))\D W_r\Big].
\end{align*}
This equation holds $\PP$-almost surely for all~$\theta\in[t,T]$ and is identical to the one satisfied by~$\lambda^{t,\lambda^{s,y}(t)}$; pathwise uniqueness thus entails~$\lambda^{s,y}=\lambda^{t,\lambda^{s,y}_t}$ $\PP$-almost surely. 
Furthermore, for all~$\theta>0$, a change of variables yields
\begin{align*}
    \lambda^{s,y}(t+\theta)
    &= S(\theta)\lambda^{s,y}(t)+ \int_0^{\theta} S(\theta-r) \Big[ Kb_0(\lambda^{s,y}(t+r))\dr + K\sigma_0(\lambda^{s,y}(t+r))\D W_{r}^t\Big].
\end{align*}
We note that~$\lambda^{s,y}(t+\cdot)$ satisfies the same equation as~$F(\lambda^{s,y}(t),W^t)$.
Pathwise uniqueness of \eqref{eq:generalised_SPDE} intervenes again and yields the claim.
\end{proof}
It appears from Lemma \ref{lemma:flow_pties} that $\lambda^{s,y}(t+\cdot)$ depends on $\Ff_t$ only through $\lambda^{s,y}(t)$. The Markov property makes this observation more precise. For $\varphi:\Hh_1\to\RR$ bounded measurable, let us define 
\begin{align}\label{eq:MarkovSemigroup}
    P_{s,t}\varphi(y):=\EE[\varphi(\lambda^{s,y}(t))].
\end{align}
On account of the flow property (Lemma \ref{lemma:flow_pties}(2)) we have in particular~$\lambda^{s,y}(s+\cdot)=F(y,W^s)$ and hence
\begin{align*}
    \lambda^{s,y}(t)=\lambda^{s,y}(s+(t-s))=F(y,W^s)(t-s)\overset{\textnormal{law}}{=}F(y,W^0)(t-s)=\lambda^{0,y}(t-s)
\end{align*} 
and hence the relation
$$
P_{s,t}\varphi(y) = \EE\Big[ \varphi\big(\lambda^{s,y}(t)\big)\Big]
= \EE\Big[ \varphi\big(\lambda^{0,y}(t-s)\big)\Big] 
= P_{0,t-s}\varphi(y).
$$
This homogeneity justifies that we narrow down our analysis to the family of semigroups~$P_t:=P_{0,t}$ for all~$t>0$.
\begin{proposition}[Markov property of $\lambda$]\label{prop:Markov}
    Let Assumptions \ref{assumption:SVEassumptions} i)-ii) and \ref{assumption:Kernel} i)-ii) hold.
    For all~$s\ge 0,\,y\in\Hh_1$, the process $(\lambda^{s,y}(t))_{t\in[s,\infty)}$ has the Markov property, that is: for all $0\le s\le t \le T$ and bounded measurable functions~$\varphi:\Hh_1\to\RR$, we have
    \begin{align*}
    \EE\left[\varphi(\lambda^{s,y}(T)) \big\lvert \Ff_t\right]=
    \EE\left[\varphi(\lambda^{s,y}(T)) \big\lvert \Ff_t^\lambda\right]
    =P_{T-t}\varphi(\lambda^{s,y}(t)),
    \end{align*}    where~$\Ff_t^\lambda:=\sigma(\{\lambda^{s,y}(r), r\in[s,t]\})$.
\end{proposition}
\begin{proof}
Let $0\le s\le t\le T$ and $y\in\Hh_1$. Leveraging the flow property in Lemma \ref{lemma:flow_pties}(2) and the independence of $W^t$ with respect to~$\mathcal{F}_t$, we have
    \begin{align*}
        \EE\Big[\varphi(\lambda^{s,y}(T)) \big\lvert \Ff_t\Big] 
        &= \EE\Big[\varphi \circ F\big(\lambda^{s,y}(t), W^t\big)(T-t) \big\lvert \Ff_t\Big] \\
        &= \EE\Big[\varphi \circ F\big(\xi, W^t\big)(T-t) \Big]\Big\lvert_{\xi=\lambda^{s,y}(t)} \\
        &= \EE\big[\varphi\big(\lambda^{0,\xi}(T-t)\big)\big] \big\lvert_{\xi=\lambda^{s,y}(t)}\\
        &=P_{T-t} \varphi(\lambda^{s,y}(t)).
    \end{align*}
    In particular, $P_{T-t} \varphi(\lambda^{s,y}(t))$ is $\Ff_t^\lambda$-measurable and~$\Ff_t^\lambda\subset\Ff_t.$ Taking conditional expectations with respect to~$\Ff_t^\lambda$ in the equation above finishes the proof.
\end{proof}


\subsection{The Feller and generalized Feller 
properties}\label{sec:Feller}
We conclude this section with an investigation of further regularity properties of the Markov semigroup.
\begin{proposition}\label{Prop:Feller} 
Let Assumptions \ref{assumption:SVEassumptions} i)-ii) and \ref{assumption:Kernel} i)-ii) hold. 
\begin{enumerate}
    \item For all $p\ge1$, $s\in[0,T]$ and $y,z\in\Hh_1$, the mild solution ~\eqref{eq:generalised_SPDE} satisfies 
    \begin{equation}\label{eq:initial_regularity}
\EE\left[\sup_{t\in[s,T]}\abs{\lambda^{s,y}(t)-\lambda^{s,z}(t)}_{\Hh_1}^p\right] \lesssim \abs{y-z}^p_{\Hh_1}.
    \end{equation}
    \item The semigroup $\{P_{t}\}_{t\ge0}$ has the Feller property: $P_t: \Cc_b(\Hh_1)\to \Cc_b(\Hh_1)$ for every~$t\ge0$.
\end{enumerate}
\end{proposition}
\begin{proof}
(1) The same arguments as in the proof of Theorem \ref{thm:SPDE_wellposedness} yield, thanks to the semigroup's strong continuity, coefficients' Lipschitz continuity, Jensen and BDG inequality, and Gr\"onwall's lemma
\begin{align*}
\EE\left[\abs{\lambda^{s,y}(t)-\lambda^{s,z}(t)}_{\Hh_1}^p\right] \lesssim \abs{y-z}^p_{\Hh_1},
\end{align*}
for all $p>\frac{2q}{q-2}$; the extension to all~$p\ge1$ follows from Jensen's inequality. For $y,z\in\Hh_1$ and $0\le s<t'<t\le T$, define $\Delta\lambda^{s,y,z}(t):=\lambda^{s,y}(t)-\lambda^{s,z}(t)-S(t-s)(y-z)$.
We then have by H\"older's inequality
 \begin{align*}
        \EE&\left[\abs{\Delta\lambda^{s,y,z}(t)-\Delta\lambda^{s,y,z}(t')}_{\Hh_1}^p\right] \\
        &\lesssim  
        \EE\left[\left(\int_{t'}^t \abs{S(t-r)K}^2_{\Hh_1} \abs{\lambda^{s,y}(r)-\lambda^{s,z}(r)}_{\Hh_1}^2 \dr\right)^{p/2}\right]\\
        &\quad + \EE\left[\left(\int_s^{t'} \abs{\big(S(t-r)-S(t'-r)\big)K}^2_{\Hh_1} \abs{\lambda^{s,y}(r)-\lambda^{s,z}(r)}_{\Hh_1}^2 \dr\right)^{p/2}\right] \\
        &\lesssim 
          \left( \int_{t'}^t \abs{S(t-r)K}^2_{\Hh_1}\dr\right)^{\frac{p-2}{2}}  \EE\int_{t'}^t \abs{S(t-r)K}^2_{\Hh_1} \abs{\lambda^{s,y}(r)-\lambda^{s,z}(r)}_{\Hh_1}^p \dr
         \\
        &\quad +\left( \int_s^{t'} \abs{\big(S(t-r)-S(t'-r)\big)K}^2_{\Hh_1}\dr\right)^{\frac{p-2}{2}} \EE\int_s^{t'} \abs{\big(S(t-r)-S(t'-r)\big)K}^2_{\Hh_1} \abs{\lambda^{s,y}(r)-\lambda^{s,z}(r)}_{\Hh_1}^p \dr \\
        &\lesssim \abs{y-z}^p_{\Hh_1}\bigg(
        \left(\int_0^{t-t'} \abs{S(r)K}^2_{\Hh_1} \dr\right)^{p/2} + \left( \int_s^{t'} \abs{\big(S(t-r)-S(t'-r)\big)K}^2_{\Hh_1}\dr\right)^{p/2}\Bigg).
    \end{align*}
    Assumption \ref{assumption:Kernel}, and in particular estimates~\eqref{eq:cond_K1} and \eqref{eq:cond_K2}, yields
    \begin{align*}
        &\int_0^{t-t'} \abs{S(r)K}^2_{\Hh_1} \dr \le (t-t')^{\frac{q-2}{2}} \left(\int_0^{t-t'} \abs{S(r)K}^q_{\Hh_1} \dr\right)^{2/q};\\
        & \int_0^s \abs{\big(S(t-r)-S(t'-r)\big)K}^2_{\Hh_1}\dr \lesssim (t-t')^{\frac{q-2}{2}} .
    \end{align*}
    Kolmogorov continuity criterion entails that for all $\gamma\in(0,(q-2)/4-1/p)$ and $p>\frac{4}{q-2}$
    \begin{align*}
        \EE\left[\sup_{0\le t'\neq t\le T} \frac{\abs{\Delta\lambda^{s,y,z}(t)-\Delta\lambda^{s,y,z}(t')}_{\Hh_1}^p}{\abs{t-t'}^{\gamma p}}\right] \le C\abs{y-z}^p_{\Hh_1}.
    \end{align*}
    Jensen's inequality grants the inequality for any~$p\ge1$. 
    In particular, since $\Delta\lambda^{s,y,z}(s)=0$, it holds~$\EE[\sup_{t\in[s,T]}\abs{\Delta\lambda^{s,y,z}(t)}^p]\le C\abs{y-z}^p_{\Hh_1}$. This yields the claim since~$\abs{S(t-s)(y-z)}_{\Hh_1}\lesssim \abs{y-z}_{\Hh_1}$ by Lemma~\ref{lemma:semigroup_pties}.
    
(2)
The estimate above implies that the map $y\mapsto \lambda^{s,y}(t)$ is continuous in probability. Hence for any $\varphi\in  \Cc_b(\Hh_1)$ and any sequence $\{y_n\}_{n\in\NN}$ in $\Hh_1$ converging to~$y\in\Hh_1$, we have
    $$
\lim_{n\to\infty}P_{t}\varphi(y_n) = \lim_{n\to\infty} \EE\big[\varphi(\lambda^{s,y_n}(t))\big]
= \EE\big[\varphi(\lambda^{s,y}(t))\big]
= P_{t}\varphi(y).
    $$
    This proves that~$P_{t}\varphi\in  \Cc_b(\Hh_1)$ and concludes the proof.
\end{proof}
\begin{remark}
    Theorem 4.9 of \cite{gawarecki2010stochastic} show that the Markov and Feller properties combined imply the strong Markov property:  $\EE\big[\varphi(\lambda^{y}(t)) \big\lvert \Ff_\tau\big]= \EE\big[\varphi\big(\lambda^{y}(t)\big)\big\lvert \lambda^{y}(\tau)\big]$ for all $\{\Ff_t\}_{t\ge0}$-stopping times~$\tau$.
\end{remark}
The strong Feller property holds if the Markov semigroup regularises the test function, in the sense that $P_t\varphi\in \Cc_b(\Hh_1)$ for all bounded measurable~$\varphi$. Hamaguchi~\cite[Section 3.2]{hamaguchi2023markovian} demonstrates, for a different Markovian lift but in a similar setup as ours where the Brownian motion~$W$ is finite-dimensional, that such a property cannot hold. This is due to the degeneracy of the noise for SPDEs such as~\eqref{eq:generalised_SPDE}. We can remedy to this issue with a detour to generalised Feller processes. 
In the finite-dimensional theory, two definitions of Feller semigroups coexist. One is related to the space of continuous bounded functions $ \Cc_b(\Hh_1)$ and another one based on the space of continuous functions vanishing at infinity $\Cc_0(\Hh_1)$. The former generalises to the infinite-dimensional case in a straightforward manner, as we just presented, but is seldom strongly continuous. The latter formulation, however, needs revamping when working with infinite dimensions. Generalized Feller processes~\cite{dorsek2010semigroup,cuchiero2020generalized,cuchiero2023ramifications} were introduced with the purpose of establishing a Feller theory in non-locally compact (typically infinite-dimensional) spaces, with the immediate benefit of granting strong continuity to the semigroup. The semigroups act on weighted spaces analogous to the space~$\Cc_0$ in non-locally compact spaces.  
We take~\cite{cuchiero2020generalized} as reference to recall the notations. 

We first note that $\Hh_1$ is a completely regular Hausdorff topological space. We say that $\varrho\colon \Hh_1\to (0,+\infty)$ is an admissible weight function if the sublevel sets~$K_R:= \{x\in \Hh_1: \varrho(x)\le R\}$ are compact for all~$R>0$~\cite[Definition 2.1]{cuchiero2020generalized}. 
The vector space
\[
\mathrm{B}^\varrho(\Hh_1) := \left\{ \varphi\colon \Hh_1\to \RR: \sup_{x\in \Hh_1}\varrho(x)^{-1} \abs{\varphi(x)} <\infty \right\},
\]
equipped with the norm
\begin{equation}\label{eq:rhonorm}
\abs{\varphi}_\varrho:= \sup_{x\in \Hh_1} \varrho(x)^{-1}\abs{\varphi(x)},
\end{equation}
is a Banach space, and~$ \Cc_b(\Hh_1) \subset \mathrm{B}^\varrho(\Hh_1)$. Moreover, the space~$\Bb^\varrho (\Hh_1)$ is defined as the closure of~$ \Cc_b(\Hh_1)$ with respect to~$\abs{\cdot}_{\varrho}$~\cite[Definition 2.3]{cuchiero2020generalized}, and is also a Banach space when equipped with the norm~\eqref{eq:rhonorm}.
A useful characterisation, proven in  \cite[Theorem 2.2]{cuchiero2023ramifications}, states that $\varphi\in\Bb^\varrho(\Hh_1)$ if and only if, for all~$R>0$, $\varphi\lvert_{K_R}\in  \Cc_b(K_R)$ and \begin{equation}\label{eq:VanishInfty}
    \lim_{R\to\infty}\sup_{x\in\Hh_1\setminus K_R}\frac{\abs{\varphi(x)}}{\varrho(x)}=0.
\end{equation} 
Generalised Feller semigroups consist of bounded, positive, linear, bounded operators and are automatically strongly continuous.
\begin{definition}[Definition 2.5 of~\cite{cuchiero2020generalized}]
A family of bounded linear operator~$P_t:\Bb^\varrho(\Hh_1)\to\Bb^\varrho(\Hh_1)$ for~$t\ge0$ is called a generalised Feller semigroup if
\begin{enumerate}
    \item[i)] $P_0=I$, the identity on~$\Bb^\varrho(\Hh_1)$,
    \item[ii)] $P_{t+s}= P_t P_s$, for all~$t,s\ge0$,
    \item[iii)] For all~$\varphi\in\Bb^\varrho(\Hh_1)$ and~$y\in\Hh_1$, $\lim_{t\downarrow0} P_t \varphi(y) = \varphi(y)$,
    \item[iv)] There exist~$C>0$ and~$\ep>0$ such that for all~$t\in[0,\ep],$ $\norm{P_t}_{\mathscr{L}(\Bb^\varrho(\Hh_1))}\le C$,
    \item[v)] $P_t$ is positive for all~$t\ge0$, that is, for any $\varphi\in\Bb^\varrho(\Hh_1)$ such that $\varphi\ge0$, then $P_t \varphi\ge0$.
\end{enumerate}
\end{definition}
Crucially, generalized Feller semigroups are strongly continuous by \cite[Theorem 2.8]{cuchiero2023ramifications}. Moreover, this theory found applications to the study of the large-time behaviour of Volterra process~\cite{friesen2024stationary,jacquier2025large}.

One can prove that a time-homogeneous Markov process has the generalized Feller property using the conditions displayed in~\cite[Theorem 2.21]{cuchiero2020generalized}.
The strategy consists in embedding a closed subspace~$\Ee$ into a dual space in order to leverage the weak-$\star$-topology. Since $\Hh_1$ is a Hilbert space, we simply define~$\Ee=\Hh_1^\star=\Hh_1$, equipped with its weak topology. The Banach-Alaoglu theorem ensures that~$\varrho(x)=1+\abs{x}_{\Hh_1}$ is an admissible weight function because closed balls with respect to the $\Hh_1$-norm are compact in the weak topology.

\begin{proposition}
Let Assumptions \ref{assumption:SVEassumptions} i)-ii) and \ref{assumption:Kernel} i)-ii) hold.
    The time-homogeneous Markov semigroup $(P_t)_{t\ge0}$ satisfies the following:
    \begin{enumerate} 
        \item There are constants $C>0$ and $\ep>0$ such that
        $$
P_t\varrho(y)\le C \varrho(y), \qquad \text{for all   }y \in\Hh_1, t\in[0,\ep].
        $$
        \item It holds that 
        $$
\lim_{t\to0} P_t \varphi(y)=\varphi(y), \qquad \text{for all   } y\in\Hh_1, \; \varphi\in \Bb^\varrho(\Hh_1).
        $$
        \item For all $\varphi$ in a dense subset of $\Bb^\varrho(\Hh_1)$, the map $y \mapsto P_t \varphi(y)$ lies in $\Bb^\varrho(\Hh_1)$.
    \end{enumerate}
    Thereofore, $(P_t)_{t\ge0}$ is a generalized Feller and a strongly continuous semigroup.
\end{proposition}
\begin{proof}\
\begin{enumerate}
    \item  The first property is a direct consequence of the estimate \eqref{eq:lambdaestimate}.
    \item  Let $\ep>0$. 
Since $\varphi\in\Bb^\varrho(\Hh_1)$ there exists $(\varphi_n)_{n\in\NN}$ in $ \Cc_b(\Hh_1)$ such that $\varphi_n$ converges to $\varphi$ with respect to~$\abs{\cdot}_{\varrho}$. Then for all $n\in\NN$ and $t\in(0,1)$,
$$
\abs{P_t \varphi(y) - \varphi(y)} \le \abs{P_t\varphi(y)-P_t \varphi_n(y)} + \abs{P_t \varphi_n(y) - \varphi_n(y)} + \abs{\varphi_n(y) - \varphi(y)}.
$$
Notice that 
$$
\abs{P_t\varphi(y)-P_t \varphi_n(y)} 
=\abs{\EE\big[ \varphi(\lambda^y_t)-\varphi_n(\lambda^y_t)\big]} \le \EE[\varrho(\lambda^y_t)]\sup_{x\in\Hh_1} \frac{\abs{\varphi(x)-\varphi_n(x)}}{\varrho(x)} 
=  \EE\big[1+\abs{\lambda^y_t}_{\varrho}\big] \abs{\varphi-\varphi_n}_{\varrho},
$$
and similarly we have~$\abs{\varphi_n(y)-\varphi(y)}\le \varrho(y) \abs{\varphi-\varphi_n}_{\varrho}$. From \eqref{eq:lambdaestimate}, we have~$\EE[\varrho(\lambda^y_t)]\lesssim C\varrho(y)$ for some constant~$C>0$ hence we let $n_0\in\NN$ be such that~$\abs{\varphi-\varphi_n}_{\varrho}(1+C)\varrho(y)\le  \ep$ for all~$n_0\le n$. For such $n$, $\varphi_n\in\Cc_b(\Hh_1)$ is a continuous function, hence $t\mapsto \varphi_n(\lambda^y(t))$ is continuous almost surely and this implies the existence of some~$\delta>0$ such that for all $t\le \delta$, we have~$\abs{\EE\big[ \varphi_n(\lambda^y_t)-\varphi_n(y)\big]}\le \ep$. It thus holds~$\abs{P_t\varphi(y)-\varphi(y)}\le2\ep$ for all~$t\le\delta$.
\item Since $(P_t)_{t\ge0}$ has the Feller property, $P_t \varphi\in  \Cc_b(\Hh_1)\subset\Bb^\varrho(\Hh_1)$  for all $\varphi\in  \Cc_b(\Hh_1)$. The latter is a dense subset of $\Bb^\varrho(\Hh_1)$, which yields the claim.
\end{enumerate}
The conclusions are precisely the statements of Thereom 2.21 of \cite{cuchiero2020generalized} and Theorem 2.8 of \cite{cuchiero2023ramifications}, respectively. 
\end{proof}


\section{It\^o formulae and Fokker--Planck equations}\label{sec:ito} Up until this point we have connected the strong solution $X$ of \eqref{eq:SVE} to (probabilistically strong) analytically mild solutions $\lambda$ of the SPDE \eqref{eq:SPDE} that take values on $\Hh_1.$ This section is devoted to the proof of It\^o formulae for $\lambda$ which yield, as byproducts, an It\^o formula for $X$ as well as Fokker-Planck equations for the law of $\lambda.$

 As we have already mentioned above, our framework allows for singular kernels $K$ that are not necessarily elements of $\Hh_1.$ This means in particular, that the mild solutions $\lambda$ \eqref{eq:SPDE_mild} cannot be viewed as $\Hh_1-$valued semimartingales of the form
 
$$\lambda(t,x)=\lambda_0(x)+\int_0^t\bigg(\partial_x\lambda(t)+K(x)b_0(\lambda(s))\bigg)\ds+\int_0^tK(x)\sigma_0(\lambda(s))\dW_s.          $$
Indeed, notice that such a representation cannot make sense on $\Hh_1$ (or any other RKHS for that matter): evaluating the solution at $x=0$ and using the lift property  \eqref{eq:liftproperty} would result to the contradicting statement $X_t=\infty$ for all $t.$ In the first part of this section, we show that such expressions can only be justified by viewing $\lambda$  as a process that takes values on the larger Hilbert space $\Hh.$ 

The absence of analytically strong solutions means that classical infinite-dimensional It\^o formulae are of little to no use for mild solutions. Viewing $\lambda$ as a semimartingale in $\Hh$ (which is not an RKHS) prevents us from obtaining any useful information for the Volterra process $X.$ In the second part of this section we show that mild It\^o formulae, similar to the ones developed in \cite{da2019mild}, provide a viable way to circumvent such problems. Last but not least, we derive a ``singular" Itô formula which combines the best of both approaches---semimartingale formulation and RKHS. 
\subsection{Strong Itô formula}\label{subsec:StrongIto} In view of the well-posedness result Theorem \ref{thm:SPDE_wellposedness} and lift property \eqref{eq:liftproperty} we can write
\begin{equation}\label{eq:RandomSPDEmildSolution}
\begin{aligned}
    \lambda(t)=S(t)\lambda_0+\int_0^t S(t-s)Kb(X_s)\ds+\int_0^t S(t-s)K\sigma(X_s)\dW_s.
\end{aligned}
\end{equation}
This in turn can be viewed as a mild solution of the evolution equation
\begin{equation}\label{eq:SPDE-RandomCoefficients}
    \left\{\begin{aligned}
        &\D\Lambda(t)=\big(\partial_x \Lambda(t)+B(t)\big)\dt+\Sigma(t)\dW_t\\&
        \Lambda(0)=\lambda_0
    \end{aligned}\right.
\end{equation}
with measurable, random coefficients $B(t)(\omega):=Kb(X_t(\omega))\in\Hh, \Sigma(t)(\omega):=K\sigma(X_t(\omega))\in\mathscr{L}(\RR^m;\Hh),$ $ \omega\in\Omega$ and values on the Hilbert space $\Hh.$

\begin{definition}[Strong Solutions] Let $T>0$. An $\Hh-$valued strong solution of \eqref{eq:SPDE-RandomCoefficients} on $[0,T]$ is an $\Hh-$valued mild solution $\lambda$ if the following hold:
\begin{enumerate}
    \item[i)] $ \Lambda(t)\in Dom(\partial_x)$ almost everywhere on $[0,T]\times\Omega$
    \item[ii)] It holds $$\EE\int_0^T|K\sigma(X_t)|^2_{\mathscr{L}_2(\RR^m;\Hh)}\dt<\infty.$$
\end{enumerate}

\end{definition}

\begin{lemma}[Semimartingale formulation] Let $T>0, \lambda_0\in L^2(\Omega;\Hh_1).$ Under Assumptions \eqref{assumption:SVEassumptions} i)-ii) and \eqref{assumption:Kernel} i)-ii), the process $\lambda$ in \eqref{eq:RandomSPDEmildSolution} is an $\Hh-$valued strong solution of \eqref{eq:SPDE-RandomCoefficients} on $[0,T].$ Moreover, $\lambda\in \Cc([0,T];\Hh)$ and for all $t\in[0,T]$ it satisfies 
\begin{equation}\label{eq:lambdaSemimart}
\lambda(t)=\lambda(0)+\int_0^t\bigg(\partial_x\lambda(s)+Kb(X_s)    \bigg)\ds +\int_0^t K\sigma(X_s)\dW_s
\end{equation}
almost surely.    
\end{lemma}

\begin{proof} As we showed in Lemma \ref{lemma:semigroup_pties}(ii) we have $Dom(\partial_x)=\Hh_1.$ Moreover, from Theorem \ref{thm:SPDE_wellposedness} we have that for all $t\in[0,T]$  $\lambda(t)\in\Hh_1$ almost surely.  Since $b$ and $\sigma$ are continuous and have linear growth we have $$\sup_{t\in[0,T]} \EE[\abs{X_t}^p] < \infty$$ (see e.g. \cite[Lemma 3.1]{abi2019affine}). From the linear growth of $\sigma,$ along with the moment estimate for the solution of the SVE we have
 \begin{equation*}
     \begin{aligned}      \EE\int_0^T|K\sigma(X_s)|^2_{\mathscr{L}_2(\RR^m;\Hh)}\ds=\EE\int_0^T\sum_{k=1}^{n}|K\sigma_{k}(X_s)|^2_{\Hh}\ds\lesssim T|K|^2_{\Hh}\bigg(1+\sup_{t\in[0,T]} \EE[\abs{X_t}^2]\bigg)<\infty,
     \end{aligned}
 \end{equation*}
  where $\sigma_k(x)\in\RR^d, k=1,\dots, m,$ are the column vectors of the diffusion matrix.
 In view of \cite[Theorem 3.2]{gawarecki2010stochastic} it follows that $\lambda\in \Cc([0,T];\Hh)$ almost surely, is a strong solution to \eqref{eq:SPDE-RandomCoefficients} and can be written as the $\Hh-$valued semimartingale \eqref{eq:lambdaSemimart}. The proof is complete.
\end{proof}

\begin{remark} It is important to observe the differences between the transport SPDEs \eqref{eq:SPDE-RandomCoefficients} and \eqref{eq:SPDE_mild}. 
The former is cast as an evolution on $\Hh$ with random coefficients while the latter is posed as an evolution equation on the RKHS $\Hh_1$ with (non-random) Lipschitz nonlinearities. More importantly,  solutions of the former do not a priori satisfy the Markov property due to the presence of random coefficients. In sharp contrast, solutions of \eqref{eq:SPDE} are Markovian as shown in Proposition \ref{prop:Markov}. The point here is that a mild solution of  \eqref{eq:SPDE}, starting from ``smooth" initial conditions $\lambda_0\in\Hh_1,$ can be viewed as a strong solution of \eqref{eq:SPDE-RandomCoefficients} that takes values on the larger space $\Hh.$ 
\end{remark}

 The point of view we adopt here allows us to obtain a ``strong" It\^o formula for the unique mild solution of \eqref{eq:SPDE} when viewed as a process with values on $\Hh.$

\begin{proposition}[Strong It\^o formula]\label{prop:StrongIto} Let $T>0$ and  $F : [0,T ]\times\Hh\rightarrow\RR$ be continuous and such that its Fr\'echet partial derivatives $\partial_tF , D F, D^2F$ are continuous and bounded on bounded subsets of $[0,T ]\times\Hh$. Let Assumptions \eqref{assumption:SVEassumptions} i)-ii) and \eqref{assumption:Kernel} i)-ii) hold. With $\lambda$ being the unique mild solution of \eqref{eq:SPDE} on $[0,T]$ with initial condition $\lambda_0\in L^2(\Omega;\Hh_1)$ we have 
\begin{equation}\label{eq:StrongIto}
  \begin{aligned}   F\big(t,\lambda(t)\big)&=F\big(s,\lambda(s)\big)
  +\int_s^t\big\langle D F(r,\lambda(r)), K\sigma(X_r)\dW_r\big\rangle_{\Hh}\\&+\int_s^t\bigg(\partial_tF\big(r,\lambda(r)\big)+ \big\langle D F(r,\lambda(r)), \partial_x\lambda(r)+Kb(X_r)\big\rangle_{\Hh}\bigg)\dr\\&
    +\frac{1}{2}\int_s^t\textnormal{Tr}\bigg[D^2F(r,\lambda(r))K\sigma(X_r)\big(K\sigma(X_r)\big)^*\bigg]\dr
  \end{aligned}
\end{equation}
almost surely for all $s\leq t\in[0,T].$   
\end{proposition}

\begin{proof} In view of \eqref{eq:lambdaSemimart}, the conclusion is a direct consequence of \cite[Theorem 2.9]{gawarecki2010stochastic} or \cite[Theorem 4.32]{da2014stochastic}.
\end{proof}

\begin{remark}\label{rem:ItoHvsH1} We emphasise once again that strong It\^o formula \eqref{eq:StrongIto} holds with respect to the Hilbert space $\Hh$ and not $\Hh_1.$ Despite its simplicity, $\Hh$ does not enjoy the RKHS property and hence such an It\^o formula is not suitable to obtain information about smooth functionals of the SVE solution $X_\cdot=ev_0(\lambda(\cdot)).$ This is due to the fact that the evaluation functional $ev_0$ is not continuous and a fortiori not differentiable when defined on $\Hh.$
\end{remark}

\subsection{Mild Itô formula}\label{subsec:MildIto}
We start this section with a definition of mild It\^o processes, as introduced in \cite{da2019mild}, adapted to our framework.
\begin{definition}[Mild It\^o Processes]\label{dfn:mildItoprocess} Let $I=[s,t]\subset\RR^+$ and $\Delta_I=\{ (t_1, t_2)\in I^2: t_1<t_2\}$ the corresponding $1-$simplex. Let $S:\Delta_I\mapsto\mathscr{L}(\Hh_1)$ be a continuous mapping satisfying for all $t_1,t_2,t_3\in I$ with  $t_1<t_2<t_3.$
 Additionally, let $Y: I\times\Omega\rightarrow\Hh_1$ and $Z: I\times\Omega\rightarrow\mathscr{L}_2(\RR^m;\Hh_1)$ be progressively measurable stochastic processes such that
 $$\int_{s}^{t_1} |S_{r,t_1}Y_r|_{\Hh_1}\dr<\infty,\;\;\int_{s}^{t_1} |S_{r,t_1}Z_r|^2_{\mathscr{L}_2(\RR^m;\Hh_1)}\dr<\infty       $$
 almost surely for all $t_1\in I.$ A progressively measurable stochastic process $\Lambda:I\times\Omega\rightarrow \Hh_1$ satisfying
$$\Lambda(t_1)=S_{s,t_1}\Lambda(s)+\int_{s}^{t_1}S_{r,t_1}Y_r\dr+\int_s^{t_1}S_{r,t_1}Z_r \dW_r$$
almost surely for all $t_1\in I$ is called a mild It\^o process (with semigroup $S,$ mild drift $Y$ and mild diffusion $Z).$
\end{definition}

The class of mild It\^o process is tailor-made to include mild solutions of a large class of SPDEs. Intuitively, the unique mild solution $\lambda$ \eqref{eq:generalised_SPDE} shall be thought of as a mild It\^o process. However, since the singular kernel $K$ is not necessarily in $\Hh_1,$ it does not satisfy the above definition. In the next lemma we show that a ``mollified" version of $\lambda$ is indeed a mild It\^o process per Definition \ref{dfn:mildItoprocess}.
\begin{lemma}\label{lem:ShiftedMildProcess} Let $\delta, T>0, \{S(t)\}_{t\geq 0}\subset\mathscr{L}(\Hh_1)$ be the shift semigroup and $\lambda$ the unique mild solution of \eqref{eq:SPDE}. Under Assumptions \ref{assumption:SVEassumptions} i)-ii) and \ref{assumption:Kernel} i)-ii) the $\Hh_1-$valued process $$\Lambda_\delta:=S(\delta)\lambda$$ is a mild It\^o process with semigroup $S_{s,t}:=S(t-s),$ mild drift $Y^\delta_t:=S(\delta)Kb_0(\lambda(t))$ and mild diffusion 
$Z^\delta=S(\delta)K\sigma_0(\lambda(t)),$ $s\leq t\in[0,T].$
\end{lemma}

\begin{proof} Let $s,t\in[0,T].$ In view of Lemma \ref{lemma:semigroup_pties},2. (see also Remark \ref{rem:H1Semigroup}) the shift semigroup $S$ satisfies all the required properties. It remains to verify the bounds on $Y^\delta, Z^\delta.$ From the linear growth of $b,$ the boundedness of the linear operator $ev_0:\Hh_1\rightarrow\RR^d$ and Corollary \ref{cor:SemigroupBound} 
\begin{equation*}
    \begin{aligned}
      \EE \int_{s}^{t} |S_{r,t}Y^\delta_r|_{\Hh_1}\dr&= \EE\int_{s}^{t} |S(t-r)S(\delta)Kb_0(\lambda(r))|_{\Hh}\dr\\&
       +\EE\int_{s}^{t} |S(t-r)S(\delta)K'b_0(\lambda(r))|_{\Hh}\dr\\&
       \lesssim \bigg(1+\sup_{r\in[0,T]}\EE|\lambda(r)|_{\Hh_1}\bigg)\int_{s}^{t}e^{c(t-r+\delta)}|K|_{\Hh}\dr\\&       +\bigg(1+\sup_{r\in[0,T]}\EE|\lambda(r)|_{\Hh_1}\bigg)\int_{s}^{t}|K'(t-r+\delta+\cdot)|_{\Hh}\dr.
    \end{aligned}
\end{equation*}
From the estimate \eqref{eq:lambdaestimate}, the prefactors are finite. For the first term on the last estimate we have 
   $\int_{s}^{t}e^{c(t-r+\delta)}|K|_{\Hh}\dr\leq Te^{T+\delta}|K|_{\Hh}.$ As for the last term, 
\begin{equation}\label{eq:uniformDeltaEstimate}
\begin{aligned}
     \int_{s}^{t}|K'(t-r+\delta+\cdot)|_{\Hh}\dr
     \le \int_0^{t-s} |K'(r+\cdot)|^2_{\Hh}\dr,
\end{aligned}    
\end{equation}  
where the last term is finite from Assumption \eqref{assumption:Kernel}.

Turning to the mild diffusion term we have 
\begin{equation*}
    \begin{aligned}
         \EE \int_{s}^{t} |S_{r,t}Z^\delta_r|^2_{\mathscr{L}_2(\RR^m;\Hh_1)}\dr&=\EE \int_{s}^{t}\sum_{k=1}^{m} |S(t-r)S(\delta)K\sigma_{k}(ev_0(\lambda(r))|^2_{\Hh_1}\dr\\&
         \lesssim_m \bigg(1+\sup_{r\in[0,T]}\EE|\lambda(r)|^2_{\Hh_1}\bigg)\int_{s}^{t}| S(t-r)S(\delta)K |^2_{\Hh_1}\dr.
    \end{aligned}
\end{equation*}
From the previous estimates we deduce that 
\begin{equation*}
     \EE \int_{s}^{t} |S_{r,t}Y^\delta_r|_{\Hh_1}\dr+\EE \int_{s}^{t} |S_{r,t}Z^\delta_r|^2_{\mathscr{L}_2(\RR^m;\Hh_1)}\dr<\infty.
\end{equation*}
The proof is complete.  
\end{proof}

Mild It\^o processes are clearly not semimartingales and are not expected to satisfy a classical It\^o formula. Fortunately, a mild It\^o formula was developed in \cite{da2019mild} (see also \cite{huber2024markovian, vogler2024lions} for applications of mild It\^o formulae; as well as \cite[Lemma 5.4]{gasteratos2023moderate} for a similar type of It\^o formula in the context of slow-fast systems of SPDEs). The following is the main result of this section and shows that a mild It\^o formula also holds for mild solutions of \eqref{eq:SPDE} despite the irregularity of coefficients.

\begin{theorem}[Mild It\^o formula]\label{thm:mildIto} Let Assumptions \ref{assumption:SVEassumptions} i)-ii) and \ref{assumption:Kernel}  i)-ii) hold, $T>0, \lambda_0\in L^2(\Omega;\Hh_1),$ $\lambda\in C([0,T];\Hh_1)$ denote the unique mild solution of \eqref{eq:SPDE} on the interval $[0,T]$ and $\Vv$ be a separable Hilbert space. Next let $F\in \Cc^{1,2}([0,T]\times\Hh_1; \Vv)$
such that, for some $\rho>0$ and a constant $C>0,$ the partial Fr\'echet derivatives $$[0,T]\times\Hh_1\ni(t,\lambda)\longmapsto \partial_t F(t,\lambda)\in \Vv, D F(t,\lambda)\in\mathscr{L}(\Hh_1;\Vv), D^2 F(t,\lambda)\in\mathscr{L}(\Hh_1;\mathscr{L}(\Hh_1, \Vv))$$
satisfy
\begin{equation}\label{eq:MildItoFgrowth}
\big|F\big(t, \lambda\big)\big|_\Vv+\big|\partial_tF\big(t, \lambda\big)\big|_\Vv+\big|DF\big(t, \lambda\big)\big|_{\mathscr{L}(\Hh_1;\Vv)}+\big|D^2 F(t,\lambda)\big|_{\mathscr{L}(\Hh_1;\mathscr{L}(\Hh_1, \Vv))}\leq C\big(1+|\lambda|^\rho_{\Hh_1}\big)
\end{equation}
for all $\lambda\in\Hh_1.$
With probability $1,$ for all $s\leq t\in[0,T],$ it holds that   
        \begin{equation}\label{eq:MildIto}
    \begin{aligned}  F\big(t, \lambda(t)\big)&=F\big(s, S(t-s)\lambda(s)\big)+\int_s^t D F\big(r, S(t-r)\lambda(r)\big)S(t-r)K\sigma_0\big(\lambda(r)\big)\dW_r\\&
  +\int_s^t\bigg[\partial_tF\big(r, S(t-r)\lambda(r)\big)+D F\big(r, S(t-r)\lambda(r)\big)S(t-r)Kb_0\big(\lambda(r)\big)\bigg]\dr\\&
  +\frac{1}{2}\int_s^t\textnormal{Tr}\bigg[ D^2F\big(r,S(t-r)\lambda(r)\big)S(t-r)K\sigma_0\big(\lambda(r)\big)\Big( S(t-r)K\sigma_0\big(\lambda(r)\big)\Big)^*\bigg]\dr.    
    \end{aligned}
\end{equation}
\end{theorem}

\begin{proof} Let $\delta\in(0,1).$ In view of Lemma \ref{lem:ShiftedMildProcess}, $\Lambda_\delta:=S(\delta)\lambda$ is a mild It\^o process. Therefore, we can apply \cite[Theorem 1]{da2019mild} for each $\delta>0$ to obtain
\begin{equation}\label{eq:deltaMildIto}
    \begin{aligned}  F\big(&t, \Lambda_\delta(t)\big)=F\big(s, S(t-s)\Lambda_\delta(s)\big)+\int_s^t D F \big(r, S(t-r)\Lambda_\delta(r)\big)S(t-r+\delta)K\sigma_0\big(\lambda(r)\big)\dW_r\\&
  +\int_s^t\bigg[ \partial_tF\big(r, S(t-r)\Lambda_\delta(r)\big)+D F\big(r, S(t-r)\Lambda_\delta(r)\big)S(t-r+\delta)Kb_0\big(\lambda(r)\big)\bigg]\dr\\&
  +\frac{1}{2}\int_s^t\textnormal{Tr}\bigg[ D^2F\big(r,S(t-r)\Lambda_\delta(r)\big)S(t-r+\delta)K\sigma_0\big(\lambda(r)\big)\Big( S(t-r+\delta)K\sigma_0\big(\lambda(r)\big)\Big)^*\bigg]\dr      \end{aligned}
\end{equation}
for all $s\leq t\in [0,T]$ almost surely. In order to pass to the limit as $\delta\to 0$ we proceed in two steps:

\noindent\underline{\textit{Step 1:} $\rho=0$.} In this case, $F$ and all its partial derivatives are  globally bounded and \eqref{eq:MildIto}
 follows by the Dominated Convergence Theorem. Indeed, by continuity of the shift semigroup (Lemma \ref{lemma:semigroup_pties}.2.) we have, for each $r\in[s,t],$ that $|\Lambda_\delta(r)-\lambda(r)|_{\Hh_1}\rightarrow 0$ as $\delta\to 0$ almost surely.  By continuity of $F$ and its derivatives it follows that 
  $F\big(t, \Lambda_\delta(t)\big)\rightarrow F\big(t, \lambda(t)\big),$   $F\big(s, S(t-s)\Lambda_\delta(s)\big)\rightarrow F\big(s, S(t-s)\lambda(s)\big),$$ DF_\lambda\big(r, S(t-r)\Lambda_\delta(r)\big)\rightarrow DF_\lambda\big(r, S(t-r)\lambda(r)\big)$ $, \partial_t F\big(r, S(t-r)\Lambda_\delta(r)\big)\rightarrow \partial_t F\big(r, S(t-r)\lambda(r))\big), D^2F\big(r,S(t-r)\Lambda_\delta (r)\big)\rightarrow D^2F\big(r,S(t-r)\lambda(r)\big)$ almost surely as $\delta\to 0$ in their respective topologies. By continuity of the shift semigroup we also have $ S(t-r+\delta)Kb_0\big(\lambda(r)\big)\rightarrow (t-r)Kb_0\big(\lambda(r)\big) , S(t-r+\delta)K\sigma_0\big(\lambda(r)\big)\rightarrow S(t-r)K\sigma_0\big(\lambda(r)\big)$ almost surely.

  Now for the first Riemann integral we use the boundedness of $\partial_t F, D F$ to bound its $L^1(\Omega)-$norm from above by 
    $$\big( |\partial_tF|_\infty+D F|_{\infty}\big)\int_s^t \bigg[1+\EE\big|S(t-r+\delta)Kb_0\big(\lambda(r)\big)\big|_{\Hh_1}\bigg]\dr.$$
 The integral of this expectation is uniformly bounded in $\delta\in(0,1).$
 In turn, from the computations in Lemma \ref{lem:ShiftedMildProcess} (and in particular \eqref{eq:uniformDeltaEstimate}) the latter is upper bounded, up to an unimportant constant and uniformly over $\delta\in(0,1),$ by 
 $$ \bigg(1+\sup_{r\in[0,T]}\EE|\lambda(r)|_{\Hh_1}\bigg)T^{1/2}\int_{0}^{T}|K'(r+\cdot)|^2_{\Hh}\dr.$$
 In view of Assumption \ref{assumption:Kernel} and \eqref{eq:lambdaestimate} the latter is finite. From similar arguments, along with the It\^o isometry , we can obtain uniform-in-$\delta\in(0,1)$ bounds for the $L^1(\Omega)$ norms of the rest of the terms in \eqref{eq:deltaMildIto}. 
 
 Passing if necessary to subsequences and invoking dominated convergence, all terms in \eqref{eq:deltaMildIto} converge to the corresponding terms in \eqref{eq:MildIto} almost surely  as $\delta\to 0.$ 
 
 \noindent\underline{\textit{Step 2:} $\rho>0$.}  We proceed with a localization argument. To this end, let $t>0$ and for each $n\in\NN$ consider the stopping times $$\tau^n=\inf\{ r\geq s: |\lambda(r)|_{\Hh_1}\geq n\}.$$
 Substituting $t$ by the bounded stopping times $t\wedge \tau^n$ we see that the arguments of $F$ and all its partial derivatives remain bounded on $[s,t\wedge \tau^n].$ Moreover, the growth conditions \ref{eq:MildItoFgrowth} imply that the latter are bounded on bounded subsets of $[0,T]\times\Hh_1.$ Thus, we can repeat all the arguments in Step 1 to obtain the almost sure equality 
   \begin{equation}\label{eq:MildItoStoppingTime}
    \begin{aligned}  F\big(&t\wedge \tau^n, \lambda(t\wedge \tau^n)\big)=F\big(s, S(t\wedge \tau^n-s)\lambda(s)\big)\\&+\int_s^{t\wedge \tau^n} D F\big(r, S(t\wedge \tau^n-r)\lambda(r)\big)S(t\wedge \tau^n-r)K\sigma_0\big(\lambda(r)\big)\dW_r\\&
  +\int_s^{t\wedge \tau^n}\bigg[\partial_tF\big(r, S(t\wedge \tau^n-r)\lambda(r)\big)+D F\big(r, S(t\wedge \tau^n-r)\lambda(r)\big)S(t\wedge \tau^n-r)Kb_0\big(\lambda(r)\big)\bigg]\dr\\&
  +\frac{1}{2}\int_s^{t\wedge \tau^n}\textnormal{Tr}\bigg[ D^2F\big(r,S(t\wedge \tau^n-r)\lambda(r)\big)S(t\wedge \tau^n-r)K\sigma_0\big(\lambda(r)\big)\Big( S(t\wedge \tau^n-r)K\sigma_0\big(\lambda(r)\big)\Big)^*\bigg]\dr.    
    \end{aligned}
\end{equation}
From Lemma \eqref{lemma:pathregularity} and Chebyshev's inequality it follows that for all $R>0, n\in\NN$
$$\PP(   \tau^n <R)\leq  \PP\bigg(  \sup_{r\in[s, t\wedge R]}|\lambda(r)|_{\Hh_1}\geq n \bigg)\leq \frac{1}{n}\EE\bigg[ \sup_{t\in[0, T]}|\lambda(t)|_{\Hh_1}\bigg]<\infty    $$
Hence the non-decreasing sequence of stopping times $\tau^n\rightarrow \infty$ almost surely. Using the latter, continuity of $F$ and its partial derivatives and the estimate 
$$ \sup_{n\in\NN, t,r\in[0,T]}\EE|S(t\wedge \tau^n-r)\lambda(r)|^p_{\Hh_1}\lesssim e^{pc T}\sup_{ r\in[0,T]} \EE| \lambda(r)|^p_{\Hh_1}<\infty, p\geq 1    $$
(with $c$ as in Corollary \ref{cor:SemigroupBound}) we can conclude once again by dominated convergence upon taking $n\to\infty.$ Since the argument is standard, we shall only sketch here the convergence of the stochastic integral in \eqref{eq:MildItoStoppingTime}.  Indeed, by writing this term as 
\begin{equation*}
    \begin{aligned}
        \int_s^{t} D F\big(r, S(t\wedge \tau^n-r)\lambda(r)\big)S(t\wedge \tau^n-r)K\sigma_0\big(\lambda(r)\big)\mathds{1}_{[0,\tau^n]}(r)\dW_r,
    \end{aligned}
\end{equation*}
taking expectation and applying It\^o isometry we have
\begin{equation*}
    \begin{aligned}
        \EE\bigg|&\int_s^{t} D F\big(r, S(t\wedge \tau^n-r)\lambda(r)\big)S(t\wedge \tau^n-r)K\sigma_0\big(\lambda(r)\big)\mathds{1}_{[0,\tau^n]}(r)\dW_r\\&-\int_s^{t} D F\big(r, S(t-r)\lambda(r)\big)S(t-r)K\sigma_0\big(\lambda(r)\big)\mathds{1}_{[0,\infty]}(r)\dW_r\bigg|^2_{\Vv}\\&
        =\int_s^{t} \EE\bigg| D F\big(r, S(t\wedge \tau^n-r)\lambda(r)\big)S(t\wedge \tau^n-r)K\sigma_0\big(\lambda(r)\big)\mathds{1}_{[0,\tau^n]}(r)\\&- D F\big(r, S(t-r)\lambda(r)\big)S(t-r)K\sigma_0\big(\lambda(r)\big)\mathds{1}_{[0,\infty]}(r)\bigg|^2_{\Vv}\dr,
    \end{aligned}
\end{equation*}
where, due to continuity of $DF$ and the semigroup $S$ and the almost sure convergence $\tau^n\rightarrow\infty,$ the integrand converges to $0$ almost surely for all $r\in[s,t],$ as $n\to\infty,$  Furthermore, by the growth condition \ref{eq:MildItoFgrowth}, triangle inequality for the kernel terms, the (pathwise) difference bound from Assumption \ref{assumption:Kernel} ii) and the linear growth estimate \eqref{eq:sigma0growth} we have 
\begin{equation*}
    \begin{aligned}
        \EE\int_s^{t} &\big| D F\big(r, S(t\wedge \tau^n-r)\lambda(r)\big)S(t\wedge \tau^n-r)K\sigma_0\big(\lambda(r)\big)\mathds{1}_{[0,\tau^n]}(r)\big|_\Vv^2\dr\\&\lesssim 
        \EE\bigg[\big(1+ \sup_{r\in[s,t]}|\lambda(r)|^{2\rho}_{\Hh_1}\big)\int_s^{t}\big(|S(t\wedge \tau^n-r)K-S(t-r)K|^2_{\Hh_1}+|S(t-r)K|^2_{\Hh_1}\big)|\sigma_0\big(\lambda(r)\big)|^2\dr\bigg]
        \\&\lesssim \EE\bigg[\big(1+ \sup_{r\in[s,t]}|\lambda(r)|^{2\rho}_{\Hh_1}\big)
        \bigg( 1+(t-\tau^n\wedge t)^{\frac{q-2}{q}}+\int_0^T|S(r)K|^2_{\Hh_1}\dr    \bigg)\big(1+ \sup_{r\in[s,t]}|\lambda(r)|^{2}_{\Hh_1}\big)\bigg]\\&
        \leq \bigg( 1+t^{\frac{q-2}{q}}+\int_0^T|S(r)K|^2_{\Hh_1}\dr    \bigg)\EE\bigg[\big(1+ \sup_{r\in[s,t]}|\lambda(r)|^{2\rho+2}_{\Hh_1}\big)
        \bigg]<\infty,
    \end{aligned}
\end{equation*}
where the deterministic prefactor is finite from Assumption \ref{assumption:Kernel} and the expectation is finite from Proposition \ref{lemma:pathregularity} (or, alternatively, Proposition \ref{Prop:Feller}.1.). Consequently, the dominated convergence theorem furnishes
\begin{equation*}   
\begin{aligned}
\lim_{n\to\infty}&\int_s^{t\wedge\tau^n} D F\big(r, S(t\wedge \tau^n-r)\lambda(r)\big)S(t\wedge \tau^n-r)K\sigma_0\big(\lambda(r)\big)\dW_r\\&=\int_s^{t} D F\big(r, S(t-r)\lambda(r)\big)S(t-r)K\sigma_0\big(\lambda(r)\big)\dW_r
  \end{aligned}
\end{equation*}
in $L^2(\Omega;\Vv)$ and (up to a subsequence) almost surely. The convergence for the rest of the terms follows by similar but simpler arguments. The proof is complete.
\end{proof}
\begin{remark}[Comparison with the path-dependent It\^o formula from \cite{viens2019martingale}]
\label{rem:Ito_VZ_comparison} In the setting of SVEs with convolution-type kernels, the mild It\^o formula in Theorem \ref{thm:mildIto} is significantly more tractable compared to the path-dependent It\^o formula from \cite[Theorems 3.10, 3.17]{viens2019martingale}. Indeed, 
the differentials are all interpreted in the sense of classical Fr\'echet differentiation on Hilbert spaces and the mild It\^o formula holds both in the case of singular and non-singular kernels without additional assumptions on the coefficients. We refer to Section \ref{sec:comparison_VZ} for a comprehensive comparison.
\end{remark}

\subsection{Singular directional differentiation and singular It\^o formula}\label{sec:SingularIto}  We have showed that $\Hh-$ valued strong solutions of the transport-type SPDE \eqref{eq:SPDE} satisfy a strong It\^o formula (Proposition \ref{prop:StrongIto}) while the $\Hh_1-$valued mild solutions \eqref{eq:SPDE} satisfy the mild It\^o formula from Theorem \ref{thm:mildIto}. As discussed in Remark \ref{rem:ItoHvsH1}, the former cannot be applied to functionals of the Volterra process $X=ev_0(\lambda)$ and $\Hh$ is not an RKHS. As for the latter, even though it leads to an It\^o formula for functionals of $X$ and a Fokker--Planck equation (Section \ref{sec:FPE}), it cannot be directly applied to obtain a backward Kolmogorov equation for the Markov semigroup \eqref{eq:MarkovSemigroup}.
\footnote[3]{This observation can already be verified in the much more general framework of mild It\^o processes that was introduced in \cite{da2019mild}. Indeed, the derivation of a mild Fokker--Planck equation (equation (71), Section 3.2.2 therein) is a straightforward consequence of a mild It\^o formula. However, to the best of our knowledge, there is no ``mild backward Kolmogorov" equation for SPDE mild  solutions that a) follows from  the mild It\^o formula or b) holds for ``non-smooth" initial data $y\notin Dom(A).$  The non-invertibility and non-commutativity of the associated composition operators and semigroups $S$ (see e.g. display (70) in the same reference) appear to be essential sources of difficulty for the aforementioned task.}. More importantly, the mild It\^o formula does not provide a semimartingale representation for smooth functionals of $\lambda.$

In this section we aim to prove an It\^o formula for time-dependent, smooth (in an appropriate sense) functionals $F:[0,T]\times\Hh_1\rightarrow\RR$ that
\begin{enumerate}
    \item[a)] is well-adapted to the structure of \eqref{eq:SPDE},
    \item[b)] results in desirable semimartingale representations for the process $F(t, \lambda(t))$, and
    \item[c)] serves as a cornerstone for the well-posedness proof of our backward Kolmogorov equation, Theorem \ref{thm:backward_equation_singular} below.
\end{enumerate} 
We achieve this goal by studying the limit of strong It\^o formulae for $\Hh_1-$valued smooth approximations of the mild solutions $\lambda$ \eqref{eq:SPDE_mild}. 

A glimpse at the strong It\^o formula on $\Hh$ \eqref{eq:StrongIto} immediately reveals the following challenge:

\noindent In order to obtain an analogous ``strong" It\^o formula on $\Hh_1,$ one needs to make sense of $\Hh_1-$directional derivatives along the singular kernel $K.$ However, as we have explained in Section \ref{subsec:Goals} $K$ cannot be an element of $\Hh_1.$ For this reason, we introduce below a concept of ``singular" directional derivatives that appears naturally in the subsequent analysis and plays a key role in the analysis of Section \ref{sec:BackwardEquation}.

\begin{definition}[Singular directional derivatives]\label{dfn:Derivative} Let $\mathcal{V}$ be a Banach space and $L^0(\Omega;\mathcal{V} )$ the vector space of $\mathcal{V}-$valued random variables endowed with the metric topology of convergence in probability. Next, let $F:\Hh_1\rightarrow L^0(\Omega;\mathcal{V} )$ be a (random) continuously Gateaux differentiable map (i.e. $\forall h\in\Hh_1$, $\Hh_1\ni y\longmapsto DF(y)(h)\in\RR$ is a well-defined continuous map) and
\begin{align*}
     &\mathscr{K}=\bigcap_{t>0}S(t)^{-1}(\Hh_1)\subset\Hh
\end{align*}
be the linear subspace defined in \eqref{eq:SingularDirections}. 
The \textit{singular directional derivative} of $F$ at the point $y\in\Hh_1$ and along the direction $h\in\mathscr{K} $ is a map $\Hh_1\times\mathscr{K}\ni (y,h)\longmapsto\mathcal{D}F(y)(h)\in L^0(\Omega;\mathcal{V} )$ defined by
$$ \mathcal{D}F(y)(h):=\lim_{\delta\to 0} DF(y)\big(S(\delta)h\big),$$
provided that the limit on the right-hand side exists in $L^0(\Omega).$ Similarly, for $m\in\NN,$ we define the $m$-th singular directional derivative $\mathcal{D}^mF(y)(h_1,\dots,h_m)$ of an $m-$times continuously Gateaux differentiable (random) function $F,$ along the directions  $\{h_i\}_{i=1}^{m}\subset\mathscr{K},$ as a map $\Hh_1\times\mathscr{K}^m\ni (y,h_1,\dots, h_m)\longmapsto\mathcal{D}^mF(y)(h_1,\dots,h_m)\in L^0(\Omega;\mathcal{V} ),$ 
\begin{align}\label{eq:mthSingularDerivativeDef}
   \mathcal{D}^mF(y)(h_1,\dots,h_m):=\lim_{\delta\to 0} D^mF(y)(S(\delta)h_1,\dots, S(\delta) h_m), 
\end{align} provided that the limit exists in $L^0(\Omega).$ The class of $m-$times singularly differentiable functions along a subset $\tilde{\mathscr{K}}\subset \mathscr{K}$ is denoted by $\mathcal{D}^m_{\!\!\tilde{\mathscr{K}}}(\Hh_1; \mathcal{V})$ and we use the shorter notation $\mathcal{D}_{\!\!\tilde{\mathscr{K}}}(\Hh_1; \mathcal{V}):=\mathcal{D}^1_{\!\!\tilde{\mathscr{K}}}(\Hh_1; \mathcal{V}).$ Finally, when dealing with real-valued functions we shall use the shorthand notation $\mathcal{D}^m_{\!\!\tilde{\mathscr{K}}}(\Hh_1):=\mathcal{D}^m_{\!\!\tilde{\mathscr{K}}}(\Hh_1; \RR).$
\end{definition}

\begin{remark}\label{rem:singular_derivative} We collect here a few observations on the definition of singular directional differentiability.
\begin{itemize}
\item[i)] Our definition of singular directional derivatives as limits of classical directional derivatives is similar to one given in \cite[Definition 3.16 and Theorem 3.17]{viens2019martingale}. To be precise, such differentials are defined in the last reference as limits of Fréchet derivatives on a (Banach) space of continuous functions while $\mathcal{D}$ is defined here as a limit of Gateaux derivatives on the Hilbert space $\Hh_1.$     
    \item[ii)] According to the above definition, a (random) function $F\in\mathcal{D}^m_{\tilde{\mathscr{K}}}(\Hh_1; \mathcal{V})$ is $m-$times continuously Gateaux differentiable. However, throughout the rest of this work, we shall only consider functions of this class that are in fact  $m-$times continuously Fr\'echet differentiable (see also Remark \ref{rem:NemytskiiRem} and Corollary \ref{cor:regularTangentProcesses}).
    \item[iii)] For the sake of generality, we have defined singular directional differentiability for random functions in terms of convergence in probability. Nevertheless, $L^0(\Omega;\mathcal{V})$ can be typically replaced by the stronger topology $L^p(\Omega;\mathcal{V}),$ $p\in[1,\infty].$ In particular, our singular differentiability results for the (random) solution flow $y\mapsto\lambda^{s,y}(t)$ (Lemma \ref{lem:SingularDifferentiability}) are true with $p=2.$  
    \item[iv)] Definition \ref{dfn:Derivative} is general enough to include both random and deterministic functions. Clearly, an  $m-$times continuously Gateaux differentiable deterministic function $F:\Hh_1\rightarrow\mathcal{V}$ is in $\mathcal{D}^m_{\tilde{\mathscr{K}}}(\Hh_1; \mathcal{V})$ if and only if the limit on the right-hand side of \eqref{eq:mthSingularDerivativeDef} exists in the topology of $\mathcal{V}.$
    \item[v)] For $h\in\Hh_1,$ the directional derivative $\mathcal{D}F(h)$ coincides with the classical Gateaux derivative  $DF(h)$ on $\Hh_1.$
    \item[vi)] Similarly to Gateaux derivatives and as highlighted before Definition \ref{assumption:admissibleweights}, we will keep the same notations~$\Dd F(y)(h_1),\Dd^2F(y)(h_1, h_2)$ for Gateaux derivatives with matrix-valued directions~$h_1,h_2\in H^1_w(\RR_+;\RR^{d\times d})$.    
    These notations will enable us to consider the derivatives in the direction~$S(\delta)K$ without carrying around~$b(y)$ or~$\sigma(y)e_i$ (as outlined e.g. in \eqref{eq:AbuseofNotations}).
    \end{itemize}
\end{remark}

We clarify the notion of singular directional differentiation by presenting a few examples and non-examples of ``singularly" smooth functions. 
\begin{example}[Singularly differentiable functionals]\label{ex:SingularDifferentiability} In general, differentiating a functional $F:\Hh_1\rightarrow\RR$ along directions that lie outside of $\Hh_1$ is not well-defined. However, singular directional derivatives are appropriate when dealing with functionals that are composed with the shift operators $S(t).$ Fortunately, and despite their particular form, such functionals arise naturally in the context of mild solutions of the SPDE \eqref{eq:SPDE_mild}.

Concrete examples of singularly differentiable functionals can be constructed as follows: 

Let $d=1$ so that $\Hh_1=H^1_w(\RR^+), \Hh=L^2_w(\RR^+)$, moreover let $t>0, H\in(0, 1/2), f\in \Cc_b^1(
\RR), w\in\mathcal{W}_a$ be an admissible weight for the singular kernel $K(\xi)=\xi^{H-1/2}.$ For some $y\in\Hh_1,$ we consider the functional $F_t:\Hh_1\rightarrow\RR$ defined by 
$$F_t(x)=f\bigg(\int_0^\infty x'(t+\xi)y'(\xi)w(\xi)d\xi\bigg),\quad x\in\Hh_1.$$
As already explained, $K\notin\Hh_1.$ Nevertheless, since $K$ is smooth away from the origin, it is possible to show that $F_t\in\mathcal{D}^1_{\{K\}}(\Hh_1).$ Indeed, for any $\delta>0,$ $K_\delta:=S(\delta)K\in\Hh_1$ and $F_t$ is Gateaux differentiable along the direction of $K_\delta.$ Writing 
$ \int_0^\infty x'(t+\xi)y'(\xi)w(\xi)d\xi=\langle S(t)\partial_\xi x, \partial_\xi y\rangle_{\Hh}$
and using chain rule, the Gateaux derivative is given by
$$DF_t(x)(K_\delta)=f'(\langle S(t)\partial_\xi x, y\rangle_{\Hh})\langle S(t)\partial_\xi K_\delta, \partial_\xi y\rangle_{\Hh}=f'(\langle S(t)\partial_\xi x, \partial_\xi y\rangle_{\Hh})\langle S(t+\delta)\partial_\xi K, \partial_\xi y\rangle_{\Hh}.$$
In fact, since $\Hh_1\ni x\mapsto \langle S(t)\partial_\xi x, \partial_\xi y\rangle_{\Hh}\in\RR$ is continuous linear and $f'\in \Cc(\RR),$ the map $\Hh_1\ni x\mapsto DF_t(x)(K_\delta)$
is continuous for any $\delta, t>0.$ From the above computations and the fact that $S(t)K\in\Hh_1$ it follows that the singular directional derivative $\mathcal{D}F_t(x)(K)$ exists and is given by
$$ \mathcal{D}F_t(x)(K)= f'(\langle S(t)\partial_\xi x, \partial_\xi y\rangle_{\Hh})\langle S(t)\partial_\xi K, \partial_\xi y\rangle_{\Hh}.$$
The same conclusion is actually true when $K$ is replaced by any element of the class $\mathscr{K}$ \eqref{eq:SingularDirections} (recall that $\Hh_1\subset \mathscr{K}\subset\Hh$)  since the latter consists of functions that are smooth away from the origin. Finally, the map $\Hh_1\ni x\mapsto \mathcal{D}F_t(x)(K)\in\RR $ is continuous, in sharp contrast to the linear map $\mathscr{K}\ni h\mapsto \mathcal{D}F_t(x)(h)\in\RR.$ Indeed, the latter is clearly not continuous when $\mathscr{K}$ is endowed with the subspace topology of $\Hh.$
\end{example}

\begin{example}[Non-differentiability of point evaluations]\label{ex:EvaluationNonDifferentiability} Let $d=1,$ $K$ be the power-law kernel and $K_\delta, \Hh_1$ as in the previous example. An important example of a functional that fails to be in $\Dd^1_{\{K\}}(\Hh_1)$ is the evaluation functional $ev_0:\Hh_1\rightarrow\RR.$ Indeed, this functional is linear and bounded and its Gateaux derivative along the regular direction $K_\delta\in\Hh_1$ is given by $$D ev_0(K_\delta)=ev_0(K_\delta)=K(\delta)\longrightarrow\infty$$
as $\delta\to 0.$ Hence $ev_0$ is not singularly differentiable along $K.$
\end{example}
\begin{remark}[Comparison to Malliavin derivatives] It is perhaps instructive to provide a comparison between  singular directional derivatives and the widely used Malliavin derivatives (see e.g. \cite{nualart2006malliavin}) for functionals $F:H_1 \rightarrow\RR,$ on an underlying abstract Wiener space $(H_1, \mathcal{K}, \gamma)$ (i.e. $\gamma$ is a Gaussian measure on a Hilbert space $H_1$ with Cameron-Martin space $\mathcal{K}\subset H_1$). Malliavin derivatives can be seen as ``restrictions" of classical Gateaux derivatives of $F$ to the smaller set of directions $\mathcal{K}.$ In contrast, singular derivatives $\mathcal{D}F(h)$ are rather extensions of Gateaux derivatives of $F:\Hh_1\rightarrow\RR$ and are defined with respect to the larger set of directions $h\in\mathscr{K}\supset\Hh_1.$ Indeed, when $h\in\Hh_1,$ $\mathcal{D}F(h)$ reduces to a classical Gateaux derivative per Remark \ref{rem:singular_derivative}.
\end{remark}

Next, we introduce an appropriate class of functionals that are classically (respectively singularly) differentiable in time (respectively space) and satisfy certain growth conditions. This is precisely the class of functionals for which we state our singular It\^o formula (Theorem \ref{thm:SingularIto}).

Define $K_\delta:=S(\delta)K$ for any~$\delta>0$ and the space of mixed singular and regular directions 
\begin{align}\label{def:spanK}
    \vspan(K,\Hh_1) := \Big\{ a_1K + a_2h: h\in\Hh_1,a_1,a_2\in\RR\Big\}. 
\end{align}
Clearly, $\vspan(K,\Hh_1)\subset\mathscr{K}$ and if the singular derivative $\Dd f(y)(K)$ exists then~$\Dd f(y)(\widetilde{K})$ also exists for any~$\widetilde{K}\in\vspan(K,\Hh_1)$.

\begin{definition}[The class $\Cc^{1,2}_{T,\mathscr{K}}$]\label{dfn:CKclass_alt} Let $T>0, \mathscr{K}_1$ as in \eqref{eq:yspace}.
    We say that $F:[0,T]\times\Hh_1\to\RR$ belongs to the class~$\Cc^{1,2}_{T,\mathscr{K}}$ if $F$ is bounded and there exists $\varpi\in L^1([0,T];\RR)$ such that the following hold:
    \begin{itemize}
        \item[i)] For all $y\in\mathscr{K}_1$, $F(\cdot,y)$ has a right-derivative~$\partial_t F$, for all $t\in[0,T)$ the map $\mathscr{K}_1\ni y\mapsto \partial_tF(t,y)\in\RR $ is continuous in the topology of $\Hh_1$ and
        \begin{align}\label{eq:defC12_dtF}
            \abs{\partial_t F(t,y)} \lesssim \varpi(t)+\abs{S(T-t)\partial_x y}_{\Hh_1}.
        \end{align}
         \item[ii)] For all~$t\in[0, T)$, $F(t,\cdot)$ is twice continuously Fréchet differentiable in~$\Hh_1$, there exists $\kappa:[0,1]^2\to\RR_+$ such that for all $\ep>0$ there is $\delta_0>0$ for which~$\delta,\delta'<\delta_0$ implies~$\kappa(\delta,\delta')<\ep$ and, for all~$y\in\mathscr{K}_1,\,h\in\mathscr{K}$, $\delta,\delta'>0,$  
\begin{equation}\label{eq:defC12_D2F_bound}
\begin{aligned}
   &\abs{DF(t,y)(h)}\lesssim  \abs{S(T-t)h}_{\Hh_1},
   \\&\abs{D^2F(t,y)(K_\delta,K_\delta)}\lesssim \varpi(t),
    \\&\abs{D^2F(t,y)\Big(K_{\delta}-K_{\delta'}, K_{\delta}+K_{\delta'} \Big)}
\lesssim 
\kappa(\delta,\delta')
\varpi(t).
    \end{aligned} 
    \end{equation}
    \item[iii)] For all $t\in[0,T), h_1\in\mathscr{K}, h_2\in\vspan(K,\Hh_1)$ the families 
    $$ \Big(  y\mapsto DF(t,y)(S(\delta)h_1)\Big)_{\delta\in(0,\delta')},\, \Big(  y\mapsto D^2F(t,y)(S(\delta)h_2, S(\delta)h_2)\Big)_{\delta\in(0,\delta')}       $$
    are equicontinuous at each~$y\in\Hh_1$ for $\delta'$ small enough.
    \end{itemize}  
\end{definition}
In Section~\ref{sec:BackwardEquation} we apply this definition to the map~$(s,y)\mapsto u(s,y):=\EE[\varphi(\lambda^{s,y}(T))]$. We prove in Section \ref{subsec:BackwardProof} that, under certain conditions including $q>4$ in Assumption~\ref{assumption:Kernel}, for any~$\delta,\delta'>0$,
\begin{align*}
& |Du(t,y)(h)|
\lesssim \abs{S(T-t)h}_{\Hh_1}\\
& |D^2u(t,y)(K_\delta, K_\delta)|^2
\lesssim 1 + \abs{S((T-t)/2)K}_{\Hh_1}^2+ \abs{S((T-t)/2)K}_{\Hh_1}^4 \\    &| D^2u(t,y)(K_{\delta}-K_{\delta'}, K_{\delta}+K_{\delta'} )|^2
   \\& \lesssim \Big(\abs{\delta-\delta'}^{\frac{q}{q-2}}+\abs{(S(\delta)-S(\delta'))S((T-s)/2) K}^2_{\Hh_1}\Big) \big(1 + \abs{S((T-t)/2)K}_{\Hh_1}^2\big),
\end{align*}
and identify $\kappa(\delta,\delta')\equiv\abs{\delta-\delta'}^{\frac{q}{q-2}}+\abs{(S(\delta)-S(\delta'))S((T-s)/2) K}^2_{\Hh_1}$. 
We stress that boundedness of $D^2u$ in the sense~$\abs{D^2u(t,y)(h_1,h_2)}\lesssim \abs{h_1}_{\Hh_1}\abs{h_2}_{\Hh_1}$ is not sufficient in this context; instead we exploit the regularisation property of the semigroup which provides time-integrable bounds. In view of this interlacing of time and space variables, this class of functions does not have a natural analogue for time-independent functions $F$.

Item ii) provides sufficient conditions for the (twice) singular differentiability of $F$ and item iii) ensures that the singular derivatives are continuous in $y$, as the following lemma guarantees. 
\begin{lemma}\label{lemma:singular_diff}
    If $F:[0,T]\times\Hh_1\to\RR$ is such that Definition \ref{dfn:CKclass_alt} ii) holds then for all $t\in[0, T),$         $F(t, \cdot)\in\mathcal{D}^1_{\!\mathscr{K}}(\Hh_1)\cap \mathcal{D}^2_{\vspan(K,\Hh_1)}(\Hh_1)$. If moreover Definition \ref{dfn:CKclass_alt} iii) holds then the maps
    $$ \Hh_1\ni y\longmapsto\mathcal{D}F(t,y)(h_1), \mathcal{D}F(t,y)(h_2, h_2)\in\RR            $$
    are continuous.
\end{lemma}
\begin{proof}
     We are going to prove that~$(DF(s,y)(S(\delta)h))_{\delta>0}$ and $(D^2F(s,y)(S(\delta)K, S(\delta)K))_{\delta>0}$ are Cauchy sequences, for any~$h\in\mathscr{K}_1$. For this purpose, take $\delta,\delta'>0$ and invoke \eqref{eq:defC12_D2F_bound} 
     \begin{align*}
         \abs{DF(s,y)(S(\delta)h)-DF(s,y)(S(\delta')h)}
         &= \abs{DF(s,y)\big((S(\delta)-S(\delta'))h\big)}\\
&\le \abs{(S(\delta)-S(\delta'))S(T-s)h}_{\Hh_1}.
     \end{align*}
     The first claim follows by continuity of the semigroup in~$\Hh_1$ (Lemma~\ref{lemma:semigroup_pties}) and the fact that $S(T-s)h\in\Hh_1$. The same reasoning applies to the second derivative since, by \eqref{eq:defC12_D2F_bound},
     \begin{align*}    &\abs{D^2F(s,y)\big(K_{\delta},K_{\delta}\big) - D^2F(s,y)\big(K_{\delta'},K_{\delta'}\big)}\\
    &=\abs{D^2F(s,y)\Big(K_{\delta}-K_{\delta'},K_{\delta}+K_{\delta'} \Big)}
    \lesssim
    \kappa(\delta,\delta') \varpi(s).
\end{align*}
These inequalities ensure that the sequences are Cauchy and therefore converge. The extension to~$\vspan(K,\Hh_1)$ presents no difficulty. 

As pointwise limits of a family of equicontinuous functions, the singular derivatives are continuous. This concludes the proof.
\end{proof}

\noindent A few comments on Definition \ref{dfn:CKclass_alt} are in order: 
\begin{remark}\
\begin{itemize}
    \item [i)] The fact that the upper bound for $\partial_t F$ is allowed to depend on $\partial_xy$ is not arbitrary. Such functions appear naturally in our study of backward Kolmogorov equations (Section \ref{sec:BackwardEquation}, Theorem \ref{thm:backward_equation_singular}). In fact, the definition of the class  $\Cc^{1,2}_{T,\mathscr{K}}$ is modelled after the behaviour of the probabilistic solution $u$ \eqref{eq:ValueFunction}  which is the ``canonical" example we have in mind.
    \item[ii)] We emphasise here that $\mathscr{K}_1$ apppearing in  Definition \eqref{dfn:CKclass_alt} is the largest subspace of $\Hh$ for which $S(T-s)\partial_x y\in\Hh_1$ for all $s\in[0,T).$ 
    \item[iii)] In view of Definition \ref{dfn:Derivative}, it is straightforward to verify that any $F\in \Cc^{1,2}_{T,\mathscr{K}}$ satisfies
 \begin{equation}\label{eq:CKSingularGrowth}
    \begin{aligned}
\abs{\mathcal{D}F(s,y)(h_1)}\lesssim \abs{S(T-s)h_1}_{\Hh_1}, \quad \abs{\mathcal{D}^2F(s,y)(K,K)}\lesssim \varpi(s)
    \end{aligned}
    \end{equation}
    for all $y\in\Hh_1$ and any singular directions $h_1\in\mathscr{K}$ \eqref{eq:SingularDirections}.
\end{itemize} 
\end{remark}

We are now ready to prove the main result of this section, a ``singular" It\^o formula for functionals in the class $\Cc^{1,2}_{T,\mathscr{K}}.$ Its proof is presented below, after a few preliminary remarks.

\begin{theorem}[Singular It\^o formula]\label{thm:SingularIto} 
Let Assumptions \ref{assumption:SVEassumptions} i)-ii) and \ref{assumption:Kernel} hold. 
Let $ T>0, t\in[0,T]$ and $F\in \Cc^{1,2}_{T,\mathscr{K}}$ per Definition \ref{dfn:CKclass_alt}. For all $(s,y)\in[0,t]\times\mathscr{K}_1$ we have
\begin{equation}\label{eq:SingularIto}
\begin{aligned}
    F(t,\lambda^{s,y}(t))&=F(s, y)+\int_s^{t}\bigg[ \partial_tF(r, \lambda^{s,y}(r))+\mathcal{D}F(r,\lambda^{s,y}(r))\Big( \partial_x \lambda^{s,y}(r)+K b_0(\lambda^{s,y}(r))\Big)\bigg]\dr
    \\&\quad +\int_s^t\frac{1}{2}\textnormal{Tr}\Big[ \mathcal{D}^2F(r,\lambda^{s,y}(r))K \sigma_0(\lambda^{s,y}(r))\big(K \sigma_0(\lambda^{s,y}(r))\big)^*\Big]\dr\\&\quad
    +\int_s^t \mathcal{D}F(r,\lambda^{s,y}(r))\big(K \sigma_0(\lambda^{s,y}(r)\dW_r\big),
\end{aligned}
   \end{equation}
    $\PP-$almost surely. In particular, for any $F\in\Cc^{1,2}_{T,\mathscr{K}},$ the process $\{F(t,\lambda^{s,y}(t))\}_{t\in[0,T]}$ is a semimartingale.
\end{theorem}

\begin{remark} Theorem \ref{thm:SingularIto} is, in a sense, interpolating between the strong and mild It\^o formulae from Sections \ref{subsec:StrongIto}, \ref{subsec:MildIto}. In particular, it shows that sufficiently smooth functionals of $\Hh_1-$valued mild solutions $\lambda$ \eqref{eq:SPDE_mild} satisfy a ``strong"-type It\^o formula that holds for initial conditions in $\mathscr{K}_1.$ Of course, such an ``intermediate" It\^o formula relies crucially on the special structure of \eqref{eq:SPDE}. Thus, an analogue of Theorem \ref{thm:SingularIto} is not expected to hold for generic SPDEs that do not enjoy the invariance and smoothing properties of \eqref{eq:SPDE}.

Furthermore, while the strong Itô formula holds when the data~$y,K\in\Hh_1$, this singular version operates a tradeoff: it allows the relaxation~$K\in\mathscr{K}$ and in exchange restricts~$y\in\mathscr{K}_1$ (recall that~$\mathscr{K}_1\subset \Hh_1\subset \mathscr{K}$).
\end{remark}
\begin{remark}
    The strategy of mollifying, applying Itô formula and then taking the limit does not apply at the level of the Volterra process $X$ itself. Its infinite quadratic variation appears in the limit. The lift, even though it is not a semimartingale either, is thus better behaved in that sense.
\end{remark}

\begin{proof}
   \noindent Throughout the proof, we consider $\mathscr{K}_1$ \eqref{eq:yspace} with the subspace topology of $\Hh_1.$ Our strategy proceeds as follows: First, we consider smooth initial data $y\in\Hh_2\subset\mathscr{K}_1$ and apply a strong It\^o formula to a ``smooth" approximation $\Lambda^{s,y}_\delta$ of $\lambda^{t,y}.$ Then, we extend this strong It\^o formula to all $y\in\mathscr{K}_1$ by a density argument. Finally, the main part of the proof is concerned with taking $\delta\to 0$ and arriving at the singular It\^o formula \eqref{eq:SingularIto}.

\noindent \underline{\textbf{Step 1:} Strong It\^o formula for $y\in\mathscr{K}_1.$} Let  $y\in\Hh_2\subset\mathscr{K}_1,$  $\delta>0, s\leq r\leq t\leq T,$ consider the process $\Lambda^{s,y}_\delta(r):=S(\delta)\lambda^{s,y}(r),$ as in the proof of Theorem \ref{thm:mildIto}, and recall that for all $y\in\Hh_2$ the latter is a strong $\Hh_1-$valued solution.  In view of Remark \ref{rem:InvariantSpaces} and Proposition \ref{prop:invariant subspaces} we have $y\in\mathscr{K}_1$ and, by invariance, $\Lambda^{s,y}_\delta(r)\in\mathscr{K}_1\implies \partial_x\Lambda^{s,y}_\delta(r)=S(\delta)\partial_x\lambda^{s,y}(r)\in\Hh_1$ almost surely for any $s.$  Since $F\in \Cc_{T,\mathscr{K}}^{1,2},$ it is twice continuously Gateaux differentiable in space and differentiable in time. 
Moreover, the estimates \eqref{eq:defC12_D2F_bound} imply that, for each fixed $r,$ the continuous Gateaux differentials $y\mapsto(h\mapsto DF(r,y)(h)),y\mapsto((h_1,h_2)\mapsto D^2F(r,y)(h_1, h_2))$ are respectively in $\Cc(\Hh_1;\mathscr{L}(\Hh_1;\RR)), $ $\Cc(\Hh_1;\mathscr{L}(\Hh_1;\Hh_1))$ by Riesz representation, 
hence $F(r,\cdot)$ is twice continuously Fr\'echet differentiable in~$\Hh_1.$ 
An application of the strong It\^o formula to the process   $\{ F(r, \Lambda^{s,y}_\delta(r)); r\in[s,t]\}$  yields
\begin{equation}
   \label{eq:SingularItoPrelimit} 
\begin{aligned}
    F(t, \Lambda^{s,y}_\delta(t))-F(s, S(\delta)y)=&\int_s^{t}\partial_t F(r, \Lambda^{s,y}_\delta(r))\dr+\int_s^t DF(r,\Lambda^{s,y}_\delta(r))\bigg(K_\delta \sigma_0(\lambda^{s,y}(r))\dW_r\bigg)\\
    &+\int_s^t\bigg[ DF(r,\Lambda^{s,y}_\delta(r))\bigg( \partial_x \Lambda^{s,y}_\delta(r)+K_\delta b_0(\lambda^{s,y}(r))\bigg)\\
    &+\frac{1}{2}\textnormal{Tr}\Big[ D^2F(r,\Lambda^{s,y}_\delta(r))K_\delta \sigma_0(\lambda^{s,y}(r))(K_\delta \sigma_0(\lambda^{s,y}(r)))^*\Big]\bigg]\dr
    .
\end{aligned}
\end{equation}
 In order to extend the singular It\^o formula to $y\in\mathscr{K}_1$ we proceed as follows: For any $y\in\mathscr{K}_1,$ consider the sequence $\{y_n\}_{n\in\NN}:=\{S(1/n)y\}_{n\in\NN}\subset\Hh_2$ that converges to $y$ in $\Hh_1.$ Applying \eqref{eq:SingularItoPrelimit} for $y=y_n$ we obtain 
\begin{equation}\label{eq:nSingularIto}
\begin{aligned}
    F(t,\Lambda^{s,y_n}_\delta(t))&-F(s, y_n)=\int_s^t\bigg[ \partial_tF(r, \Lambda^{s,y_n}_\delta(r))+DF(r,\Lambda^{s,y_n}_\delta(r))\bigg( \partial_x \Lambda^{s,y_n}_\delta(r)+K b_0(\Lambda^{s,y_n}_\delta(r))\bigg)\bigg]\dr
    \\&+\int_s^t\frac{1}{2}\textnormal{Tr}\Big[ D^2F(r,\Lambda^{s,y_n}_\delta(r))K \sigma_0(\Lambda^{s,y_n}_\delta(r))(K \sigma_0(\Lambda^{s,y_n}_\delta(r)))^*\Big]\dr\\&
    +\int_s^t DF(r,\Lambda^{s,y_n}_\delta(r))\big(K \sigma_0(\Lambda^{s,y_n}_\delta(r))\dW_r\big).
\end{aligned}
\end{equation}
Again by invariance, and for all $s, t$ and $\delta>0$ the sequence $\{\Lambda^{s,y_n}_\delta(r)\}_{n\in\NN}=\{S(\delta)\lambda^{s,y_n}(r)\}_{n\in\NN}\subset\Hh_2$ hence $\{\partial_x\Lambda^{s,y_n}_\delta(r)\}_{n\in\NN}\subset\Hh_1$ almost surely. By definition of $\mathscr{K}_1,$ the same is true with $y_n$ replaced by~$y.$

By assumption, the real-valued functions  $b_0(y), \sigma_0(y), F(r,y)$ $,DF(r,y)(h_1),$ $D^2F(r,y)(h_1, h_2), $ $\partial_tF(r,y)$ are jointly continuous functions of $(y, h_1, h_2)\in \mathscr{K}_1\times\Hh_1\times\Hh_1.$ By Proposition \ref{Prop:Feller}, $\lambda^{s,y}(r)$ is continuous in probability with respect to $y\in\Hh_1.$ Thus, by passing to a subsequence in $n\in\NN,$ and with $\delta>0$ fixed, the left-hand side and all integrands on the right-hand side of \eqref{eq:nSingularIto} converge almost surely, as $n\to\infty,$ to the same expressions, with $y_n$ replaced by $y.$ Finally, from linear growth of $b_0,\sigma_0$ and the integrable bounds in Definition~\ref{dfn:CKclass_alt}, one can interchange the almost sure limit as $n\to\infty$ via dominated convergence (or via dominated convergence for stochastic integrals as in Step 1). Thus, for any $\delta>0,$ \eqref{eq:SingularItoPrelimit} remains true for $y\in\mathscr{K}_1$ and the proof is complete.

\underline{\textbf{Step 2:} $\delta\to 0$.} We fix $y\in\mathscr{K}_1$ and study the limit of \eqref{eq:SingularItoPrelimit} as $\delta\to 0.$ Starting from the two terms on the left-hand side, the almost sure convergence $ |\Lambda^{s,y}_\delta(t)-\lambda^{s,y}(t)|_{\Hh_1}\to 0,$ continuity of $F(t,\cdot)$ along with the convergence $\lim_{\delta\to 0}|(S(\delta)-I)y|_{\Hh_1}=0$ furnish
\begin{equation}
    \label{eq:SingularItoDeltaLimit1}
    \lim_{\delta\to 0} \bigg(F(t, \Lambda^{s,y}_\delta(t))-F(s, S(\delta)y)\bigg)=F(t, \lambda^{s,y}(T))-F(s,y)
\end{equation}
almost surely.

\noindent The time-derivative term requires some extra care since $\partial_tF(r,\cdot)$ is only defined and continuous on $\mathscr{K}_1$ (recall Definition \ref{dfn:CKclass_alt}). First, we notice that $\Lambda_\delta^{s,y}(r)\in\Hh_2\subset\mathscr{K}_1$ almost surely for all $s$ and $\delta>0.$ Moreover, $\Lambda_\delta^{s,y}(r)\rightarrow \lambda^{s,y}(r)\in\mathscr{K}_1$ almost surely and by continuity it follows that 
$$   \partial_tF(r,\Lambda^{s,y}_\delta(r))\rightarrow \partial_tF(r,\lambda^{s,y}(r))      $$
almost surely. Moreover, the growth estimate \eqref{eq:defC12_dtF} implies
\begin{align*}    \int_s^t\EE|\partial_tF(r,\Lambda^{s,y}_\delta(r))|\dr&\lesssim  \int_s^t \Big(\varpi(r) + \EE\abs{S(T-r)\partial_x \Lambda^{s,y}_\delta(r)}_{\Hh_1}\Big)\dr\\&\lesssim 1+ \int_s^t\EE\abs{S(t-r)\partial_x S(\delta)\lambda^{s,y}(r)}_{\Hh_1}\dr\\&
\lesssim 1+\int_s^t\EE\abs{S(t-r)\partial_x\lambda^{s,y}(r)}_{\Hh_1}\dr,
\end{align*}
where we used the semigroup property on the second line and Assumption \ref{assumption:Kernel} to claim that~$\int_t^T \varpi(r)\dr$ is finite.

In order to apply the dominated convergence theorem we need to show that the last integral is also finite. To this end, we use the semigroup property once again to obtain
\begin{equation*}
    \begin{aligned}
        S(t-r)\lambda^{s,y}(r)&=S(t-r)S(r-s)y+S(t-r)\int_s^r S(r-\theta)Kb_0(\lambda^{s,y}(\theta))\D\theta\\&
        \quad+S(t-r)\int_s^r S(r-\theta)K\sigma_0(\lambda^{s,y}(\theta))\dW_\theta\\&
        =S(t-s)y+\int_s^r S(t-\theta)Kb_0(\lambda^{s,y}(\theta))\D\theta+\int_s^r S(t-\theta)K\sigma_0(\lambda^{s,y}(\theta))\dW_\theta.
    \end{aligned}
\end{equation*}
From linear growth of $b_0, \sigma_0,$ the BDG inequality for the stochastic integral, estimate \eqref{eq:lambdaestimate} 
Assumption \ref{assumption:Kernel} and 
the estimate
\begin{align}\label{eq:ShiftedDxlambda}
    \int_s^t\bigg(\int_s^r|S(t-\theta)\partial_x K|^2_{\Hh_1}\D \theta\bigg)^{\frac{1}{2}}\dr 
    &= \int_s^t\bigg(\int_s^r|S(t-r)S(r-\theta)\partial_xK|^2_{\Hh_1}\dr\bigg)^{\frac{1}{2}}\dr \nonumber\\
    &= \int_0^{t-s}\bigg(\int_0^{r-s} |S(r+\theta)\partial_x K|^2_{\Hh_1}\D \theta\bigg)^{\frac{1}{2}}\dr <+\infty 
\end{align}
\noindent (which follows by substitution and crude bounds on the domain of integration) we deduce that, indeed, 
$$   \int_s^t\EE\abs{S(t-r)\partial_x\lambda^{s,y}(r)}_{\Hh_1}\dr<\infty$$
(notice that similar estimates can be found in the proof of Proposition \ref{prop:invariant subspaces}). 
Hence, by virtue of the dominated convergence theorem it follows that \begin{equation}  \label{eq:SingularItoDeltaLimit2}
\lim_{\delta\to 0}\int_s^t\partial_tF(r,\Lambda^{s,y}_\delta(r))\dr=\int_s^t\partial_tF(r,\lambda^{s,y}(r))\dr
\end{equation}
in $L^1(\Omega).$

Next, we study the first order spatial derivatives of $F.$ To this end, we invoke again Proposition \ref{prop:invariant subspaces} which says that for all $y\in\mathscr{K}_1, s\geq t$, $\partial_x \lambda^{s,y}(r)\in\mathscr{K}$ is an admissible singular direction.

By continuity of the (Gateaux) derivative $DF(r,\cdot)(h),$ and singular directional differentiability of $F$ 
(recall Lemma \ref{lemma:singular_diff}) along any direction in $\mathscr{K}$ we have, for all $s,r,y,$ 
\begin{equation}\label{eq:SingularItoDFConvergence}
    \begin{aligned}        DF(r,\Lambda^{s,y}_\delta(r))\bigg( \partial_x S(\delta)\lambda^{s,y}(r)&+K_\delta b_0(\lambda^{s,y}(r))\bigg)\\&=DF(r,S(\delta)\lambda^{s,y}(r))\bigg(S(\delta)\bigg(\partial_x \lambda^{s,y}(r)+K b_0(\lambda^{s,y}(r))\bigg)\bigg)\\&
        \longrightarrow \mathcal{D}F(r,\lambda^{s,y}(r))\bigg(\partial_x \lambda^{s,y}(r)+K b_0(\lambda^{s,y}(r))\bigg),
    \end{aligned} 
    \end{equation}
    as $\delta\to 0,$ almost surely.

    As before, we shall pass the limit as $\delta\to 0$ under the sign of expectation and Riemann integration to deduce the  $L^1(\Omega)$ limit   \begin{equation}\label{eq:SingularItoDeltaLimit3}
    \begin{aligned}
        \lim_{\delta\to 0}\int_s^t DF(r,\Lambda^{s,y}_\delta(r))&\bigg( \partial_x S(\delta)\lambda^{s,y}(r)+K_\delta b_0(\lambda^{s,y}(r))\bigg)\dr\\&=\int_s^t\mathcal{D}F(r,\lambda^{s,y}(r))\bigg(\partial_x \lambda^{s,y}(r)+K b_0(\lambda^{s,y}(r))\bigg)\dr.
    \end{aligned}
     \end{equation}
     Indeed, since $F$ satisfies \eqref{eq:defC12_D2F_bound} and $b$ has linear growth we have the pathwise estimate     
     \begin{align*}
     \bigg|DF(r,\Lambda^{s,y}_\delta(r))\bigg( \partial_x S(\delta)\lambda^{s,y}(r)&+K_\delta b_0(\lambda^{s,y}(r))\bigg)\bigg|\lesssim\bigg| S(T-r)\partial_x S(\delta)\lambda^{s,y}(r)+ S(T-r)K_\delta b_0(\lambda^{s,y}(r))\bigg|_{\Hh_1}\\&
     \lesssim \big| S(T-r)\partial_x \lambda^{s,y}(r)\big|_{\Hh_1}+ \big|S(T-r)K\big|_{\Hh_1} \big(1+\big|\lambda^{s,y}(r)\big|_{\Hh_1}\big).
     \end{align*}
 
Arguing as above, the right hand side is integrable over $(s,\omega)\in[t,T]\times\Omega$ and independent of $\delta.$ Hence \eqref{eq:SingularItoDeltaLimit3} holds true by dominated convergence.

A similar argument holds for the second derivative terms. Indeed, twice singular directional differentiability of $F(r,\cdot)$ along directions in $\vspan(K,\Hh_1)$ (Lemma \ref{lemma:singular_diff}) implies (at least up to a subsequence) the almost sure limit 
\begin{align*}
    \lim_{\delta\to0} \textnormal{Tr}\big[ D^2F(r,\Lambda^{s,y}_\delta(r))&K_\delta \sigma_0(\lambda^{s,y}(r))(K_\delta \sigma_0(\lambda^{s,y}(r)))^*\big]\\
    &=\lim_{\delta\to0}\sum_{i=1}^{m}D^2F(r,S(\delta)\lambda^{s,y}(r))\big(K_\delta \sigma_0(\lambda^{s,y}(r)e_i, K_\delta \sigma_0(\lambda^{s,y}(r))e_i\big)\\&
    = \sum_{i=1}^{m}\mathcal{D}^2F(r,\lambda^{s,y}(r))\big(K \sigma_0(\lambda^{s,y}(r)e_i, K \sigma_0(\lambda^{s,y}(r))e_i\big) \\&
    =\textnormal{Tr}\big[ \mathcal{D}^2F(r,\lambda^{s,y}(r))K \sigma_0(\lambda^{s,y}(r))(K \sigma_0(\lambda^{s,y}(r)))^*\big],
\end{align*}
which holds all $s,t,y$ and $\{e_i\}_{i=1,\dots m}$ is the standard basis of $\RR^m.$

Then, \eqref{eq:defC12_D2F_bound} and linear growth of $\sigma$ allows the application of dominated convergence to obtain
\begin{equation}\label{eq:SingularItoDeltaLimit4}
    \begin{aligned}
        \lim_{\delta\to 0}\int_s^t\textnormal{Tr}&\big[ D^2F(r,\Lambda^{s,y}_\delta(r))K_\delta \sigma_0(\lambda^{s,y}(r))(K_\delta \sigma_0(\lambda^{s,y}(r)))^*\big]\dr
        \\&
        =\lim_{\delta\to 0}\sum_{i=1}^{m}\int_s^tD^2F(r,\Lambda^{s,y}_\delta(r))\bigg(K_\delta \sigma_0(\lambda^{s,y}(r))e_i, K_\delta \sigma_0(\lambda^{s,y}(r))e_i\bigg)\dr  
        \\&=\int_s^t\textnormal{Tr}\big[ \mathcal{D}^2F(r,\lambda^{s,y}(r))K \sigma_0(\lambda^{s,y}(r))(K \sigma_0(\lambda^{s,y}(r)))^*\big]\dW
    \end{aligned}
     \end{equation}
     in $L^1(\Omega).$

Turning to the stochastic integral in \eqref{eq:SingularItoPrelimit}, we write in coordinates 
$$ \int_s^tDF(r,\Lambda^{s,y}_\delta(r))\bigg(K_\delta \sigma_0(\lambda^{s,y}(r))\dW_r\bigg)=\sum_{i=1}^{m} \int_s^tDF(r,\Lambda^{s,y}_\delta(r))\bigg(K_\delta \sigma_0(\lambda^{s,y}(r))e_i\bigg)\dW^i_r.$$
For each $i=1,\dots,m,$ It\^o isometry along with the bounds \eqref{eq:defC12_D2F_bound}, linear growth of $\sigma_0$ and Assumption \ref{assumption:Kernel} furnish
\begin{equation*}
    \begin{aligned}   \EE&\bigg|\int_s^tDF(r,\Lambda^{s,y}_\delta(r))\bigg(K_\delta \sigma_0(\lambda^{s,y}(r))e_i\bigg)\dW^i_r\bigg|^2=   \int_s^t\EE\bigg|DF(r,\Lambda^{s,y}_\delta(r))\bigg(K_\delta \sigma_0(\lambda^{s,y}(r))e_i\bigg)\bigg|^2\dr\\&
    \leq  \int_s^t \EE|S(T-r)K_\delta \sigma_0(\lambda^{s,y}(r))e_i\bigg|_{\Hh_1}^2\dr
    \lesssim \left(1+\sup_{r\in[s,T]}\EE \abs{\lambda^{s,y}(r)}^2_{\Hh_1}\right) \int_s^t|S(T-r)K\big|^2_{\Hh_1}\dr<\infty.
\end{aligned}
\end{equation*}
Furthermore, the arguments leading to \eqref{eq:SingularItoDFConvergence} also imply the almost sure convergence 
\begin{equation*}
    \begin{aligned}        DF(r,\Lambda^{s,y}_\delta(r))\bigg( K_\delta \sigma_0(\lambda^{s,y}(r))e_i\bigg)
        \longrightarrow \mathcal{D}F(r,\lambda^{s,y}(r))\bigg(K \sigma_0(\lambda^{s,y}(r))e_i\bigg).
    \end{aligned} 
    \end{equation*}
By the dominated convergence theorem for stochastic integrals it follows that, for all $t,T,y,$ 
\begin{equation} \label{eq:SingularItoDeltaLimit5}
    \begin{aligned}   \lim_{\delta\to 0}\int_s^tDF(r,\Lambda^{s,y}_\delta(r))\bigg(K_\delta &\sigma_0(\lambda^{s,y}(r))\dW_r\bigg)=\int_s^t\mathcal{D}F(r,\lambda^{s,y}(r))\bigg(K\sigma_0(\lambda^{s,y}(r))\dW_r\bigg)
\end{aligned}
\end{equation}
in $L^0(\Omega).$

Finally, by passing to a subsequence over which the limits \eqref{eq:SingularItoDeltaLimit1}, \eqref{eq:SingularItoDeltaLimit2} and \eqref{eq:SingularItoDeltaLimit3}-\eqref{eq:SingularItoDeltaLimit5} hold almost surely, we take $\delta\to 0$ in \eqref{eq:SingularItoPrelimit} to obtain \eqref{eq:SingularIto} for all $(t,y)\in[0,T]\times\mathscr{K}_1.$ 
\end{proof}

We conclude with a technical remark on the proof strategy of Theorem \ref{thm:SingularIto}:
\begin{remark}\label{rem:singularItoProof}
As we already saw in Example \ref{ex:SingularDifferentiability}, singular derivatives at a fixed point $y\in\Hh_1$ $\mathcal{D}F(y)(\cdot)$ are not continuous with respect to directions in $\mathscr{K}.$ For this reason, in the above proof, we first extend the ``smooth" It\^o formula from initial conditions $y\in\Hh_2$ to $y\in\mathscr{K}_1$ by taking $n\to\infty$ and then pass to the ``singular" limit $\delta\to 0.$ Taking limits in this order allows us to use joint continuity properties of standard differentials $DF$ which may not be available in the singular limit.
\end{remark}


\subsection{Applications of the It\^o formulae}  \label{sec:applications_mildito}
In this section we present a number of applications of Theorems \ref{thm:mildIto}, \ref{thm:SingularIto}. The first two, an It\^o formula with random coefficients for the SVE $X$ and a mild Fokker--Planck Equation (FPE) for the law of mild solutions $\lambda$ \eqref{eq:SPDE_mild} follow from the mild It\^o formula. The third, a singular FPE follows from the mild It\^o formula. Before we proceed, we make here the following important remarks:

\begin{remark}[Comparisons between the mild and singular It\^o formulae and FPEs]\label{rem:FPEcomparison}\

\begin{itemize}
    \item[i)] In contrast to the the mild It\^o formula, the singular one does not lead to an It\^o formula for~$X.$ Indeed, the latter holds for functionals in the class $\Cc^{1,2}_{T,\mathscr{K}}$ (Definition \ref{dfn:CKclass_alt}).  Thus it cannot be applied to evaluation functionals $ev_0:\Hh_1\rightarrow\RR$ which, per Example \ref{ex:EvaluationNonDifferentiability}, are not singularly differentiable.
    \item[ii)] In Corollary \ref{cor:FPEuniqueness}, we show that the singular FPE \eqref{eq:singular_FPE} with initial data $\{\delta_y\}_{y\in\mathscr{K}_1}$ is \textit{uniquely} solved by the law of the state process $\lambda^{0,y}.$ This is done by duality and essentially follows from existence of solutions to the backward PDE \ref{thm:backward_equation_singular}(1). In contrast, due to absence of a ``mild" backward PDE we only show existence of solutions to  the mild FPE \eqref{eq:mildFP}.
   \item[iii)] The mild and singular It\^o formulae and FPEs apply to different classes of test functions. In particular the class $\Cc^{1,2}_{T,\mathscr{K}}$ are space-time test functions and satisfy much stronger conditions compared to (spatial only) test functions in $\Cc_b^2(\Hh_1).$  
\end{itemize}
    
\end{remark}

\subsubsection{Itô formula for Volterra processes}\label{subsubsec:ItoSVE}
For the first application of this section, we  ``project" the mild It\^o formula \eqref{eq:MildIto} back to the SVE \eqref{eq:SVE} and obtain a finite-dimensional It\^o formula for functionals of the non-semimartingale process $X.$ The main result of this subsection is the following:

\begin{corollary}
    [It\^o formula for SVEs]\label{prop:ItoSVE} Let $T>0, p\geq 2q/(q-2), X_0\in L^p(\Omega;\Hh_1),$ $X, \lambda$  denote the unique solutions of \eqref{eq:SVE} and \eqref{eq:SPDE} respectively on $[0,T].$ Next, let $f\in \Cc^{1,2}([0,T]\times\RR^d; \RR)$ such that  $$[0,T]\times\RR^d\ni(t,x)\longmapsto \partial_t f(t,x)\in \RR, \nabla_x f(t,x)\in\RR^d, D^2_x f(t,x)\in\mathscr{L}(\RR^d;\RR^d)$$
satisfy the growth conditions \eqref{eq:MildItoFgrowth} with $\Vv=\RR^d$ and $\Hh_1$ replaced by $\RR$. With probability $1,$ for all $s\leq t\in[0,T],$ it holds that
        \begin{equation}\label{eq:SVEIto}
    \begin{aligned}  f\big(t, X_t\big)&=f\big(s, \lambda(s, t-s)\big)+\int_s^t \nabla_x f\big(r, \lambda(r, t-r)\big)K(t-r)\sigma\big(X_r\big)\dW_r\\&
  +\int_s^t\bigg[\partial_tf\big(r, \lambda(r, t-r)\big)+ \nabla_x f\big(r,\lambda(r, t-r)\big)K(t-r)b\big(X_r\big)\bigg]\dr\\&
  +\frac{1}{2}\int_s^t\textnormal{Tr}\bigg[ D^2_xf\big(r,\lambda(r, t-r)\big)K(t-r)\sigma\big(X_r\big)\Big( K(t-r)\sigma\big(X_r\big)\Big)^{\top}\bigg]\dr.    
    \end{aligned}
\end{equation}
\end{corollary}

\begin{proof} We shall apply Theorem \ref{thm:mildIto} with $V:=\RR$ and $F:[0,T]\times\Hh_1\rightarrow V$ given by $F(t,x):=f(t, ev_0(x)).$ From the lift property \ref{eq:liftproperty} it follows that $F(t, \lambda(t))=f(t, X_t).$ In order to show that $F$ is Fr\'echet differentiable in the spatial variable we omit the time-dependence as it does not affect the proof. 
To this end, Lemma \ref{lemma:Hwproperties} guarantees that  $ev_0:\Hh_1\rightarrow \RR^d$ is well-defined and continuous. Since it is a linear operator, $ev_0\in \Cc^2(\Hh_1)$ with directional Fr\'echet derivatives along $h_1, h_2\in\Hh_1$ given by $Dev_0(h_1)=ev_0(h_1), D^2ev_0(h_1,h_2)=0.$ For $f\in \Cc^2(\RR^d;\RR),$ let 
    $F:=f\circ ev_0$. By Fr\'echet chain rule we have $F\in \Cc^2(\Hh_1;\RR)$ and $D_hF(x)=Df(ev_0(x))D_hev_0=Df(x(0))h(0)$
for $x\in\Hh_1, h\in\Hh_1.$ For the second derivative, an application of the product rule yields:
\begin{equation*}
\begin{aligned}
    D^2F(x)(h_1,h_2)&= D^2f(ev_0(x))\bigg(D ev_0(h_1), Dev_0(h_2)\bigg)+Df(ev_0(x))D^2ev_0(h_1, h_2)\\&=D^2f(x(0))h_1(0)h_2(0).
\end{aligned}
\end{equation*}
Hence the function $F$ satisfies the required differentiability assumptions. Finally, notice that $F$ inherits the growth properties of $f$ since the evaluation operator is linear and bounded. By noticing that $ev_0(S(t-r)x)=x(t-r)$ and~$ev_0(S(t-r)h)=h(t-r)$, the proof is complete.
\end{proof}

Several observations on the SVE It\^o formula are collected in the following remark

\begin{remark}\
\begin{enumerate}
    \item[i)] Letting $s=0$ in \eqref{eq:SVEIto} we see that the first term on the right-hand side reduces to $f(0,\lambda(0,t))=f(0, X_0(t))$ which is a function of the initial curve of the SVE \eqref{eq:SVE}.
    \item[ii)] In general, the It\^o formula for $X$ is tied to the SPDE solution $\lambda.$ The presence of this infinite-dimensional process on the right-hand side of \eqref{eq:SVEIto} compensates for the lack of semimartingality of $X.$ Nevertheless, a smooth functional $f(X)$ of an SVE remains an SVE (albeit with a random, non-convolutional kernel).
    \item[iii)]  
    The It\^o formula \eqref{eq:SVEIto} can also be obtained from the (singular) functional Itô formula derived in \cite[Theorem 3.17]{viens2019martingale}. 
    Nevertheless, we emphasise that the assumptions in force here are weaker than those needed for the functional Itô formula. We refer to Section \ref{sec:comparison_VZ} for a more detailed comparison. 
    A similar It\^o formula for SVEs was obtained in \cite{huber2024markovian} by using the OU lift (see Section \ref{subsec:OUlift}) in negative Sobolev spaces.
    \item[iv)] The mild and SVE It\^o formulae obtained above differ from strong It\^o formulae due to the presence of the processes $S(t-s)\lambda(s), \lambda(s, t-s)$  that appear respectively as arguments in the differentials of the chosen test functions. Nevertheless, for fixed $t\in[0,T],$ these processes are in fact semimartingales in $s\in[0,t].$
\end{enumerate}   
\end{remark}

\begin{remark} All results in this section continue to hold for SVEs with (probabilistically) weak solutions.
\end{remark}

\subsubsection{Mild Fokker--Planck equation}\label{sec:FPE}
\noindent From the mild It\^o formula \eqref{eq:MildIto}, it is possible to derive a mild Fokker--Planck Equation (FPE) for the law of $\lambda^{s,y}(t)$, defined as the measure $\mu_t^y\in\mathscr{P}(\Hh_1)$ such that $\langle \mu_t^y,\varphi\rangle:=\EE[\varphi(\lambda^{0,y}(t))],  \varphi\in \Cc_b(\Hh_1),$ and for $y\in\Hh_1$. As shown in \cite[Section 3.2.2]{da2019mild} and in particular, remark 3 therein, we have:
\begin{equation*}
    \begin{aligned}
        \langle \mu_t^y, \varphi\rangle&
        =\int_{\Hh_1} \varphi(S(t)z) \D\mu_0^y(z)+ \int_0^t\int_{\Hh_1}D\varphi\big(S(t-s)z \big)S(t-s) Kb_0(y)\D\mu_s^y(z)\ds\\&
  +\frac{1}{2}\int_0^t\int_{\Hh_1}
  \frac{1}{2}{\rm Tr} \bigg[D^2\varphi\big(S(t-s)z\big)S(t-s)K\sigma_0(z) \big(S(t-s)K \sigma_0(z)\big)^\ast \bigg]  
  \D\mu_s^y(z)\ds.  
    \end{aligned}
\end{equation*}
Letting, for each $s\geq 0, \varphi\in \Cc_b^2(\Hh_1)$, 
\begin{equation}\label{eq:mildgenerator}
\begin{aligned}
     \mathcal{L}_{s}\varphi(z):=D\varphi\big(S(s)z \big)S(s)K b_0(z)
     +\frac{1}{2}{\rm Tr} \bigg[D^2\varphi\big(S(s)z\big)S(s)K\sigma_0(z) \big(S(s)K \sigma_0(z)\big)^\ast \bigg],
\end{aligned}
\end{equation}
and noticing that~$\mu_0^y=\delta_y$ we can write the mild FPE in the form
\begin{equation}\label{eq:mildFP}
\begin{aligned}
    \langle \mu_t^y, \varphi\rangle&= \varphi(S(t)y) + \int_0^t\big\langle \mu_s^y,\mathcal{L}_{t-s}\varphi\big\rangle \ds.
\end{aligned}
\end{equation}

 \subsubsection{Singular Fokker--Planck equation}
Similarly to the mild FPE being derived from the mild Itô formula in Section~\ref{sec:FPE}, we are able to obtain a singular FPE. Recall that the law~$\mu^y_t\in\mathscr{P}(\Hh_1)$ of~$\lambda^{0,y}(t),t\in[0,T], y\in\mathscr{K}_1$ \eqref{eq:yspace} is defined by~$\langle\mu_t^y,\varphi\rangle := \EE[\varphi(\lambda^{0,y}(t))]$ for all~$\varphi\in \Cc_b(\Hh_1)$. Letting $\Phi\in\Cc^{1,2}_{T,\mathscr{K}}$ as per Definition \ref{dfn:CKclass_alt} (note that for each fixed $t\in[0,T]$ the section $\Phi_t(\cdot):=\Phi(t,\cdot)\in C_b(\Hh_1)$), we apply the singular Itô formula \eqref{eq:SingularIto} (with $s=0$) and take expectations. Thus, the stochastic integral vanishes and we obtain
\begin{equation*}
\begin{aligned}
    \langle \mu^y_t,\Phi_t\rangle 
    &=    \langle \delta_y,\Phi_0\rangle
    +\int_0^{t} \int_{\mathscr{K}_1} \bigg[ \partial_t \Phi(s,z)+\mathcal{D}\Phi(s,z)\bigg( \partial_x z+K b_0(s,z)\bigg)\bigg] \D\mu_s( z) \dr
    \\&\quad +\int_0^{t} \int_{\mathscr{K}_1} \frac{1}{2}\textnormal{Tr}\Big[ \mathcal{D}^2 \Phi(s,z)K \sigma_0(s,z)(K \sigma_0(s,z))^*\Big]  \D\mu_s( z) \ds \\
    &=:\langle \delta_y,\Phi_0\rangle + \int_0^t\int_{\mathscr{K}_1} \mathcal{L}\Phi(s,z)  \D\mu_s(z)\ds,
\end{aligned}
\end{equation*}
where $\delta_y$ is a Dirac measure centered at the function $y\in\mathscr{K}_1.$ The fact that the spatial integrals on the right hand side are restricted to the subspace $\mathscr{K}_1\subset\Hh_1$ is a consequence of invariance (Proposition~\ref{prop:invariant subspaces})\footnote[4]{We stress out that invariance plays, once again, a crucial role: for all $z\in\mathscr{K}_1,$ $\partial_xz\in\mathscr{K}$ is an ``admissible" singular direction per Definition \ref{dfn:Derivative}. Without this property, it is not clear whether the directional derivative $\mathcal{D}\Phi(r,z)( \partial_x z)$ is well-defined.}. We note here that both the test function $\Phi\in\Cc^{1,2}_{T,\mathscr{K}}$ and the formal ``generator'' 
\begin{equation}\label{eq:singularFPEgenerator}  \mathcal{L}\Phi(t,y)=\partial_t\Phi(t,y)+\mathcal{D}\Phi(t,z)\bigg( \partial_x z+K b_0(t,z)\bigg)+\frac{1}{2}\textnormal{Tr}\Big[ \mathcal{D}^2 \Phi(t,z)K \sigma_0(t,z)(K \sigma_0(t,z))^*\Big],
\end{equation}
where $ (t,y)\in[0,T]\times\mathscr{K}_1 $, depend on the time variable.

An important difference between the mild and singular FPEs is that the latter admits unique solutions in the sense of the following corollary:
\begin{corollary}[Uniqueness of singular Fokker--Planck equation]\label{cor:FPEuniqueness}
 Let $T>0, y\in\mathscr{K}_1$, $\mathcal{L}$ as in~\eqref{eq:singularFPEgenerator}, Assumptions \ref{assumption:SVEassumptions}i), iii) and \ref{assumption:Kernel} with $q>4$ hold.    If  $\{\mu_t\}_{t\in[0,T]}\subset\mathscr{P}(\Hh_1)$ is a family of probability measures that satisfies, for all $\Phi\in\Cc^{1,2}_{T,\mathscr{K}},$  the singular Fokker--Planck equation
 \begin{equation}\label{eq:singular_FPE}
  \langle \mu_t,\Phi_t\rangle=\langle \delta_y,\Phi_0\rangle + \int_0^t\int_{\mathscr{K}_1} \mathcal{L}\Phi(s,z)  \D\mu_s(z)\ds,   
 \end{equation}
then $\mu_t^y=\textnormal{Law}(\lambda_t^y)$ for all $t\in[0,T].$ In other words, the singular FPE admits a unique solution.  If $\sigma$ is constant, the same conclusion holds when $q>2.$ 
\end{corollary} 
The proof relies on a classical duality argument and follows from our existence result, Theorem \ref{thm:backward_equation_singular}(1),  of the backward Kolmogorov equation. A similar argument can also be found for example in \cite[Corollary 3.7]{diehl2017stochastic}.
\begin{proof}  Let $\gamma_0>0$ and $\varphi\in C_b^2(\Hh_1)$ with $\gamma_0-$H\"older continuous second derivative. Next, consider the function $u_t(y):=P_{T-t}\varphi(y)=\EE[\varphi(\lambda_T^{t,y})], y\in\Hh_1, t\in[0,T].$ Under our assumptions and by virtue of Theorem \ref{thm:backward_equation_singular}(1), $u\in \Cc^{1,2}_{T,\mathscr{K}}$ solves the Cauchy problem \eqref{eq:BackwardPDECauchyProblem} with terminal condition $\varphi$. Thus the second term on the right-hand side of \eqref{eq:singular_FPE} vanishes and, for all $t,$ we have $ \langle \mu^y_t,u_t\rangle=\langle \delta_y,u_0\rangle=u(0,y),$ i.e. the function $[0, T]\ni t\mapsto \langle \mu^y_t,u_t\rangle\in\RR$ is constant. 
Consequently, any two solutions $\mu^y_t, \nu^y_t$ of the singular FPE must satisfy
$$\langle \mu^y_T, \varphi\rangle=\langle \mu^y_T, u_T\rangle=u(0,y)=\langle\nu^y_T, u_T\rangle=\langle\nu^y_T, \varphi\rangle,$$
for all~$\varphi\in \Cc_b^2(\Hh_1)$ with the aforementioned regularity. A fortiori, the last equality holds for all test functions $\varphi_{h}(x):=\exp(i\langle x,h\rangle_{\Hh_1}), x,h\in\Hh_1$ (these are indeed smooth and their derivatives of all orders are uniformly bounded in $x$) after splitting into real and imaginary parts. This means that the Fourier transforms $\hat{\mu^y}, \hat{\nu^y}$ agree and hence $\mu^y_T=\nu^y_T$ upon inversion. Since $T>0$ was arbitrary the proof is complete.    \end{proof}

\section{Kolmogorov backward equation}\label{sec:BackwardEquation}

From this point on, Assumption \ref{assumption:SVEassumptions}(iii) is in place. This section is devoted to the proof of our main result, Theorem \ref{thm:backward_equation_singular}. Before we proceed, we recall the notation $\Hh_2:=H^2_w(\RR^+;\RR^d),$ where $w\in\mathcal{W}_a$ is an admissible weight per Assumption \ref{assumption:Kernel}.

In particular, and in analogy to our proof of the singular It\^o formula \eqref{eq:SingularIto}, we shall obtain the backward Kolmogorov equation \eqref{eq:BackwardPDECauchyProblem} as a limit of classical Kolmogorov equations associated to a ``mollified" process $\lambda_\delta.$ The latter solves the same equation \eqref{eq:SPDE} as $\lambda,$ but with $K$ replaced by the shifted kernel $K_\delta:=S(\delta)K$ which is smooth away from the origin by Assumption \ref{assumption:Kernel}. Note that~$\lambda_\delta$ is different from~$\Lambda_\delta=S(\delta)\lambda$ defined in the previous section. For any~$\lambda_0\in\Hh_2$ and~$\delta>0$, $\lambda_\delta$ is an $\Hh_1$--mild solution to the corresponding SPDE and, furthermore, Assumption~\ref{assumption:Kernel} guarantees that~$\lambda_\delta(t)\in\Hh_2$. Thus~$\lambda_\delta$ is an analytically strong solution in~$\Hh_1$ and, for $T>0, \varphi\in \Cc^2_b(\Hh_1),$ the function 
\begin{equation}\label{eq:vdelta} u_\delta(t,y)=\EE[\varphi(\lambda_\delta^{t,y}(T))], \quad (t,y)\in[0,T]\times\Hh_1
\end{equation} 
 satisfies, for all $(t, y)\in[0,T]\times\Hh_2,$ 
 the classical (backward) Kolmogorov equation
\begin{equation}\label{eq:KolmoBackwardDelta}
   \begin{aligned}   u_\delta(t,y)=\varphi(y)+\int_t^T\bigg(\big\langle Dv_\delta(s,y),  \partial_x y+ K_\delta b_0(y)\big\rangle_{\Hh_1}+\frac{1}{2}\textnormal{Tr}\Big[ D^2v_\delta(s,y)K_\delta \sigma_0(y)(K_\delta \sigma_0(y))^*\Big]             \bigg)\ds.
   \end{aligned} 
\end{equation}
The proof of this fact can be derived from~\cite[Theorem 9.25]{da2014stochastic} and the homogeneity of the Markov semigroup discussed in Section~\ref{subsec:MarkovSemigroup} (which associates the Markov semigroups~$P_{T-t}$ and~$P_{t,T}$). Equation~\eqref{eq:KolmoBackwardDelta} can also be written in differential form after applying the temporal derivative on both sides of the equality.

Upon taking the limit $\delta\to0$, our main result states that~$u(t,y):=\EE[\varphi(\lambda^{t,y}(T))]$, $(t,y)\in[0,T]\times\Hh_1$ is the unique $\Cc^{1,2}_{T,\mathscr{K}}$ solution of the backward Kolmogorov equation
\begin{equation}\label{eq:BackwardPDECauchyProblem}
    \left\{\begin{aligned}
    &\partial_tu(t,y)+\mathcal{D}u(t,y)\Big( \partial_x y+K b_0(y)  \Big)+\frac{1}{2}\textnormal{Tr}\Big[ \mathcal{D}^2u(t,y)K \sigma_0(y)(K \sigma_0(y))^*\Big]=0\;,   (t,y)\in[0,T)\times\mathscr{K}_1,\\&
         u(T,y)=\varphi(y),\quad  y\in\Hh_1.
         \end{aligned}\right.
    \end{equation}

We state this result below in a rigorous fashion. Its proof requires several preparatory lemmas and is postponed to Section \ref{subsec:BackwardProof}.

\begin{theorem}[Kolmogorov backward equation] \label{thm:backward_equation_singular} Let $T>0, \gamma_0>0$ and $\varphi\in \Cc_b^2(\Hh_1)$ with $\gamma_0$-H\"older continuous $D^2\varphi$ 
and $\lambda$ be the SPDE mild solution \eqref{eq:SPDE_mild}. Under Assumptions \ref{assumption:SVEassumptions}(i), (iii) and \ref{assumption:Kernel} with $q>4$ and with
$\mathscr{K}, \mathscr{K}_1, \mathcal{D}, \Cc_{T,\mathscr{K}}^{1,2}$ as in \eqref{eq:SingularDirections}, \eqref{eq:yspace} and Definitions \ref{dfn:Derivative}, \ref{dfn:CKclass_alt} respectively the following hold:

\begin{enumerate}
    \item (Existence) The  function \begin{equation}\label{eq:ValueFunction}
     [0,T]\times\Hh_1\ni (t, y)\longmapsto u(t,y):=\EE[\varphi(\lambda^{t,y}(T))]\in\RR
 \end{equation}
  is in 
 $\Cc^{1,2}_{T,\mathscr{K}}.$
 Moreover, for all $(t,y)\in[0,T]\times \mathscr{K}_1$,
 $u$ solves the Kolmogorov equation \eqref{eq:BackwardPDECauchyProblem}.
\item (Uniqueness) Let $u\in \Cc_{T,\mathscr{K}}^{1,2}$ be a solution of the Cauchy problem~\eqref{eq:BackwardPDECauchyProblem}. Then $u$ admits the probabilistic representation 
   $ u(t,y):=\EE[\varphi(\lambda^{t,y}(T))].$ Consequently, \eqref{eq:BackwardPDECauchyProblem} admits a unique solution in the space $\Cc_{T,\mathscr{K}}^{1,2}.$ 
    \item (Computation of conditional expectations) For all $0\leq s\leq t\leq T, y\in\mathscr{K}_1,$  
\begin{align}\label{eq:conditionalExpectationsBackward}
    \EE[\varphi(\lambda^{s,y}(T))\big| \mathcal{F}_t]=u(t,\lambda^{s,y}(t)),
\end{align}      
almost surely,
where $u$ is the unique solution to \eqref{eq:BackwardPDECauchyProblem}. Furthermore, the following martingale representation stands
\begin{align*}
    u(t,\lambda^{s,y}(t)) = u(s,y) + \int_s^t \Dd u(r,\lambda^{s,y}(r)) \Big(K\sigma_0(\lambda^{s,y}(r))\D W_r\Big),\quad s\le t\le T.
\end{align*}
\end{enumerate}
All of the above remain true when $q>2$ and the noise coefficient $\sigma$ is constant.
\end{theorem}

An immediate corollary of the previous theorem is the following dynamic characterization for conditional expectations of the SVE \eqref{eq:SVE}.

\begin{corollary}[Conditional expectations of $X$]\label{cor:SVEConditionalExp} Let $T>0,$ $X$ denote the unique solution of the SVE~\eqref{eq:SVE} on $[0,T]$ with initial curve $X_0\in\mathscr{K}_1$ and $\phi\in \Cc_b^{2}(\RR^d)$ with $D^2\phi$ $\gamma_0-$H\"older continuous for some $\gamma_0>0.$ Under Assumptions \ref{assumption:SVEassumptions} i), iii) and \ref{assumption:Kernel} with $q>4$ we have
$$     \EE[\phi(X_T)\big| \mathcal{F}_t]=u(t,\lambda^{0,X_0}(t))\;,\;\; t\in[0,T],      $$
where $u$ is the unique solution of the Cauchy problem \eqref{eq:BackwardPDECauchyProblem} with terminal condition $\varphi=\phi\circ ev_0$ and $\lambda^{0, X_0}$ is the unique mild solution of  \eqref{eq:SPDE} with initial condition $X_0.$ If $\sigma$ is constant,  the same conclusion holds for all $q>2.$  
\end{corollary}

\begin{proof} The assumptions on $b,\sigma, K, X_0$ guarantee well-posedness for $\lambda^{0, X_0}$ per Theorem~\ref{thm:SPDE_wellposedness}. Moreover, the smoothness of $\phi,$ along with the computations in Corollary~\ref{prop:ItoSVE}  imply that the terminal condition $\varphi=\phi\circ ev_0\in \Cc_b^{2}(\Hh_1)$ with a H\"older continuous second derivative and then Theorem~\ref{thm:backward_equation_singular} (1), (2) asserts the well-posedness of the backward equation. Finally, from the lift property of $\lambda$~\eqref{eq:lift_property} and Theorem~\ref{thm:backward_equation_singular} 3) applied to this terminal condition we conclude that 
$$   \EE[\phi(X_T)\big| \mathcal{F}_t]=\EE[\varphi(\lambda^{0, X_0}_T)\big| \mathcal{F}_t]       $$
 almost surely.   
\end{proof}

The rest of this section is organised as follows: With Definition \ref{dfn:Derivative} of singular directional derivatives, we shall show that, for each $ 0\leq s\leq t, $ the map
$$\Hh_1 \ni y\longmapsto \lambda^{s,y}(t)\in L^0(\Omega;\Hh_1) $$
belongs to the class $\mathcal{D}^1_{\!\mathscr{K}}(\Hh_1;\Hh_1)\cap \mathcal{D}^2_{\vspan(K,\Hh_1)}(\Hh_1;\Hh_1)$  (Section \ref{subsec:singularDifferentiablity}). Before doing so, we introduce in Section \ref{subsec:SingularTangents} two processes $\zeta_h, \zeta_{K,K}$ that are natural candidates for the differentials $\mathcal{D}_y\lambda^{s,y}(t)(h),$ $ \mathcal{D}_y\lambda^{s,y}(t)(K,K)$ for $h\in\mathscr{K}$. 
Section \ref{subsec:ConvergenceDeltaTo0} is devoted to several technical approximation lemmas for (singular) tangent processes that are required in the proof of Theorem \ref{thm:backward_equation_singular}. The latter is presented in Section \ref{subsec:BackwardProof}.

Sections \ref{subsec:SingularTangents}, \ref{subsec:singularDifferentiablity} and \ref{subsec:ConvergenceDeltaTo0} provide a number of estimates regarding $\lambda,\zeta_h,\zeta_{K,K}$ and their mollified versions. All the proofs follow the same strategy that we outline here for the convenience of the reader. 
The first step consists in framing the process of interest into a Gr\"onwall-type inequality. For this purpose we will often rely on a Volterra--Gr\"onwall inequality which we state now and postpone its proof to Appendix~\ref{appendix:Gronwall}. 

For a kernel~$k\in L^1([0,T],\RR^+)$, we call $r$ its resolvent if it is the unique solution to
\begin{align*}
    r(t) - k(t) = \int_0^t r(t-s)k(s)\ds,\quad \text{for all  }t\in[0,T],
\end{align*}
or, said more concisely, $r-k=r\star k$.
\begin{lemma}[Volterra--Gr\"onwall inequality]\label{lemma:Volterra_Gronwall}
    Let $k\in L^1([0,T],\RR^+)$ and $r$ be its resolvent. Let $x,f\in L^1([0,T],\RR)$ and assume that, for almost all $t\in[0,T]$,
    \begin{align*}
        x(t)\le f(t)+\int_0^t k(t-s)x(s)\ds.
    \end{align*}
    Then it holds, for almost all $t\in[0,T]$,
    \begin{align*}
        x(t) \le f(t) + \int_0^t r(t-s)f(s)\ds.
    \end{align*}
\end{lemma}
In the second step, we shall give an estimate for the source term~$f(t)$ (or convergence properties thereof as the relaxation parameter $\delta$ tends to zero); the third and last step is analogous for the other term~$\int_0^t r(t-s)f(s)\ds$. The restriction on the range of~$q$ arises because the kernel~$\overline{K}:=\abs{S(\cdot)K}^2_{\Hh_1}$ (and its resolvent) belongs to~$L^{q/2}_T$ according to Assumption~\ref{assumption:Kernel}.

\subsection{Properties of singular tangent processes}\label{subsec:SingularTangents}

We begin here our study of differentiability of the random solution flow of \eqref{eq:SPDE}. We emphasise once again that the abuse of notation outlined in \eqref{eq:AbuseofNotations} is also applied at the level of the tangent processes $D_y\lambda^y(h_1)=\zeta_{h},D_y\lambda^y(h_1, h_2)=\zeta_{h_1, h_2}.$ In particular, we will often plug in the matrix-valued kernel $K$ in the place of $h_1, h_2.$ Of course, as explained above, expressions like $zeta_{K}, zeta_{K, K}$ will always be interpreted componentwise in the direction.

Before proceding to our analysis, we make the following important remark on the (functional) differentiability of the nonlinear terms $b,\sigma.$

\begin{remark}[Fr\'echet differentiability of coefficients]\label{rem:NemytskiiRem} The nonlinear operators $b_0, \sigma_0:\Hh_1\rightarrow\RR^d, \RR^{d\times m}$ are twice Fr\'echet differentiable (for example, $Db_0(\lambda)(h)=\nabla b(ev_0(\lambda))ev_0(h)$ for $\lambda, h\in\Hh_1$) since they only depend on the finite-dimensional projections $ev_0(\lambda).$ Consequently, the same is true for the ``singular" coefficients $Kb_0, K\sigma_0:\Hh_1\rightarrow L^2_w(\RR^+; \RR^d),L^2_w(\RR^+; \RR^{d\times m}).$ This is in sharp contrast to Nemytskii operators $B:\Hh_1\rightarrow\Hh_1$ of the form $B(\lambda)(x):=b(\lambda(x)), x\in\RR^+,$ which, as is well known, are only Gateaux differentiable (except for trivial cases; see e.g. \cite[Remark 6]{gasteratos2023moderate}). This additional regularity means that (classical) directional derivatives $D^i_y\lambda^y, i=1,2,$ as defined in Corollary \ref{cor:regularTangentProcesses} below, can be taken along any direction in $\Hh_1.$ This differs to the setting of SPDEs with Nemytskii operators where, typically, tangent processes are well-defined under additional smoothness of the direction (see e.g. \cite[Chapter 4.2]{cerrai2001second}).   
\end{remark}

\begin{lemma}[Well-posedness of the first variation equation]\label{lem:zetaKWP} Let $y\in\Hh_1, 0\leq s< t\leq T$ and $h\in\mathscr{K}$ satisfy
\begin{equation}\label{eq:zetaDirectionIntegrability}
    \int_s^T|S(r-s)h|^2_{\Hh_1}\dr<\infty.
\end{equation}
\noindent Under assumptions \ref{assumption:SVEassumptions}i), iii), \ref{assumption:Kernel}, the integral equation 
\begin{equation}\label{eq:zetaKEqRough}
    \begin{aligned}
         \zeta(t) 
    &= S(t-s)h + \int_{s}^t S(t-r)K Db_0(\lambda^{s,y}(r)) (\zeta(r)) \D r\\
    &+ \int_{s}^t S(t-r)K D\sigma_0(\lambda^{s,y}(r)) (\zeta(r)) \D W_r\;, t\in(s, T],
    \end{aligned}
\end{equation}
admits a unique $\Hh_1-$valued (mild) solution with $\zeta(s)=h$ per Definition \ref{dfn:mildsolutions} with $p=2.$ Furthermore, for all $t\in(s,T],$ the map $\mathscr{K}\ni h\longmapsto \zeta(t)\in L^2(\Omega,\Hh_1)$ is linear.
\end{lemma}
\begin{proof}

Let $p\in[1, \infty)$ and $\mathscr{Z}^p_{s,T}(\Hh_1)$ denote the Banach space of of $\{\mathcal{F}_t\}_{t\geq 0}-$adapted processes $\zeta\in L^p([s,T]; L^p(\Omega;\Hh_1)),$ endowed with the norm
\begin{equation}\label{eq:Z2Space}    |\zeta|^p_{\mathscr{Z}^p_{s,T}}:=\int_s^T\EE[|\zeta(r)|^p_{\Hh_1} ]\dr.
\end{equation}

\noindent We shall first show that the map
\begin{equation}\label{eq:zetaFixedPoint}
\begin{aligned}
 \zeta\longmapsto\mathcal{I}_1(h,\zeta)(t)&:= S(t-s)h + \int_{s}^t S(t-r)K Db_0(\lambda^{s,y}(r)) (\zeta(r)) \D r
    \\&+ \int_{s}^t S(t-r)K D\sigma_0(\lambda^{s,y}(r)) (\zeta(r)) \D W_r\;,\;\;t\in[s,T],
\end{aligned}
\end{equation}
    maps $\mathscr{Z}^2_{s,T}(\Hh_1)$ to itself. Then, we prove that $\mathcal{I}_1(h,\cdot)$ is a contraction.
    Indeed, due to boundedness of the Fr\'echet derivatives $Db_0, D\sigma_0,$ integrability of $h,$ BDG inequality and H\"older's inequality with exponent $q/2>1$ we have 
    \begin{equation*}
        \begin{aligned}
            \EE|\mathcal{I}_1(h,\zeta)(t)|^2_{\Hh_1}\lesssim |S(t-s)h|^2_{\Hh_1}+2\int_s^t|S(t-r)K|^2_{\Hh_1}\EE|\zeta(r)|^2_{\Hh_1}\dr.
        \end{aligned}
    \end{equation*}
    By Young's inequality for convolutions in the form $|f\star g|_{L^1}\leq |f|_{L^1}|g|_{L^1}$ it follows that
    \begin{equation*}
        \begin{aligned}        |\mathcal{I}_1(h,\zeta)|_{\mathscr{Z}^2_{s,T}}^2&=  \int_s^T\EE|\mathcal{I}_1(h,\zeta)(t)|^2_{\Hh_1}
            \dt\\&\lesssim \int_s^T|S(t-s)h|^2_{\Hh_1}\dt+\int_s^T\int_s^t|S(t-r)K|^2_{\Hh_1}\EE|\zeta(r)|^2_{\Hh_1}\dr\dt\\&
            \leq \int_s^T|S(t-s)h|^2_{\Hh_1}\dt+\bigg(\int_s^T|S(t)K|^2_{\Hh_1}\dt\bigg)|\zeta|_{\mathscr{Z}^2_{s,T}}^2<\infty.
        \end{aligned}
    \end{equation*}
    Turning to the contraction property, fix $\zeta_1, \zeta_2\in \mathscr{Z}^2_{s,T}.$ Using linearity (in $\zeta$) of the right-hand side and repeating similar arguments, we arrive at the bound
    \begin{equation*}
        \begin{aligned}  |\mathcal{I}_1(h,\zeta_2)-\mathcal{I}_1(h,\zeta_1)|_{\mathscr{Z}^2_{s,T}}^2\lesssim \bigg(\int_s^T|S(t)K|^2_{\Hh_1}\dt\bigg)|\zeta_2-\zeta_1|_{\mathscr{Z}^2_{s,T}}^2
         \end{aligned}
    \end{equation*}
    which holds up to a constant independent of $T.$ Taking $T$ sufficiently small, the prefactor on the right-hand side can be taken to be strictly smaller than $1.$ This implies the contraction property which, by Banach's fixed point theorem, yields the existence of a unique solution $\zeta$ to \eqref{eq:zetaKEqRough} in the space $\mathscr{Z}^2_{s,T}$ for $T$ small enough. Well-posedness for all $T>0$ then follows by concatenation of unique solutions on small enough time intervals. Furthermore, the moment estimate from Definition \ref{dfn:mildsolutions}i) with $p=2$ follows automatically from the fact that $|\zeta|_{\mathscr{Z}^2_{s,T}}<\infty.$

   The fixed point property $|\zeta-\mathcal{I}_1(h,\zeta)|_{\mathscr{Z}^2_{s,T}}=0$ implies that $\zeta$ only satisfies \eqref{eq:zetaKEqRough} for almost all $t\in(s, T].$ However, in view of \eqref{eq:zetaFixedPoint}, standard arguments show that $\zeta\in \Cc((s,T];L^2(\Omega;\Hh_1))$ with possible discontinuity only at the initial time $s$ (recall that the initial condition $h$ need not be an element of $\Hh_1$). 
   This implies that $\PP(\forall t\in(s, T],   \zeta(t)=\mathcal{I}_1(h,\zeta)(t)    )=1.$ Uniqueness then follows for all $t\in[s, T],$ provided that we look for solution processes which are modified to equal $h$ at time $s.$ The proof is complete.
   
   Finally, linearity in the initial condition follows by Fr\'echet differentiability of the coefficients $b_0, \sigma_0,$ linearity of the semigroup and a standard Gr\"onwall argument. In particular, for $h_1, h_2\in\mathscr{K},$ the process $\Delta\zeta:=\zeta_{h_1+h_2}-\zeta_{h_1}-\zeta_{h_2}$ (i.e. with $\zeta_{h_1+h_2}(0)=h_1+h_2$) satisfies
   \begin{equation*}
       \begin{aligned}
         \Delta\zeta(t)=\int_{s}^t S(t-r)K Db_0(\lambda^{s,y}(r)) (\Delta\zeta(r)) \D r+\int_{s}^t S(t-r)K D\sigma_0(\lambda^{s,y}(r)) (\Delta\zeta(r)) \dW_r
       \end{aligned}
   \end{equation*}
   and Gr\"onwall's inequality yields $\Delta\zeta(t)=0$ for all $t\in(s, T]$ almost surely.
   \end{proof}

Our subsequent analysis of the backward Kolmogorov equation requires control of high moments for solutions to \eqref{eq:zetaKEqRough}. The following lemma provides such moment estimates, along with continuity bounds for $\zeta$ with respect to initial conditions $y$ of $\lambda.$

\begin{lemma}[Properties of $\zeta_h$]\label{lem:zetaKrough} Let $y,z\in\Hh_1, 0\leq s< t\leq T, h\in\mathscr{K}$ satisfying \eqref{eq:zetaDirectionIntegrability} and $\zeta^{s,y}_h$ be the unique $\Hh_1-$valued solution to \eqref{eq:zetaKEqRough}. Under assumptions \ref{assumption:SVEassumptions}i),iii), \ref{assumption:Kernel} and for any $p\geq 1$ there exists a constant $C_T>0$ such that 
\begin{align}\label{eq:estimate_zeta}
    &\EE|\zeta^{s,y}_{h}(t)|^p_{\Hh_1}\leq    C_T|S(t-s)h|^{p}_{\Hh_1}.
\end{align}  
Moreover, the following hold:
\begin{enumerate}
    \item If $\sigma$ is constant then 
\begin{align}    
    \EE|\zeta^{s,y}_{K}(t)-\zeta^{s,z}_{K}(t)|_{\Hh_1}^2\lesssim \abs{y-z}_{\Hh_1}^2.
    \label{eq:regularity_zeta2}
\end{align}
    \item If, instead, Assumption \ref{assumption:Kernel} holds with $q\ge 4$, then
\begin{align}\label{eq:regularity_zeta4}
    \EE|\zeta^{s,y}_{K}(t)-\zeta^{s,z}_{K}(t)|^4_{\Hh_1}\lesssim \abs{y-z}^4_{\Hh_1}.
\end{align}
\end{enumerate}

The above conclusions hold without change for the process $\zeta^{s,y}_{\delta, h}$ which solves the same equation, albeit with $K$ replaced by $K_\delta$ in the integral terms.
\end{lemma}

\begin{proof} 
    
    Turning to the desired moment bound (and recalling that, by definition of $\mathscr{K}$ \eqref{eq:SingularDirections},  $S(t-s)h\in\Hh_1$),  Jensen's, BDG and H\"older inequalities furnish
    \begin{equation*}
        \begin{aligned}
            \EE|\zeta^{s,y}_{h}(t) |^p_{\Hh_1}&\leq |S(t-s)h|^p_{\Hh_1}+\EE\bigg(\int_{s}^t| S(t-r)K  |_{\Hh_1}|\zeta^{s, y}_{h}(r)|_{\Hh_1} \D r\bigg)^{p}\\&
            +\EE\bigg(\int_{s}^t |S(t-r)K |^2_{\Hh_1} |\zeta^{s, y}_{h}(r)|^2\D r\bigg)^{\frac{p}{2}}\\&
            \leq |S(t-s)h|^p_{\Hh_1}+    C_T\int_{s}^t  \EE|\zeta^{s, y}_{h}(r)|^p\D r,
        \end{aligned}
    \end{equation*} 
    provided that $p>\frac{2q}{q-2}.$
    By Gr\"onwall's inequality we obtain the desired estimate for large $p.$ For $p'\leq p$ we apply Jensen's inequality and the bound for large $p$ to obtain
    $$ \EE|\zeta^{s,y}_{h}(t) |^{p'}_{\Hh_1}\leq   (\EE|\zeta^{s,y}_{h}(t) |^{p}_{\Hh_1})^{p'/p}\leq C_T|S(t-s)h|^{p'}_{\Hh_1}. $$ 

    (1) Now we assume that $\sigma$ is constant; this entails the stochastic integral vanishes in~\eqref{eq:zetaKEqRough}. Leveraging the Lipschitz continuity of~$Db$ and Cauchy--Schwarz inequality we obtain
    \begin{align*}
        \EE|\zeta^{s,y}_{h}(t)-\zeta^{s,z}_{h}(t)|^2_{\Hh_1}
        \lesssim &\, \EE\left(\int_s^t \abs{S(t-r)K}_{\Hh_1}\abs{\lambda^{s,y}(r)-\lambda^{s,z}(r)}_{\Hh_1}\abs{\zeta^{s,z}_h(r)}_{\Hh_1}\dr\right)^2\\
        &\quad+ \EE\left(\int_s^t \abs{S(t-r)K}_{\Hh_1}|\zeta^{s,y}_{h}(r)-\zeta^{s,z}_{h}(r)|_{\Hh_1}\dr\right)^2 \\
        &\le \int_s^t \abs{S(t-r)K}^2_{\Hh_1}\dr \int_s^t \EE\Big[\abs{\lambda^{s,y}(r)-\lambda^{s,z}(r)}_{\Hh_1}^2\abs{\zeta^{s,z}_h(r)}_{\Hh_1}^2\Big]\dr\\
        &\quad+\int_s^t \abs{S(t-r)K}^2_{\Hh_1}\dr \int_s^t \EE|\zeta^{s,y}_{h}(r)-\zeta^{s,z}_{h}(r)|_{\Hh_1}^2\dr
    \end{align*}
    Cauchy--Schwarz inequality, the moment bound~\eqref{eq:estimate_zeta} and Proposition \ref{Prop:Feller} yield
    \begin{align*}
        \int_s^t \EE\Big[\abs{\lambda^{s,y}(r)-\lambda^{s,z}(r)}_{\Hh_1}^2\abs{\zeta^{s,z}_h(r)}_{\Hh_1}^2\Big]\dr
        &\lesssim \int_s^t \EE\Big[\abs{\lambda^{s,y}(r)-\lambda^{s,z}(r)}_{\Hh_1}^4\Big]^\half \EE\Big[\abs{\zeta^{s,z}_h(r)}_{\Hh_1}^4\Big]^\half\dr \\
        &\lesssim \abs{y-z}^2_{\Hh_1} \int_s^t \abs{S(r-s)K}^2_{\Hh_1}\dr.
    \end{align*}
    Gr\"onwall's inequality thus yields the claim.
    
    (2) Let Assumption \ref{assumption:Kernel} hold with $q\ge4$ instead. In a similar fashion, the moment bound~\eqref{eq:estimate_zeta}, Proposition \ref{Prop:Feller} and Assumption~\ref{assumption:Kernel}, along with Cauchy--Schwarz inequality, yield 
    \begin{align*}
        \EE|\zeta^{s,y}_{h}(t)-\zeta^{s,z}_{h}(t)|^4_{\Hh_1}
        \lesssim &\, \EE\left(\int_s^t \abs{S(t-r)K}^2_{\Hh_1}\abs{\lambda^{s,y}(r)-\lambda^{s,z}(r)}^2_{\Hh_1}\abs{\zeta^z_h(r)}^2_{\Hh_1}\dr\right)^{2}\\
        &\qquad+ \EE\left(\int_s^t \abs{S(t-r)K}^2_{\Hh_1}\EE|\zeta^{s,y}_{h}(r)-\zeta^{s,z}_{h}(r)|^2_{\Hh_1}\dr \right)^{2}.
    \end{align*}
    The moment bound~\eqref{eq:estimate_zeta} and Assumption~\ref{assumption:Kernel}, along with Cauchy--Schwarz inequality, yields 
\begin{align*}
    \EE&\left(\int_s^t \abs{S(t-r)K}^2_{\Hh_1}\abs{\lambda^{s,y}(r)-\lambda^{s,z}(r)}^2_{\Hh_1}\abs{\zeta^{s,z}_h(r)}^2_{\Hh_1}\dr\right)^{2}\\
    &\le \int_s^t \abs{S(t-r)K}^4_{\Hh_1} \dr \;\EE\int_s^t \abs{\lambda^{s,y}(r)-\lambda^{s,z}(r)}^4_{\Hh_1}\abs{\zeta^{s,z}_h(r)}^4_{\Hh_1}\dr \\
    &\lesssim \int_s^t \left(\EE\abs{\lambda^{s,y}(r)-\lambda^{s,z}(r)}^8\right)^\half \left(\EE\abs{\zeta^{s,z}_h(r)}^8\right)^\half\dr\\
    &\lesssim \sup_{r\in[s,t]}\left(\EE\abs{\lambda^{s,y}(r)-\lambda^{s,z}(r)}^{8}_{\Hh_1}\right)^\half \int_s^t \abs{S(t-r)K}^4_{\Hh_1}\dr \lesssim  \abs{y-z}_{\Hh_1}^4,
\end{align*}
A further application of Cauchy--Schwarz inequality yields
$$
\EE\left(\int_s^t \abs{S(t-r)K}^2_{\Hh_1}\EE|\zeta^{s,y}_{h}(r)-\zeta^{s,z}_{h}(r)|^2_{\Hh_1}\dr \right)^{2}
\le \int_s^t \abs{S(t-r)K}^4_{\Hh_1}\dr \int_s^t \EE\abs{\zeta^{s,y}_{h}(r)-\zeta^{s,z}_{h}(r)}^4_{\Hh_1}\dr.
$$
Collecting the previous estimates and applying Gr\"ownall's inequality yields the claim.

    Finally, due to the estimate   
    $$|S(t-r)K_\delta|_{\Hh_1}=|S(\delta)S(t-r)K|_{\Hh_1}\lesssim |S(t-r)K|_{\Hh_1},$$  
    all of the above remains true for the process $\zeta^{s,y}_{\delta, h}.$
\end{proof}

The proof strategy for the well-posedness and estimates of the singular tangent processes $\zeta_{K,K}$ is analogue to that of~$\zeta_h$, with the difference that we can only obtain $L^2_T$ estimates. This notably implies that we rely on a version of Gr\"onwall's inequality for Volterra processes, Lemma~\ref{lemma:Volterra_Gronwall}

Besides the latter, estimates in Lemmas \ref{lem:zetaKKRough}, \ref{lem:zetadeltaKL1limit} and \ref{lem:2ndTangentConvergence} rely on the following auxiliary lemma.
\begin{lemma}\label{lemma:resolvent}
    Let Assumption \ref{assumption:Kernel} hold. 
    The real-valued kernel~$t\mapsto \overline{K}(t):=\abs{S(t)K}_{\Hh_1}^2$ has a resolvent~$\overline{R}$ and it belongs to $L^{q/2}_T$. 
\end{lemma}
\begin{proof}
    We can check, for instance, that the conditions displayed in \cite{zhang2010stochastic} hold. Assumption \ref{assumption:Kernel} entails that $\sup_{t\in[0,T]} \abs{\int_0^t \overline{K}(s)\ds} <\infty$ and
    \begin{align*}
        &\sup_{t\in[0,T]}\abs{\int_t^{t+\ep} \overline{K}(t+\ep-s)\ds}
        = \abs{\int_0^\ep \overline{K}(s)\ds}
        \le \ep^{\frac{q}{q-2}} \int_0^T \abs{S(s)K}_{\Hh^1}^q\ds, 
    \end{align*}
    which converges to zero as $\ep\to0$. Lemma 2.1 from~\cite{zhang2010stochastic} thus ensures the existence of the resolvent~$\overline{R}\in L^1_T$.  Moreover, $\overline{K}\in L^{q/2}_T$ hence $\overline{R}\star\overline{K}\in L^{q/2}_T$ by Young's convolution inequality. Finally, $\overline{R}\in L^{q/2}_T$ since $\overline{R}=\overline{K} + \overline{R}\star\overline{K}$.
\end{proof}

\begin{lemma}[Well-posedness of the second variation equation]\label{lem:zetaKKWP} Let $y\in\Hh_1, 0\leq s< t\leq T$ and $h_1, h_2\in\vspan(K, \Hh_1)\subset\mathscr{K}.$
Next let $\lambda^{s,y}, \zeta^{s,y}_{h_1}, \zeta^{s,y}_{h_2}$ be the unique $\Hh_1-$valued solutions to \eqref{eq:generalised_SPDE} and~\eqref{eq:zetaKEqRough} respectively. 
\noindent Under assumptions \ref{assumption:SVEassumptions}i), iii) and \ref{assumption:Kernel} with $q>4$, the integral equation
\begin{equation}\label{eq:zetah1h2Rough}
\begin{aligned}
     \zeta(t) &= \int_s^t S(t-r)K D^2b_0(\lambda^{s,y}(r)) (\zeta^{s,y}_{h_1}(r),\zeta^{s, y}_{h_2}(r)) \D r\\ 
    &\quad+ \int_s^t S(t-r)K Db_0(\lambda^{s, y}(r)) (\zeta(r)) \D r \\
    &\quad + \int_s^t S(t-r )K D^2\sigma_0(\lambda^{s, y}(r)) (\zeta^{s, y}_{h_1}(r),\zeta^{s, y}_{h_2}(r)) \D W_r \\
    &\quad+ \int_s^t S(t-r)K D\sigma_0(\lambda^{s, y}(r)) (\zeta(r)) \D W_r
\end{aligned}  
\end{equation}
admits a unique $\Hh_1-$valued (mild) solution with $\zeta(0)=0,$ per Definition \ref{dfn:mildsolutions} with $p=2.$ Furthermore, for all $t\in(s,T],$ the map $\vspan(K, \Hh_1)\ni(h_1, h_2)\longmapsto \zeta^{s,y}_{h_1, h_2}(t):=\zeta(t)\in L^2(\Omega,\Hh_1)$ is bilinear. Finally, if $\sigma$ is constant well-posedness holds for all $q>2.$
\end{lemma}

\begin{remark}[On $H>1/4$]\label{rem:Honefourth}
The well-posedness of \eqref{eq:zetah1h2Rough} is conditional to the existence of the third integral on the right-hand side which is intimately linked to the (ir)regularity of~$\zeta^{s,y}_{h_i},$ $i=1,2$. In order to better explain the condition $q\geq 4,$ we specialize the discussion to the case $h_1=h_2=K$ (where we recall the abuse of notations \eqref{eq:AbuseofNotations}). Then, the estimate~\eqref{eq:estimate_zeta} and Itô isometry yield
\begin{align}\label{eq:Honefourth_pbterm}
    \EE\abs{\int_s^t S(t-r )K D^2\sigma_0(\lambda^{s, y}(r)) (\zeta^{s, y}_{K}(r),\zeta^{s, y}_{K}(r)) \D W_r}^2_{\Hh_1}
    \lesssim \int_s^t \abs{S(t-r)K}^2_{\Hh_1} \abs{S(r-s)K}^4_{\Hh_1}\dr.
\end{align}
This at least requires $r\mapsto \abs{S(r)K}_{\Hh_1}$ to belong to $L^4_{T}$, hence the assumption~$q\ge4$. 
For the power-law kernel $K(t)=t^{H-\half}$, we verified in Example~\ref{example:powerlaw} that such a condition was satisfied if~$4< \frac{2}{1-2H}$, which is equivalent to~$H>1/4$.
We do not know at this stage whether this limitation can be relaxed but we emphasise that it is not a peculiarity of this choice of lift. Rather this treshold seems to be fundamental for the existence of the second derivative.
In the case of additive noise, stochastic integrals vanish from the equations of $\zeta_K$ and $\zeta_{K,K}$ thus enabling the use of $L^1_T$ norms and circumventing the aforementioned problem. This condition is intimately related to the study of the second tangent process which has not been treated elsewhere in the literature. However, a similar yet stricter condition (which translates to $H>1/3$ for power-law kernels) has appeared in the context of dynamic programming principles for SVE \cite[Remark 3.2]{hamaguchi2025global}.
\end{remark}

\begin{proof} In view of Assumption \ref{assumption:Kernel}, the initial directions $h_i\in \vspan(K, \Hh_1) $ automatically satisfy the integrability condition \eqref{eq:zetaDirectionIntegrability} which suffices for existence of the tangent processes $\zeta^{s,y}_{h_i}, i=1,2.$ We now proceed to the proof of well-posedness in the case $h_1=h_2=h.$ By means of polarization, we can then naturally define, for any $h_1, h_2\in \vspan(K,\Hh_1),$ 
\begin{equation}\label{eq:2ndTangentPolarization}
    \zeta^{s,y}_{h_1,h_2}:=\frac{1}{4}\bigg(\zeta^{s,y}_{h_1+h_2,h_1+h_2}- \zeta^{s,y}_{h_1-h_2,h_1-h_2}   \bigg) 
\end{equation} 
with similar properties. In particular, if the processes on the right-hand side are unique mild solutions to linear evolution equations, it is straightforward to deduce that $\zeta^{s,y}_{h_1,h_2}$ satisfies \eqref{eq:zetah1h2Rough}. Since $h\mapsto\zeta^{s,y}_h$ is linear and the Fréchet derivatives~$(h_1,h_2)\mapsto D^2b_0(y)(h_1,h_2),D^2\sigma_0(y)(h_1,h_2)$ are bilinear for any~$y\in\Hh_1$, it follows that $(h_1,h_2)\mapsto\zeta^{s,y}_{h_1,h_2}$ is bilinear.

Turning to well-posedness, we consider
the mapping
\begin{equation}\label{eq:2ndVarSolutionMap}
    \begin{aligned}
      \zeta\longmapsto I_2(h, \zeta)(t)&:= \int_s^t S(t-r)K D^2b_0(\lambda^{s,y}(r)) (\zeta^{s,y}_{h}(r),\zeta^{s, y}_{h}(r)) \D r\\ 
    &+ \int_s^t S(t-r)K Db_0(\lambda^{s, y}(r)) (\zeta(r)) \D r \\
    & + \int_s^t S(t-r )K D^2\sigma_0(\lambda^{s, y}(r)) (\zeta^{s, y}_{h}(r),\zeta^{s, y}_{h}(r)) \D W_r \\
    &+ \int_s^t S(t-r)K D\sigma_0(\lambda^{s, y}(r)) (\zeta(r)) \D W_r\;,\;\;t\in[s, T],
    \end{aligned}
\end{equation}
and show that it has a unique fixed point in the Banach space  $\mathscr{Z}^2_{s,T}(\Hh_1)$ \eqref{eq:Z2Space} as in the proof of Lemma~\ref{lem:zetaKrough}. Starting from the linear terms, we use It\^o's isometry, Cauchy--Schwarz ineuality and boundedness of the Fr\'echet derivatives $Db, D\sigma$ to obtain  
\begin{equation*}
    \begin{aligned}
        \EE\bigg|   \int_s^t S(t-r)K Db_0(\lambda^{s, y}(r)) (\zeta(r)) \D r&+\int_s^t S(t-r)K D\sigma_0(\lambda^{s, y}(r)) (\zeta(r)) \D W_r\bigg|^2_{\Hh_1}\\&
        \lesssim \int_s^t |S(t-r)K|^2|\zeta(r)|^2_{\Hh_1}\dr.
    \end{aligned}
\end{equation*}
Integrating in $t$ and applying Young's convolution inequality as in Lemma \ref{lem:zetaKWP} we get
\begin{equation}\label{eq:2ndTangentWP1}
    \begin{aligned}
        \bigg|   \int_s^t S(t-r)K Db_0(\lambda^{s, y}(r)) (\zeta(r)) \D r&+\int_s^t S(t-r)K D\sigma_0(\lambda^{s, y}(r)) (\zeta(r)) \D W_r\bigg|^2_{\mathscr{Z}^2_{s,T}}&
        \lesssim 1+|\zeta|^2_{\mathscr{Z}^2_{s,T}}.
    \end{aligned}
\end{equation}
As for the source terms, It\^o's isometry, H\"older's inequality and the bound \eqref{eq:estimate_zeta} furnish
\begin{equation*}
    \begin{aligned}
        &\EE\bigg|\int_s^t S(t-r)K D^2b_0(\lambda^{s,y}(r)) (\zeta^{s,y}_{h}(r),\zeta^{s, y}_{h}(r)) \D r+ \int_s^t S(t-r )K D^2\sigma_0(\lambda^{s, y}(r)) (\zeta^{s, y}_{h}(r),\zeta^{s, y}_{h}(r)) \D W_r\bigg|^2_{\Hh_1}\\&
        \lesssim \int_s^t \abs{S(t-r)K}^2_{\Hh_1} \abs{S(r-s)h}^4_{\Hh_1}\dr
        =\int_0^{t-s} \abs{S(t-s-r)K}^2_{\Hh_1} \abs{S(r)h}^4_{\Hh_1}\dr
        =:(\overline{K}\star\overline{h}^2)(t-s),
    \end{aligned}
\end{equation*}
where we recall the notation $\overline{h}(t):=\abs{S(t)h}^2_{\Hh_1}, h\in\vspan(K, \Hh_1)$ from Lemma \ref{lemma:resolvent}. The latter then satisfies the bound
\begin{equation}\label{eq:K*K^2bound}
\begin{aligned}
    \overline{K}\star\overline{h}^2 (t-s) &=\int_0^{(t-s)/2} \abs{S(t-s-r)K}^2_{\Hh_1} \abs{S(r)h}^4_{\Hh_1}\dr+\int_{(t-s)/2}^{t-s}\abs{S(t-s-r)K}^2_{\Hh_1} \abs{S(r)h}^4_{\Hh_1}\dr\\&
    \leq \int_0^{(t-s)/2} \abs{S((t-s)/2-r)S((t-s)/2)K}^2_{\Hh_1} \abs{S(r)h}^4_{\Hh_1}\dr\\&
    +   \int_{(t-s)/2}^{t-s}\abs{S(t-s-r)K}^2_{\Hh_1} \abs{S(r-(t-s)/2)S((t-s)/2)h}^4_{\Hh_1}\dr\\&
    \le C_T\bigg(|S((t-s)/2)K|^2_{\Hh_1}\abs{\overline{h}}_{L^2_T}^2+ \abs{\overline{K}}_{L^1_T} \abs{S((t-s)/2)h}^4_{\Hh_1}\bigg)
\end{aligned}
\end{equation}
which is finite from Assumption \ref{assumption:Kernel} with $q\geq 4.$ The latter, along with \eqref{eq:2ndTangentWP1} implies that, for all $h\in\vspan(K, \Hh_1),$
$|I_2(h, \zeta)(t)|_{\mathscr{Z}^2_{s,T}}\lesssim 1+ |\zeta|^2_{\mathscr{Z}_{s,T}}<\infty$
i.e. $I_2(h,\cdot)$ maps $\mathscr{Z}_{s,T}(\Hh_1)$ to itself. As for the contraction property, the proof follows in exactly the same manner as that of Lemma \ref{lem:zetaKWP}. In particular, there exists $T_0>0$ small enough and $C_{T_0}<1$ such that  for any  $\zeta_1, \zeta_2\in\mathscr{Z}_{s,T}(\Hh_1),$ $ |\mathcal{I}_1(h,\zeta_2)-\mathcal{I}_1(h,\zeta_1)|_{\mathscr{Z}^2_{s,T}}\leq C_{T_0}|\zeta_2-\zeta_1|_{\mathscr{Z}^2_{s,T}}.$ By Banach's fixed point theorem and concatenation in $T>0$ we obtain a unique solution $\zeta\equiv \zeta^{s,y}_{h,h}$ in $\mathscr{Z}^2_{s,T}(\Hh_1)$ for all $T>0.$ The integrability bound from Definition \ref{dfn:mildsolutions}i) follows automatically by finiteness in $\mathscr{Z}^2_{s,T}-$norm. Existence and uniqueness for all (as opposed to almost all) $t\in[s, T]$ is true since $\zeta\in C((s, T]; \Hh_1)$ (which in turn follows by standard arguments).

Finally if $\sigma$ is constant, the stochastic integrals in \eqref{eq:2ndVarSolutionMap} are identicallty zero. This allows us to work in the space $\mathscr{Z}^1_{s,T}(\Hh_1)$ \eqref{eq:Z2Space} which in turn leads to well-posedness for all $q>2.$ Indeed, we have thanks to Lemma \ref{lem:zetaKrough} and Young's inequality
\begin{align*}
    \int_s^T \EE&\abs{\int_s^t S(t-r)K D^2b_0(\lambda^{s,y}(r)) (\zeta^{s,y}_{h}(r),\zeta^{s, y}_{h}(r)) \D r}_{\Hh_1} \dt
    \lesssim \int_s^T \int_s^t \abs{S(t-r)K}_{\Hh_1} \EE\abs{\zeta^{s,y}_{h}(r)}^2_{\Hh_1}\dr\dt \\
    &\le \int_s^T \int_s^t \abs{S(t-r)K}_{\Hh_1} \abs{S(r-s)h}_{\Hh_1}^2 \dr\dt 
    \le \int_s^T \abs{S(T-r)K}_{\Hh_1} \dr \int_s^T \abs{S(T-r)h}^2_{\Hh_1} \dr <\infty. 
\end{align*}
This ensures that $I_2(h,\cdot)$ maps $\mathscr{Z}^1_{s,T}(\Hh_1)$ to itself. The rest of the proof unfolds analogously to the first case. 
\end{proof}

\begin{lemma}[Properties of $\zeta_{h_1, h_2}$]\label{lem:zetaKKRough} Let $y,z\in\Hh_1, 0\leq s< t\leq T,$~ $h_1,h_2\in \vspan(K,\Hh_1)$ and $\zeta^{s,y}_{h_1, h_2}$ be the corresponding unique mild solution to \eqref{eq:zetah1h2Rough}. \\
(1) Let Assumption \ref{assumption:SVEassumptions}i), iii) and~\ref{assumption:Kernel} with $q\ge4$ hold. 
For any $p\in[2,q/2]$
\begin{align}\label{eq:estimate_zetah1h2}
       \EE|\zeta^{s,y}_{h_1, h_2}(t)|^{p}_{\Hh_1}
&\lesssim|\overline{h}_1|_{L^{p}_T}^{p/2}|\overline{h}_2|_{L^{p}_T}^{p/2}\big(1 + \overline{K}((t-s)/2)^{p/2}\big) + \overline{h}_1((t-s)/2)^{p/2} \overline{h}_2((t-s)/2)^{p/2},
   \end{align}
   where we denote~$\overline{h}_i(t):=\abs{S(t)h_i}_{\Hh_1}^2$, for $i=1,2$ and~$t\ge0$
In the particular case $h_1=h_2=K$ we also have
    \begin{align*}
 \EE|\zeta^{s,y}_{K, K}(t)-\zeta^{s,z}_{K, K}(t)|^{2}_{\Hh_1}
\lesssim \abs{y-z}^{2\gamma_0}_{\Hh_1}\left(1+|S((t-s)/2)K|^{2}_{\Hh_1}+|S((t-s)/2)K|^4_{\Hh_1}\right).
  \end{align*}
(2) If instead Assumption~\ref{assumption:Kernel} holds with $q\ge2$ and $\sigma$ is constant, the estimate~\eqref{eq:estimate_zetah1h2} is satisfied with $p\in[1,q/2]$. Moreover we have
    \begin{align*}
 \EE|\zeta^{s,y}_{K, K}(t)-\zeta^{s,z}_{K, K}(t)|_{\Hh_1}
\lesssim \abs{y-z}^{\gamma_0}_{\Hh_1}\left(1+|S((t-s)/2)K|_{\Hh_1}+|S((t-s)/2)K|^2_{\Hh_1}\right).
  \end{align*}
  The same conclusions remain true for the process 
$\zeta^{s,y}_{\delta, h_1,h_2}$ which solves \eqref{eq:zetah1h2Rough} with $K, \lambda^{s,y}, \zeta_K^{s,y}$ replaced by $K_\delta, \lambda_\delta^{s,y}, \zeta_{\delta, K_\delta}^{s,y}$ respectively. 
\end{lemma}

\begin{remark}
An alternative estimate one gets from the proof is
\begin{align*}
    \EE|\zeta^{s,y}_{K, K}(t)|^{p}_{\Hh_1}
\lesssim 1+\overline{K}^{p/2}\star\big(\overline{h}_1^{p/2}\overline{h}_2^{p/2}\big) (t-s).
\end{align*}
    In particular, with the choice~$h_1=h_2=K$ these estimates translate into
\begin{align}\label{eq:zetaKK_estimate}
      &\EE|\zeta^{s,y}_{K, K}(t)|^{p}_{\Hh_1}
      \lesssim 1+ \overline{K}^{p/2}\star\overline{K}^{p} (t-s)
\lesssim 1+|S((t-s)/2)K|^{p/2}_{\Hh_1}+|S((t-s)/2)K|^p_{\Hh_1}.
\end{align}
\end{remark}

\begin{proof}

\textbf{Case 1) Moment bound.}
Note that for any~$h\in\vspan(K,\Hh_1)$ we have $\int_0^T \abs{S(t)h}_{\Hh_1}^q\dt <\infty$ by Assumption~\ref{assumption:Kernel}. 
Turning to the desired a priori bound, Jensen's and H\"older inequalities furnish the estimate, for $q\ge4$ and $p\in[2,q/2]$, 
    \begin{equation*}
        \begin{aligned}
            \EE|\zeta^{s,y}_{h_1,h_2}(t) |^{p}_{\Hh_1}
            &\le
            \EE\int_s^t \abs{S(t-r)K}^{p}_{\Hh_1} \abs{\zeta_{h_1}^{s,y}(r)}^{p}_{\Hh_1}\abs{\zeta_{h_2}^{s,y}(r)}^{p}_{\Hh_1}\dr
            +\EE\left(\int_{s}^t |S(t-r)K |^2_{\Hh_1} |\zeta^{s, y}_{h_1,h_2}(r)|^2\D r\right)^{p/2}\\
            &\lesssim   \int_s^t \abs{S(t-r)K}^{p}_{\Hh_1} \abs{S(r-s)h_1}^{p}_{\Hh_1}\abs{S(r-s)h_2}^{p}_{\Hh_1}\dr 
            \\&\qquad + \abs{ \overline{K}}_{L^2_T}^{p-2}\int_{s}^t  |S(t-r)K |^2_{\Hh_1}\EE|\zeta^{s, y}_{h_1,h_2}(r)|^{p}\D r.
        \end{aligned}
    \end{equation*} 
   We remark that, since $q\ge4$, $\overline{K}\in L^2_T$ and Lemma \ref{lemma:resolvent} entails that $\overline{K}$ has a resolvent $\overline{R}\in L^2_T$. Recalling the notation~$\overline{h}_i(t)=\abs{S(t)h_i}_{\Hh_1}^2$, for $i=1,2$ and~$t\ge0$, we note that~$\overline{h}_i\in L^{q/2}_T$. Lemma~\ref{lemma:Volterra_Gronwall} thus yields
\begin{align}\label{eq:zetah1h2_integral_estimate}
    \EE|\zeta^{s,y}_{h_1,h_2}(t) |^{p}_{\Hh_1}
    \leq \overline{K}^{p/2}\star\big(\overline{h}_1^{p/2}\overline{h}_2^{p/2}\big) (t-s) + \overline{R}\star\Big(\overline{K}^{p/2}\star\Big(\overline{h}_1^{p/2}\overline{h}_2^{p/2}\Big)\Big) (t-s).
\end{align}
By the Cauchy--Schwarz and Young convolution inequality (the latter being used as $\abs{f\star g}_{L^2}\le \abs{f}_{L^1}\abs{g}_{L^2}$) we get the bound
\begin{equation}\label{eq:Lq_bound_zetah1h2}
    \begin{aligned}    &\overline{R}\star\overline{K}^{p/2}\star\big(\overline{h}_1^{p/2}\overline{h}_2^{p/2}\big)(t-s)
=\int_0^{t-s}\overline{R}(t-s-r)\Big(\overline{K}^{p/2}\star\Big(\overline{h}_1^{p/2}\overline{h}_2^{p/2}\Big)\Big)(r)\dr\\&
    \leq \abs{\overline{R}}_{L^2_{t-s}}|\overline{K}^{p/2}|_{L^2_{t-s}}|\overline{h}_1^{p/2}\overline{h}_2^{p/2}|_{L^{1}_{t-s}}
    \le \abs{\overline{R}}_{L^2_{t-s}}|\overline{K}^{p/2}|_{L^2_{t-s}}|\overline{h}_1^{p/2}|_{L^{2}_{t-s}} |\overline{h}_2^{p/2}|_{L^{2}_{t-s}}.
\end{aligned}
\end{equation}
The other term is bounded as in \eqref{eq:K*K^2bound}
\begin{align*}
    &\overline{K}^{p/2}\star\big(\overline{h}_1^{p/2}\overline{h}_2^{p/2}\big) (t-s)\\
    &\lesssim \overline{K}^{p/2}((t-s)/2) \int_0^{(t-s)/2}\overline{h}_1^{p/2}(r)\overline{h}_2^{p/2}(r)\dr 
    + \overline{h}_1^{p/2}((t-s)/2)\overline{h}_2^{p/2}((t-s)/2) \int_0^{(t-s)/2} \overline{K}^{p/2}(r)\dr\\
    &\lesssim  \overline{K}^{p/2}((t-s)/2)|\overline{h}_1^{p/2}|_{L^{2}_{t-s}} |\overline{h}_2^{p/2}|_{L^{2}_{t-s}} + \overline{h}_1^{p/2}((t-s)/2)\overline{h}_2^{p/2}((t-s)/2)
\end{align*}
and this yields the estimate~\eqref{eq:estimate_zetah1h2}.

\textbf{Case 1) Regularity.} In virtue of the Lipschitz continuity of~$Db,D\sigma$ and the H\"older continuity of~$D^2b,D^2\sigma$ (see Assumption \ref{assumption:SVEassumptions}) and H\"older's inequality for expectations, we obtain via similar computations
\begin{align}\label{eq:Comps_zetaKK_reg}
    \EE|\zeta^{s,y}_{K,K}(t)-\zeta^{s,z}_{K,K}(t) |^{2}_{\Hh_1}
            &\leq 
            \EE\int_s^t \abs{S(t-r)K}^{2}_{\Hh_1} \abs{\lambda^{s,y}(r)-\lambda^{s,z}(r)}^{2\gamma_0}_{\Hh_1} \abs{\zeta_K^{s,y}(r)}^{4}_{\Hh_1}\dr\\
            &\quad + \EE\int_s^t \abs{S(t-r)K}^{2}_{\Hh_1} \abs{\zeta_K^{s,y}(r)+\zeta_K^{s,z}(r)}^{2}_{\Hh_1} \abs{\zeta_K^{s,y}(r)-\zeta_K^{s,z}(r)}^{2}_{\Hh_1}\dr \nonumber\\
            &\quad +\EE\int_{s}^t |S(t-r)K |^2_{\Hh_1} |\zeta^{s, y}_{K,K}(r)-\zeta^{s, z}_{K,K}(r)|^2\D r\nonumber\\
            &\lesssim \sup_{r\in[s,T]} \left(\EE\abs{\lambda^{s,y}(r)-\lambda^{s,z}(r)}^{4\gamma_0}_{\Hh_1} \right)^\half 
            \int_s^t \abs{S(t-r)K}^{2}_{\Hh_1}  \left(\EE\abs{\zeta_K^{s,y}(r)}^{8}_{\Hh_1}\right)^\half\dr  \nonumber\\
            &\quad +  \int_s^t  \abs{S(t-r)K}^{2}_{\Hh_1}  \left(\EE\abs{\zeta_K^{s,y}+\zeta^{s,z}_K}^4\right)^\half \left(\EE\abs{\zeta_K^{s,y}-\zeta^{s,z}_K}^4\right)^\half \dr \nonumber\\
            &\quad + \int_{s}^t |S(t-r)K |^2_{\Hh_1} \EE|\zeta^{s, y}_{K,K}(r)-\zeta^{s, z}_{K,K}(r)|^{2} \D r \nonumber
\end{align}
Estimates obtained in Proposition~\ref{Prop:Feller} and~\ref{eq:regularity_zeta4} combined with
the Volterra--Gr\"onwall inequality  (Lemma \ref{lemma:Volterra_Gronwall}) furnish
\begin{align*}
    \EE|\zeta^{s,y}_{K,K}(t)-\zeta^{s,z}_{K,K}(t) |^{2}_{\Hh_1}
    &\leq \abs{y-z}^{2\gamma_0}_{\Hh_1}\left(\overline{K}\star\overline{K}^{2} (t-s) + \overline{R}\star(\overline{K}\star\overline{K}^{2}) (t-s)\right) \\
    &\qquad+ \abs{y-z}^{2}\big(\overline{K}\star\overline{K}(t-s) + \overline{R}\star(\overline{K}\star\overline{K}) (t-s)\big)
\end{align*}
and thus estimates~\eqref{eq:K*K^2bound} and~\eqref{eq:zetaKK_estimate} prove the claim.

\textbf{Case 2)} Let Assumption \ref{assumption:Kernel} hold with any $q\ge2$ and $\sigma$ be constant. 
The disappearance of the stochastic integrals yield the integral equation 
\begin{align*}
    \zeta^{s,y}_{h_1,h_2}(t) = \int_s^t S(t-r)K D^2b_0(\lambda^{s,y}(r)) (\zeta^{s,y}_{h_1}(r),\zeta^{s, y}_{h_2}(r)) \D r + \int_s^t S(t-r)K Db_0(\lambda^{s, y}(r)) (\zeta^{s, y}_{h_1,h_2}(r)) \D r.
\end{align*}
Applications of Cauchy--Schwarz inequality yield, for any~$p\in[1,q/2]$,
\begin{align*}
    \EE&\abs{\zeta^{s,y}_{h_1,h_2}(t)}_{\Hh_1}^{p}
    \lesssim \int_s^t \abs{S(t-r)K}_{\Hh_1}^p \EE\abs{\zeta^{s,y}_{h_1}(r)\zeta^{s, y}_{h_2}(r)}_{\Hh_1}^p\dr + \EE\left(\int_s^t \abs{S(t-r)K}_{\Hh_1} \abs{\zeta^{s,y}_{h_1,h_2}(r)}_{\Hh_1}\dr\right)^{p} \\
    &\lesssim 
    \int_s^t \abs{S(t-r)K}_{\Hh_1}^p\left(\EE\abs{\zeta^{s,y}_{h_1}(r)}^{2p}_{\Hh_1}\EE\abs{\zeta^{s, y}_{h_2}(r)}_{\Hh_1}^{2p}\right)^\half\dr\\
    &\quad + \left(\int_s^t \abs{S(t-r)K}_{\Hh_1}\dr\right)^{p-1} \int_s^t \abs{S(t-r)K}_{\Hh_1} \EE\big\lvert\zeta^{s,y}_{h_1,h_2}(r)\big\lvert_{\Hh_1}^{p} \dr\\
    &\lesssim \int_s^t \abs{S(t-r)K}_{\Hh_1}^p\abs{S(r-s)h_1}_{\Hh_1}^p\abs{S(r-s)h_2}_{\Hh_1}^p\dr + \int_s^t \abs{S(t-r)K}_{\Hh_1} \EE\big\lvert\zeta^{s,y}_{h_1,h_2}(r)\big\lvert_{\Hh_1}^{p} \dr.
\end{align*}
We apply the Volterra--Gr\"onwall inequality (Lemma \ref{lemma:Volterra_Gronwall}) with the kernel $\overline{K}^\half$ and its resolvent~$\overline{R}_\half$, followed by Cauchy--Schwarz inequality
\begin{align*}
    \EE\big\lvert\zeta^{s,y}_{h_1,h_2}(t)\big\lvert_{\Hh_1}^{p}
    &\lesssim \overline{K}^{p/2}\star \Big(\overline{h}_1^{p/2}\overline{h}_2^{p/2}\Big)(t-s) + \overline{R}_\half \star\Big( \overline{K}^{p/2}\star\Big(\overline{h}_1^{p/2}\overline{h}_2^{p/2}\Big)\Big)(t-s).
\end{align*}
This is the same inequality as \eqref{eq:zetah1h2_integral_estimate} albeit with $\overline{R}_\half$ instead of $\overline{R}$, hence the same arguments yield the claim.

The estimate for $\EE|\zeta^{s,y}_{K,K}(t)-\zeta^{s,z}_{K,K}(t) |_{\Hh_1}$ is derived by the same computations as in \eqref{eq:Comps_zetaKK_reg} albeit working with the $L^1(\Omega)$ norm instead of~$L^2(\Omega)$. 

Finally, it is straightforward to verify that, due to the estimate   $|S(t-r)K_\delta|_{\Hh_1}\lesssim |S(t-r)K|_{\Hh_1},$
    the last implication of Lemma \ref{lem:zetaKrough} and the boundedness of $D^ib,D^i\sigma, i=1, 2$
    all of the above remains true for the process $\zeta^{s,y}_{\delta, K_\delta,  K_\delta}.$
\end{proof}
When $h_1, h_2\in\Hh_1,$ the processes $\zeta_{h_1}, \zeta_{h_1, h_2}$ coincide with the classical \textit{tangent processes} of the solution flow $y\mapsto \lambda^y,$ i.e. Gateaux derivatives of the mild solution with respect to initial data (in fact the solution flow is Fr\'echet differentiable with respect to initial data as shown in the next corollary). Indeed, differentiation of the mild solution map $y\mapsto\mathcal{I}(y, \cdot)$ \eqref{eq:lambdaSolutionMap} implies: 
\begin{corollary}\label{cor:regularTangentProcesses} Let $s\ge0,h_1, h_2\in\Hh_1, \zeta_{h_1}, \zeta_{h_1, h_2}$ as in \eqref{eq:zetaKEqRough}, \eqref{eq:2ndTangentPolarization}. The map $\Hh_1\ni y\longmapsto\lambda^{s,y}(t)\in L^0(\Omega;\Hh_1)$ is twice continuously Fr\'echet differentiable along any directions $h_1, h_2\in\Hh_1$, with derivatives denoted by~$D_y \lambda^{s,y}(t)$ and~$D_y^2 \lambda^{s,y}(t)$. Moreover, for any $0\leq s\leq t\leq T, y\in\Hh_1$ it holds
$$D_y\lambda^{s,y}(t)(h_1)=\zeta^{s,y}_{h_1}(t)\;,\;\;  D^2_y\lambda^{s,y}(t)(h_1, h_2)=\zeta^{s,y}_{h_1,h_2}(t)      $$
almost surely.    
\end{corollary}

\begin{proof} Even though $\lambda$ solves an SPDE with singular coefficients, the proof of this result follows from standard fixed point results depending on parameters (e.g. \cite[Theorem 3.9]{gawarecki2010stochastic} or \cite[Chapter]{cerrai2001second}. To be a bit more precise, a close inspection of Assumptions \ref{assumption:SVEassumptions},\ref{assumption:Kernel}, along with the continuous Fr\'echet differentiability of the nonlinearities $b_0=b\circ ev_0, \sigma_0=\sigma\circ ev_0,$ asserts that the map $\Hh_1\times \mathscr{Z}^2_{s,T}(\Hh_1)\ni (y, \lambda)\mapsto \mathcal{I}(y, \lambda)\in \mathscr{Z}^2_{s,T}(\Hh_1)$ satisfies
Hypotheses C.1., C.2. from \cite[Appendix C]{cerrai2001second} and the argument is complete by invoking Proposition C.0.5. from the last reference. The fact that $\mathcal{I}$ maps the space $\mathscr{Z}^2_{s,T}(\Hh_1)$ \eqref{eq:Z2Space} to itself follows from standard modifications of Theorem \ref{thm:SPDE_wellposedness}.
\end{proof}

\subsection{Singular directional differentiability of $y\mapsto\lambda^y$}\label{subsec:singularDifferentiablity}

In this section, we shall show that, for any $t\in[s,T],$ the map $\Hh_1\ni y\mapsto\lambda^{s,y}(t)\in L^2(\Omega;\Hh_1)$ belongs to the class $\mathcal{D}^1_{\!\mathscr{K}}(\Hh_1;\Hh_1)\cap \mathcal{D}^2_{\vspan(K,\Hh_1)}(\Hh_1;\Hh_1)$ in the sense of Definition \ref{dfn:Derivative}. For any $h\in\mathscr{K}$, the singular directional derivatives $ \mathcal{D}_y\lambda^{s,y}(t)(h), \mathcal{D}_y\lambda^{s,y}(t)(K, K)$ are well-defined and coincide with the processes $\zeta^{s,y}_h(t), \zeta^{s,y}_{K,K}(t)$ as introduced in the previous section.

Let us here define $\overline{\Delta K}_\delta(t):=\abs{(S(\delta)-I)S(t)K}^2_{\Hh_1} =\abs{S(t)(K_\delta-K)}^2_{\Hh_1}$.
\begin{lemma}\label{lem:SingularDifferentiability} For any $ y\in\Hh_1, 0\leq s< t\leq T$ and $q>2$ as in Assumption \ref{assumption:Kernel} the following hold
\begin{enumerate}
    \item For any $p\geq 1, 
    \delta>0, h\in\mathscr{K}$ satisfying \eqref{eq:zetaDirectionIntegrability} and $h_\delta:=S(\delta)h,$ there exists a constant $C_T>0$ such that 
    \begin{align}\label{eq:zetaKdiffBnd}
        \EE|\zeta^{s,y}_{h_\delta}(t)-\zeta^{s,y}_{h}(t)|^p_{\Hh_1}\leq C_T |(S(\delta)-I)S(t-s)h|^p_{\Hh_1}. 
    \end{align}
    \item If $q\geq 4,$ there exists a constant $C_T^\delta>0$ with $\lim_{\delta\to 0}C_T^\delta=0$ such that     
    \begin{align}\label{eq:zetaKKdiffBnd_q4}
        \EE|\zeta^{s,y}_{K_\delta, K_\delta}(t)-\zeta^{s,y}_{K, K}(t)|^2_{\Hh_1}\lesssim C^\delta_T  \big(1 
        + \overline{K}((t-s)/2)\big) + \overline{\Delta K}_\delta((t-s)/2) \overline{K}((t-s)/2).
    \end{align} 
    If instead $q\ge2$ and $\sigma$ is constant then
    \begin{align}\label{eq:zetaKKdiffBnd_q2}
        \EE|\zeta^{s,y}_{K_\delta, K_\delta}(t)-\zeta^{s,y}_{K, K}(t)|_{\Hh_1}\lesssim C^\delta_T  \big(1 
        + \overline{K}((t-s)/2)^\half\big) + \overline{\Delta K}_\delta((t-s)/2)^\half \overline{K}((t-s)/2)^\half.
    \end{align} 
\end{enumerate}
Hence, for all $t\in(s, T],$ the limits
$$  \lim_{\delta\to 0}D_y \lambda^{s,y}(t)(h_\delta)=  \zeta_h^{s,y}(t)\;,\;\; \lim_{\delta\to 0}D^2_y \lambda^{s,y}(t)(K_\delta, K_\delta)=  \zeta_{K,K}^{s,y}(t)    $$
hold in probability, and in $L^2(\Omega),$ 
in the topology of $\Hh_1.$ In particular, $$\lambda^{s,\cdot}(t)\in \mathcal{D}^1_{\!\mathscr{K}}(\Hh_1;\Hh_1)\cap \mathcal{D}^2_{\vspan(K,\Hh_1)}(\Hh_1;\Hh_1)$$ per Definition \ref{dfn:Derivative}
and, for all $h\in\mathscr{K}$ \eqref{eq:SingularDirections}, $  \mathcal{D}_y\lambda^{s,y}(t)(h)= \zeta_h^{s,y}(t)$ and $\mathcal{D}^2_y\lambda^{s,y}(t)(K, K)= \zeta_{K,K}^{s,y}(t).$  
\end{lemma}
\begin{proof}\
\begin{enumerate}
    \item In view of the linearity of $h\mapsto \zeta^{s,y}_h$ and the estimate \eqref{eq:estimate_zeta} we get 
    \begin{align*}
        \EE|\zeta^{s,y}_{h_\delta}(t)-\zeta^{s,y}_{h}(t)|^p_{\Hh_1} = \EE|\zeta^{s,y}_{h_\delta-h}(t)|^p_{\Hh_1} 
        \lesssim \abs{S(t-s)(h_\delta-h)}^p_{\Hh_1},
    \end{align*}
    which yields the first claim.
    \item  Let us start with the case $q\ge4$. Invoking the bilinearity of $(h_1,h_2)\mapsto \zeta^{s,y}_{h_1,h_2}$ as well as the estimate \eqref{eq:estimate_zetah1h2} with $p=4$ we obtain
    \begin{align*}
        &\EE\abs{\zeta^{s,y}_{K_\delta, K_\delta}(t)-\zeta^{s,y}_{K, K}(t)}^2_{\Hh_1} 
        = \EE\abs{\zeta^{s,y}_{K_\delta-K, K_\delta+K}(t)}^2_{\Hh_1} \\
        &\quad \lesssim |\overline{\Delta K}_\delta|_{L^{2}_{t-s}} \big(|\overline{K}|_{L^{2}}+ |\overline{K_\delta}|_{L^{2}}\big)  \Big(1 
        + \overline{K}((t-s)/2)\Big)\\& + \overline{\Delta K}_\delta((t-s)/2) \big(\overline{K_\delta}((t-s)/2)+\overline{K}((t-s)/2)\big).
    \end{align*}
    By continuity of the semigroup, $\overline{\Delta K}_\delta^2= |(S(\delta)-I)S(\cdot)K|^2_{\Hh_1} $ goes to zero pointwise on $[0,T]$, as $\delta\to0$. Thus, by dominated convergence the integral $|\overline{\Delta K}_\delta|_{L^{2}_{t-s}}$ converges to $0$ as $\delta\to 0.$ For a constant $C_T^\delta$ that vanishes as $\delta\to 0$ and up to unimportant constants it follows that 
    \begin{align*}
        &\abs{\zeta^{s,y}_{K_\delta, K_\delta}(t)-\zeta^{s,y}_{K, K}(t)}^2_{\Hh_1} 
        \lesssim C^\delta_T  \big(1 
        + \overline{K}((t-s)/2)\big) + \overline{\Delta K}_\delta((t-s)/2) \overline{K}((t-s)/2).
    \end{align*}
    The case of additive noise ($\sigma$ constant) follows the same lines but with the $L^1(\Omega)$ norm instead. 
\end{enumerate}
Finally, since $K_\delta\in\Hh_1$ for all $\delta>0$ (by Assumption \ref{assumption:Kernel}) Corollary \ref{cor:regularTangentProcesses} asserts that
    $D_y\lambda^{s,y}(t)(h_\delta)=\zeta_{h_\delta}^{s,y}(t), D^2_y\lambda^{s,y}(t)(K_\delta, K_\delta)=\zeta_{K_\delta, K_\delta}^{s,y}(t).$
    The convergence to the singular derivatives unfolds, concluding the proof.
\end{proof}

\subsection{Convergence of the mollified processes}\label{subsec:ConvergenceDeltaTo0}
We recall that the mollified ~$\lambda_\delta^{s,y}$, $\zeta_{\delta,h}^{s,y}$, $\zeta_{\delta,h_1,h_2}^{s,y}$ are solutions to the integral equations~\eqref{eq:generalised_SPDE}, \eqref{eq:zetaKEqRough} and~\eqref{eq:zetah1h2Rough} respectively, where the kernel~$K$ is replaced by~$K_\delta=S(\delta)K$. The auxiliary results in this section will be used in the following section, for the proof of our main result Theorem \ref{thm:backward_equation_singular}.
\begin{lemma}[Convergence of $\lambda_\delta$]\label{lem:lambdaConvergence}
Let $p\geq 1$, $0\leq s< T$ and~$\delta>0$. Under Assumptions \ref{assumption:SVEassumptions} and \ref{assumption:Kernel} with $q\geq 2$ we have
\begin{align}\label{eq:lambdaConvergence}
    \sup_{r\in [s, T]}\sup_{t\in[r,T]} \EE\abs{\lambda^{r,y}(t)-\lambda_\delta^{r,y}(t)}_{\Hh_1}^{p} \lesssim \delta^{\frac{p(q-2)}{2q}}.
\end{align}
\end{lemma}
\begin{proof}  Since $\lambda^{r,y}, \lambda_\delta^{r,y}$ have the same initial conditions we can write 
\begin{equation}\label{eq:lambda-lambdadelta}
    \begin{aligned}
        \lambda^{r,y}(t)-\lambda_\delta^{r,y}(t)&=\int_r^t S(t-\theta)(K-K_\delta)b_0\big(\lambda^{r,y}(\theta)\big)\D\theta\\
        &\quad+\int_r^t S(t-\theta)K_\delta\big[ b_0\big(\lambda^{r,y}(\theta)\big)-b_0\big(\lambda_\delta^{r,y}(\theta)\big)\big]\D\theta\\&
        \quad+\int_r^t S(t-\theta)(K-K_\delta)\sigma_0\big(\lambda^{r,y}(\theta)\big)\dW_\theta\\&
        \quad+\int_r^t S(t-\theta)K_\delta\big[ \sigma_0\big(\lambda^{r,y}(\theta)\big)-\sigma_0\big(\lambda_\delta^{r,y}(\theta)\big)\big]\dW_\theta.
    \end{aligned}
\end{equation}
We detail only the estimates for stochastic integral terms since the same method applies to the other terms. By BDG inequality and the boundedness of~$\sigma_0$ we have
\begin{align*}
    \EE\abs{\int_r^t S(t-\theta)(K-K_\delta)\sigma_0\big(\lambda^{r,y}(\theta)\big)\D W_\theta}_{\Hh_1}^p \lesssim \left(\int_r^t \abs{S(t-\theta)(K-K_\delta)}_{\Hh_1}^2 \D \theta\right)^{p/2}.
\end{align*}
On the one hand, the fundamental theorem of calculus yields 
\begin{align*}
    &\int_0^t \abs{S(t-\theta)(K-K_\delta)}_{\Hh}^2 \D \theta
    = \int_0^t \abs{S(t-\theta)-S(t-\theta+\delta)K}_{\Hh}^2 \D \theta
    = \int_0^t \abs{\int_0^\delta \partial_x S(t-\theta+s)K\ds}_{\Hh}^2 \D \theta\\
    &\le \delta \int_0^\delta \int_0^t  \abs{\partial_x S(t-\theta+s)K}_{\Hh}^2\D \theta \ds  
    \le \delta^2\sup_{s\in[0,\delta]} \int_0^{t+s}  \abs{\partial_xS(t-\theta+s)K}_{\Hh}^2\D \theta 
    \le \delta^2 \int_0^T \abs{\partial_x S(\theta)K}^2_{\Hh}\D \theta,
\end{align*}
which is finite by Assumption \ref{assumption:Kernel}. 
The other integrability condition, \eqref{eq:cond_K2} from Assumption~\ref{assumption:Kernel} entails
\begin{align*}
    \int_0^t \abs{\partial_x S(t-\theta)(K-K_\delta)}_{\Hh}^2 \D \theta 
    &= \int_0^t \abs{\partial_x(S(\theta)-S(\theta+\delta))K}_{\Hh}^2 
    \le C_T \delta^{\frac{q-2}{q}}.
\end{align*}
Suppose now that $p\ge\frac{2q}{q-2}$. Regarding the second stochastic integral in~\eqref{eq:lambda-lambdadelta}, the Lipschitz continuity and H\"older's inequality yield
\begin{align*}
    &\EE\abs{\int_r^t S(t-\theta)K_\delta\big[ \sigma_0\big(\lambda^{r,y}(\theta)\big)-\sigma_0\big(\lambda_\delta^{r,y}(\theta)\big)\big]\dW_\theta}^p_{\Hh_1}\\
    &\le \EE\left( \int_r^t \abs{S(t-\theta)K_\delta}_{\Hh_1}^2 \abs{\lambda^{r,y}(\theta)-\lambda_\delta^{r,y}(\theta)}_{\Hh_1}^2 \D \theta\right)^{p/2} \\
    &\le \left(\int_r^t \abs{S(t-\theta)K_\delta}_{\Hh_1}^q \D \theta\right)^{p/q} \EE\left(\int_r^t \abs{\lambda^{r,y}(\theta)-\lambda_\delta^{r,y}(\theta)}_{\Hh_1}^{\frac{2q}{q-2}} \D \theta\right)^{\frac{p}{2}\frac{q-2}{q}} \\
    &\le C_T^{p/q}\int_r^t \EE\abs{\lambda^{r,y}(\theta)-\lambda_\delta^{r,y}(\theta)}_{\Hh_1}^{p} \D \theta,
\end{align*}
where $C_T$ is finite by assumption. It thus suffices to invoke Grönwall's lemma to conclude the proof for $p\ge\frac{2q}{q-2}$. If $p'< \frac{2q}{q-2}$ then
\begin{align*}
    \EE\abs{\lambda^{r,y}(t)-\lambda_\delta^{r,y}(t)}_{\Hh_1}^{p'} 
    \le
    \left(\EE\abs{\lambda^{r,y}(t)-\lambda_\delta^{r,y}(t)}_{\Hh_1}^{p}\right)^{p'/p} \lesssim \delta^{\frac{p'(q-2)}{2q}},
\end{align*}
and the former case yields the claim. 
\end{proof}

We continue with an integral convergence result, where we emphasise that integration takes place with respect to $r$, the starting time of the process~$\zeta^{r,y}_K$.
\begin{lemma}[Convergence of the tangent processes]\label{lem:zetadeltaKL1limit}
Let Assumptions \ref{assumption:SVEassumptions} and \ref{assumption:Kernel} hold.
For $0\leq s< t\leq T$ it holds 
    \begin{align*}
    \lim_{\delta\to0} \int_s^t \EE|\zeta^{r,y}_{ K}(t)-\zeta^{r,y}_{ \delta, K}(t)|_{\Hh_1}^2 \dr =0.
\end{align*}
The same conclusion holds when $K$ is replaced by $K_\delta.$ Furthermore, the same conclusion holds when $K$ (respectively $K_\delta$) is replaced by a smooth direction $h\in\Hh_1$ (respectively $h_\delta=S(\delta)h$).
\end{lemma}
\begin{proof} We have 
\begin{align*}
    &\zeta^{r,y}_{ K}(t)-\zeta^{r,y}_{\delta, K}(t) 
    \\&= \int_{r}^t S(t-\theta)\bigg( K Db_0(\lambda^{r,y}(\theta)) (\zeta^{r,y}_{ K}(\theta))- K_\delta Db_0(\lambda_\delta^{r,y}(\theta)) (\zeta^{r,y}_{\delta, K}(\theta))\bigg)\D \theta\\
    &+ \int_{r}^t S(t-\theta)\bigg(K D\sigma_0(\lambda^{r,y}(\theta)) (\zeta^{r,y}_{ K}(\theta))-K_\delta D\sigma_0(\lambda_\delta^{r,y}(\theta))(\zeta^{r,y}_{\delta, K}(\theta))    \bigg) \D W_\theta.  
\end{align*}
Again, both terms above can be treated similarly so we shall only focus on the stochastic integral. To this end, Itô isometry furnishes 
\begin{equation*}
\begin{aligned}
    \EE\bigg|&\int_{r}^t S(t-\theta)\bigg(K D\sigma_0(\lambda^{r,y}(\theta)) (\zeta^{r,y}_{ K}(\theta))-K_\delta D\sigma_0(\lambda_\delta^{r,y}(\theta))(\zeta^{r,y}_{\delta, K}(\theta))    \bigg) \D W_\theta\bigg|^2_{\Hh_1}\\&
    \lesssim\EE\int_{r}^t \bigg|S(t-\theta)\bigg(K D\sigma_0(\lambda^{r,y}(\theta)) (\zeta^{r,y}_{ K}(\theta))-K_\delta D\sigma_0(\lambda_\delta^{r,y}(\theta))(\zeta^{r,y}_{\delta, K}(\theta))    \bigg)\bigg|_{\Hh_1}^2 \D\theta.
\end{aligned}
\end{equation*}
We proceed by decomposing the integrand as follows
\begin{equation*}
\begin{aligned}
\bigg|&S(t-\theta)\bigg(K D\sigma_0(\lambda^{r,y}(\theta)) (\zeta^{r,y}_{ K}(\theta))-K_\delta D\sigma_0(\lambda_\delta^{r,y}(\theta))(\zeta^{r,y}_{\delta, K}(\theta))    \bigg)\bigg|_{\Hh_1}^2\\
&\lesssim
\bigg|S(t-\theta)(K-K_\delta)D\sigma_0(\lambda^{r,y}(\theta)) (\zeta^{r,y}_{ K}(\theta))\bigg|_{\Hh_1}^2\\
&\quad
+ \bigg|S(t-\theta)K_\delta \bigg[D\sigma_0(\lambda^{r,y}(\theta))- D\sigma_0(\lambda_\delta^{r,y}(\theta))   \bigg] (\zeta^{r,y}_{ K}(\theta))\bigg|_{\Hh_1}^2\\&
\quad+ \bigg|S(t-\theta)K_\delta D\sigma_0(\lambda_\delta^{r,y}(\theta)) (\zeta^{r,y}_{ K}(\theta)-\zeta^{r,y}_{ \delta, K}(\theta))\bigg|_{\Hh_1}^2\\&
\lesssim |S(t-\theta)(K-K_\delta)|^2_{\Hh_1
}|\zeta^{r,y}_{K}(\theta)|_{\Hh_1}^2
\\&
\quad+|S(t-\theta)K_\delta|^2_{\Hh_1}|  \lambda^{r,y}(\theta)-\lambda_\delta^{r,y}(\theta)|^2_{\Hh_1}|\zeta^{r,y}_{K}(\theta)|_{\Hh_1}^2\\&
\quad
+|S(t-\theta)K_\delta|_{\Hh_1}^2|\zeta^{r,y}_{ K}(\theta)-\zeta^{r,y}_{ \delta, K}(\theta)|_{\Hh_1}^2
\end{aligned}
\end{equation*}
where we used the boundedness and Lipschitz continuity of $D\sigma$ (Assumption~\ref{assumption:SVEassumptions}). We recall the estimate $\EE|\zeta^{r,y}_{K}(\theta)|_{\Hh_1}^p\lesssim \abs{S(\theta-r)K}_{\Hh_1}^p$ 
from Lemma \ref{lem:zetaKrough}
 which holds for all~$p\ge1$. 
In view of the latter and using notation from Lemma \ref{lemma:resolvent} and the proof of Lemma \ref{lem:zetaKKRough}  the first term can be estimated as follows:
\begin{align}\label{eq:int_K-Kdelta_zeta}
    \int_r^t|S(t-\theta)(K-K_\delta)|^2_{\Hh_1
}\EE|\zeta^{r,y}_{K}(\theta)|_{\Hh_1}^2\D\theta
\leq \int_r^t|S(t-\theta)(K-K_\delta)|^2_{\Hh_1
}|K(\theta-r)|_{\Hh_1}^2\D\theta =\overline{K}\star\overline{\Delta K}_\delta(t-r),
\end{align}
where we recall $\overline{\Delta K}_\delta(t)= |S(t)(K-K_\delta)|^2_{\Hh_1}$.  As in \eqref{eq:K*K^2bound}, the latter can be upper bounded, up to unimportant constants, by 
$$C_T^\delta\overline{K}((t-r)/2)+ |(S(\delta)-I)S((t-r)/2)K|^2_{\Hh_1},$$
where $C_T^\delta$ tends to zero as~$\delta$ goes to zero. 

Lemmas \ref{lem:zetaKrough} and~\ref{lem:lambdaConvergence}, along with an application of the Cauchy--Schwarz, yield
\begin{align}\label{eq:bound_KastK}
    &\EE\int_r^t |S(t-\theta)K_\delta|^2_{\Hh_1}|  \lambda^{r,y}(\theta)-\lambda_\delta^{r,y}(\theta)|^2_{\Hh_1}|\zeta^{r,y}_{K}(\theta)|_{\Hh_1}^2 \D \theta \nonumber\\
    &\le \int_r^t |S(t-\theta)K|^2_{\Hh_1} \left(\EE |\lambda^{r,y}(\theta)-\lambda_\delta^{r,y}(\theta)|^4_{\Hh_1}\right)^{\half} \left(\EE|\zeta^{r,y}_{K}(\theta)|_{\Hh_1}^4\right)^\half \D \theta \nonumber\\
    &\le \sup_{\theta\in[r,T]} \left(\EE |\lambda^{r,y}(\theta)-\lambda_\delta^{r,y}(\theta)|^4_{\Hh_1}\right)^{\half} \int_r^t |S(t-\theta)K|^2_{\Hh_1} \abs{S(\theta-r)K}_{\Hh_1}^2\D\theta \nonumber \\&\leq \delta^{\frac{q-2}{q}} \int_0^{t-r}|S(t-r-\theta)K|^2_{\Hh_1} \abs{S(\theta)K}_{\Hh_1}^2\D\theta =\delta^{\frac{q-2}{q}}\overline{K}\star\overline{K}(t-r).
\end{align}
Similar to \eqref{eq:K*K^2bound}, $\overline{K}\star\overline{K}(t-r)$ is upper bounded, up to unimportant constants, by $\abs{S\left(\frac{t-r}{2}\right) K}^2_{\Hh_1}.$ 
Collecting the previous estimates \eqref{eq:int_K-Kdelta_zeta} and \eqref{eq:bound_KastK} and applying the Volterra--Gr\"ownall's inequality (Lemma~\ref{lemma:Volterra_Gronwall}) this entails 
\begin{align*}
    \EE|&\zeta^{r,y}_{ K}(t)-\zeta^{r,y}_{ \delta, K}(t)|_{\Hh_1}^2
    \\&\le \overline{K}\star\overline{\Delta K}_\delta(t-r)+\delta^{\frac{q-2}{q}}\overline{K}\star\overline{K}(t-r)\\&+\int_r^t \overline{R}(t-s) \bigg(\overline{K}\star\overline{\Delta K}_\delta(s-r)+\delta^{\frac{q-2}{q}}\overline{K}\star\overline{K}(s-r)\bigg)\ds\\& \lesssim C_T^\delta\overline{K}((t-r)/2)+ |(S(\delta)-I)S((t-r)/2)K|^2_{\Hh_1}+ \delta^{\frac{q-2}{q}} \abs{S\left((t-r)/2\right) K}^2_{\Hh_1} \\&+\overline{R}\star\overline{K}\star\overline{\Delta K}_\delta(t-r)+\delta^{\frac{q-2}{q}}\overline{R}\star\overline{K}\star\overline{K}(t-r).
\end{align*}
The first line above converges to $0$ pointwise in $r\in[s, T]$ and moreover it is  bounded, uniformly in $\delta,$ by the integrable function $   \overline{K}((t-\cdot)/2)+ |S((t-\cdot)/2)K|^2_{\Hh_1}+ \delta^{\frac{q-2}{q}}\abs{S\left((t-
\cdot)/2\right) K}^2_{\Hh_1}\in L^1([s, t]).$ 
Thus, by dominated convergence it follows that 
$$   \lim_{\delta\to 0}\int_{s}^t \bigg(C_T^\delta\overline{K}((t-r)/2)+ |(S(\delta)-I)S((t-r)/2)K|^2_{\Hh_1}+ \delta^\frac{q-2}{q}\abs{S\left((t-r)/2\right) K}^2_{\Hh_1}\bigg)\dr=0.$$
Turning to the second line, we have 
$$  \int_{s}^t\bigg(\overline{R}\star\overline{K}\star\overline{\Delta K}_\delta(t-r)+\delta^2\overline{R}\star\overline{K}\star\overline{K}(t-r)\bigg)\dr\leq | \overline{R}\star\overline{K}\star\overline{\Delta K}_\delta|_{L^1_{t-s}}+\delta^{\frac{q-2}{q}}|\overline{R}\star\overline{K}\star\overline{K}|_{L^1_{t-s}}.     $$
By Young's convolution inequality the latter is upper bounded by 
$$   \abs{\overline{R}}_{L^1_{t-s}}\bigg( |\overline{K}\star\overline{\Delta K}_\delta|_{L^1_{t-s}}+ \delta^\frac{q-2}{q}|\overline{K}\star\overline{K}|_{L^1_{t-s}}\bigg) $$
which in turn converges to $0$ as $\delta\to 0$  the computations in ~\eqref{eq:int_K-Kdelta_zeta}. The proof is complete.
\end{proof}
\begin{remark}
    We are only able to obtain an integrated convergence in Lemma~\ref{lem:zetadeltaKL1limit} because the resolvent~$\overline{R}$ is only known to be in~$L^{q/2}_T$, with no guarantee that~$\overline{R}(t)<\infty$ for all ~$t>0$. 
\end{remark}
We state the following result which is a variation of \eqref{eq:zetaKdiffBnd} and is proved in the same fashion.
\begin{lemma}\label{lem:ZetaDeltaKdeltaConvergence}
Let Assumptions \ref{assumption:SVEassumptions} and \ref{assumption:Kernel} hold.
Let $0\leq s\leq r\leq t\leq T$. For any $p\geq 1$ there exists a constant $C_T>0$ such that 
    \begin{align*}
        &\EE\left[ \big| \zeta^{r, y}_{\delta, K}(t)-\zeta^{r, y}_{\delta, K_\delta}(t)  \big|_{\Hh_1}^p\right]
        \le C_T \abs{S(t-r)(K-K_\delta)}^p_{\Hh_1}.
    \end{align*}
    In particular, for any $p\le q$ with $q$ as in Assumption \ref{assumption:Kernel}, we have  \begin{align}\label{eq:L1_cvg_zetadeltaK}
        \lim_{\delta\to0} \int_s^t \EE\left[ \big| \zeta^{r,y}_{\delta, K}(t)-\zeta^{r,y}_{\delta, K_\delta}(t)  \big|_{\Hh_1}^p\right]\dr=0.
    \end{align}
\end{lemma}

We conclude this section with a technical lemma on the convergence of the mollified second tangent process, in integral form as Lemma~\ref{lem:zetadeltaKL1limit}.
\begin{lemma}[Convergence of the second tangent processes]\label{lem:2ndTangentConvergence}
Let Assumption \ref{assumption:SVEassumptions} and \ref{assumption:Kernel} with $q>4$ hold. For any $s<t\le T$ we have
    \begin{align*}
    \lim_{\delta\to0} 
    \int_s^{t} \EE|\zeta^{r,y}_{K_
    \delta, K_\delta}(t)-\zeta^{r,y}_{ \delta, K_\delta, K_\delta}(t)|_{\Hh_1}^2 \dr
    =0.
\end{align*}
If instead Assumption \ref{assumption:Kernel}  holds with $q>2$ and $\sigma$ is constant then
$$
\lim_{\delta\to0} 
    \int_s^{t} \EE|\zeta^{r,y}_{K_
    \delta, K_\delta}(t)-\zeta^{r,y}_{ \delta, K_\delta, K_\delta}(t)|_{\Hh_1} \dr
    =0.
$$
\end{lemma}
\begin{remark}
    The proof goes through with $K$ instead of $K_\delta$ but the result is stated this way to complement Lemma~\ref{lem:SingularDifferentiability}.
\end{remark}
\begin{proof}
Let us begin we the assumption that $q>4$. 
    We can decompose the difference as follows
    \begin{align}
    &  \zeta^{r,y}_{K_
    \delta, K_\delta}(t)-\zeta^{r,y}_{ \delta, K_\delta, K_\delta}(t) \nonumber\\
    =& \int_r^t \bigg(S(t-\theta)K D^2 b_0(\lambda^{r,y}(\theta))(\zeta^{r,y}_{K_
    \delta}(\theta),\zeta^{r,y}_{K_\delta}(\theta))- S(t-\theta)K_\delta 
    D^2 b_0(\lambda_\delta^{r,y}(\theta))(\zeta^{r,y}_{\delta,K_\delta}(\theta),
    \zeta^{r,y}_{\delta,K_\delta}(\theta))\Big)\D\theta \nonumber\\
    &+\int_r^t \bigg(S(t-\theta)K D^2\sigma_0(\lambda^{r,y}(\theta))(\zeta^{r,y}_{K_
    \delta}(\theta),\zeta^{r,y}_{K_\delta}(\theta))- S(t-\theta)K_\delta 
    D^2\sigma_0(\lambda_\delta^{r,y}(\theta))(\zeta^{r,y}_{\delta,K_\delta}(\theta),
    \zeta^{r,y}_{\delta,K_\delta}(\theta))\Big)\D W_\theta \nonumber\\
    &+\int_r^t \Big(S(t-\theta)K D b_0(\lambda^{r,y}(\theta))(\zeta^{r,y}_{K_\delta,K_\delta}(\theta)) 
    -S(t-\theta)K_\delta D b_0(\lambda_\delta^{r,y}(\theta))(\zeta^{r,y}_{\delta,K_\delta,K_\delta}(\theta)) \Big)\D\theta \nonumber\\
    &+\int_r^t \Big(S(t-\theta)K D \sigma_0(\lambda^{r,y}(\theta))(\zeta^{r,y}_{K_\delta,K_\delta}(\theta)) 
    -S(t-\theta)K_\delta D \sigma_0(\lambda_\delta^{r,y}(\theta))(\zeta^{r,y}_{\delta,K_\delta,K_\delta}(\theta)) \Big)\D W_\theta \nonumber\\
    &=: \mathbf{T_1} + \mathbf{T_2} + \mathbf{T_3} + \mathbf{T_4} .
    \label{eq:decomposition_ItoIV}
    \end{align}
     The first and second terms can be dealt with in the same way. We apply Jensen's inequality, leverage the bilinearity of $(h_1,h_2)\mapsto D^2b_0(x)(h_1,h_2)$, boundedness and H\"older
    continuity of~$x\mapsto D^2b(x)$ (Assumption~\ref{assumption:SVEassumptions}) to compute
    \begin{align*}
        \EE\abs{\mathbf{T_1}}_{\Hh_1}^2 
        \le &\,\EE\int_r^t \abs{ S(t-\theta)(K-K_\delta)}_{\Hh_1}^2 \abs{D^2 b_0(\lambda^{r,y}(\theta))(\zeta^{r,y}_{K_
    \delta}(\theta),\zeta^{r,y}_{K_\delta}(\theta)) }^2\D \theta \\
        &+ \EE\int_r^t \abs{ S(t-\theta)K_\delta}_{\Hh_1}^2 \abs{D^2 \big[b_0(\lambda^{r,y}(\theta))-b_0(\lambda_\delta^{r,y}(\theta))\big](\zeta^{r,y}_{K_
    \delta}(\theta),\zeta^{r,y}_{K_\delta}(\theta)) }^2\D \theta\\
        &+ \EE\int_r^t \abs{ S(t-\theta)K_\delta}_{\Hh_1}^2 \abs{D^2 b_0(\lambda_\delta^{r,y}(\theta))(\zeta^{r,y}_{K_
    \delta}(\theta)-\zeta^{r,y}_{\delta,K_
    \delta}(\theta),\zeta^{r,y}_{K_\delta}(\theta)+\zeta^{r,y}_{\delta,K_\delta}(\theta)) }^2\D \theta \\
        \le&\,\int_r^t \abs{ S(t-\theta)(K-K_\delta)}_{\Hh_1}^2 \EE \abs{\zeta^{r,y}_{K_\delta}(\theta)) }^4_{\Hh_1} \D \theta \\
        &+\EE \int_r^t \abs{ S(t-\theta)K}_{\Hh_1}^2 \abs{\lambda^{r,y}(\theta)-\lambda_\delta^{r,y}(\theta)}^{2\gamma_0}_{\Hh_1} \abs{\zeta^{r,y}_{K_\delta}(\theta)) }^4_{\Hh_1}\D \theta\\
        &+ \EE\int_r^t \abs{ S(t-\theta)K}_{\Hh_1}^2 \abs{\zeta^{r,y}_{K_
    \delta}(\theta)-\zeta^{r,y}_{\delta,K_
    \delta}(\theta)}^2_{\Hh_1} \big(\abs{\zeta^{r,y}_{K_\delta}(\theta)) }^2_{\Hh_1}+\abs{\zeta^{r,y}_{\delta,K_\delta}(\theta)) }^2_{\Hh_1}\big) \D \theta\\
    =:& \,\mathbf{T_{11}} + \mathbf{T_{12}}+ \mathbf{T_{13}}.
    \end{align*}
    By Lemma~\ref{lem:zetaKrough} we have
    \begin{align*}
        \mathbf{T_{11}}\lesssim  \overline{K}^2\star\overline{\Delta K}_\delta(t-r)
    \end{align*}
    which converges to zero by dominated convergence analogously to~\eqref{eq:int_K-Kdelta_zeta} with~$\overline{K}\in L^2_T$. Furthermore, Cauchy--Schwarz inequality and estimate~\eqref{eq:estimate_zeta} yield
    \begin{align*}
        \mathbf{T_{12}}
        &\le  \int_r^t \abs{ S(t-\theta)K}_{\Hh_1}^2 \left(\EE\abs{\lambda^{r,y}(\theta)-\lambda_\delta^{r,y}(\theta)}^{4\gamma_0}_{\Hh_1}\right)^{\half}\left(\EE\abs{\zeta^{r,y}_{K_\delta}(\theta)) }^8_{\Hh_1}\right)^\half\D \theta \\
        &\le \sup_{\theta\in[r,T]} \left(\EE\abs{\lambda^{r,y}(\theta)-\lambda_\delta^{r,y}(\theta)}^{4\gamma_0}_{\Hh_1} \right)^{\half} \int_r^t \abs{ S(t-\theta)K}_{\Hh_1}^2 \abs{S(\theta-r)K}^4_{\Hh_1}\D\theta\\
        &\le \delta^{2\gamma_0} \overline{K}\star\overline{K}^2(t-r)
    \end{align*}
    which also tends to zero by Lemma~\ref{lem:lambdaConvergence} since $\overline{K}\star\overline{K}^2(t-r)$ can be bounded as in~\eqref{eq:K*K^2bound}. Cauchy--Schwarz inequality and estimate \eqref{eq:zetaKdiffBnd} give
    \begin{align}
       \mathbf{T_{13}}
    &\le  \left(\int_r^t \abs{ S(t-\theta)K}_{\Hh_1}^4 \abs{S(\theta-r)K}^4_{\Hh_1}\D\theta \int_r^t \EE\abs{\zeta^{r,y}_{K_
    \delta}(\theta)-\zeta^{r,y}_{\delta,K_
    \delta}(\theta)}^4_{\Hh_1} \D\theta \right)^{\half}
    \label{eq:integral_zetadelta_four}
    \end{align}
    The first integral is finite as for~\eqref{eq:bound_KastK}, while the second requires some extra work. 
    For any $t\in[r,T]$, the BDG inequality, the initial decomposition in the proof of Lemma~\ref{lem:zetadeltaKL1limit} and Cauchy--Schwarz inequality entail
    \begin{align*}
     &   \EE\abs{\zeta^{r,y}_{K_
    \delta}(t)-\zeta^{r,y}_{\delta,K_
    \delta}(t)}^4_{\Hh_1}\\
    &\lesssim \EE\bigg[\bigg(\int_r^t \bigg\{\abs{S(t-\theta)(K-K_\delta)}^2_{\Hh_1} \abs{\zeta^{r,y}_{K_
    \delta}(\theta)}^2_{\Hh_1}  + \abs{S(t-\theta)K_\delta}^2_{\Hh_1} \abs{\lambda^{r,y}(\theta)-\lambda_\delta^{r,y}(\theta)}^2_{\Hh_1}\abs{\zeta^{r,y}_{K_
    \delta}(\theta)}^2_{\Hh_1}  \\
    &\qquad\qquad + \abs{S(t-\theta)K_\delta}^2_{\Hh_1} \abs{\zeta^{r,y}_{K_
    \delta}(\theta)-\zeta^{r,y}_{\delta,K_
    \delta}(\theta)}^2_{\Hh_1} \bigg\}  \D\theta \bigg)^2\bigg]\\
    &\le \int_r^t \abs{S(t-\theta)(K-K_\delta)}^4_{\Hh_1} \D\theta \int_r^t \EE\abs{\zeta^{r,y}_{K_
    \delta}(\theta)}^4_{\Hh_1}\D\theta \\
    &\qquad\qquad +\int_r^t \abs{S(t-\theta)K_\delta}^4_{\Hh_1} \D\theta \left(\sup_{\theta\in[r,T]}\EE \abs{\lambda^{r,y}(\theta)-\lambda_\delta^{r,y}(\theta)}^8_{\Hh_1}\right)^\half \int_r^t \left(\EE
    \abs{\zeta^{r,y}_{K_
    \delta}(\theta)}^8_{\Hh_1} \right)^\half\D\theta\\
    &\qquad\qquad + \int_r^t \abs{S(t-\theta)K_\delta}^4_{\Hh_1} \D\theta  \int_r^t \EE\abs{\zeta^{r,y}_{K_
    \delta}(\theta)-\zeta^{r,y}_{\delta,K_
    \delta}(\theta)}^4_{\Hh_1} \D\theta.
    \end{align*}
    The standard Grönwall inequality, along with Lemmas~\ref{lem:zetaKrough}, \ref{lem:lambdaConvergence} and Assumption \ref{assumption:Kernel} show that there exists~$C^\delta_T$, independent of~$t$, which tends to zero as $\delta\to0$, and such that
     \begin{align*}
       \EE\abs{\zeta^{r,y}_{K_
    \delta}(t)-\zeta^{r,y}_{\delta,K_
    \delta}(t)}^4_{\Hh_1}\le C^\delta_T. 
    \end{align*}
    Therefore Equation~\eqref{eq:integral_zetadelta_four} vanishes as $\delta$ goes to zero. 
    It is straigthforward to show that the same estimates hold for $\mathbf{T_2}$ hence, gathering estimates, we conclude that there is a (possibly different) $C^\delta_T$ such that
    \begin{align}\label{eq:def_Fdelta1}
        \EE\abs{\mathbf{T_1}}^2_{\Hh_1} + \EE\abs{\mathbf{T_2}}^2_{\Hh_1} \lesssim \overline{K}^2\star \overline{\Delta K}_{\delta}(t-r) + C^\delta_T \Big(\overline{K}\star\overline{K}^2(t-r) 
        +  \big( \overline{K}^2\star\overline{K}^2(t-r)\big)^\half\Big).
    \end{align}
    We deal with the third and fourth terms $\mathbf{T_3},\mathbf{T_4}$   in~\eqref{eq:decomposition_ItoIV} in the same way. Jensen's and H\"older's inequalities, combined with \eqref{eq:zetaKK_estimate} with $q=4$, yield
    \begin{align*}
        \EE\abs{\mathbf{T_3}}^2_{\Hh_1} 
        \le&\, \EE \int_r^t \abs{S(t-\theta)(K-K_\delta)}^2_{\Hh_1} \abs{D b_0(\lambda^{r,y}(\theta))(\zeta^{r,y}_{K_\delta,K_\delta}(\theta))}^2\D\theta\\
        &+ \EE\int_r^t \abs{ S(t-\theta)K_\delta}^2_{\Hh_1} \abs{D b_0(\lambda^{r,y}(\theta))(\zeta^{r,y}_{K_\delta,K_\delta}(\theta)) - D b_0(\lambda_\delta^{r,y}(\theta))(\zeta^{r,y}_{K_\delta,K_\delta}(\theta)) }^2_{\Hh_1} \D\theta \\
        &+ \EE\int_r^t \abs{S(t-\theta)K_\delta }^2_{\Hh_1}  \abs{ D b_0(\lambda_\delta^{r,y}(\theta))(\zeta^{r,y}_{K_\delta,K_\delta}(\theta)) - D b_0(\lambda_\delta^{r,y}(\theta))(\zeta^{r,y}_{\delta,K_\delta,K_\delta}(\theta))}^2 \D\theta \\
        \lesssim&\,\EE \int_r^t \abs{S(t-\theta)(K-K_\delta)}^2_{\Hh_1} \abs{\zeta^{r,y}_{K_\delta,K_\delta}(\theta)}_{\Hh_1}^2\D\theta\\
        &+ \EE\int_r^t \abs{ S(t-\theta)K_\delta}^2_{\Hh_1} \abs{\lambda^{r,y}(\theta) - \lambda_\delta^{r,y}(\theta)}^2_{\Hh_1} \abs{\zeta^{r,y}_{K_\delta,K_\delta}(\theta) }^2_{\Hh_1} \D\theta\\
        &+ \EE\int_r^t \abs{ S(t-\theta)K_\delta}^2_{\Hh_1} 
        \abs{\zeta^{r,y}_{K_\delta,K_\delta}(\theta)-\zeta^{r,y}_{\delta,K_\delta,K_\delta}(\theta)}^2_{\Hh_1} \D\theta \\
        \lesssim&\, \int_r^t 
        \overline{\Delta K_\delta}(t-\theta)
        \Big(1+\overline{K}\star\overline{K}^2 (\theta-r)  \Big)\D\theta \\        &+\left(\sup_{\theta\in[r,T]} \EE\abs{\lambda^{r,y}(\theta) - \lambda_\delta^{r,y}(\theta)}^{\frac{2q}{q-4}}_{\Hh_1} \right)^{\frac{q-4}{q}} \int_r^t \abs{S(t-\theta)K_\delta}^2_{\Hh_1} \left(\EE\abs{\zeta^{r,y}_{K_\delta,K_\delta}(\theta) }^{q/2}_{\Hh_1} \right)^{4/q}\D\theta \\
        &+ \int_r^t \abs{ S(t-\theta)K_\delta}^2_{\Hh_1} \EE\abs{\zeta^{r,y}_{ K_\delta,K_\delta}(\theta)-\zeta^{r,y}_{\delta,K_\delta,K_\delta}(\theta)}^2_{\Hh_1} \D\theta\\
        &=: \mathbf{T_{31}} +  \mathbf{T_{32}} +  \mathbf{T_{33}}.
    \end{align*}
    Regarding the first term, Cauchy--Schwarz inequality and Young's inequalitiy $\abs{f\star g}_{L^2}\le \abs{f}_{L^2}\abs{g}_{L^1}$ yield 
    \begin{align*}
         \int_r^t 
        \overline{\Delta K}_\delta(t-\theta) \overline{K}\star\overline{K}^2 (\theta-r)  \D\theta 
         &\le  \left(\int_r^t\overline{\Delta K}_\delta(t-\theta)^2 \D\theta\right)^\half \left(\int_r^t \abs{\overline{K}\star\overline{K}^2 (\theta-r)}^2 \D\theta\right)^\half\\
         &\le \abs{\overline{\Delta K}_\delta^2}_{L^2_{t-r}}         \abs{\overline{K}}_{L^2_{t-r}}^3,
    \end{align*}
    hence~$\mathbf{T_{31}}\lesssim  \lvert \overline{\Delta K}_\delta\lvert_{L^1_{t-r}}+ \lvert \overline{\Delta K}_\delta^2\lvert_{L^2_{t-r}}$
     goes to zero as $\delta\to0$. By \eqref{eq:zetah1h2_integral_estimate} and~\eqref{eq:Lq_bound_zetah1h2}, H\"older's and Young's inequalities we get 
    \begin{align*}
         \int_r^t \abs{S(t-\theta)K_\delta}^2_{\Hh_1} \left(\EE\abs{\zeta^{r,y}_{K_\delta,K_\delta}(\theta) }^{q/2}_{\Hh_1} \right)^{4/q}\D\theta
         &\lesssim \int_r^t \overline{K}(t-\theta)\big(1+\overline{K}^{q/4}\star\overline{K}^{q/2}(\theta)\big)^{4/q}\D\theta\\
         &\le \abs{\overline{K}}_{L^{q/2}_{t-r}} \abs{\big(1+\overline{K}^{q/4}\star\overline{K}^{q/2}\big)^{4/q}}_{L^{\frac{q}{q-2}}_{t-r}}\\
         &\lesssim \abs{\overline{K}}_{L^{q/2}_{t-r}} \Big(1+\abs{\overline{K}^{q/4}\star\overline{K}^{q/2}}_{L^{\frac{4}{q-2}}_{t-r}}^{4/q}\Big)
         \\
         &\le \abs{\overline{K}}_{L^{q/2}_{t-r}} \Big(1+\abs{\overline{K}^{q/4}}_{L^{\frac{4}{q-2}}_{t-r}}^{4/q}\abs{\overline{K}^{q/2}}_{L^1_{t-r}}^{4/q}\Big)
    \end{align*}
    which is finite because $\overline{K}\in L^q$ with $q>4$. Hence we have~$\mathbf{T_{32}}\lesssim \delta^\frac{q-2}{q}$ and thus there is $C_T^\delta\to0$ such that $\mathbf{T_{31}}+\mathbf{T_{32}}\lesssim C^\delta_T$.
    We notice that once more the same estimates can be derived for~$\mathbf{T_4}$. 
    Gathering this estimate with~\eqref{eq:def_Fdelta1} we finally obtain
    \begin{align*}
        \EE\abs{\zeta^{r,y}_{K_
    \delta, K_\delta}(t)-\zeta^{r,y}_{ \delta, K_\delta, K_\delta}(t) }^2  
    \lesssim\, \mathbf{F}_\delta(t-r) + \int_r^t \overline{K}(t-\theta)\EE\abs{\zeta^{r,y}_{K_\delta,K_\delta}(\theta)-\zeta^{r,y}_{K_\delta,K_\delta}(\theta)}^2_{\Hh_1} \D\theta, 
    \end{align*}
    where
    \begin{align}
         \mathbf{F}_\delta(t-r) := \overline{K}^2\star \overline{\Delta K}_{\delta}(t-r) + C^\delta_T \Big(1+\overline{K}\star\overline{K}^2(t-r) 
        +  \big(  \overline{K}^2\star\overline{K}^2(t-r)\big)^\half\Big).
    \end{align}
    The Volterra--Gr\"onwall inequality (Lemma~\ref{lemma:Volterra_Gronwall}) followed by Cauchy--Schwarz inequality implies
    \begin{align*}
        \EE\abs{\zeta^{r,y}_{K_
    \delta, K_\delta}(t)-\zeta^{r,y}_{ \delta, K_\delta, K_\delta}(t) }^2
    \lesssim \mathbf{F}_\delta(t-r) + \int_r^t \overline{R}(t-\theta) \mathbf{F}_\delta(\theta-r)\D\theta\le \mathbf{F}_\delta(t-r) + \abs{\overline{R}}_{L^2} \abs{\mathbf{F}_\delta}_{L^2_{t-r}
    }.
    \end{align*}
    We have already established that $\mathbf{F}_\delta$ tends to zero pointwise. Young convolution's inequality shows that $\mathbf{F}_\delta$ also converges in $L^2_T$:
    \begin{align*}
        \abs{\mathbf{F}_\delta}_{L^2_T}
        &\le \abs{\overline{\Delta K}_\delta\star\overline{K}^2}_{L^2_T}+ C^\delta_T\Big(1+\abs{\overline{K}\star\overline{K}^2}_{L^2_T} + \abs{(\overline{K}^2\star\overline{K}^2)^\half}_{L^2_T}\Big)\\
        &\le \abs{\overline{\Delta K}_\delta}_{L^2_T}\abs{\overline{K}^2}_{L^1_T} +C^\delta_T\Big( 1+  \abs{\overline{K}}_{L^2_T}\abs{\overline{K}^2}_{L^1_T} + \abs{\overline{K}^2}_{L^1_T}\Big).
    \end{align*}
    This proves the pointwise convergence of $\EE|\zeta^{r,y}_{K_
    \delta, K_\delta}(t)-\zeta^{r,y}_{ \delta, K_\delta, K_\delta}(t)|^2 $. The $L^1_T$ convergence follows by dominated convergence since 
    \begin{align}
        \mathbf{F}_\delta(t-r) \lesssim \overline{K}\star\overline{K}^2(t-r) 
        +  \big(  \overline{K}^2\star\overline{K}^2(t-r)\big)^\half
    \end{align}
    and the right side is integrable in $r$ by Young's convolution inequality. 

    The additive noise case is proved in the same way but with $L^1(\Omega)$ norms instead of $L^2(\Omega)$.
\end{proof}

\subsection{Proof of Theorem \ref{thm:backward_equation_singular}}\label{subsec:BackwardProof}
 \textit{\textbf{(1) (Existence)}}     We start by showing that the value function \eqref{eq:ValueFunction} is sufficiently regular. In particular, we will prove that 
$u\in\Cc_{T,\mathscr{K}}^{1,2}$ per Definition \ref{dfn:CKclass_alt}. Then we will show that $u$ satisfies \eqref{eq:BackwardPDECauchyProblem} by taking the limit, as $\delta\to 0,$ of the ``smooth" Kolmogorov equation \eqref{eq:KolmoBackwardDelta} which is satisfied by $u_\delta$ \eqref{eq:vdelta}.

\underline{\textbf{Step 1}: $u\in \mathcal{D}^1_{\!\mathscr{K}}(\Hh_1)\cap \mathcal{D}^2_{\vspan(K,\Hh_1)}(\Hh_1).$}  We prove first that $u$ is twice continuously Gateaux differentiable along $\Hh_1-$directions under the assumption that $q>4$; the case $q>2$ with $\sigma$ constant is proved by exploiting the analogous estimates and convergence results.

Differentiating under the sign of expectation and using chain rule for Fr\'echet derivatives, along with Corollary \ref{cor:regularTangentProcesses}, we have for all $s\in[t,T),y,h_1,h_2\in\Hh_1$
\begin{equation}\label{eq:vGateauxDerivatives}
    \begin{aligned}
       &Du(s,y)(h_1)=\EE\big[D\varphi\big(\lambda^{s,y}(T)\big)\big(D_y\lambda^{s,y}(T)(h_1)\big)\big]=\EE\big[D\varphi\big(\lambda^{s,y}(T)\big)\big(\zeta^{s,y}_{h_1}(T)\big)\big],\\&
       D^2u(s,y)(h_1, h_2)\\&\quad\quad=\EE\big[D^2\varphi\big(\lambda^{s,y}(T)\big)\big(D_y\lambda^{s,y}(T)(h_1), D_y\lambda^{s,y}(T)(h_2)\big)\big]+\EE\big[D\varphi\big(\lambda^{s,y}(T)\big)\big(D^2_y\lambda^{s,y}(T)(h_1, h_2)\big)\big]
\\&\quad\quad=\EE\big[D^2\varphi\big(\lambda^{s,y}(T)\big)\big(\zeta^{s,y}_{h_1}(T), \zeta^{s,y}_{h_2}(T)\big)\big]+\EE\big[D\varphi\big(\lambda^{s,y}(T)\big)\big(\zeta^{s,y}_{h_1,h_2}(T)\big)\big],
    \end{aligned}
\end{equation}
where $\zeta^{s,y}_{h_1}, \zeta^{s,y}_{h_1,h_2}$ are respectively given by \eqref{eq:zetaKEqRough}, \eqref{eq:zetah1h2Rough}.
Notice that, due to Proposition \ref{Prop:Feller}, Lemmas~\ref{lem:zetaKrough} and \ref{lem:zetaKKRough} these are continuous functions of $y\in\Hh_1.$ Moreover, it is straightforward to verify from this computation that $u$ is in fact twice Fr\'echet differentiable.

Turning to the setting of Definition \ref{dfn:CKclass_alt} we immediately see that $u$ is bounded since $\varphi\in \Cc^2_b(\Hh_1)$ and in view of \eqref{eq:estimate_zeta} we have
\begin{align}
    |Du(s,y)(h_1)|\leq \EE\big[|D\varphi\big(\lambda^{s,y}(T)\big)|_{\Hh_1}|\zeta^{s,y}_{h_1}(T)|_{\Hh_1}\big]\leq |D\varphi|_{\infty}\EE\big[|\zeta^{s,y}_{h_1}(T)|_{\Hh_1}\big]\lesssim |S(T-s)h_1|_{\Hh_1},
\end{align}
thereby satisfying the first estimate in \eqref{eq:defC12_D2F_bound}. Assume first that $q\ge4$.  
For any $h\in \vspan(K,\Hh_1)$ let $\overline{h}=\abs{S(\cdot)h}^2_{\Hh_1}$, the estimates \eqref{eq:estimate_zeta} and \eqref{eq:estimate_zetah1h2} (with $p=4$) then yield 
\begin{align*}
& |D^2u(s,y)(h_1, h_2)|^2\leq |D^2\varphi|_{\infty}^2 |\zeta^{s,y}_{h_1}(T)|_{L^2(\Omega;\Hh_1)}^2 |\zeta^{s,y}_{h_2}(T)|_{L^2(\Omega;\Hh_1)}^2+|D\varphi|_{\infty}^2 |\zeta^{s,y}_{h_1,h_2}(T)|_{L^2(\Omega;\Hh_1)}^2\\&
    \qquad\lesssim \overline{h}_1(T-s)\overline{h}_2(T-s)
    +|\overline{h}_1|_{L^2_T}|\overline{h}_2|_{L^2_T}\big(1 +
\overline{K}((T-s)/2)\big) + \overline{h}_1((T-s)/2)\overline{h}_2((T-s)/2)
    \\&
    \qquad\lesssim |\overline{h}_1|_{L^2_T}|\overline{h}_2|_{L^2_T}\big(1 +
\overline{K}((T-s)/2)\big) + \overline{h}_1((T-s)/2)\overline{h}_2((T-s)/2).   
\end{align*}
This furnishes, for any~$\delta>0$,
\begin{align}\nonumber
 |D^2u(s,y)(K_\delta, K_\delta)|^2
&\lesssim \int_0^T \abs{S(r+\delta)K}^2_{\Hh_1}\dr \big(1 + \abs{S((T-s)/2)K}_{\Hh_1}^2\big) + \abs{S((T-s)/2)K}_{\Hh_1}^4\\
&\lesssim 1 + \abs{S((T-s)/2)K}_{\Hh_1}^2+ \abs{S((T-s)/2)K}_{\Hh_1}^4 =: \varpi(s)^2.
\label{eq:BoundD2V}
\end{align}
Thus $u$ satisfies the second estimate in ~\eqref{eq:defC12_D2F_bound} since the right side of the last estimate is independent of~$\delta$ and integrable over $s$ in virtue of Assumption~\ref{assumption:Kernel}. The latter also helps us derive, for all~$\delta,\delta'>0$,
\begin{align*}
    &\abs{D^2u(t,y)\Big(K_{\delta}-K_{\delta'}, K_{\delta}+K_{\delta'} \Big)}^2\\
    &\lesssim \int_0^T \abs{S(t)\big(K_{\delta}-K_{\delta'}\big)}_{\Hh_1}^2\dt \int_0^T \abs{S(t)\big(K_{\delta}+K_{\delta'}\big)}_{\Hh_1}^2\dt \big(1 + \abs{S((T-s)/2)K}_{\Hh_1}^2\big) \\
    &\qquad +\abs{S((T-s)/2) \big(K_{\delta}-K_{\delta'}\big)}^2_{\Hh_1} \abs{S((T-s)/2) \big(K_{\delta}+K_{\delta'}\big)}^2_{\Hh_1}.
\end{align*}
The strong continuity of the semigroup combined with \eqref{eq:cond_K2} and \eqref{eq:diff_StK_H} yield
\begin{align*}
    \int_0^T \abs{S(t)\big(K_{\delta}-K_{\delta'}\big)}_{\Hh_1}^2\dt 
    \lesssim \int_0^T \abs{S(t+\delta-\delta')K-S(t)K}^2_{\Hh_1}\dt \lesssim \abs{\delta-\delta'}^{\frac{q}{q-2}}.
\end{align*}
Moreover, the strong continuity helps again to derive that
\begin{align*}
    \abs{S((T-s)/2) \big(K_{\delta}-K_{\delta'}\big)}^2_{\Hh_1} 
    =  \abs{(S(\delta)-S(\delta'))S((T-s)/2) K}^2_{\Hh_1} 
\end{align*}
converges as $\delta-\delta'$ tends to zero. In particular the function~$\kappa:[0,1]^2\to\RR$ defined by $\kappa(\delta,\delta'):=\abs{\delta-\delta'}^{\frac{q}{q-2}}+\abs{(S(\delta)-S(\delta'))S((T-s)/2) K}^2_{\Hh_1}$ satisfies the conditions of Assumption \ref{dfn:CKclass_alt} and, since~$s\mapsto\varpi(s)=\abs{S((T-s)/2)K}_{\Hh_1}^2$ is integrable by Assumption~\ref{assumption:Kernel}, the last bound in~\eqref{eq:defC12_D2F_bound} holds.

If instead $q\ge2$ and $\sigma$ is constant, then the analogous estimates yield 
\begin{align*}\nonumber
& |D^2u(s,y)(K_\delta, K_\delta)|^2
\lesssim 1 + \abs{S((T-s)/2)K}_{\Hh_1}+ \abs{S((T-s)/2)K}_{\Hh_1}^2,\\
&\abs{D^2u(t,y)\Big(K_{\delta}-K_{\delta'}, K_{\delta}+K_{\delta'} \Big)}
\lesssim \kappa(\delta,\delta') \big(1 + \abs{S((T-s)/2)K}_{\Hh_1}\big)
\end{align*}
and the same conclusion follows.

Hence Lemma \ref{lemma:singular_diff} guarantees that  $u(s,\cdot)\in \mathcal{D}^1_{\!\mathscr{K}}(\Hh_1)\cap \mathcal{D}^2_{\vspan(K,\Hh_1)}(\Hh_1)$ per Definition \ref{dfn:Derivative}. We can identify the singular directional derivatives to be respectively given by
\begin{equation}\label{eq:vSingularDerivatives}
    \begin{aligned}
        &\mathcal{D}u(s,y)(h)=\EE[D\varphi\big(\lambda^{s,y}(T)\big)\big(\zeta^{s,y}_{h}(T)\big)]\;,\\&
\mathcal{D}^2u(s,y)(K, K)=\EE[D^2\varphi\big(\lambda^{s,y}(T)\big)\big(\zeta^{s,y}_K(T), \zeta^{s,y}_{K}(T)\big)]+\EE\big[D\varphi\big(\lambda^{s,y}(T)\big)\big(\zeta^{s,y}_{K, K}(T)\big)\big].
    \end{aligned}
\end{equation}
Indeed, denoting $h_\delta:=S(\delta)h\in\Hh_1$ for any $h\in\mathscr{K}, \delta>0$, and in view of \eqref{eq:vGateauxDerivatives} and Lemma \ref{lem:SingularDifferentiability} we have, for any $s,y,$
\begin{equation*}
    \begin{aligned}
      \big|\EE\big[D\varphi\big(\lambda^{s,y}(T)\big)\big(\zeta^{s,y}_{h_\delta}(T)\big)\big]- \EE\big[D\varphi\big(\lambda^{s,y}(T)\big)\big(\zeta^{s,y}_{h_\delta}(T)\big)\big]\big|\leq |D\varphi|_{\infty}\EE|      \zeta^{s,y}_{h_\delta}(T)-\zeta^{s,y}_{h}(T)|\longrightarrow 0, 
    \end{aligned}
\end{equation*}
and
\begin{equation*}
    \begin{aligned}      &\big|\EE\big[D^2\varphi\big(\lambda^{s,y}(T)\big)\big(\zeta^{s,y}_{K}(T), \zeta^{s,y}_{K}(T)\big)\big]-\EE\big[D^2\varphi\big(\lambda^{s,y}(T)\big)\big(\zeta^{s,y}_{K_\delta}(T), \zeta^{s,y}_{K_\delta}(T)\big)\big]\big|\\&   +\big|\EE\big[D\varphi\big(\lambda^{s,y}(T)\big)\big(\zeta^{s,y}_{K,K}(T)\big)\big]-\EE\big[D\varphi\big(\lambda^{s,y}(T)\big)\big(\zeta^{s,y}_{K_\delta,K_\delta}(T)\big)\big]    \big|\\&
    \leq |D^2\varphi|_{\infty}\big|\zeta^{s,y}_{K}(T)-\zeta^{s,y}_{K_\delta}(T)|^2_{L^2(\Omega;\Hh_1)}
    +|D\varphi|_{\infty}\big|\zeta^{s,y}_{K,K}(T)-\zeta^{s,y}_{K_\delta,K_\delta}(T)|_{L^1(\Omega;\Hh_1)}
    \longrightarrow 0,
    \end{aligned}
\end{equation*}
as $\delta\to 0.$ 
In particular, $\mathcal{D}u(s,y):\mathscr{K}\to\RR$ and~$\mathcal{D}^2u(s,y):\vspan(K,\Hh_1)\times\vspan(K,\Hh_1)\to\RR$ are linear and bilinear, respectively.

We now check the equicontinuity required in Definition \ref{dfn:CKclass_alt} iii).
 From \eqref{eq:vGateauxDerivatives}, Lemmas \ref{Prop:Feller} \ref{lem:zetaKrough}, Lipschitz continuity of $D\varphi$ and Cauchy--Schwarz inequality, for any $y,z\in\Hh_1,\,\delta>0$,
 \begin{align*}
     &\abs{Du(s,y)(K_\delta)-Du(s,z)(K_\delta)}^2\\
     &\qquad\le \EE\abs{\lambda^{s,y}(T)-\lambda^{s,z}(T)}_{\Hh_1}^2 \EE\abs{\zeta^{s,y}_{K_\delta}(T)}^2_{\Hh_1}
     +\EE\abs{\lambda^{s,z}(T)}_{\Hh_1}^2 \EE\abs{\zeta^{s,y}_{K_\delta}(T)-\zeta^{s,z}_{K_\delta}(T)}^2_{\Hh_1}\\
     &\qquad\le \abs{y-z}^2_{\Hh_1} \abs{S(T-s)K_\delta}^2_{\Hh_1} + C_T \abs{y-z}^2_{\Hh_1} \le C_T \abs{y-z}^2_{\Hh_1},
 \end{align*}
 where $C_T$ does not depend on $\delta$. Similarly, with the help of Propositon~\ref{Prop:Feller}, Lemmas \ref{lem:zetaKrough} and \ref{lem:zetaKKRough}, Lipschitz continuity of $D\varphi$ and $\gamma_0$--H\"older continuity of $D^2\varphi$, along with Cauchy--Schwarz inequality, we obtain 
 \begin{align*}
     &\abs{D^2u(s,y)(K_\delta,K_\delta)-D^2u(s,z)(K_\delta,K_\delta)}\\
     &\le \EE\Big[\abs{\lambda^{s,y}(T)-\lambda^{s,z}(T)}_{\Hh_1}\abs{\zeta^{s,y}_{K_\delta}(T)}^2_{\Hh_1} \Big] 
     + \EE\Big[\abs{\zeta^{s,y}_{K_\delta}(T)-\zeta^{s,z}_{K_\delta}(T)}_{\Hh_1} \abs{\zeta^{s,y}_{K_\delta}(T)+\zeta^{s,z}_{K_\delta}(T)}_{\Hh_1}\Big] \\
     &\quad + \EE\Big[\abs{\lambda^{s,y}(T)-\lambda^{s,z}(T)}_{\Hh_1}^{\gamma_0} \abs{\zeta^{s,y}_{K_\delta,K_\delta}(T)}_{\Hh_1} \Big] 
     + \EE\Big[\abs{\zeta^{s,y}_{K_\delta,K_\delta}(T)-\zeta^{s,z}_{K_\delta,K_\delta}(T)}_{\Hh_1} \Big]\\
     &\lesssim \abs{y-z}_{\Hh_1} \abs{S(T-s)K_\delta}^2_{\Hh_1} + \abs{y-z}_{\Hh_1} \abs{S(T-s)K_\delta}_{\Hh_1} 
     + \abs{y-z}^{\gamma_0}_{\Hh_1} 
     \left(\EE\big\lvert\zeta^{s,y}_{K_\delta,K_\delta}\big\lvert_{\Hh_1}^{q/2} \right)^{2/q} \\
     &\quad + \abs{y-z}_{\Hh_1}^{\gamma_0}\left(1+ \abs{S((T-s)/2)K_\delta}_{\Hh_1}+ \abs{S((T-s)/2)K_\delta}_{\Hh_1}^2\right)
 \end{align*}
Thanks to the estimate $\abs{S(\delta)S(t)K}_{\Hh_1}\le C_T \abs{S(t)K}_{\Hh_1}$ for any $t>0$ and some $C_T>0$ independent of~$\delta$ we conclude
\begin{equation*}
\begin{aligned}
    |D^2u(s,y)(K_\delta,K_\delta)&-D^2u(s,z)(K_\delta,K_\delta)|
    \\&\lesssim \Big(\abs{y-z}_{\Hh_1}^{\gamma_0}\vee\abs{y-z}_{\Hh_1}\Big)\bigg(1+ \abs{S((T-s)/2)K}_{\Hh_1}+ \abs{S((T-s)/2)K}_{\Hh_1}^2\bigg).
    \end{aligned}
\end{equation*}
This guarantees the equicontinuity demanded by Definition~\ref{dfn:CKclass_alt} iii).
In particular, Lemma \ref{lemma:singular_diff} ensures that the maps 
$$  y\longmapsto \mathcal{D}u(s,y)(h),  \mathcal{D}^2u(s,y)(K,K)$$
are continuous for any fixed $s,h.$

\underline{\textbf{Step 2}: $u$ solves \eqref{eq:BackwardPDECauchyProblem} on $[0,T]\times\Hh_2.$}

Fix $t\in[0,T]$ and recall that the ``mollified" value function \eqref{eq:vdelta} 
\begin{equation*}
u_\delta(s,y)=\EE[\varphi(\lambda_\delta^{s,y}(T))], (s,y)\in[t,T]\times\Hh_1
\end{equation*} 
 satisfies, for all $t\in[0,T]$ and $y\in\Hh_2$, the PDE
\begin{equation}\label{eq:ExistenceKolmoBackwardDelta}
   \begin{aligned}   u_\delta(t,y)=\varphi(y)+\int_t^T\bigg(\big\langle Dv_\delta(s,y),  \partial_x y+ K_\delta b_0(y)\big\rangle_{\Hh_1}+\frac{1}{2}\textnormal{Tr}\Big[ D^2u_\delta(s,y)K_\delta \sigma_0(y)(K_\delta \sigma_0(y))^*\Big]             \bigg)\ds.
   \end{aligned} 
\end{equation}
Throughout this step we fix $y\in\Hh_2$ and aim at taking the limit $\delta\to0$ on both sides of the equation. Starting from the left-hand side, Lemma \ref{lem:lambdaConvergence} along with the dominated convergence theorem yield
\begin{equation}
    \label{eq:existenceDeltaLimit1}
    \lim_{\delta\to 0}u_\delta(t,y)=\lim_{\delta\to 0}\EE[\varphi(\lambda_\delta^{t,y}(T))]=\EE[\varphi(\lambda^{t,y}(T))]=u(t,y).
\end{equation}
Recalling \eqref{eq:vSingularDerivatives}, we turn to the first derivative terms. From the linearity of the derivatives and the triangle inequality we obtain
\begin{equation}\label{eq:Dvdeltaestimate}
\begin{aligned}
    &|\mathcal{D}u\big(s,y\big) \big( K b_0(y)\big)-Du_\delta\big(s,y\big) \big( K_\delta b_0(y)\big)|\\&
    \leq \abs{b_0(y)} |\mathcal{D}u\big(s,y\big) \big( K \big)-Du_\delta\big(s,y\big) \big( K_\delta \big)|\\    
    &\lesssim\bigg|\EE\left[ D\varphi\big(\lambda^{s,y}(T)\big)\big(\zeta^{s,y}_{ K}(T)-\zeta^{s,y}_{\delta, K}(T)\big)\right]\bigg|
    +\bigg|\EE\left[ \big(D\varphi\big(\lambda^{s,y}(T)\big)-D\varphi\big(\lambda_\delta^{s,y}(T)\big)\zeta^{s,y}_{\delta, K}(T)\right]\bigg|
    \\&
    \quad+\bigg|  \EE\left[ D\varphi\big(\lambda_\delta^{s,y}(T)\big)\big(\zeta^{s,y}_{\delta, K}(T)-\zeta^{s,y}_{\delta, K_\delta}(T)\big)\right]     \bigg|\\&
    \leq |D\varphi|_{\infty}\EE\left[ \big|\zeta^{s,y}_{ K}(T)-\zeta^{s,y}_{\delta, K}(T)\big|_{\Hh_1}\right]
+|D\varphi|_{Lip}\EE\left[ \big|\lambda^{s,y}(T)-\lambda_\delta^{s,y}(T)\big|_{\Hh_1}\big|\zeta^{s,y}_{\delta, K}(T)\big|_{\Hh_1}\right]\\&
\quad +|D\varphi|_{\infty}\EE\left[ \big| \zeta^{s,y}_{\delta, K}(T)-\zeta^{s,y}_{\delta, K_\delta}(T)  \big|_{\Hh_1}\right].
\end{aligned}  
\end{equation}
Lemmas \ref{lem:lambdaConvergence}, \ref{lem:zetadeltaKL1limit}, \ref{lem:ZetaDeltaKdeltaConvergence} and the a priori bound from Lemma \ref{lem:zetaKrough}  furnish
    \begin{align*}
    \lim_{\delta\to0} \int_t^T \EE|\zeta^{s,y}_{ K}(T)-\zeta^{s,y}_{ \delta, K}(T)|_{\Hh_1} \dr =\lim_{\delta\to0} \int_t^T \EE|\zeta^{s,y}_{ \delta,K}(T)-\zeta^{s,y}_{ \delta, K_\delta}(T)|_{\Hh_1} \ds=0,
\end{align*}
and 
\begin{align*}
    \lim_{\delta\to0} &\int_t^T\EE\left[ \big|\lambda^{s,y}(T)-\lambda_\delta^{s,y}(T)\big|_{\Hh_1}\big|\zeta^{s,y}_{\delta, K}(T)\big|_{\Hh_1}\right]\dr\\&\lesssim \lim_{\delta\to0}\sup_{r\in[s, T]}\EE\left[ \big|\lambda^{s,y}(T)-\lambda_\delta^{s,y}(T)\big|^2_{\Hh_1}\right]^{\frac{1}{2}}\int_t^T|S(T-r)K|_{\Hh_1}\ds=0,
\end{align*}
where we also applied the Cauchy--Schwarz inequality. Thus, integrating \eqref{eq:Dvdeltaestimate} with respect to $s\in[t,T]$ we obtain 
\begin{equation}\label{eq:existenceDeltaLimit2}
    \lim_{\delta\to 0}\int_t^T Du_\delta\big(s,y\big) \big(K_\delta b_0(y)\big)\ds=\int_t^T \mathcal{D}u\big(s,y\big) \big(K b_0(y)\big)\ds.
\end{equation}
\noindent Since $\partial_xy\in\Hh_1$ is a smooth direction, a similar yet simpler argument yields 
\begin{equation}\label{eq:existenceDeltaLimit3}
    \lim_{\delta\to 0}\int_t^T Du_\delta\big(s,y\big) \big(\partial_xy\big)\ds=\int_t^T Du\big(s,y\big) \big(\partial_xy\big)\ds.
\end{equation}
To avoid repetition, its proof shall be omitted. Next, we show convergence of the second-derivative terms in \eqref{eq:ExistenceKolmoBackwardDelta}. Let  $\{e_k\}_{k=1,\dots m}$ be the standard basis of $\RR^m.$ In view of \eqref{eq:vSingularDerivatives} and the bilinearity of~$\Dd^2u(s,y)$ and~$D^2u(s,y)$, we apply the triangle inequality and the boundedness of $\sigma$ to obtain the decomposition 
\begin{equation}\label{eq:BackwardPDE2ndDerivativeConvergence}
    \begin{aligned}
        &\big|\mathcal{D}^2u\big(s,y\big) \big(K\sigma_0(y)e_k,K\sigma_0(y)e_k\big)-D^2u_\delta\big(s,y\big) \big(K_\delta\sigma_0(y)e_k,K_\delta\sigma_0(y)e_k\big)\big|\\
        &\leq \abs{\sigma_0(y)}^2 \big|\mathcal{D}^2u\big(s,y\big) \big(K,K\big)-D^2u_\delta\big(s,y\big) \big(K_\delta,K_\delta\big)\big|\\&
        \lesssim \big|\EE[  D\varphi\big(\lambda_\delta^{s,y}(T)\big)\zeta^{s,y}_{\delta, K_\delta,K_\delta}(T)-D\varphi\big(\lambda^{s,y}(T)\big)\zeta^{s,y}_{K,K}(T)   ]\big|\\& \quad +\big|\EE\big[D^2\varphi\big(\lambda_\delta^{s,y}(T)\big)(\zeta^{s,y}_{\delta, K_\delta}(T),\zeta^{s,y}_{\delta, K_\delta}(T))-D^2\varphi\big(\lambda^{s,y}(T)\big)\big(\zeta^{s,y}_{K}(T),\zeta^{s,y}_{K}(T)\big)\big]\big|\\&
    \leq \big|\EE\big[  \big( D\varphi\big(\lambda_\delta^{s,y}(T)\big)-D\varphi\big(\lambda^{s,y}(T)\big)\big)\zeta^{s,y}_{\delta, K_\delta,K_\delta}(T)\big]\big|
    +\big|\EE\big[ D\varphi\big(\lambda^{s,y}(T)\big)\big(\zeta^{s,y}_{\delta, K_\delta,K_\delta}(T)-  \zeta^{s,y}_{ K_\delta,K_\delta}(T)     \big)\big]\big|\\&
    \quad+\big|\EE\big[ D\varphi\big(\lambda^{s,y}(T)\big)\big(\zeta^{s,y}_{K_\delta,K_\delta}(T)-  \zeta^{s,y}_{ K,K}(T)     \big)\big]\big| \\&+\bigg|\EE\bigg[\bigg(D^2\varphi\big(\lambda_\delta^{s,y}(T)\big)-D^2\varphi\big(\lambda^{s,y}(T)\big)\bigg)(\zeta^{s,y}_{\delta, K_\delta}(T),\zeta^{s,y}_{\delta, K_\delta}(T))\bigg]\bigg|\\&
 \quad +\bigg|\EE\bigg[D^2\varphi\big(\lambda^{s,y}(T)\big)(\zeta^{s,y}_{\delta, K_\delta}(T)-\zeta^{s,y}_{K_\delta}(T),\zeta^{s,y}_{\delta, K_\delta}(T)+\zeta^{s,y}_{K_\delta}(T))\bigg]\bigg|\\&
 \quad +\bigg|\EE\bigg[D^2\varphi\big(\lambda^{s,y}(T)\big)(\zeta^{s,y}_{ K_\delta}(T)-\zeta^{s,y}_{K}(T),\zeta^{s,y}_{K_\delta}(T)+\zeta^{s,y}_{K}(T))\bigg]\bigg|\\&
 \leq |D\varphi|_{Lip}\EE\big[\big|  \lambda_\delta^{s,y}(T)-\lambda^{s,y}(T)\big|_{\Hh_1} |\zeta^{s,y}_{\delta, K_\delta,K_\delta}(T)\big|_{\Hh_1}\big] +|D\varphi|_{\infty}\EE\big[\big|  \zeta^{s,y}_{\delta, K_\delta,K_\delta}(T)-  \zeta^{s,y}_{ K_\delta,K_\delta}(T)     \big|_{\Hh_1} \big]\\&
 \quad +|D\varphi|_{\infty}\EE\big[\big|  \zeta^{s,y}_{ K_\delta,K_\delta}(T)-  \zeta^{s,y}_{K,K}(T)     \big|_{\Hh_1} \big]
 +|D^2\varphi|_{Lip}\EE\big[\big|  \lambda_\delta^{s,y}(T)-\lambda^{s,y}(T)\big|_{\Hh_1} |\zeta^{s,y}_{\delta, K_\delta}(T)\big|^2_{\Hh_1}\big]\\&
 \quad +|D^2\varphi|_{\infty}\left(\EE\big[\big| \zeta^{s,y}_{\delta, K_\delta}(T)-\zeta^{s,y}_{K_\delta}(T)\big|^2_{\Hh_1}\big] \EE\big[\big| \zeta^{s,y}_{\delta, K_\delta}(T)+\zeta^{s,y}_{K_\delta}(T)\big|^2_{\Hh_1}\big]\right)^\half\\
 &\quad +|D^2\varphi|_{\infty}\left(\EE\big[\big| \zeta^{s,y}_{ K_\delta}(T)-\zeta^{s,y}_{K}(T)\big|^2_{\Hh_1}\big]\EE\big[\big| \zeta^{s,y}_{ K_\delta}(T)+\zeta^{s,y}_{K}(T)\big|^2_{\Hh_1}\big]\right)^\half\\
 &=: \sum_{i=1}^6 \mathbf{V}_i^\delta(s),
    \end{aligned}
\end{equation}  
where we used Cauchy--Schwarz inequality at the end. 
In view of Lemmas  \ref{lem:zetaKKRough} and \ref{lem:lambdaConvergence} the first term satisfies 
\begin{equation}\label{eq:Existence2ndDerivativeLimit1}
    \begin{aligned}
        \int_t^T\mathbf{V}_1^\delta(s)\ds&
        \leq \sup_{s\in[t,T]}\big|  \lambda_\delta^{s,y}(T)-\lambda^{s,y}(T)\big|_{L^2(\Omega;\Hh_1)}\int_t^T \big|\zeta^{s,y}_{\delta, K_\delta,K_\delta}(T)\big|_{L^2(\Omega;\Hh_1)}\ds\\&
        \leq \delta^{\frac{q-2}{2q}}   \int_t^T\bigg(     1+|S((T-s)/2)K|_{\Hh_1}+|S((T-s)/2)K|^2_{\Hh_1}\bigg)\ds
    \end{aligned}
\end{equation}
which goes to zero
as $\delta\to 0,$ where the last integral is finite by Assumption \ref{assumption:Kernel} and the second line follows by the Cauchy--Schwarz inequality.

Turning to the second term, Lemma \ref{lem:2ndTangentConvergence} along with Jensen's inequality furnish
  \begin{align}\label{eq:Existence2ndDerivativeLimit2}
   \lim_{\delta\to0}
    \bigg(\int_t^{T} \mathbf{V}_2^\delta(s) \ds\bigg)^2 \lesssim \lim_{\delta\to0}
    \int_t^{T} \EE|\zeta^{s,y}_{K_
    \delta, K_\delta}(T)-\zeta^{s,y}_{ \delta, K_\delta, K_\delta}(T)|_{\Hh_1}^2 \ds 
    =0.
\end{align}
In view of Lemma \ref{lem:SingularDifferentiability} and the dominated convergence theorem, the third term  satisfies
\begin{equation}\label{eq:Existence2ndDerivativeLimit3}
    \begin{aligned}
      \lim_{\delta\to0}\int_t^T\ \mathbf{V}_3^\delta(s)  \ds
      & \lesssim \lim_{\delta\to0}\int_t^T\bigg(C_T^\delta\bigg(1+|S((T-s)/2)K|_{\Hh_1}\bigg)+ |(S(\delta)-I)S((T-s)/2)K|^2_{\Hh_1}\bigg)\ds=0.
    \end{aligned}
\end{equation}
Regarding the fourth term, Lemmas \ref{lem:lambdaConvergence}, the a priori bounds from \ref{lem:zetaKrough}, along with the Cauchy--Schwarz inequality furnish
\begin{equation}\label{eq:Existence2ndDerivativeLimit4}
    \begin{aligned}
\lim_{\delta\to0}
\int_t^T \mathbf{V}_4^\delta(s)\ds
&\le \lim_{\delta\to0}
\int_t^T \EE\big[\big|  \lambda_\delta^{s,y}(T)-\lambda^{s,y}(T)\big|_{\Hh_1} |\zeta^{s,y}_{\delta, K_\delta}(T)\big|^2_{\Hh_1}\big]\ds\\&
\leq \lim_{\delta\to0}\sup_{s\in[t,T]}\big|  \lambda_\delta^{s,y}(T)-\lambda^{s,y}(T)\big|_{L^2(\Omega;\Hh_1)}\int_t^T  \big|\zeta^{s,y}_{\delta, K_\delta}(T)\big|^2_{L^4(\Omega;\Hh_1)}\ds\\&
\lesssim \lim_{\delta\to0}\delta^{\frac{q-2}{2q}} \int_t^T|S(T-s)K|^{2}_{\Hh_1}\ds = 0.
    \end{aligned}
\end{equation}
As for the fifth and sixth terms, Lemma \ref{lem:zetadeltaKL1limit} (with $K$ replaced by $K_\delta$), the estimates \eqref{eq:estimate_zeta}, \eqref{eq:zetaKdiffBnd} and Cauchy--Schwarz inequality imply
\begin{equation}\label{eq:Existence2ndDerivativeLimit56}
    \begin{aligned}
       \lim_{\delta\to0}
\int_t^T \big(\mathbf{V}_5^\delta(s)&+\mathbf{V}_6^\delta(s)\big)\ds
\\&\le \lim_{\delta\to0}\int_t^T\Big(\EE\big| \zeta^{s,y}_{\delta, K_\delta}(T)-\zeta^{s,y}_{K_\delta}(T)\big|^2_{\Hh_1}+\EE\big| \zeta^{s,y}_{ K_\delta-K}(T)\big|^2_{\Hh_1}\Big)\ds \int_t^T \abs{S(T-s)K}_{\Hh_1}^2\ds\\&
       \lesssim \lim_{\delta\to0}\int_t^T\EE\big| \zeta^{s,y}_{\delta, K_\delta}(T)-\zeta^{s,y}_{K_\delta}(T)\big|^2_{\Hh_1}\ds+ \int_t^T|(S(\delta)-I)S(T-s)K|^p_{\Hh_1}\ds= 0.      \end{aligned}
\end{equation}
Collecting the convergence results \eqref{eq:BackwardPDE2ndDerivativeConvergence} to \eqref{eq:Existence2ndDerivativeLimit56} we deduce that
\begin{equation}\label{eq:existenceDeltaLimit4}
    \lim_{\delta\to 0} \int_t^T D^2u_\delta\big(s,y\big) \big(K_\delta\sigma_0(y)e_k,K_\delta\sigma_0(y)e_k\big)\ds=    \int_t^T \mathcal{D}^2u\big(s,y\big) \big(K\sigma_0(y)e_k,K\sigma_0(y)e_k\big)\ds,  
\end{equation}
for all $k=1,\dots, m$.
From the latter, along with \eqref{eq:existenceDeltaLimit1}, \eqref{eq:existenceDeltaLimit2}, \eqref{eq:existenceDeltaLimit3}, we take $\delta\to 0$ in \eqref{eq:ExistenceKolmoBackwardDelta} to conclude that 
\begin{equation}\label{eq:BackwardPDEStep1}
    \begin{aligned}  u(t, y)-\varphi(y)=\int_t^T\bigg[Du(s,y)\big( \partial_x y\big)+\mathcal{D}u(s,y)\big(K b_0(y)\big)+\frac{1}{2}\textnormal{Tr}\Big[ \mathcal{D}^2u(s,y)K \sigma_0(y)(K \sigma_0(y))^*\Big]             \bigg]\ds.
    \end{aligned}
\end{equation}
holds for all $(t,y)\in [0,T]\times\Hh_2.$ This concludes Step 2.

\underline{\textbf{Step 3}: $u$ solves \eqref{eq:BackwardPDECauchyProblem} on $[0,T]\times\mathscr{K}_1$ and $\partial_tu$ satisfies \eqref{eq:defC12_dtF}.} 
For any $y\in\mathscr{K}_1$ (recall \eqref{eq:yspace}), the sequence $\{y_n\}_{n\in\NN}:=\{S(1/n)y\}_{n\in\NN}\subset\Hh_2$ converges to $y$ in $\Hh_1.$ Moreover, since $y\in\mathscr{K}_1,$  $\partial_xy_n=S(1/n)\partial_x y\in\Hh_1$ for all $n\in\NN,$  hence $\partial_xy\in\mathscr{K}$ (recall \eqref{eq:SingularDirections}) is an admissible singular direction per Definition \ref{dfn:Derivative}.

From Step 1, $u(s,\cdot), Du(s,\cdot),D^2u(s,\cdot)$ are continuous in $\Hh_1.$ Again from Step 1, the singular differentials
$\Hh_1 \ni y \mapsto\mathcal{D}u(t,y)(h_1), \mathcal{D}^2u(t,y)(h_2,h_2)$  for any $h_1\in\mathscr{K}, h_2\in\vspan(K,\Hh_1).$ Thus the following pointwise limits hold:
\begin{equation*}
   \begin{aligned}
   &\lim_{n\to\infty}u(s, y_n)=u(s, y),
      \\&\lim_{n\to\infty}Du(s,y_n)\big( \partial_x y_n\big)=  \lim_{n\to\infty}Du(s,y_n)\big(S(1/n)\partial_x y\big)=\mathcal{D}u(s,y)\big( \partial_x y\big)\\&      \lim_{n\to\infty}\mathcal{D}u(s,y_n)\big(K b_0(y_n)\big)=\mathcal{D}u(s,y)\big(K b_0(y)\big)\\&
      \lim_{n\to\infty}\textnormal{Tr}\big[ \mathcal{D}^2u(s,y_n)K \sigma_0(y_n)(K \sigma_0(y_n))^*\big]= \textnormal{Tr}\big[ \mathcal{D}^2u(s,y)K \sigma_0(y)(K \sigma_0(y))^*\big],     \end{aligned} 
\end{equation*}
where we recall that (since $W$ is finite dimensional) the trace is given by a finite sum and we implicitly used Lemma \ref{lem:SingularDifferentiability} 1) and continuity of $b_0, \sigma_0$ as functions on $\Hh_1.$ In turn, the integrable (over $s\in[t,T]$) upper bounds
\eqref{eq:defC12_D2F_bound} (which transfer to the singular differentials \eqref{eq:vSingularDerivatives}) and the dominated convergence theorem allow us to interchange the limit as $n\to\infty$ and the Riemann integrals over $s$ in \eqref{eq:BackwardPDEStep1} to obtain
\begin{equation*}
    \begin{aligned}  u(t, y)-\varphi(y)=\int_t^T\bigg[\mathcal{D}u(s,y)\big( \partial_x y\big)+\mathcal{D}u(s,y)\big(K b_0(y)\big)+\frac{1}{2}\textnormal{Tr}\big[ \mathcal{D}^2u(s,y)K \sigma_0(y)(K \sigma_0(y))^*\big]             \bigg]\ds
    \end{aligned}
\end{equation*}
for all $(t,y)\in[0,T]\times\mathscr{K}_1.$ 
Clearly, the last display shows that $(0,T)\ni t\longmapsto u(\cdot,y)\in\RR$ is differentiable for all $y\in\mathscr{K}_1$ and
$$ 
\partial_tu(t, y)=- \mathcal{D}u(t,y)\big( \partial_x y\big)-\mathcal{D}u(t,y)\big(K b_0(y)\big)-\frac{1}{2}\textnormal{Tr}\big[ \mathcal{D}^2u(t,y)K \sigma_0(y)(K \sigma_0(y))^*\big].              
$$
We have seen that $Du$ and $D^2u$ satisfy the estimates~\eqref{eq:defC12_D2F_bound}, hence there is~$\varpi\in L^1([0,T];\RR^+)$, derived in~\eqref{eq:BoundD2V}, such that also
\begin{align*}
    &\abs{\Dd u(s,y)(\partial_x y +Kb_0(y))}\lesssim \abs{S(T-s)\partial_x y}_{\Hh_1} + \abs{S(T-s)K}_{\Hh_1},\\
    &\abs{\Dd u(s,y)(K,K)} \lesssim \varpi(s).
\end{align*}
Therefore there is $\hat\varpi\in L^1([0,T];\RR^+)$ given by~$\hat\varpi(s):= \varpi(s)+\abs{S(T-s)K}_{\Hh_1}$ such that~$
    \abs{\partial_t u(s,y)}\lesssim \hat\varpi(s).$ The proof is complete.\\
\noindent \textit{\textbf{(2) (Uniqueness)}}
Suppose $u:[0,T]\times\Hh_1\to\RR$ is a $\Cc^{1,2}_{T,\mathscr{K}}$ solution of the backward PDE~\eqref{eq:BackwardPDECauchyProblem}.  
Let $(t,y)\in[0,T]\times\mathscr{K}_1, 0\leq t\leq s\leq T.$ By virtue of Proposition \ref{prop:invariant subspaces} we have $\lambda^{t,y}(s)\in\mathscr{K}_1$. From an application of the singular It\^o formula (Theorem \ref{thm:SingularIto}) to the process $\{ u(s, \lambda^{t,y}(s)); s\in[t,T]\}$ (and noting that $u(T,y)=\varphi(y)$) we obtain 
\begin{equation*}
\begin{aligned}
    \varphi\big(\lambda^{t,y}(T)\big)-u(t, y)&=u(T,\lambda^{t,y}(T))-u(t, \lambda^{t,y}(t))\\&=
    \int_t^{T}\bigg[ \partial_tu(s, \lambda^{t,y}(s))+\mathcal{D}u(s,\lambda^{t,y}(s))\bigg( \partial_x \lambda^{t,y}(s)+K b_0(\lambda^{t,y}(s))\bigg)\bigg]\ds
    \\&\quad+\int_t^{T}\frac{1}{2}\textnormal{Tr}\Big[ \mathcal{D}^2u(s,\lambda^{t,y}(s))K \sigma_0(\lambda^{t,y}(s))(K \sigma_0(\lambda^{t,y}(s)))^*\Big]\ds\\&
    \quad +\int_t^T \mathcal{D}u(s,\lambda^{t,y}(s))\big(K \sigma_0(\lambda^{t,y}(s))\dW_s\big),
\end{aligned}
   \end{equation*}
   almost surely. Taking expectation, the stochastic integral vanishes and we obtain 
   \begin{equation*}
\begin{aligned}
    \EE\varphi\big(\lambda^{t,y}(T)\big)-u(t, y)=
    \int_t^{T}\EE\bigg[& \partial_tu(s, \lambda^{t,y}(s))+\mathcal{D}u(s,\lambda^{t,y}(s))\big( \partial_x \lambda^{t,y}(s)+K b_0(\lambda^{t,y}(s))\big)\\&+\frac{1}{2}\textnormal{Tr}\big[ \mathcal{D}^2u(s,\lambda^{t,y}(s))K \sigma_0(\lambda^{t,y}(s))(K \sigma_0(\lambda^{t,y}(s)))^*\big]\bigg]\ds.
\end{aligned}
   \end{equation*}
   Since $u$ is a pointwise solution of the Cauchy problem \eqref{eq:BackwardPDECauchyProblem}, the right-hand side vanishes and the desired probabilistic representation follows.

\noindent \textit{\textbf{(3) (Computation of conditional  expectations)}} The martingale representation is derived by applying the singular Itô formula \eqref{eq:SingularIto} and cancelling the finite variation terms thanks to the PDE~\eqref{eq:BackwardPDECauchyProblem}. We conclude by proving \eqref{eq:conditionalExpectationsBackward}. To this end, let $0\leq s\leq t\leq T, y\in\mathscr{K}_1.$ By the Markov property and time-homogeneity of $\lambda$ (Proposition \ref{prop:Markov}) we have 
\begin{equation*}
\begin{aligned}
     \EE[\varphi(\lambda^{s,y}(T))\big| \mathcal{F}_t]&=\EE[\varphi(\lambda^{0,y'}(T-t))\big]\big|_{y'=\lambda^{s,y}(t)} 
     \\&=\EE[\varphi(\lambda^{t,y'}(T))\big]\big|_{y'=\lambda^{s,y}(t)} 
     \\&=u(t,\lambda^{s,y}(t)),
\end{aligned}    
\end{equation*}   
where we used the invariance of $\mathscr{K}_1$ (Proposition \ref{prop:invariant subspaces}) to guarantee that $\lambda^{s,y}(t)\in\mathscr{K}_1$ on the last line.
The proof is complete.

\section{Connections and comparisons between infinite-dimensional SVE lifts}\label{sec:Conclusion}

We conclude this work by comparing our setting and results to those found in the context of different Markovian SVE lifts.

\subsection{Ornstein-Uhlenbeck lift} \label{subsec:OUlift}
Throughout this section, $\mathfrak{L}$ denotes the Laplace transform defined on a sufficiently large space of functions or measures over $\RR^+$ and we set $d=m=1$ in \eqref{eq:SVE}. As we have already mentioned in Section \ref{sec:Intro}, the transport SPDE \eqref{eq:SPDE} is related to the forward curve lift for SVEs. In this section, we show that the \textit{Ornstein-Uhlenbeck (OU) lift}, a well-studied process that has often been employed to restore Markovianity of SVEs such as \eqref{eq:SVE}, can be recovered by a transformation of the SPDE \eqref{eq:SPDE}. 

Initially introduced by Carmona--Coutin in \cite{carmona1998fractional} and further developed or employed  in stochastic analysis, stochastic control and mathematical finance settings (often with an emphasis on the special class of affine SVEs) \cite{cuchiero2019markovian, cuchiero2020generalized, harms2019affine, hamaguchi2023markovian, huber2024markovian, abi2021linear}, the OU lift is described as follows.

Assuming that $K$ is a completely monotone kernel, one can write 
\begin{equation}\label{eq:KLaplaceTransform}
    K(x)=\int_{\textnormal{supp}(\mu)} e^{-xz}\D\mu(z)=\mathfrak{L}(\mu)\equiv \mathfrak{L}(\rho)\;,\;\;x>0
\end{equation} as the Laplace transform of a measure $\mu$ with $\textnormal{supp}(\mu)\subset\RR^+.$ The last equality then holds if we further assume that $\mu$ is equivalent to Lebesgue measure with a strictly positive density $\frac{d\mu}{dx}(x)=\rho(x)>0, x\in\textnormal{supp}(\mu).$ 
Under these structural conditions one can then express, for all $t>0,$ the SVE solution \eqref{eq:SVE} with initial curve 
\begin{equation}\label{eq:OULiftInitialCurve}
    X_0(t)=\int_{\textnormal{supp}(\mu)}y_0(x)e^{-tx}\D\mu(x),
\end{equation}
for some $y_0\in L^2(\mu),$ by
$X_t=\langle Y_t, 1\rangle_{L^2(\mu)}.$ Here, for each $z\in\textnormal{supp}(\mu),$ $Y_t(z)$ is a random field that satisfies the SDE 
\begin{equation}\label{eq:OULiftDifferential}
    \D Y_t(z)=\bigg(-z\cdot Y_t(z)+b\big( \langle  Y_t, 1\rangle_{L^2(\mu)}  \big)\bigg)\dt+\sigma\big( \langle  Y_t, 1\rangle_{L^2(\mu)}  \big)\dW_t\;,\;\; Y_0(z)=y_0(z).   
\end{equation}

Notice that, when $b=0$ and for each fixed $z,$ $Y_t(z)$ is an Ornstein-Uhlenbeck process, with mean reversion rate $z.$ It is given in integral form by
\begin{equation}\label{eq:OULiftIntegral}
    \begin{aligned}
       Y_t(z)=e^{-zt}y_0(z)+\int_0^t e^{-z(t-s)}b\big( \langle  Y_s, 1\rangle_{L^2(\mu)}  \big)\ds+\int_0^t e^{-z(t-s)}\sigma\big( \langle  Y_s, 1\rangle_{L^2(\mu)}  \big)\dW_s.
    \end{aligned}
\end{equation}
Borrowing some notation from \cite{hamaguchi2023markovian}, the above integral equation admits a unique solution that takes values in the space $$\Hh_\mu:=\bigg\{y:\textnormal{supp}(\mu)\rightarrow\RR \bigg| \|y\|^2_{\Hh_\mu}:=\int_{\textnormal{supp}(\mu) } r(z)y(z)^2\D\mu(z)<\infty   \bigg\},$$
where $r$ is a non-increasing, nonnegative, Borel measurable function such that $1\wedge (z^{-1/2} )\leq  r(z) \leq 1$ for any $z\geq 0$ and $r\in L^1(\mu).$
The domain of the multiplication operator $(Ay)(z)=zy(z),$ which appears on the right-hand side, is given by the subspace 
$$\Vv_\mu=\bigg\{y\in\Hh_\mu \bigg| \|y\|^2_{\Vv_\mu}:=\int_{\textnormal{supp}(\mu) }(z+1)r(z)y(z)^2\D\mu(z)<\infty   \bigg\}.$$

In view of \eqref{eq:KLaplaceTransform}, \eqref{eq:OULiftInitialCurve} and  by virtue of strong uniqueness of the SVE, \eqref{eq:SVE} it is straightforward to verify that
\begin{equation*}
    \begin{aligned}
      \langle  Y_t, 1\rangle_{L^2(\mu)}  &=\int_{\textnormal{supp}(\mu)}y_0(x)e^{-tx}\D\mu(x)+\int_0^tK(t-s)b\big( \langle  Y_s, 1\rangle_{L^2(\mu)}  \big)\ds+\int_0^t K(t-s)\sigma\big( \langle  Y_s, 1\rangle_{L^2(\mu)}  \big)\dW_s\\&
      = X_t,
    \end{aligned}
\end{equation*}
and the map $\Vv_\mu\ni y\mapsto \langle y, 1\rangle_{L^2(\mu)}\in\RR $ is a continuous linear functional. From these facts, along with Markovianity of the $\Vv_\mu-$valued unique solutions \eqref{eq:OULiftIntegral} (as established e.g. in \cite[Theorem 2.19]{hamaguchi2023markovian}), $Y$ may be called a Markovian lift of $X.$

We further point out that there are striking analogies between Hamaguchi's spaces $\Vv_\mu\subset\Hh_\mu$ and our spaces~$\Hh_1\subset\Hh$. Recall that, when the kernel is singular (e.g. $K(t)=t^{H-\half},H\in(0,\half)$),  the coefficients of the SPDE~\eqref{eq:SPDE} belong to~$\Hh\setminus\Hh_1$ while the projection map~$y\mapsto ev_0(y)$ is only continuous with respect to~$\Hh_1$. Similarly, the coefficients of the evolution equation~\eqref{eq:OULiftDifferential} belong to~$\Hh_\mu\setminus\Vv_\mu$ but the projection map~$y\mapsto\langle y,1\rangle_{L^2(\mu)}$ is only continuous with respect to~$\Vv_\mu$. In both cases, the issue is solved by resorting to the mild solution theory in the smaller space.

With all necessary notation in place, we are ready to show that \eqref{eq:OULiftIntegral} can be recovered via a linear transformation of the unique $\Hh_1$ mild solution $\lambda$ from Theorem \ref{thm:SPDE_wellposedness} above.

\begin{proposition} Let $X_0, K, \mu, \rho$ and $y_0\in L^2(\mu)$ satisfy \eqref{eq:KLaplaceTransform}-\eqref{eq:OULiftDifferential}. Under Assumptions \ref{assumption:Kernel}, \ref{assumption:SVEassumptions} the following hold:
\begin{enumerate}
    \item $X_0\in\Hh_1$ and, for all $t>0,$ the unique $\Hh_1-$valued mild solution \eqref{eq:SPDE_mild} satisfies $$\lambda^{0, X_0}(t)\in\mathfrak{L}\big(L^1(\RR^+)\big)\;,\;\;\PP-\textnormal{almost surely}.$$
    \item For all $t\in \RR^+$  and $\PP\times\mu-$almost all
    $(\omega, z)\in\Omega\times\textnormal{supp}(\mu),$  the unique $\Hh_\mu-$valued solution~$Y^{0,y_0}$ of \eqref{eq:OULiftIntegral} with $Y_0=y_0$ satisfies
    \begin{equation}\label{eq:OULambdaLink}
        Y^{0,y_0}_t(z)=\tfrac{1}{\rho(z)}\mathfrak{L}^{-1}(\lambda^{0, X_0}(t))(z),
    \end{equation}
    where the inverse Laplace transform $\mathfrak{L}^{-1}$ is applied to the spatial variable. 
\end{enumerate}
\end{proposition}

\begin{proof}

\begin{enumerate}
    \item For any $s\in[0,t], x\in\RR^+$ we have $K(t-s+x)=\mathfrak{L}(e^{-(t-s)\cdot}\rho)(x)$
    and $X_0(t+x)=\mathfrak{L}(e^{-t\cdot}\rho y_0)(x).$ Moreover, by the assumptions, we have $| e^{-(t-s)\cdot}\rho|_{L^1(\RR^+)}=K(t-s)<\infty,$ $ |e^{-t\cdot}\rho y_0|_{L^1(\RR^+)}<\infty.$ Interchanging spatial Laplace transform and temporal integration we get
    \begin{equation}\label{eq:lambdaLaplace}
    \begin{aligned}
        \lambda^{0, X_0}(t,x)&=X_0(t+x)+\int_0^tK(t-s+x)b_0(\lambda(s))\ds+\int_0^tK(t-s+x)\sigma_0(\lambda(s))\dW_s\\&
        =\mathfrak{L}(e^{-t\cdot}\rho y_0)(x)+ \int_0^t\mathfrak{L}(e^{-(t-s)\cdot}\rho)(x)b_0(\lambda(s))\ds+\int_0^t\mathfrak{L}(e^{-(t-s)\cdot}\rho)(x)\sigma_0(\lambda(s))\dW_s\\&
        =\mathfrak{L}\bigg(e^{-t\cdot}\rho y_0+\int_0^t e^{-(t-s)\cdot}\rho b_0(\lambda(s))\ds+\int_0^t e^{-(t-s)\cdot}\rho \sigma_0(\lambda(s))\dW_s\bigg)(x)
         \end{aligned}
    \end{equation}
    the conclusion follows.
    \item By the previous part, $\lambda^{0,X_0}(t,\cdot)$ is in the range of the Laplace transform. Taking inverse Laplace transform in \eqref{eq:lambdaLaplace} and dividing by the strictly positive density $\rho$ we obtain 
    \begin{equation*}
    \begin{aligned}
           \tfrac{1}{\rho(z)}\mathfrak{L}^{-1}(\lambda^{0, X_0}(t))(z)=e^{-tz} y_0(z)+\int_0^t e^{-(t-s)z}b_0(\lambda(s))\ds+\int_0^t e^{-(t-s)z} \sigma_0(\lambda(s))\dW_s,
     \end{aligned}
    \end{equation*}
    with $z\in\text{supp}(\mu).$ Setting
    $Y^{0,y_0}_t=\tfrac{1}{\rho}\mathfrak{L}^{-1}(\lambda^{0, X_0}(t))$ and recalling the lift property \eqref{eq:liftproperty} of $\lambda$ it follows that for all $t\in\RR^+,z\in\textnormal{supp}(\mu)$ 
\begin{equation*}
    \begin{aligned}
          Y^{0,y_0}_t(z) =e^{-tz} y_0(z)+\int_0^t e^{-(t-s)z}b(X(s))\ds+\int_0^t e^{-(t-s)z} \sigma(X(s))\dW_s
     \end{aligned}
    \end{equation*}
    almost surely. Finally, integrating with respect to $z$ against $\mu$ yields 
    \begin{equation*}
    \begin{aligned}
          \langle Y^{0, y_0}_t, 1\rangle_{L^2(\mu)} =X_0(t)+\int_0^t K(t-s)b(X(s))\ds+\int_0^t K(t-s) \sigma(X(s))\dW_s=X_t.
     \end{aligned}
    \end{equation*}
    Combining the last two displays we see that $Y^{0, y_0}$ coincides with the unique mild solution \eqref{eq:OULiftIntegral} of \eqref{eq:OULiftDifferential}. The proof is complete.
\end{enumerate}
\end{proof}

\begin{remark} The identity \eqref{eq:OULambdaLink} can be inverted to obtain an expression for $\lambda$ in terms of the OU lift $Y$:
\begin{equation}\label{eq:LambdaOULink}
    \lambda(t,x)=\mathfrak{L}(\rho Y_t)(x)=\int_0^\infty e^{-zx}\rho(z)Y_t(z)\D z=\int_0^\infty e^{-zx}Y_t(z)\D\mu(z),\;\;\forall (t,x)\in\RR^+\times\RR^+.
\end{equation}  
Following the rationale of the current section, it is reasonable to expect that analogues of our singular It\^o formula and backward Kolmogorov equation for $\lambda$ can be obtained for the process $Y$ via \eqref{eq:OULambdaLink}, \eqref{eq:LambdaOULink} and
Theorems \ref{thm:SingularIto}, \ref{thm:backward_equation_singular} from the previous sections. As explained in Section \ref{sec:Intro}, we choose here to work with \eqref{eq:SPDE} since a) it provides a natural choice for mathematical finance applications (in particular stochastic volatility modelling), b) it unifies Musiela's classical SPDE theory and study of singular SVEs under the same mathematical framework, c) it induces fewer structural assumptions ($X_0,K$ do not need to be completely monotone). For these reasons, we refrain from further exploring the connections between the two lifts in the present paper and leave such considerations for future work.
\end{remark}

\subsection{Path-dependent PDEs and the work of Viens and Zhang}\label{sec:comparison_VZ}
The theory of SVE lifts, that the present paper contributes to, mainly benefits from the lift solving a stochastic equation in infinite dimensions. Its solution has better properties than the original Volterra process (Markov, flow, Feller, see section~\ref{Sec:LiftRegularity}) and it is thus desirable to work exclusively in this new space. In parallel to this literature, Viens and Zhang~\cite{viens2019martingale} proposed a complementary, more ``hands-on" approach. Inspired by the work of Dupire~\cite{dupire2019functional}, their motivation was to figure out the representation of the conditional expectation~$\EE[\phi(X_{[0,T]})\lvert \Ff_t]$ with respect to~$\Ff_t$-measurable random variables, where~$X$ solves 
\begin{align}\label{eq:VZ_SVE}
    X_t = x + \int_0^t K(t,r)b(X_r)\dr + \int_0^t K(t,r)\sigma(X_r)\D W_r.
\end{align}
By considering the auxiliary $\Ff_t$-measurable process $\Theta^t=(\Theta^t_s)_{s\ge t}$
\begin{align}\label{eq:VZ_Theta}
    \Theta^t_s:= x + \int_0^t K(s,r)b(X_r)\dr + \int_0^t K(s,r)\sigma(X_r)\D W_r,\quad s\ge t,
\end{align}
and the orthogonal decomposition~$X_s=\Theta^t_s + (X_s-\Theta^t_s)$, they deduce that the conditional expectation can be represented as 
\begin{align}\label{eq:VZ_u}
    u\Big(t,\big(\widehat{\Theta}^t_s\big)_{s\in[0,T]}\Big) = \EE\Big[\phi(X_{[0,T]})\lvert\Ff_t\Big],
\end{align}
where~$\widehat{\Theta}^t_s:=X_s\one_{s\in[0,t)}+\Theta^t_s\one_{s\in[t,T]}$ corresponds to the concatenation of~$X$ and~$\Theta^t$ 
(see \cite{viens2019martingale,wang2022path,bonesini2023rough} for the Gaussian, regular Volterra and singular Volterra cases, respectively). 

Taking advantage of the semimartingality of~$(\Theta_s^t)_{t\in[0,s]}$, the innovative result of Viens and Zhang~\cite[Theorem 3.17]{viens2019martingale} is the following functional Itô formula for regular functionals~$f:[0,T]\times C[0,T]\to\RR$:
\begin{align}
    \D f(t,\widehat{\Theta}^t)& = \left(\partial_t f(t,\widehat{\Theta}^t) + \partial_{\bm{x}} f(t,\widehat{\Theta}^t)(K(\cdot,t)b(X_t)) + \half \partial_{\bm{x}}^2 f(t,\widehat{\Theta}^t)(K(\cdot,t)\sigma(X_t),K(\cdot,t)\sigma(X_t))\right)\dt \nonumber\\
    &\quad + \partial_{\bm{x}} f(t,\widehat{\Theta}^t)(K(\cdot,t)\sigma(X_t))\D W_t.
    \label{eq:VZ_Ito}
\end{align}
This formula applies to both regular and singular kernels, even though in the latter case~$K(\cdot,t)\notin C[t,T]$. Hence, in a similar fashion as our singular derivative (Definition~\ref{dfn:Derivative}), the object~$\partial_{\bm{x}}u$ is defined as the limit of the Fréchet derivatives
\begin{equation}\label{eq:VZ_singularderivative}
    \partial_{\bm{x}}f(t,\bm{x})(K(\cdot,t))=\lim_{\delta\to0} Df(t,\bm{x})(K^\delta(\cdot,t))
\end{equation} 
where~$K^\delta(\cdot,t)$ is a continuous truncation of~$K(\cdot,t)$. Analogously to the vertical derivative of Dupire~\cite{dupire2019functional}, the ``past''~$X_{[0,t]}$ is left untouched and only the semimartingale part~$\Theta_s$ is differentiated. This desirable property comes at a technical price as one has to resort to a more intricate domain for~$f$. This domain allows for a discontinuity at~$t$; denoting~$D[0,T]$ for real-valued càdlàg paths, it reads:
\begin{align}\label{eq:VZ_statespace}
    \Big\{(t,\bm{x})\in[0,T]\times D[0,T] : \bm{x}\lvert_{[t,T]}\in C[t,T] \Big\}.
\end{align}
Before comparing the necessary assumptions for~\eqref{eq:VZ_Ito} to hold and those set forth in this paper, let us discuss the (dis)similarities between~$\Theta^t$ and~$\lambda(t)$. 

The process $\Theta^t$ can be seen as a lift of~$X$ in the space~$C[t,T]$ (even though it is not presented as such) since~$\Theta^t_t=X_t$. It is very close in spirit to the lift~$\lambda$ studied in this paper and in fact we have~$\lambda(t,x)=\Theta^{t}_{t+x}$ for all~$x\in[0,T-t]$. Moreover, the Itô formula for Volterra processes derived in Corollary~\ref{prop:ItoSVE} via the mild Itô formula can also be obtained from~\eqref{eq:VZ_Ito}. This emphasises again that~$\Theta^t$ and~$\lambda(t)$ coincide in a pointwise sense.  However, the dynamics (and pathwise properties) of the two processes differ. This is evident from looking at the $t$-dependence in~\eqref{eq:VZ_Theta} (only in the integral, while $s$ is fixed) versus~\eqref{eq:SPDE_mild} (both in the integral and in the kernel). The consequences of this distinction are fully visible in the associated Itô formulae. 
As a consequence of the mild formulation~\eqref{eq:SPDE_mild}, our singular Itô formula (Theorem~\ref{thm:SingularIto}) features the additional term~$\Dd u(t,\lambda(t))(\partial_x \lambda(t))$ compared to~\eqref{eq:VZ_Ito}. Heuristically, this extra term comes from the differentiation inside the integral since $\partial_x$ is the generator of the shift semigroup~$S$.

The assumptions of Corollary~\ref{prop:ItoSVE} are however lighter than those necessary for~\eqref{eq:VZ_Ito}. In the former, the differentials are all interpreted in the sense of classical Fr\'echet differentiation on Hilbert spaces and the mild It\^o formula holds both in the case of singular and non-singular kernels without additional assumptions on the coefficients. Our convenient framework is tied to SVEs with convolution-type kernels while the results in \cite{viens2019martingale} also hold for more general kernels (we recall nevertheless that we do cover the Mandelbrot-van Ness fractional Brownian motion thanks to its representation with a random initial curve; see Section \ref{subsec:examples}). 
In order to guarantee the convergence~\eqref{eq:VZ_singularderivative}, Viens and Zhang instead impose regularity conditions on~$Df,D^2f$ with respect to the $L^\infty$-norm of the direction and to some rate of decay around the singularity~\cite[Definition 3.16]{viens2019martingale}. To the best of our knowledge, these conditions have only been verified for a few conditional expectations~\eqref{eq:VZ_u}, all pertaining to Gaussian Volterra processes ($b=0,\sigma=1$) (\cite[Theorem 4.1]{viens2019martingale}, \cite[Proposition 2.22]{bonesini2023rough}, \cite[Lemma 4.8]{pannier2023path}). All of these led to the well-posedness of associated backward path-dependent PDEs. The situation where~$X$ follows general Volterra dynamics~\eqref{eq:VZ_SVE} is, at the moment of writing, an open problem. 

In contrast, the framework introduced in this paper is tailor-made to include the conditional expectations
\begin{align*}   u(t,\lambda(t))=\EE\big[\phi(X_T)\lvert\Ff_t\big]. 
\end{align*}
The convenient formulation of the Markovian lift in a Hilbert space paves the way for the singular Itô formula (Theorem \ref{thm:SingularIto}) and the backward PDE (Theorem~\ref{thm:backward_equation_singular}). Furthermore, the choice of the forward curve/Musiela parametrisation avoids the awkward jump at~$t$ in the state space~\eqref{eq:VZ_statespace} but so far does not cover path-dependent payoffs.

\section*{Acknowledgements}
IG acknowledges financial support from the EPSRC grant EP/T032146/1 throughout the period 08/2022-07/2024. AP is grateful for the grant PEPS ``Jeunes chercheurs et jeunes chercheuses" which funded several visits between the two authors spent at Univeristé Paris Cité, Imperial College London and TU Berlin.

\appendix
\section{}\label{sec:appendix}
\subsection{Proof of Lemma \ref{lemma:Hwproperties}}\label{section:Appendix}
We start by showing that $L^2_w(\RR^+;\RR^d)$ is separable. To this end, note that the map 
    $$L_w^2(\RR^+;\RR^d)\ni f\longmapsto T_w(f):=fw^{1/2}\in L^2(\RR^+;\RR^d)$$
    is linear and bounded.  In fact, $T_w$ is a Hilbert space isometry. Indeed, for any $f,g\in L_w^2(\RR^+;\RR^d)  $ we have $\langle T_w(f), T_w(g) \rangle_{L^2}=\langle f, g\rangle_{L^2_w}$ and the inverse $T_w^{-1}:L^2(\RR^+;\RR^d)\rightarrow L^2_w(\RR^+;\RR^d),$ explicitly given by
    $T^{-1}(f)=fw^{-1/2},$ is also a bounded linear map. Thus, since $L^2(\RR^+;\RR^d)$ is a separable Hilbert space, the same is true for $L^2_w(\RR^+;\RR^d).$ Turning to separability of $H^1_w(\RR^+;\RR^d),$ let $J: H^1_w(\RR^+;\RR^d)\rightarrow L^2_w(\RR^+;\RR^d)\times L^2_w(\RR^+;\RR^d)$ be the linear map $ f\mapsto J(f):=(f, f').$ Clearly, $J$ is an isometry and thus $H^1_w(\RR^+;\RR^d)$ can be identified with the closed linear subspace $J(H^1_w(\RR^+;\RR^d))$ of $L^2_w(\RR^+;\RR^d)\times L^2_w(\RR^+;\RR^d).$ As a finite product of separable spaces, the latter is a separable Banach space and the same then follows for $H^1_w(\RR^+;\RR^d).$

    It remains to show that $H^1_w(\RR^+;\RR^d)$ is an RKHS. To this end let $x\in\RR^+, u\in H^1_w(\RR^+;\RR^d).$ Since $u$ is weakly differentiable and its weak derivative $f'\in L^1_{loc}(\RR^+;\RR^d),$ it follows that $u$ has a continuous representative (also denoted by $u$) which is furthermore absolutely continuous on any compact subinterval of $\RR^+.$ Thus for any $ L>x$ and $y\in(x, L)$ we can write
    $$u(x)-u(y)=\int_x^y u'(r)\dr.$$
    Using the Cauchy--Schwarz inequality and the admissibility of $w$ it follows that 
\begin{equation*}
\begin{aligned}
    \abs{u(y)-u(x)} &= \abs{\int_x^y u'(r) \sqrt{\frac{w(r)}{w(r)}} \dr }
\\&\le \left(\int_x^y u'(r)^2 w(r) \dr \right)^\half \left(\int_x^y w(r)^{-1}\dr\right)^\half 
\le |u|_{H^1_w}|w^{-1}|_{L^1(0,L)}.
\end{aligned}
\end{equation*}
By triangle inequality, $$
\abs{u(x)} \leq \abs{u}_{H^1_w} |w^{-1}|_{L^1(0,L)}+|u(y)|$$
which in turn implies
 $$  \abs{u(x)}^2w(y) \leq 2\abs{u}_{H^1_w}^2 |w^{-1}|^2_{L^1(0,L)}w(y)+2|u(y)|^2w(y).
$$
Integrating $y$ over $(x, L)$ we obtain
$$|w|_{L^1(x,L)}\abs{u(x)}^2\leq 2 \abs{u}_{H^1_w}^2 |w^{-1}|^2_{L^1(0,L)}|w|_{L^1(x,L)}+2|u|^2_{L^2_w}
$$
which in turn furnishes the estimate
$$|ev_x(u)|=\abs{u(x)}\leq C_{x}|u|_{H^1_w}
$$
with 
$$C^2_x:=2|w|^{-1}_{L^1(x,L)}\bigg(|w^{-1}|^2_{L^1(0,L)}|w|_{L^1(0,L)}+1\bigg).$$
From this we conclude that $ev_x$ is a bounded linear operator and the proof is complete.

\subsection{Proof of Lemma \ref{lemma:semigroup_pties}}\label{app:shift}

 We shall assume for simplicity that $d=1$ i.e. that $\Hh$ consists of real-valued functions. The vector-valued analogue of this result follows immediately by working componentwise.
\begin{enumerate}
\item We have 
\begin{equation*}
\begin{aligned}
&|S(t)f|^2_{\Hh}=\int_{\RR^+}f^2(s+t)w(s)\ds=\int_{t}^{\infty}f^2(s)w(s-t)\ds\leq \sup_{s\geq t}\frac{w(s-t)}{w(s)}|f|^2_{\Hh}.
\end{aligned}   
\end{equation*}
    \item Let $f\in\Hh$ and $\{\tilde{f}_n
\}_{n\in\NN}\subset \Cc_c^{\infty}(\RR^+)$
a sequence of functions that converges to $fw^{1/2}$ in $L^2(\RR^+).$ The sequence $\{f_n\}_n:=\{\tilde{f}_n w^{-1/2}\}_n\subset \Cc^{\infty}_c(\RR^+)$ converges to $f$ in $\Hh.$ Next let $t_0, t\geq 0$ with $|t-t_0|< 1, \epsilon>0$ and fix  $n=n(\epsilon)$ large enough to guarantee that 
$$ |S(t)(f-f_n)|_{\Hh}+|S(t_0)(f-f_n)|_{\Hh}\leq C_{t,t_0}|f-f_n|_{\Hh}<\epsilon,$$
where we used 1. to obtain the first inequality.
Finally, for a compact set $M\subset\RR^+$ that contains the support of $S(t)\tilde{f}_n-S(t_0)\tilde{f}_n$ we have 
$$ |S(t)f_n-S(t_0)f_n|_{\Hh}= |S(t)\tilde{f}_n-S(t_0)\tilde{f}_n|_{L^2(\RR^+)}\lesssim\sup_{s\in M}|\tilde{f}_n(s+t)-\tilde{f}_n(s+t_0)|     $$
and the right hand side converges to $0$ as $t\rightarrow t_0$ by uniform continuity of $f_n.$ A combination of the last two displays and the triangle inequality yields $|S(t)f-S(t_0)f|_{\Hh}<\epsilon.$ Since $\epsilon$ is arbitrary the proof is complete.
\item Let $\partial_x$ be the realization of the derivative operator on $\Hh.$ For any $f\in Dom(\partial_x)$ we have $\partial_xf\in\Hh$ and $\tfrac{d}{dt}S(t)f= \partial_xS(t)f=S(t)\partial_xf.$ Thus, $Dom(\partial_x)\subset\Hh_1.$ To show the reverse inclusion, let $f\in\Hh_1.$ 
For $t>0$ we have 
\begin{equation*}
\begin{aligned}    
\frac{S(t)f(s)-f(s)}{t}-\partial_x f(s)=\frac{f(s+t)-f(s)}{t}-f'(s)=\int_{0}^{1}[f'(s+rt)-f'(s)]dr.
\end{aligned}
\end{equation*}
An application of Fubini's theorem along with (2) and the dominated convergence theorem yield
\begin{equation*}
\begin{aligned}    
\bigg|\frac{S(t)f-f}{t}-\partial_xf\bigg|^2_{\Hh}\leq \int_{0}^{1}|S(rt)f'-f'\big|_{\Hh}^2dr\longrightarrow 0
\end{aligned}
\end{equation*}
as $t\to 0^{+}.$
Hence $f\in Dom(\partial_x).$ 
\end{enumerate} 
The more general statements with $\Hh_m$ and $\Hh_{m+1}$ are straightforward and we refer the reader to Remark \ref{rem:H1Semigroup} for a few more details. The proof is complete.

\subsection{Proof of Corollary \ref{cor:HmDensity}}\label{app:density}
 We assume again for simplicity that $d=1$ i.e. $\Hh$ consists of real-valued functions (the vector-valued analogue follows by working componentwise). In the proof of Lemma \ref{lemma:semigroup_pties}.2., we constructed, for any $f\in\Hh,$ a sequence $\{f_n\}_{n\in\NN}\subset \Cc_c^\infty(0,\infty)$ that converges to $f$ in $\Hh.$ Thus $\Cc_c^\infty(0,\infty)$ is dense in~$\Hh.$ Since $w\in L^1_{loc}(\RR^+)$ it follows that $\Cc_c^\infty(0,\infty)\subset \Hh_{m}$ and the first conclusion follows. Finally, from Lemma \ref{lemma:semigroup_pties}.3., $\Hh_{m+1}=Dom(\partial_x)$ where $\partial_x$ is the generator of the strongly continuous shift semigroup $\{S(t)\}_{t\geq 0}\subset\mathscr{L}(\Hh_m).$ By virtue of \cite[Corollary 2.5]{pazy2012semigroups}, $\partial_x$ is densely defined, hence $\Hh_{m+1}\subset\Hh_m$ is dense. In fact, from this abstract result we see that for any $m\in\NN$ and $f\in\Hh_m,$ the sequence $\{n\int_0^{1/n}S(t)f\dt\}_{n\in\NN}\subset\Hh_{m+1}$ converges, as $n\to\infty$, to $f$ in the topology of $\Hh_m.$

\subsection{Volterra--Gr\"onwall inequality}\label{appendix:Gronwall}

We recall that the resolvent $r$ for a kernel~$k\in L^1([0,T],\RR^+)$ is defined as the unique solution to
\begin{align*}
    r(t) - k(t) = \int_0^t r(t-s)k(s)\ds,\quad \text{for all  }t\in[0,T],
\end{align*}
or, said more concisely, $r-k=r\star k$.
The construction of $r$, as shown in~\cite[Lemma 2.1]{zhang2010stochastic} for instance, reveals that~$r\in L^1([0,T],\RR^+)$. 
Several forms of the Volterra-Grönwall inequality appeared in the literature, see e.g. \cite[9.8.2]{gripenberg1990volterra} and \cite[Lemma 2.2]{zhang2010stochastic}. For completeness we provide here a proof of Lemma \ref{lemma:Volterra_Gronwall} which is a version tailored to our purposes.
\begin{proof}
    We may write $x=f-g+k\star x$ for some $g\ge0$. The solution to this equation is given by $x=f-g + r\star(f-g)$ and hence
    \begin{align*}
        k\star x 
        = x-(f-g)
        =r\star(f-g).
    \end{align*}
    Therefore, $x\le f + r\star(f-g)\le f+r\star f$, since $r\ge0$.
\end{proof}

\bibliographystyle{alpha}
\nocite{}
\bibliography{bib}
\end{document}